\documentclass{amsart}

\usepackage{amsmath, amssymb, amsthm, amsfonts, marginnote, stmaryrd}

\usepackage{color}

\input xy
\xyoption{all}

\usepackage{hyperref}

\swapnumbers
\numberwithin{equation}{section}

\theoremstyle{plain}
\newtheorem{theorem}[subsubsection]{Theorem}
\newtheorem{lemma}[subsubsection]{Lemma}
\newtheorem{prop}[subsubsection]{Proposition}
\newtheorem{cor}[subsubsection]{Corollary}
\newtheorem{conj}[subsubsection]{Conjecture}

\theoremstyle{definition}
\newtheorem{defn}[subsubsection]{Definition}

\newtheorem{remark}[subsubsection]{Remark}
\newtheorem{exam}[subsubsection]{Example}

\newtheorem{ex}[subsubsection]{Example}

\def\module{\operatorname{-mod}}
\def\bimod{\operatorname{-mod-}}
\def\rmod{\operatorname{mod-}}

\def\sfA{\mathsf{A}}
\def\sfB{\mathsf{B}}

\def\sfQ{\mathsf{Q}}
\def\sfQ{\mathsf{Q}}

\def\sfV{\mathsf{V}}
\def\sfW{\mathsf{W}}
\def\sfX{\mathsf{X}}

\def\sfp{\mathsf{p}}

\def\sfBun{\mathsf{Bun}}
\def\sfHk{\mathsf{Hk}}

\def\AA{\mathbb{A}}
\def\BB{\mathbb{B}}
\def\CC{\mathbb{C}}

\def\EE{\mathbb{E}}

\def\GG{\mathbb{G}}

\def\NN{\mathbb{N}}

\def\PP{\mathbb{P}}

\def\QQ{\mathbb{Q}}

\def\RR{\mathbb{R}}

\def\SS{\mathbb{S}}

\def\UU{\mathbb{U}}

\def\WW{\mathbb{W}}
\def\XX{\mathbb{X}}
\def\YY{\mathbb{Y}}
\def\ZZ{\mathbb{Z}}

\newcommand\cA{\mathcal{A}}

\newcommand\cC{\mathcal{C}}
\newcommand\cD{\mathcal{D}}
\newcommand\cE{\mathcal{E}}
\newcommand\cF{\mathcal{F}}
\newcommand\cG{\mathcal{G}}
\newcommand\cH{\mathcal{H}}

\newcommand\cK{\mathcal{K}}
\newcommand\cL{\mathcal{L}}

\newcommand\cN{\mathcal{N}}
\newcommand\cO{\mathcal{O}}

\newcommand\cQ{\mathcal{Q}}
\newcommand\cR{\mathcal{R}}
\newcommand\cS{\mathcal{S}}

\newcommand\cU{\mathcal{U}}

\newcommand\cW{\mathcal{W}}

\newcommand\cY{\mathcal{Y}}
\newcommand\cZ{\mathcal{Z}}

\def\bI{\mathbf{I}}
\def\bJ{\mathbf{J}}

\def\bP{\mathbf{P}}

\newcommand\frB{\mathfrak{B}}
\newcommand\frC{\mathfrak{C}}

\newcommand\frO{\mathfrak{O}}
\newcommand\frP{\mathfrak{P}}

\newcommand\frR{\mathfrak{R}}

\newcommand\frV{\mathfrak{V}}
\newcommand\frW{\mathfrak{W}}
\newcommand\frX{\mathfrak{X}}
\newcommand\frY{\mathfrak{Y}}
\newcommand\frZ{\mathfrak{Z}}

\newcommand\fra{\mathfrak{a}}
\newcommand\frb{\mathfrak{b}}
\newcommand\frc{\mathfrak{c}}

\newcommand\frg{\mathfrak{g}}

\newcommand\frh{\mathfrak{h}}

\newcommand\frn{\mathfrak{n}}
\newcommand\frp{\mathfrak{p}}
\newcommand\frq{\mathfrak{q}}

\newcommand\tilW{\widetilde{W}}

\newcommand\bc{\textup{bc}}

\newcommand\bimon{\mathit{bimon}}
\newcommand\bub{\mathit{bub}}
\newcommand{\Bun}{\textup{Bun}}

\newcommand\ev{\textup{ev}}

\newcommand\Gal{\textup{Gal}}

\newcommand{\Gr}{\textup{Gr}}

\newcommand{\Hk}{\textup{Hk}}

\newcommand\id{\textup{id}}

\newcommand\inv{\textup{inv}}

\newcommand\Lie{\textup{Lie}}
\newcommand\Loc{\textup{Loc}}

\newcommand{\Perf}{\textup{Perf}}

\newcommand{\Pic}{\textup{Pic}}

\newcommand{\qc}{\textup{QC}}

\newcommand{\Res}{\textup{Res}}

\newcommand\Spec{\textup{Spec }}
\newcommand\Spf{\textup{Spf}}
\newcommand\St{\mathit{St}}

\newcommand{\univ}{\textup{univ}}

\newcommand\Aut{\textup{Aut}}
\newcommand\Hom{\textup{Hom}}
\newcommand\End{\textup{End}}

\newcommand{\Gm}{\GG_m}

\newcommand{\der}{\textup{der}}

\newcommand\xch{\mathbb{X}^*}
\newcommand\xcoch{\mathbb{X}_*}

\renewcommand\a\alpha
\renewcommand\b\beta
\newcommand\g\gamma
\renewcommand\d\delta
\newcommand\D\Delta
\newcommand{\e}{\epsilon}

\renewcommand{\th}{\theta}
\newcommand{\Th}{\Theta}
\newcommand{\ph}{\varphi}
\newcommand{\Sig}{\Sigma}
\newcommand{\s}{\sigma}
\renewcommand{\t}{\tau}

\newcommand{\y}{\eta}
\newcommand{\z}{\zeta}

\renewcommand{\l}{\lambda}
\renewcommand{\L}{\Lambda}
\newcommand{\om}{\omega}
\newcommand{\Om}{\Omega}
\newcommand{\io}{\iota}

\DeclareMathOperator{\Eis}{Eis}

\newcommand{\incl}{\hookrightarrow}
\newcommand{\isom}{\stackrel{\sim}{\to}}
\newcommand{\bij}{\leftrightarrow}

\newcommand{\twtimes}[1]{\times^{#1}}

\newcommand{\wt}[1]{\widetilde{#1}}
\newcommand{\wh}[1]{\widehat{#1}}
\newcommand\quash[1]{}
\newcommand\un{\underline}
\newcommand\ov{\overline}
\newcommand\bs{\backslash}
\newcommand\ot{\otimes}
\newcommand{\op}{\oplus}
\newcommand{\tl}[1]{[\![#1]\!]}
\newcommand{\lr}[1]{(\!(#1)\!)}

\newcommand{\cohog}[2]{\textup{H}^{#1}({#2})}     

\newcommand{\oll}[1]{\overleftarrow{#1}}
\newcommand{\orr}[1]{\overrightarrow{#1}}

\newcommand\pt{\textup{pt}}

\newcommand\vn{\varnothing}
\newcommand\hs{\heartsuit}
\newcommand\ds{\diamondsuit}
\newcommand\xr{\xrightarrow}
\renewcommand\c{\circ}

\newcommand{\quot}{/\hspace{-.25em}/}
\newcommand{\Sph}{\mathit{sph}}

\newcommand{\Wh}{\textup{Wh}}

\newcommand{\Maps}{\textup{Maps}}

\newcommand{\sph}{\mathit{sph}}

\newcommand{\aff}{\mathit{aff}}

\newcommand{\Sh}{\mathit{Sh}}
\newcommand{\Sta}{\mathit{Stk}}
\newcommand{\Sch}{\mathit{Sch}}

\newcommand{\beq}{\begin{equation}}
\newcommand{\eeq}{\end{equation}}

\newcommand{\ssupp}{\mathit{ss}}

\newcommand{\oo}{\infty}

\newcommand{\ol}{\overline}

\newcommand{\Fun}{\operatorname{Fun}}

\newcommand\sss{\subsubsection}

\newcommand\Corr{\operatorname{Corr}}
\newcommand\ver{\mathit{vert}}
\newcommand\hor{\mathit{horiz}}
\newcommand\adm{\mathit{adm}}
\newcommand\inj{\mathit{inj}}

\newcommand\na{\natural}
\renewcommand\r{\rho}
\newcommand\bt{\boxtimes}
\newcommand\bu{\bullet}
\newcommand\mt{\mapsto}
\newcommand\nb{\nabla}
\newcommand\Rad{\operatorname{Rad}}

\newcommand\tGm{\widetilde{\mathbb{G}}_{m}}

\title{Automorphic gluing functor in Betti Geometric Langlands}
\author{David Nadler}
\address{Department of Mathematics\\University of California, Berkeley\\Berkeley, CA  94720-3840}
\email{nadler@math.berkeley.edu}
\author{Zhiwei Yun}
\address{Department of Mathematics, MIT,  77 Massachusetts Ave., Cambridge, MA 02139}
\email{zyun@mit.edu}
\date{\today}
\subjclass[]{}

\dedicatory{}	
\date{\today}

\begin{document}


\maketitle

\begin{abstract}

We study automorphic categories of nilpotent sheaves under  degenerations of smooth curves to nodal Deligne-Mumford curves. Our constructions  realize affine Hecke operators as the result of bubbling projective lines from marked points. We use this to construct a ``gluing functor" from the automorphic category of a nodal
Deligne-Mumford  curve to the automorphic category of a smoothing. 
\end{abstract}

\tableofcontents



\section{Introduction}

This is part of a series of papers devoted to Verlinde formulas in Betti Geometric Langlands~\cite{BN}.
By this, we mean the idea -- developed in conjunction with D. Ben-Zvi -- to study automorphic and spectral categories   associated to smooth curves via degenerations to nodal curves. Our aims are to produce results  for automorphic categories parallel to established results for spectral categories~\cite{BN-spec}. We seek to express  automorphic categories for smooth curves by ``gluing"  automorphic categories for   marked genus zero curves. 
This paper constructs the functors underlying this gluing.

\subsection{First statement of the main result}\label{ss:intro main} Here we state the main result of this paper in a simplified situation in terms of familiar objects in geometric Langlands.  The proof  involves several new constructions that are characteristic of our approach to Betti geometric Langlands, and will be sketched in the subsequent sections of the introduction.

We will work with schemes and stacks over the complex numbers. 
We will  study sheaves of complex vector spaces in the classical topology. All constructions will be derived without further comment. Thus we will write sheaf in place of complex of sheaves, category in place of dg or $\oo$-category, etc.

%
%

\begin{defn}\label{def intro:sep nod deg A1} A {\em separating nodal degeneration over $\AA^{1}$} is a flat family of projective curves $\tau:\frY\to \AA^{1}$   such that:
\begin{enumerate}
\item the total space $\frY$ is a smooth surface;
\item $\tau$ is smooth over $\Gm=\AA^{1}\bs\{0\}\subset \AA^{1}$ with connected fibers;
\item the special fiber $Y=\tau^{-1}(0)$ is scheme-theoretically the union of two smooth transversely intersecting curves  $Y_{-}, Y_{+} \subset \frY$ with nodes $R=Y_{-}\cap Y_{+}$.
\end{enumerate} 
\end{defn}

\begin{remark}
In the main body of the paper, we allow the base of the family $\t$ to be an arbitrary smooth curve. Here in the introduction for simplicity we discuss our results when the base is $\AA^1$.
\end{remark}

Let $G$ be a connected reductive group, $B\subset G$ a Borel subgroup, $N = [B, B]$ its unipotent radical, and $H = B/N$ the universal Cartan. 
Let $G\lr{z}$ be the corresponding loop group, $\bI\subset  G\lr{z}$ the Iwahori, and $\bI^{\c} \subset \bI$ its pro-unipotent radical, so that $H \simeq \bI/\bI^\c$.

On the one hand, let $\Bun_{G,N}(Y_{\pm},R)$ be the moduli stack of $G$-bundles over $Y_{\pm}$ with $N$-reductions along the points $R=Y_{-}\cap Y_{+}$. Let $\cA(Y_{\pm},R)$ be the category of nilpotent sheaves on $\Bun_{G,N}(Y_{\pm},R)$, i.e.~sheaves with  nilpotent singular support. 

Let $\cH$ be the universal monodromic version of the affine Hecke category (see \S\ref{sss:aff Hk}). Its objects are sheaves on the ind-stack $\bI^{\c}\bs G\lr{z}/\bI^{\c}$ that are monodromic under the left and right $H$-actions. For each $y\in R$, the affine Hecke category $\cH$ acts on $\cA(Y_{\pm },R)$  on the left by modifications at $y$. 
The actions evidently commute making $\cA(Y_{\pm},R)$ a left $\cH^{\ot R}$-module category.

We would like to regard $\cA(Y_{- },R)$ as a right  $\cH^{\ot R}$-module category. We do this by using the  monoidal equivalence $\cH\simeq \cH^{op}$ that sends each Wakimoto operator $J_\lambda$ to its conjugate by the long intertwiner $\wh\D(w_{0})$, i.e.,  $\wh\D(w_{0})\star J_{\l}\star \wh\nb(w_{0})$; and  sends each standard object $\wh\D(w)$ in the finite Hecke category to $\wh\D(w_{0}w^{-1}w_{0}^{-1})$. For the source of this anti-equivalence, see Remark \ref{r:two eq H}.

%

On the other hand, consider the moduli  $\Pi:\Bun_{G}(\t) \to \AA^1 $ classifying a point of the base $\AA^1$ and a  $G$-bundle along the corresponding fiber of $\t$.   
Consider the  real analytic family of curves $\tau_{>0} = \tau|_{\RR_{>0}}:\frY_{>0} = \t^{-1}(\RR_{>0}) \to \RR_{>0}$ obtained by restricting to the positive real ray $\RR_{>0} \subset \AA^1(\CC)$.
Let $\cA(\frY_{>0})$ denote the category of nilpotent sheaves on the corresponding real analytic moduli $ \Bun_{G}(\tau_{>0}) =\Pi^{-1}(\RR_{>0})$.

Here the notion of nilpotent sheaves  involves a conic closed Lagrangian in $T^{*}\Bun_{G}(\t)$ which we call the {\em universal nilpotent cone}. It is a family version of the global nilpotent cone. Its construction is of independent interest and will be discussed in more detail in \S\ref{ss:intro univ cone}. Also, in Appendix \ref{app:A}, we develop foundations of sheaf theory on certain real analytic stacks.

The main result of this paper can be stated as follows.

\begin{theorem}[See Theorem \ref{th:bar complex}]\label{thm:main coarse} There is a canonical functor
\begin{equation}
\xymatrix{
\alpha:\cA(Y_-, R) \otimes_{\cH^{\ot R}} \cA(Y_+, R) \ar[r] & \cA(\frY_{>0}).
}
\end{equation}
In particular, for any $\b>0$, with fiber curve denoted by $Y_{\b}=\t^{-1}(\b)$ and automorphic category of nilpotent sheaves $\cA(Y_\beta)$, further restricting from $\Bun_{G}(\t_{>0})$ to $\Bun_{G}(Y_{\b})$ gives a functor
\begin{equation}
\xymatrix{
\alpha_{\b}:\cA(Y_-, R) \otimes_{\cH^{\ot R}} \cA(Y_+, R) \ar[r] & \cA(Y_{\b}).
}
\end{equation}

\end{theorem}

The proof of this result involves the following major steps:
\begin{itemize}
\item Starting from the family of curves $\t:\frY\to \AA^{1}$, we will construct a cosimplicial family of curves $\pi^{[n]}: \frX^{[n]}\to A^{[n]}\cong \AA^{n+1}$  called the {$k$th twisted cosimplicial bubbling of $\t$}. The central fiber $X(0)$ of $\pi^{[0]}$ is obtained from $Y=\t^{-1}(0)$ by turning the nodes to orbifold points with automorphism group $\mu_{k}$. The central fiber $X(n)$ of $\pi^{[n]}$ is the obtained by inserting $n$ copies of $[\PP^{1}/\mu_{k}]$ at each node of $X(0)$.

\item Rewriting $\cA(Y_\pm, R)$ as $\cA(X_{\pm},Q)$, where $X_{\pm}$ are Deligne-Mumford curves obtained from $Y_{\pm}$ by turning the nodes $R$ to orbifold points $Q\cong R\times \BB\mu_{k}$, and $X(0)=X_{-}\cup_{Q}X_{+}$.

\item Rewriting the affine Hecke category $\cH$ as the bubbling Hecke category $\cH^{\bub}$ as in
Theorem~\ref{thm:equiv of hecke cats}. The bubbling Hecke category $\cH^{\bub}$ consists of nilpotent sheaves on the moduli of $G$-bundles on the Deligne-Mumford curve $[\PP^{1}/\mu_{k}]$.

\item Identifying the bar complex calculating $\cA(X_-, Q) \otimes_{\cH^{\bub,\ot R}} \cA(X_+, Q)$ as sheaf categories coming from the moduli of $G$-bundles on the twisted cosimplicial bubbling. More precisely,  the $n$th term in the bar complex of $\cA(X_-, Q) \otimes_{\cH^{\bub,\ot R}} \cA(X_+, Q)$ is the category of nilpotent sheaves on the central fiber $X(n)$ of $\pi^{[n]}$. The maps in the bar complex come as left adjoints to nearby cycles for the families  $\{\pi^{[n]}\}$. The functor $\a$ in the main theorem also comes as the left adjoint to nearby cycles for $\pi^{[0]}$.
\end{itemize}


Below we give more details of these constructions.

\subsection{Deligne-Mumford curves and automorphic sheaves} 

We fix a maximal torus $T\subset B$, with coweight lattice $\XX_*(T)= \Hom(\GG_m, T)$, Weyl group $W = N_G(T)/T$, and extended affine Weyl group $ \tilW  = \XX_*(T) \rtimes W$.

\subsubsection{Bundles on Deligne-Mumford curves}
In the work of Solis~\cite{Solis} on the compactification of the moduli of bundles on nodal degenerations of smooth curves, one is naturally led to consider nodal Deligne-Mumford (DM for short) curves. Inspired by this, we will also use DM curves in the proof of the main theorem.

Let $X$ be a smooth connected projective DM curves $X$. We will assume there is an open dense $X^\circ \subset X$ that is a scheme and whose complement $Q = X \setminus X^\circ$ is finitely many twisted points $x \in X $ (i.e., points with nontrivial automorphism group) each with cyclic group $\mu_{k_x}$ as automorphisms, for some $k_x\in \NN$. Thus the formal neighborhood of $x \in X$ is modeled on $(\Spec \CC\tl{u})/\mu_{k_x}$, where $\zeta\in \mu_{k_x}$ acts by $\zeta\cdot u = \zeta u$. We will write $\un X$ for the   coarse moduli of $X$, and likewise $\un Q\subset \un X$  and $\un x\in \un X$  for the respective coarse moduli of $Q\subset  X$ and  $x\in Q$.

Given a $G$-bundle $\cE$ on $X$, its restriction $\cE|_{x}$ to a twisted point $x \in X$ gives a $G$-bundle on the classifying stack $x \simeq \BB\mu_{k_x}$. A $G$-bundle on $\BB\mu_{k_x}$ is represented by a homomorphism $\wt \y_x:\mu_{k_x} \to G$ up to conjugacy, so up to isomorphism such bundles are parametrized by the set $\YY_{k_x}$ of $\tilW$-orbits on $\frac{1}{k_x} \XX_*(T)$. Each $\wt\y_x\in \frac{1}{k_x} \XX_*(T)$ determines a parahoric subgroup $\bP_{\wt\y_x}\subset G\lr{z}$ in the loop group of $G$ centered at $\un x \in \un X$, and we write $\bP^\circ_{\wt \y_x} \subset  \bP_{\wt \y_x}$ for its pro-unipotent radical and   $L_{\wt \y_x} = \bP_{\wt \y_x}/\bP^\circ_{\wt \y_x}$ for its Levi quotient.

Given a collection $\eta = (\y_x)_{x\in Q}$ with $\y_x\in \YY_{k_x}$, let $\Bun_{G}(X)_\eta$ (resp. $\Bun_{G, 1}(X, Q)_\eta$) denote the moduli of $G$-bundles on $X$ in the class indexed by $\y_x$ at each twisted point $x\in Q$ (resp. together with a rigidification along $Q$). 
Given lifts $\wt \eta = (\wt \y_x)_{x\in Q}$ with $\y_x\in\frac{1}{k_x} \XX_*(T)$, let $\Bun_{G, \bP_{\wt \eta}}(\un X, \un Q)$ (resp. $\Bun_{G, \bP^\circ_{\wt \eta}}(\un X, \un Q)$) denote the moduli of $G$-bundles on $\un X$ with  $\bP_{\wt \y_x}$  (resp. $\bP^\circ_{\wt \y_x}$) level structure at  each point $\un x\in \un Q$.

\begin{prop}[See Proposition \ref{p:Bun on tw}]\label{th:intro Bun tw} Assume  for simplicity that the derived group $G^{\der}$ is simply-connected. Then there is a canonical commutative diagram with horizontal isomorphisms 
\beq\label{eq:twist=para}
\xymatrix{
\ar[d]\Bun_{G, 1}(X, Q)_\eta  \ar[r]^-\sim &  \Bun_{G, \bP^\circ_{\wt \eta}}(\un X, \un Q)\ar[d]\\
\Bun_{G}(X)_\eta \ar[r]^-\sim &  \Bun_{G, \bP_{\wt \eta}}(\un X, \un Q) 
}
\eeq
with the bottom row resulting from the top by quotienting by changes of rigidification along $Q$ on the left hand side and quotienting by the natural $\prod_{x\in Q} L_{\wt \y_x}$-action   on the right hand side.
\end{prop}

Thus for a fixed curve, there is no difference between allowing twisted points or parahoric  structure.

When $G^{\der}$ is not assumed to be simply-connected, $\bP_{\wt\y}$ should be replaced by a slightly larger group $\bP'_{\wt\y}$. In \S\ref{s:Bun tw}, we will prove a more general version of the above isomorphism for nodal DM curves, as well as a family version.


We will be most interested in collections $\eta = (\y_x)$ where each $\y_x$ is {\em strongly regular}, in the sense that its centralizer is a maximal torus of $G$. When $G^{\der}$ is simply-connected we may take $\wt \y_x = \rho^\vee/h$ where $\rho^\vee$ is  half the sum of the positive coroots, and $h$ is the Coxeter number. When $\y_x$ is strongly regular,  the parahoric  $\bP_{\wt\y_x}$ is in fact an Iwahori subgroup so its Levi $L_{\wt \y_x}$ is canonically isomorphic to $H$.

\subsubsection{Automorphic categories of nilpotent sheaves} Fix a collection $\eta = (\y_x)$ with each $\y_x$ strongly regular. We have the moduli stack $\Bun_{G, 1}(X, Q)_\eta$ with its natural  $H^{\un Q}$-action  by changes of rigidification along $Q$.

The cotangent bundle $T^*\Bun_{G, 1}(X, Q)_\eta$ admits a description in terms of Higgs bundles with a corresponding Hitchin system. The zero fiber of the Hitchin map $\cN \subset T^* \Bun_{G, 1}(X, Q)_\eta$ is a closed conic Lagrangian of everywhere nilpotent Higgs fields.  

Let $\Sh (\Bun_{G, 1}(X, Q)_\eta)$ denote the stable category of sheaves on $\Bun_{G, 1}(X, Q)_\eta$. 
To each sheaf 
 $\cF\in \Sh (\Bun_{G, 1}(X, Q)_\eta)$, we can assign its singular support $\ssupp(\cF) \subset T^* \Bun_{G, 1}(X, Q)_\eta$ which is a closed conic coisotropic subset.

We say  $\cF\in \Sh (\Bun_{G, 1}(X, Q)_{\y})$ is a {\em nilpotent sheaf} if its singular support satisfies $\ssupp(\cF) \subset \cN$. This implies $\cF$ is weakly constructible for some stratification of  $ \Sh (\Bun_{G, 1}(X, Q)_\eta)$, but we place no finiteness condition on the stalks or the support of $\cF$. Our main object of study is the automorphic category of nilpotent sheaves to be denoted 
\begin{equation*}
\cA(X, Q) = \Sh_{\cN} (\Bun_{G, 1}(X, Q)_\eta)
\end{equation*}
The natural $H^{\un Q}$-action on   $\Bun_{G, 1}(X, Q)_\eta$ by changes of rigidification provides  a $\Sh_{0}(H)^{\ot \un Q}$-action on $\cA(X, Q)$ where $\Sh_{0}(H)$ denotes the tensor category of sheaves with locally constant cohomology on $H$ under convolution. Note the Fourier dual description  $\Sh_{0}(H) \simeq \qc(H^\vee)$ as quasi-coherent sheaves on the dual torus.

The diagram \eqref{eq:twist=para} implies an equivalence
\begin{equation}
\cA(X, Q)\simeq \cA(\un X, \un Q)
\end{equation}
as $\Sh_{0}(H)^{\ot \un Q}$-modules. We will match the affine Hecke action on the right side with a symmetry on the left given by the bubbling Hecke category.


\subsection{Bubbling Hecke category}

Fix $k\in\NN$. Let $\PP^1$ be the projective line with affine coordinate $x$. Consider the $\mu_k$-action on $\PP^1$ given by $\zeta\cdot x = \zeta x$.  Let ${}_{k}P_{k} = [ \PP^1/\mu_k]$ denote the  quotient DM curve with orbifold points $0_{k}, \infty_{k} \in {}_{k}P_{k}$ each with automorphisms  $\mu_k$.

Let $\eta_{0}\in \YY_{k}$ be strongly regular.   Set $\y_{\infty}=-\y_{0}$ and $\y=(\y_{0},\y_{\infty})$.

\begin{defn}
The bubbling Hecke category is the automorphic category of nilpotent sheaves
\begin{equation*}
\cH^\bub = \cA({}_{k}P_{k}, \{0_{k}, \infty_{k}\}) =   \Sh_{\cN} (\Bun_{G, 1}({}_{k}P_{k}, \{0_{k}, \infty_{k}\})_\eta)
\end{equation*}
\end{defn}

We have the $\Sh_0(H)$-actions on $\cH^\bub$ at $0_{k}, \infty_{k}$ induced by the $H$-actions on $\Bun_{G, 1}({}_{k}P_{k}, \{0_{k}, \infty_{k}\})_\eta$ at $0_{k}, \infty_{k}$.
The inverse map $\iota:H \to H$, $\iota(h)= h^{-1}$ gives an identification  $\iota_!:\Sh_0(H)\to \Sh_0(H)^{\star op} $ with the monoidal opposite.
 We will apply  this at $\infty$ to regard $\cH^\bub$ as a left $\Sh_0(H)$-module at $0$ and a right $\Sh_0(H)$-module at $\infty$, and hence a $\Sh_0(H)$-bimodule.

In \S\ref{ss:bub Hk}, we will construct a family $\pi:\frP\to \AA^1$ of DM curves with the following properties. Its generic restriction $\frP|_{\Gm} \to \Gm$ is isomorphic to the constant DM curve ${}_{k}P_{k}  \times \Gm \to \Gm$. Its special fiber $\frP|_0$ is isomorphic to the nodal DM curve ${}_{k}P_{k} \vee_{\BB\mu_k} {}_{k}P_{k}$ where  $\infty_{k} \simeq \BB\mu_k$ in the first ${}_{k}P_{k}$ is glued to $0_{k}\simeq \BB\mu_k$ in the second ${}_{k}P_{k}$. 

We consider the moduli  of $G$-bundles along the fibers of $\pi$ as a family over $\AA^1$. Nearby cycles of nilpotent sheaves in this  family  admit a left adjoint providing a functor
\beq\label{eq:mult}
\xymatrix{
m = \psi^L :\cH^\bub \otimes_{\Sh_0(H)} \cH^\bub \ar[r] & \cH^\bub 
}\eeq
We also construct analogous families $\pi^{[n]}:\frP^{[n]} \to \AA^{[n]}$, for $[n] = \{0, \ldots, n\}$, with special fiber $\frP^{[n]}|_0$ isomorphic to a chain of $n+1$ components 
${}_{k}P_{k} \vee_{\BB\mu_k} \cdots \vee_{\BB\mu_k} {}_{k}P_{k}$.  
We consider the moduli  of $G$-bundles along the fibers of $\pi^{[n]}$ as a family over $ \AA^{[n]}$.
The left adjoints to iterated nearby cycles in these families  provide associativity constraints on   the multiplication $m$.

\begin{theorem} \label{thm:intro hbub}

The left adjoint $m$ of \eqref{eq:mult}, together with its associativity constraints, makes $\cH^\bub$ into a monoidal category in  $\Sh_0(H)$-bimodules.

\end{theorem}

The point of Theorem~\ref{thm:intro hbub} is not the abstract construction  of the monoidal category $\cH^\bub$, but rather its geometric realization in terms of degenerations of curves. It will allow us to establish results about more general automorphic categories under degenerations of curves. 

In fact, we will identify $\cH^\bub$ as a monoidal category with the universal version of the local affine Hecke category.

\begin{theorem}[See Theorem \ref{th:Hbub}]\label{thm:equiv of hecke cats}
There is a canonical equivalence of monoidal categories $\cH^\bub\simeq \cH$.
\end{theorem}

\begin{remark}\label{r:two eq H}
In fact, as one finds in Theorem \ref{th:Hbub}, there are two canonical equivalences  $\cH^\bub\simeq \cH$. One is suited to studying left modules and the other to right. Here in the introduction, we default to the one compatible with left modules as for example found in Theorem~\ref{th:intro matching actions}.

The difference between the two canonical equivalences  $\cH^\bub\simeq \cH$ gives a monoidal auto-equivalence $\t: \cH\simeq \cH$. Geometrically,  $\t$ results from the pointed $\cH$-bimodule given by nilpotent sheaves on $\Bun_{G, N}(\PP^1, \{0, \infty\})$ pointed by the extension by zero of a universal local system on the open stratum (here the left $\cH$-action  is by modification at $0$ and the right action at $\infty$).  The anti-equivalence $\s:\cH\to \cH^{op}$ mentioned in \S\ref{ss:intro main} is the composition of $\t$ with the inverse on the loop group.
\end{remark}


Before continuing, let us mention some compatibilities with related Hecke categories.

First, the inclusion $G\subset G\lr{z}$ gives a monoidal embedding of the universal finite Hecke category $\cH_f = \Sh_\bimon(N\bs G/N)$ into the affine Hecke category $\cH$. Under the equivalence of Theorem~\ref{thm:equiv of hecke cats}, $\cH_f$ corresponds to a finite bubbling Hecke category $\cH_f^\bub$ inside of the bubbling Hecke category $\cH^\bub$. 
More directly, one can define $\cH^\bub_f$ as the full subcategory of $\cH^\bub$ of sheaves supported on a certain open locus of $\Bun_{G, 1}({}_{k}P_{k}, \{0_{k}, \infty_{k}\})_\eta$. The  multiplication map of $\cH^\bub_f$ transported back to $\cH_f$ can be identified with the left adjoint  to the  ``limit functor of the wonderful degeneration" studied in~\cite{CY}. 

Second, our techniques similarly provide  a  global construction of the spherical Hecke category $\cH_\sph = \Sh(G\tl{z} \bs G\lr{z}/G\tl{z})$  of sheaves on the ind-stack $G\tl{z} \bs G\lr{z}/G\tl{z}$ where $G\tl{z} \subset G\lr{z}$ is the arc subgroup. Here we start with $\PP^1$ instead of $[\PP^{1}/\mu_{k}]$, and define
the bubbling spherical Hecke category to be the automorphic category of sheaves (automatically nilpotent)
\begin{equation*}
\cH_\sph^\bub = \cA(\PP^1) =   \Sh(\Bun_{G}(\PP^1)).
\end{equation*}
Our techniques equip $\cH_\sph^\bub$ with a monoidal structure and provide a monoidal equivalence
$\cH_\sph^\bub \simeq \cH_\sph$. 

We expect one can go further and consider bubbling at all points of $\PP^1$ to establish the following. The monoidal structure on the usual spherical Hecke category $\cH_\sph$ can be upgraded to an $E_3$-structure by combining convolution, giving the usual monoidal or  $E_1$-structure, with a fusion product, giving an additional compatible $E_2$-structure. One can also transport an evident $E_3$-structure from the spectral description of $\cH_\sph$ as coherent sheaves on $G^\vee$-local systems on $\PP^1$.

\begin{conj}
Bubbling at any point of $\PP^1$ equips
$\cH_\sph^\bub$ with the structure of $E_3$-category, and there is an equivalence of $E_3$-categories  
$\cH_\sph^\bub \simeq \cH_\sph$.
\end{conj}

\subsection{Bubbling Hecke actions}

\sss{Bubbling from a smooth point}
Let $X$ be a smooth connected projective DM curve with a single twisted point $x\in X$ with automorphisms $\mu_k$. Set $\eta \in \YY_{k}$ be strongly regular and consider the automorphic category 
\begin{equation*}
\cA(X, \{x\}) = \Sh_{\cN} (\Bun_{G, 1}(X, \{x\})_\eta)
\end{equation*}
with its $\Sh_0(H)$-action at $x_0$ 
 induced by the $H$-action on $\Bun_{G, 1}(X, \{x_0\})_\eta$ at $x_0$.

We construct a family $\pi:\frX\to \AA^1$ of DM curves with the following properties. Its generic restriction $\frX|_{\Gm} \to \Gm$ is isomorphic to the constant DM curve $X  \times \Gm \to \Gm$. Its special fiber $\frX|_0$ is isomorphic to the nodal DM curve ${}_{k}P_{k} \vee_{\BB\mu_k} X$ where  $\infty_{k} \simeq \BB\mu_k$ in the component ${}_{k}P_{k}$ is glued to $x\simeq \BB\mu_k$ in the component $X$.

We consider the moduli stack of $G$-bundles along the fibers of $\pi$ as a family over $\AA^1$. Nearby cycles of nilpotent sheaves in this  family  admit a left adjoint providing a functor
\beq\label{eq:action}
\xymatrix{
a = \psi^L:\cH^\bub \otimes_{\Sh_0(H)} \cA(X, \{x\}) \ar[r] &\cA(X, \{x\})
}\eeq

We also construct analogous families $\pi^{[n]}:\frX^{[n]} \to \AA^{[n]}$, for $[n] = \{0, \ldots, n\}$, with special fiber $\frX^{[n]}|_0$ isomorphic to chains of $n+1$ components 
${}_{k}P_{k} \vee_{\BB\mu_k} \cdots \vee_{\BB\mu_k} {}_{k}P_{k} \vee_{\BB\mu_k} X$.  The case $n=0$ recovers the family $\pi: \frX=\frX^{[0]}\to \AA^{1}$. The collection $\{\frX^{[n]}\}_{n\ge0}$ form a cosimplicial stack, and is a special case of the {\em twisted cosimplicial bubbling} construction in \S\ref{s:bub}. We consider the moduli stacks of $G$-bundles along the fibers of $\pi^{[n]}$ as a family over $ \AA^{[n]}$. The left adjoints to iterated nearby cycles in these families  provide associativity constraints on the original functor $a$ compatible with the multiplication $m$ of the monoidal category $\cH^\bub$.

\begin{theorem} \label{thm:intro hbub mod}
The functor $a$ of \eqref{eq:action}, together with its associativity constraints coming from the cosimplicial bubbling $\{\pi^{[n]}: \frX^{[n]}\to \AA^{[n]}\}$, makes $\cA(X, \{x\})$ into a $\cH^\bub$-module in  $\Sh_0(H)$-modules.
\end{theorem}


The top row of diagram~\eqref{eq:twist=para} gives a canonical isomorphism
 \begin{equation*}
\xymatrix{
\Bun_{G, 1}(X, \{x\})_\eta  \ar[r]^-\sim &  \Bun_{G, N}(\un X, \un x)
}
\end{equation*}
which provides  a canonical equivalence
 \beq\label{eq:aut cat trad form}
\xymatrix{
\cA(X, \{x\}) = \Sh_\cN(\Bun_{G, 1}(X, \{x\})_\eta)  \ar[r]^-\sim &  \Sh_\cN(\Bun_{G, N}(\un X, \un x))=\cA(\un X, \{\un x\})
}
\eeq

\begin{theorem}[See Theorem \ref{th:matching actions}]\label{th:intro matching actions}
The canonical equivalence~\eqref{eq:aut cat trad form} intertwines the $\cH^\bub$-module structure on $\cA(X, \{x\}) $ with the  $\cH$-module structure on $ \cA(\un X, \un x)$ via  the monoidal equivalence   $\cH^\bub\simeq \cH$ of Theorem \ref{thm:equiv of hecke cats}.
\end{theorem}

In the main body of the paper the theorem is proved in the more general setting with multiple twisted points instead of just one.

%


%
%


\sss{Bubbling from nodes}
Let $\t: \frY\to \AA^1$ be a separating nodal degeneration  as in Definition \ref{def intro:sep nod deg A1}. Let $A^{[0]}$ be an affine line, with coordinate $\e$, viewed as a branched cover of $\AA^{1}$ via the $k$th power map $A^{[0]}\to \AA^1$, $\e\mapsto t=\e^k$.  We construct a family $\pi:\frX\to A^{[0]}$ of DM curves with the following properties.  Let $A^{[0],\times}=A^{[0]}\bs \{0\}$.
 The generic restriction $\frX|_{A^{[0],\times}} \to A^{[0],\times}$ is isomorphic 
 to the base-change $\frY \times_{\Gm}A^{[0],\times}$. 
 The  special fiber $\frX|_{0}$ is isomorphic to the 
   nodal DM curve $X_- \vee_{Q} X_+$ where $X_\pm$ is the smooth DM curve with coarse moduli $Y_\pm$ but with twisted nodes $Q \subset X_\pm $ with automorphisms $\mu_k$ in place of the original nodes $R \subset Y_\pm$.  Note $R=\un Q$.
   
As in \S\ref{ss:intro main}, we form the real analytic family $\frY_{>0} = \tau^{-1}(\RR_{>0})$, which is isomorphic to $\frX_{>0}=\pi^{-1}(\RR_{>0})$.    
   
We consider the moduli stack of $G$-bundles along the fibers of $\pi$ as a family over $A^{[0]}$. Nearby cycles of nilpotent sheaves in this  family  admit a left adjoint providing a functor
\beq\label{eq:coequal}
\xymatrix{
a: \cA(X_-, Q) \otimes_{\Sh_0(H)} \cA(X_+, Q) \ar[r] & \cA(\frY_{>0}).
}\eeq

Moreover, for each $n\ge0$, we construct a relative  DM curve $\pi^{[n]}: \frX^{[n]}\to A^{[n]}\cong \AA^{n+1}$. The central fiber of $\pi^{[n]}$ is isomorphic to a chain of  components 
$X_- \vee_{Q} (R\times {}_{k}P_{k}) \vee \cdots  \vee (R\times {}_{k}P_{k}) \vee_{Q} X_+$. Here the gluing of the adjacent copies of $R\times {}_{k}P_{k}$ identifies $R\times \infty_{k}$ of one copy with $R\times 0_{k}$ of the next copy. The case $n=0$ recovers the family $\pi:\frX=\frX^{[0]}\to A^{[0]}$. Again this is the twisted cosimplicial bubbling construction given in \S\ref{s:bub}.



By considering the moduli of $G$-bundles along fibers of $\pi^{[n]}$, and considering nilpotent sheaves on them, we get a semi-simplicial category with functors given by iterated nearby cycles in these families. Passing to  left adjoints to nearby cycles provides the bar resolution calculating the tensor product of the right $\cH^{\bub,\ot R}$-module $\cA(X_-, Q)$ and the left $\cH^{\bub,\ot R}$-module $\cA(X_+, Q)$. We obtain the following result.

\begin{theorem}\label{thm:main}
The functor $a$ of \eqref{eq:coequal} descends to  a functor
\begin{equation}
\xymatrix{
\alpha:\cA(X_-, Q) \otimes_{\cH^{\bub,\ot R}} \cA(X_+, Q) \ar[r] & \cA(\frY_{>0})
}
\end{equation}
\end{theorem}

The construction of the functor $\alpha$ of Theorem~\eqref{thm:main} already has implications for automorphic categories (as  for example in \S\ref{s:low genus} below). We expect it is possible to resolve the following conjecture by establishing a base-change identity for nearby cycles.

\begin{conj}
The functor $\alpha$ of Theorem~\eqref{thm:main} is an equivalence.
\end{conj}

%
%
%
%
%
%
%
%

\begin{remark}
The arguments establishing Theorem~\ref{thm:intro hbub}, Theorem~\ref{thm:intro hbub mod} 
and Theorem~\ref{thm:main} involve several independent ingredients: 1) The (twisted) cosimplicial bubbling construction carried out in \S\ref{s:bub} gives cosimplicial stacks of DM curves  organizing iterated degenerations; 2) In \S\ref{s:diagram}, to pass from a diagram of moduli spaces of $G$-bundles to the diagram of categories of nilpotent sheaves on them, we invoke Gaitsgory-Rozenblyum's criterion~\cite{GR} for extending functors from categories to correspondence categories; 3) To verify this criterion, we use the commuting nearby cycles criterion of~\cite{N} as recalled in Appendix \ref{app:B}; 4) To talk about sheaves on real analytic stacks such as $\Bun_{G}(\t_{>0})$ (where $\t_{>0}: \frY_{>0}\to \RR_{>0}$), we develop the notion of relative real analytic spaces over an algebraic stack and sheaf theory on them in Appendix~\ref{app:A}.
\end{remark}

\subsection{Universal nilpotent cone}\label{ss:intro univ cone}
The categories involved in the statement and proof of Theorem~\ref{thm:main} are categories of nilpotent sheaves on the moduli stack of $G$-bundles along fibers of certain families of curves. We emphasize here that the nilpotent singular support condition is crucial in applying results of \cite{N} on commuting nearby cycles over higher dimensional bases.

To define the notion of nilpotent sheaves that constitute the category $\cA(\frY_{>0})$, we need to specify a conic Lagrangian in the cotangent bundle to the total space $\Bun_{G}(\t_{>0})$ 
that generalizes the global nilpotent cone for a single curve.  We state here the solution to this problem.


Let $S$ be a smooth base scheme. 

Let $\pi:\frX\to S$ be a family of smooth connected projective curves. We could also allow twisted points or marked points but will leave them aside for simplicity.

Let $\Pi: \Bun_G(\pi) \to S$ denote the moduli of relative $G$-bundles. So an $S'$-point of $\Bun_G(\pi)$ classifies a map $S' \to S$ together with a $G$-bundle on the fiber $S' \times_S \frX$.

Consider the relative cotangent bundle $T_\Pi^* \to \Bun_G(\pi)$ defined as the quotient $ T_\Pi^* = T^*\Bun_G(\pi)/\Pi^*(T^*S)$. Then  $T^{*}_{\Pi}$ can be identified with the moduli of Higgs bundles along fibers of $\pi$, and it admits a relative Hitchin system. The zero fiber of the relative Hitchin map gives a closed conic substack $\cN_{\Pi} \subset T^*_\Pi$ that we  refer to as the {\em relative global nilpotent cone}. For any $s\in S$, under the canonical identification $T^*_\Pi|_s \simeq T^*\Bun_G(\frX_s)$, the fiber $\cN_{\Pi}|_s$  is the usual global nilpotent cone for $\Bun_G(\frX_s)$.

To define nilpotent sheaves in the family $\Pi: \Bun_G(\pi) \to S$, we establish the following. 

\begin{theorem}[See Theorem \ref{th:Eis cone closed}, Corollary \ref{c:univ cone non-char} and Proposition \ref{p:univ cone base change}]
There exists a  canonically constructed closed conic Lagrangian $\wt\cN_\Pi \subset T^*\Bun_G(\pi)$ such that:
\begin{enumerate}
\item The map $T^*\Bun_G(\pi)\to T^{*}_{\Pi}$ sends $\wt\cN_{\Pi}$ bijectively onto the relative global nilpotent cone $\cN_{\Pi}$;
\item In particular, $\wt\cN_{\Pi}$ is non-characteristic with respect to the map $\Pi$, i.e., the intersection $\wt\cN_\Pi \cap \Pi^*(T^*S)$ is the zero-section of  $T^*\Bun_G(\pi)$;
\item The formation of $\wt\cN_{\Pi}$ commutes with arbitrary base change.
\end{enumerate}
\end{theorem}

We refer to the specified  closed conic Lagrangian $\wt\cN_\Pi \subset T^*\Bun_G(\pi)$ as the {\em universal nilpotent cone}.  The construction of $\wt\cN_{\Pi}$ uses the microlocal geometry of the map $\Bun_{B}\to \Bun_{G}$, generalizing an idea of Ginzburg \cite{Gin}.

The notion of the universal nilpotent cone allows us make precise the expectation that the automorphic category $\cA(Y)$ ``varies locally constantly with the curve $Y$'', as predicted by the Betti geometric Langlands conjecture. A special case of this expection can be stated as:
\begin{conj} In the situation of \S\ref{ss:intro main}, the restriction functor to any point $\b\in\RR_{>0}$  is an equivalence of automorphic categories
\begin{equation*}
i^{*}_{\b}: \cA(\frY_{>0})\isom \cA(Y_{\b}).
\end{equation*}
\end{conj}
See Conjecture \ref{c:res eq} for a more general version.

\sss{Hecke preserves nilpotent sheaves} We also prove the key result that sheaves with singular support in the universal nilpotent cone are preserved by Hecke operators. More precisely, consider a base-change $\theta: S'\to S$ with $\pi':\frX' = \frX \times_S S' \to S'$, and  a commutative diagram 
\begin{equation*}
\xymatrix{  & \frX\ar[d]^-{\pi}\\
S'\ar[r]^-\theta\ar[ur]^-{\s} & S}
\end{equation*}
Given a spherical kernel $\cK\in \cH_\sph$ invariant under automorphisms of the disk, there is  a natural Hecke functor
\begin{equation*}
\xymatrix{
H^{\Sph}_{\s,\cK}: \Sh(\Bun_{G}(\pi))\ar[r] &  \Sh(\Bun_{G}(\pi'))
}\end{equation*}

\begin{theorem}[See Theorem \ref{th:sph Hk pres}]\label{th:intro sph Hk pres} The spherical Hecke functor $H^{\Sph}_{\s,\cK}$ preserves universal nilpotent singular support: if a sheaf $\cF$ has singular support in $\wt \cN_\Pi$, then 
the sheaf $H^{\Sph}_{\s,\cK}(\cF)$ has singular support in $\wt \cN_{\Pi'}$. 
A similar statement holds for affine Hecke modifications.
\end{theorem}

\begin{remark} Theorem \ref{th:intro sph Hk pres} implies the invariance of spherical Hecke operators on nilpotent sheaves proved in~\cite[Theorem 5.2.1]{NY}. Namely, for a single curve $X$, the Hecke modification $H_{\cK}: \Sh(\Bun_{G}(X))\to \Sh(X\times \Bun_{G}(X))$ along a moving point in $X$ sends $\Sh_{\cN}(\Bun_{G}(X))$ to $\Sh_{0_{X}\times \cN}(X\times \Bun_{G}(X))$, i.e.~the singular support of the result is contained in the product of the zero section $0_X \subset T^{*}X$ and the nilpotent cone $\cN \subset T^{*}\Bun_{G}(X)$.
To recover this statement, we can apply Theorem \ref{th:intro sph Hk pres} to the constant curve $\frX=X$ over $S=\pt$ and $S'=X, \s=\id_{X}:S'=X\to X$. Note that $\Bun_{G}(\pi')=X\times \Bun_{G}(X)$ and its universal nilpotent cone is $0_X\times \cN$ by the base change property in Proposition \ref{p:univ cone base change}. The proof of Theorem \ref{th:intro sph Hk pres} is different from that in \cite{NY} even in this special case. 
\end{remark}


\subsection{Further directions}
We briefly mention further results in the series~\cite{NY-comp}, \cite{NY-2pt} devoted to Verlinde formulas.

\subsubsection{Compatibilities}
The functor $a$ of \eqref{eq:coequal} and thus the functor $\alpha$ of Theorem~\eqref{thm:main} satisfy essential  compatibilities with other structures.
For example, we show in this paper that $\alpha$ is compatible with spherical Hecke operators along sections of the family  $\pi:\frX\to A^{[0]}$ and similarly affine Hecke operators acting along sections of marked points with Iwahori level structure. 

In the  paper~\cite{NY-comp}, we show  $\alpha$ is compatible with parabolic induction and Whittaker normalization. (Both of these statements are to be expected as they are satisfied under the spectral gluing of~\cite{BN-spec}.)
To  informally state this, we write $\Eis^{\lambda}$ for the universal Eisenstein series of twist $\lambda \in \XX_*(H)$ obtained by induction of the universal local system on $\Bun_H^\lambda$. We write $\Wh$ for the nilpotent Whittaker object, i.e.~the nilpotent sheaf co-representing the functor of ``taking the first Whittaker coefficient''. 

\begin{theorem}[\cite{NY-comp}]
\begin{enumerate}
\item The functor $a$ of  \eqref{eq:coequal} maps the tensor of Eisenstein series $\Eis^{\lambda_-} \otimes \Eis^{\lambda_+}$ to the Eisenstein series $\Eis^{\lambda_- + \lambda_+}$, for any $\l_{-},\l_{+}\in \xcoch(H)$. 

\item The functor $a$ of  \eqref{eq:coequal} maps the tensor of Whittaker objects $\Wh_- \otimes \Wh_+$ to the extended Whittaker object $\Wh \otimes_{\cO(H^\vee/W)^{\otimes R}} \cO(H^\vee)^{\otimes R}$, where the  $ \cO(H^\vee/W)^{\otimes R}$-action on $\Wh$ comes from monodromy around the vanishing cycles of the degeneration $\t:\frY\to \AA^{1}$.
\end{enumerate}
\end{theorem}

\subsubsection{Langlands duality for the universal affine Hecke category}

In Theorem~\ref{thm:equiv of hecke cats}, we have identified the bubbling Hecke category $\cH^\bub$ 
with the universal affine Hecke category $\cH$.
This is an identification of monoidal categories in bimodules for $\Sh_0(H) \simeq \qc(H^\vee)$. 

If we constrain the $H$-monodromies to be trivial (so supported scheme-theoretically at the identity  $e\in H^{\vee}$) or pro-unipotent (so supported topologically at   $e\in  H^\vee$), then Bezrukavnikov's tamely ramified local Langlands equivalence \cite{Be} applies. We plan to show in~\cite{NY-2pt} that it is possible to establish such duality without specifying the monodromy, as stated in the following conjecture.

\begin{conj}\label{thm:univ aff hecke duality}
There is an equivalence of monoidal categories 
\begin{equation*}
\cH  \simeq \qc^!((\wt G^\vee \times_{G^\vee} \wt G^\vee)/G^\vee)
\end{equation*}
extending Bezrukavnikov's  equivalence.
\end{conj}

\subsubsection{Betti Geometric Langlands in genus one}\label{s:low genus}
A primary motivation for the results of this paper is the potential application to low genus situations (genus one or genus zero with marked points). 
Using a degeneration of a genus one curve to a nodal pair of rational curves, our main result implies the following.

\begin{cor}[of Theorem \ref{thm:main coarse}]\label{cor:glue in genus 1}
There is a canonical map from the cocenter of the universal affine Hecke category to the Betti automorphic category of   
a genus one curve $E$:
\begin{equation*}
\xymatrix{
hh(\cH)=\cH \otimes_{\cH \ot \cH^{op}} \cH \ar[r] & \cA(E) 
}
\end{equation*}
\end{cor}

Combining this with Conjecture~\ref{thm:univ aff hecke duality} and results of \cite{BNP-hecke}, one expects to deduce the following.

\begin{conj} \label{thm:genus one} The map of Corollary~\ref{cor:glue in genus 1} is an equivalence. Hence
the Betti geometric Langlands equivalence holds for the genus one curve $E$:
\begin{equation*}
\cA(E) \simeq \qc^!_\cN(\Loc_{G^\vee}(E)) 
\end{equation*}
\end{conj}


\subsection{Acknowledgements}
The motivation for this paper, constructing Verlinde formulas in Betti Geometric Langlands,  is the result of many  discussions with David Ben-Zvi. We are grateful to him for his many insights. We thank Martin Olsson for illuminating discussions about twisted curves, and Nick Rozenblyum for advice on homotopical algebra.

DN was partially supported by NSF grant DMS-1802373. ZY was partially supported by the Packard Fellowship and the Simons Investigator grant.


\section{Bubbling of a nodal degeneration}\label{s:bub}


In this section, starting with a separating nodal degeneration $\t: \frY\to C$ of curves, we construct a cosimplicial scheme $\{\t^{[n]}:\frY^{[n]}\to C^{[n]}\}_{n\ge0}$ called the ``cosimplicial bubbling '' of $\t$ such that the special fibers of $\t^{[n]}$ are further degenerations of the special fiber of $\frY$ with several rational components added. We also construct a twisted version of $\t^{[n]}$ by making the nodes of its fibers into orbifold points.

The construction of the cosimplicial bubbling is basis of our approach to the Betti geometric Langlands conjecture.

\subsection{The initial family}

\begin{defn}\label{def:base}
By a {base} we mean a triple $(C,g, c_{0})$ where $C$ is a smooth  curve, $g:C\to \AA^1$ an \'etale morphism   such that $g^{-1}(0)=\{c_{0}\}$ is a single point.  We denote $C^{\times}:=C\bs \{c_{0}\}$.
\end{defn}

The basic example of a  base is $C = \AA^1$ with $g=\id$ and $c_{0}=0$.

%
%

\begin{defn}\label{def:sep nodal curve} Fix a base $(C, g, c_{0})$.

A {\em separating nodal degeneration} is  a flat  projective map  $\tau:\frY\to C$  
  such that:
\begin{enumerate}
\item $\frY$ is a smooth surface.
\item $\tau$ is smooth over $C^{\times}=C \setminus \{c_0\}$ with connected fibers.
\item The special fiber $Y=\tau^{-1}(c_0)$ is scheme-theoretically the union of two smooth transversely intersecting curves designated $Y_{-}, Y_{+} \subset \frY$.
\end{enumerate} 

Note that we do not assume $Y_{-}$ or $Y_{+}$ is connected. We write $R=Y_{-}\cap Y_{+}$ for the set of nodes of $Y$. When we view $R$ as a subscheme of $Y_{-}$ (resp. $Y_{+}$) we denote it by $R_{-}$ (resp. $R_{+}$).
\end{defn}

\begin{remark}
In the definition, we are free to choose which component of the special fiber $Y=\tau^{-1}(c_0)$ is designated $Y_-$ and which $Y_+$. We regard this choice as part of the structure of the separating nodal degeneration, and can also consider the alternative separating nodal degeneration with the choice swapped.
\end{remark}

\begin{ex}\label{ex: nodal degen}
Let $Z$ be a smooth connected projective curve with a finite set of points $R_{0}\subset Z$. Let $\frY$ be the blowup of $Z\times \AA^{1}$ at the points $(z,0)$ for $z\in R_{0}$. The resulting projection $\tau :\frY\to \AA^1$ is a separating nodal degeneration with a canonical identification of the general fiber $\frY|_{\GG_m} \simeq Z \times \GG_m$. The special fiber $\frY|_0 = Y$ has components $Y_-\simeq Z$ and $Y_+\simeq \coprod_{z\in R_{0}} \PP^1$ (the union of the exceptional divisors corresponding to $z\in R_{0}$)  intersecting transversely in the nodes $R$ which are in bijection with $R_{0}$. 
\end{ex}


\sss{Reformulation of input data}\label{sss:YA2}
We will make frequent use of the following constructions. Let $t$ denote the coordinate of $\AA^1$ and also its pullback under $g\circ \tau:\frY\to \AA^1$. 
Since $c_0 = g^{-1}(0)$, and $Y = \tau^{-1}(c_0) = Y_+ \cup Y_- $,  we see $Y_{+}+Y_{-}$ is the principal divisor on $\frY$ of the regular function 
 $t = g \circ \tau$.
 Thus we have a canonical isomorphism $\cO_{\frY}(Y_{+})\ot \cO_{\frY}(Y_{-})\simeq \cO_{\frY}$. Throughout what follows, we denote the tautological section $1\in \Gamma(\frY,\cO_{\frY}(Y_{+}))$ by $x$ (resp. $1\in \Gamma(\frY, \cO_{\frY}(Y_{-}))$  by $y$). Then we have the relation $xy=t$ as regular functions on $\frY$. The equation $x=0$ (resp. $y=0$) cuts out the reduced curve $Y_+ \subset \frY$ (resp. $Y_- \subset \frY)$, and the equations $x=0, y=0$ cut out the reduced set of points $R\subset \frY$. 

Recall for any stack $\cY$,  a map $\cY\to [\AA^{1}/\Gm]$ (here $\Gm$ acts on $\AA^{1}$ by dilation) is the same data as a pair $(\cL, a)$ where $\cL$ is a line bundle on $\cY$ and $a$ is a section of $\cL$. Applying this observation to $\frY$, the pair $(\cO_{\frY}(Y_{+}),x)$ gives a map $a_{+}: \frY\to [\AA^{1}/\Gm]$, and the pair $(\cO_{\frY}(Y_{-}),y)$ gives a map $a_{-}: \frY\to [\AA^{1}/\Gm]$. The canonical trivialization of $\cO_{\frY}(Y_{+})\ot \cO_{\frY}(Y_{-})$ given by $t$ gives a canonical lifting of the map $(a_{+},a_{-}):\frY\to [\AA^{1}/\Gm]^{2}$ to a map
\begin{equation*}
\xymatrix{
\fra: \frY\ar[r] &  [\AA^{2}/\Gm]
}
\end{equation*}
where $\Gm$ acts on $(u,v)\in \AA^{2}$ by $\l\cdot (u,v)=(\l u, \l^{-1} v)$.  Note that $\fra^{-1}(0\times \AA^{1})=Y_{+}$ and $\fra^{-1}(\AA^{1}\times 0)=Y_{-}$. We have a commutative diagram
\begin{equation}\label{YA2}
\xymatrix{ \frY\ar[r]^-{\fra}\ar[d]^{\t} & [\AA^{2}/\Gm]\ar[d]^{(u,v)\mapsto uv}\\
C\ar[r]^-{g} & \AA^{1}
}
\end{equation}

It is easy to check the following.

\begin{lemma}
Given a base $(C,g,c_{0})$ and a flat projective map $\t: \frY\to C$ of relative dimension $1$, $\t$ is a separating nodal degeneration in the sense of Definition \ref{def:sep nodal curve} if and only if it can be completed into a commutative diagram \eqref{YA2} such that $\fra$ is smooth.
\end{lemma}

\begin{defn}\label{def:rig sep nodal curve} A {\em rigidification} of a separating nodal degeneration $\tau:\frY\to C$  is a trivialization of the cotangent line $\om_{Y_-}|_{r}$ for each node $r\in R$. 
\end{defn}

Note that
\begin{equation*}
\cO(Y_+)|_R\simeq (\cO(-Y_-)|_{Y+})|_R\simeq \cO_{Y_{+}}(-R)|_{R}\simeq \om_{Y_-}|_R.
\end{equation*}
Therefore the trivialization of $\om_{Y_-}|_R$ gives a lifting of $\fra|_{R}:R\to [\{0\}/\Gm^{\WW}]$ to $R\to \{0\}$.



\subsection{Cosimplicial bubbling}\label{ss:cosim bub}


Starting with a separating nodal degeneration $\t: \frY\to C$,  the goal of this subsection is to construct a sequence of families of curves $\t^{[n]}: \frY^{[n]}\to C^{[n]}$ over higher dimensional bases with more nodes in the special fibers.  This is an elaboration of a construction of Gieseker \cite[Lemma 4.2]{Gies}. 

We first make the construction for the standard example of separating nodal degeneration, namely $\{xy=t\}_{t\in \AA^{1}}$, and then base change to $\frY$ using the diagram \eqref{YA2}. In this case, the detailed construction already appeared in Drinfeld \cite[Section 3.2]{Dr}.

\subsubsection{Cosimplicial preliminaries}\label{sss:cosim prel}

 Let $\D$ be the simplex category  with objects the ordered sets  $[n]=\{0,\cdots, n\}$ for $n\ge 0$, and morphisms $\ph:[m] \to [n]$ order-preserving maps $i \leq j \implies \ph(i)\leq \ph(j)$. 

By definition, a cosimplicial object of a category $\cC$ is a functor $\Delta \to \cC$. We will work with  cosimplicial schemes, cosimplicial group-schemes, and cosimplicial stacks. An action of a cosimplicial group-scheme on a 
cosimplicial scheme is a term-wise action compatible with the cosimplicial maps, and given such an action, the term-wise quotients naturally form a cosimplicial stack.
Note also if $Y \to X$ is a map of schemes, and $X^\Delta \to X$ is a cosimplicial scheme over $X$ regarded as a constant cosimplicial scheme,  then the fiber product $Y \times_X X^\Delta$ is naturally a cosimplicial scheme.

To start, we define $\AA^{\D}$ to be the cosimplicial scheme  sending the ordered set $[n]$ to the affine space
$\AA^{[n]} = \Spec \CC[t_i | i\in [n]] \simeq \AA^{n+1}$ and  an order-preserving  map $\ph: [m]\to [n]$ to the morphism $a_{\ph}: \AA^{[m]}\to \AA^{[n]}$ defined by 
\begin{equation*}
a_{\ph}^{*}t_{j}=\left\{
\begin{array}{ll} 
\prod_{i\in [m], \ph(i)=j}t_{i} & j\in \ph([m])\\
 1 &  j\not \in \ph([m])
\end{array}
\right.
\end{equation*}
 In particular,  an
 order-preserving  inclusion $\ph: [m]\hookrightarrow  [n]$ is sent to the closed embedding  $a_{\ph}: \AA^{[m]}\hookrightarrow \AA^{[n]}$
 given by 
 \begin{equation*}
a_{\ph}^{*}t_{j}=\left\{
\begin{array}{ll} 
t_{i} & j = \ph(i)\\
 1 &  j\not \in \ph([m])
\end{array}
\right.
\end{equation*}
We have the natural map $\Pi^\Delta: \AA^\Delta \to \AA^1$ to the constant cosimplicial affine line defined term-wise by the product of coordinates.
 
Similarly, we define $\Gm^{\D}\subset \AA^{\D}$ to be the cosimplicial group-scheme
sending  $[n]$ to the torus
$\Gm^{[n]} = \Spec \CC[t_i, t_i^{-1} | i\in [n]] \simeq \Gm^{n+1}$ with maps given by the same formulas $a_{\ph}$.
There is a natural $\Gm^{\D}$-action  on $\AA^{\D}$ given term-wise by coordinate-wise scaling.
We have the homomorphism $\Pi^{\D}:\Gm^{\D}\to \Gm$ to the constant cosimplicial multiplicative group
defined term-wise by the product of coordinates.
We define $S\Gm^{\D} \subset \Gm^{\D}$ to be the kernel cosimplicial group-scheme sending 
$[n]$ to the torus $S\Gm^{[n]}=\ker(\Pi^{[n]}:\Gm^{[n]}\to \Gm)$. Note  $S\Gm^{[n]}$ is non-canonically isomorphic to $\Gm^{n}$, and in particular $S\Gm^{[0]}$ is the trivial group.

\sss{Standard bubbling}\label{sss:std bub}
Let $(\AA^{1}, \id, 0)$ be the standard base, and $\t_{\WW}: \WW=\AA^{2}\to \AA^{1}$ be the map $(x,y)\mapsto t=xy$. Then $\t_{\WW}$ satisfies the requirement for being a separating nodal degeneration except that $\t_{\WW}$ is not proper. We will construct a toric variety $\WW^{[n]}$ by explicitly gluing coordinate charts.

Notation: for an interval $[a,b]\subset \ZZ_{\ge0}$, let $t_{[a,b]}$ denote the product $t_{a}t_{a+1}\cdots t_{b}$. 
(If $a=b$, the interval $[a,b]$ is a point and we understand $t_{[a,b]}=t_a$;
if $a>b$, the interval $[a,b]$ is empty and we understand $t_{[a,b]}=1$.)

For $n\ge0$, we begin with  
 \begin{equation*}
 \ov\WW^{[n]}=\WW\times_{\AA^{1}} \AA^{[n]} =\{(x,y,t_{0},\cdots, t_{n})\in \AA^{n+3}|xy=t_{[0,n]}\}.
\end{equation*}
We will construct  $\WW^{[n]}$ by birationally modifying $\ov\WW^{[n]}$.

For $0\le i\le n$, let us define an affine scheme over $\AA^{[n]}$ by 
\begin{equation*}
\UU^{[n]}_{i}=\{(x_{i},y_{i},t_{0},\cdots, t_{n})\in \AA^{n+3}|x_{i}y_{i}=t_{i}\}.
\end{equation*}
Equip $\UU^{[n]}_{i}$ with the map to $\WW=\AA^{2}$ by
\begin{equation}\label{map from Ui}
(x_{i},y_{i},t_{0},\cdots, t_{n})\mapsto (t_{[0,i-1]}x_{i},t_{[i+1,n]}y_{i}).
\end{equation}



The scheme $\WW^{[n]}$ is obtained by gluing $\UU^{[n]}_{0},\UU^{[n]}_{1},\cdots, \UU^{[n]}_{n}$ (over $\AA^{[n]}$). For $0\le i<j\le n$, let $\UU^{[n]}_{i,j}\subset\UU^{[n]}_{i}$ be the open subset defined by inverting $y_{i}$ and $t_{[i+1,j-1]}$, and let $\UU^{[n]}_{j,i}\subset \UU^{[n]}_{j}$ be the open subset defined by inverting $x_{j}$ and $t_{[i+1,j-1]}$. Then the change of coordinates 
\begin{equation*}
\xymatrix{
x_{j}=y^{-1}_{i}t^{-1}_{[i+1,j-1]}  & y_{j}=y_{i}t_{[i+1, j]}
}
\end{equation*} 
defines an isomorphism $
\UU^{[n]}_{i,j}\simeq \UU^{[n]}_{j,i}$  which we use to glue $\UU^{[n]}_{i}$ and $\UU^{[n]}_{j}$. We leave it to the reader to check that these gluings are compatible over triple overlaps so define a scheme $\WW^{[n]}$. Moreover, the maps \eqref{map from Ui} are compatible with the gluing and gives a map
\begin{equation}\label{wn}
w^{[n]}: \WW^{[n]}\to \WW.
\end{equation}
We also get a map
\begin{equation*}
\t^{[n]}_{\WW}: \WW^{[n]}\to\AA^{[n]}
\end{equation*}
by taking the coordinates $(t_{0},\cdots, t_{n})$ on each $\UU^{[n]}_{i}$. When $n=0$, we recover $\WW^{[0]}= \WW$, $\tau^{[0]} _{\WW}= \tau_{\WW}$.

\sss{Toric symmetry on standard bubbling}\label{sss:torus action std bub}
The scheme $\WW^{[n]}$ is toric, but we will only use part of the torus action. Let $\Gm^{\WW}=\Gm$ that acts hyperbolically on $\WW$. Then $\Gm^{\WW}\times S\Gm^{[n]}$ act on $\ov\WW^{[n]}$ where $\Gm^{\WW}$ acts hyperbolically on $\WW$ and $S\Gm^{[n]}$ acts on $\AA^{[n]}$ preserving the product of coordinates. This action has a unique lifting to an action on $\WW^{[n]}$: $(\mu,\l_{0},\cdots, \l_{n})\in \Gm^{\WW}\times S\Gm^{[n]}$ acts on $\UU^{[n]}_{i}$ by $(x_{i},y_{i},t_{0},\cdots, t_{n})\mapsto (\mu\l_{[0,i-1]}^{-1}x_{i},\mu^{-1}\l_{[i+1,n]}^{-1}y_{i},  \l_{0}t_{0},\cdots, \l_{n}t_{n})$.

\sss{Cosimplicial structure on standard bubbling}\label{sss:cosim std bub}
The assignment $\WW^{\D}: [n]\mapsto \WW^{[n]}$ has the structure of a  cosimplicial scheme over $\AA^{\D}$. Indeed, for  any  order-preserving map $\ph: [m]\to [n]$, we construct a map $w_{\ph}: \WW^{[m]}\to \WW^{[n]}$ covering the map $a_{\ph}: \AA^{[m]}\to \AA^{[n]}$ as follows.

For $0\le a\le n$, the indices $i$ such that $\ph(i)=\ph(a)$ form an interval $[a',a'']$. Define a map $f_{a}: \UU^{[m]}_{a}\to \UU^{[n]}_{\ph(a)}$ over $\WW$ by the formula
\begin{equation*}
\xymatrix{
f_{a}^{*}(t_{j})=\prod_{0\le i\le m, \ph(i)=j}t_{i} & 
f_{a}^{*}(x_{\ph(a)})=x_{a}t_{[a',a-1]} &
f_{a}^{*}(y_{\ph(a)})=y_{a}t_{[a+1,a'']}
}
\end{equation*}
We leave it to the reader to check that these maps $\{f_{a}\}_{0\le a\le m}$ are compatible on the overlaps of $\UU^{[m]}_{a}$ so define the desired map $w_{\ph}$.


\subsubsection{General cosimplicial bubbling}\label{sss:cons cosim}
Fix a base $(C,g,c_{0})$. Define the  cosimplicial scheme
\begin{equation*}
C^\Delta = C \times_{\AA^1} \AA^\Delta
\end{equation*}
given by the fiber product of the maps $g:C \to \AA^1$ and $\Pi^\Delta: \AA^\Delta \to \AA^1$. We will also write
 $\Pi^\Delta: C^\Delta \to C$ for the induced projection.
 For $n\geq 0$, we denote elements of $C^{[n]}$ by tuples $(c; t_0, \ldots, t_n)$ where $c\in C$ and $(t_0, \ldots, t_n) \in \AA^{n+1}$ with  $g(c) = t_0 \cdots t_n$. 
 For a non-decreasing  map $\ph: [m]\to [n]$, we denote the corresponding structure morphism by $\a_{\ph}: C^{[m]}\to C^{[n]}$.
Note the 
natural $S\Gm^{\D}$-action  on $\AA^{\D}$ acts along the fibers of $\Pi^\Delta: \AA^\Delta \to \AA^1$, 
and hence 
extends to a  $S\Gm^{\D}$-action  on $C^{\D}$ along the fibers of $\Pi^\Delta: C^\Delta \to C$.

Now let $\tau:\frY \to C$ be a separating nodal degeneration. Recall the diagram \eqref{YA2}. Define the cosimplicial scheme
\begin{equation*}
\frY^{\D}=\frY\times_{[\WW/\Gm^{\WW}]}[\WW^{\D}/\Gm^{\WW}].
\end{equation*}
formed using $\fra: \frY\to [\AA^{2}/\Gm]=[\WW/\Gm^{\WW}]$ and the  $\Gm^{\WW}$-equivariant map $w^{[n]}: \WW^{[n]}\to \WW$ in \eqref{wn}.  The action of $\Gm$ on each $\WW^{[n]}$ comes from the first factor in the $\Gm^{\WW}\times S\Gm^{[n]}$-action on $\WW^{[n]}$ constructed in \S\ref{sss:torus action std bub}. From the construction, $\frY^{[n]}$ is equipped with the map
\begin{equation*}
\t^{[n]}: \frY^{[n]}\to C\times_{\AA^{1}, \Pi^{[n]}\c\t^{[n]}_{\WW}}[\WW^{[n]}/\Gm^{\WW}]\to C\times_{\AA^{1},\Pi^{[n]}}\AA^{[n]}=C^{[n]}.
\end{equation*}
In other words we have a map of cosimplicial schemes
\begin{equation*}
\t^{\D}: \frY^{\D}\to C^{\D}
\end{equation*}
that fit into a term-wise commutative diagram 
\begin{equation}\label{YW}
\xymatrix{ \frY^{\D}\ar[d]^{\t^{\D}}\ar[r]^-{\fra^{\D}} & [\WW^{\D}/\Gm^{\WW}]\ar[d]^{\t_{\WW}^{\D}}\\
C^{\D}\ar[r]^-{g^{\D}} & \AA^{\D}
}
\end{equation}
extending the diagram \eqref{YA2}.

For a non-decreasing  map $\ph: [m]\to [n]$, we denote the corresponding structure morphism by $\y_{\ph}: \frY^{[m]}\to \frY^{[n]}$.

\begin{defn}\label{def:bub} The cosimplicial scheme $\frY^{\D}$ over $C^{\D}$ is called the {\em cosimplicial bubbling} of the separating nodal degeneration $\t:\frY\to C$.
\end{defn}

\sss{Exceptional loci}\label{sss:exc loci}
The preimage of $\frY^{[n]}$ over $\frY\bs R$ is simply the base change $(\frY\bs R)\times_{\AA^{1}}\AA^{[n]}$. Let $\frY^{[n]}_{\t}$ be the complement, i.e., the preimage of $R\subset \frY$ in $\frY^{[n]}$. This is the exceptional locus of the birational map $\frY^{[n]}\to \ov\frY^{[n]}:=\frY\times_{\AA^{1}}\AA^{[n]}$.

Let $\WW^{[n]}_{\t}\subset \WW^{[n]}$ be the preimage of $0\in \WW$. From the definition we get
\begin{equation*}
\frY^{[n]}_{\t}\simeq R\times_{[\{0\}/\Gm^{\WW}]}[\WW^{[n]}_{\t}/\Gm^{\WW}].
\end{equation*}
If $\t:\frY\to C$ is equipped with a rigidification (see Definition \ref{def:rig sep nodal curve}), i.e., the trivialization of $\om_{Y_-}|_R$, hence a lifting of $\fra|_{R}:R\to [\{0\}/\Gm^{\WW}]$ to $R\to \{0\}$. This gives an $S\Gm^{[n]}$-equivariant isomorphism for each $n\ge0$
\begin{equation}\label{exc locus}
\frY^{[n]}_{\t}\simeq R\times \WW^{[n]}_{\t}.
\end{equation}

Next we will give a list of properties enjoyed by the cosimplicial bubbling.  Before that we need some more notation.

\sss{Subspace notation}\label{sss:subspace}
For $n\geq 0$ and $S\subset [n]$, set $\AA^{[n]}_{S} \subset \AA^{[n]}$ 
to be the coordinate subspace  where $t_i = 0$, for $i\notin S$. Set $\AA^{[n],S}=\AA^{[n]}_{S^{c}}$ where $S^{c}=[n]\bs S$. 

For $i\in [n]$, we will also write $i\subset [n]$  for the subset with the single element $i$, and $i^c \subset [n]$ for the complementary subset $[n] \setminus i$. Thus $\AA^{[n],i} \subset \AA^{[n]}$
denotes  the hyperplane cut out by $t_i= 0$, and
 $\AA^{[n]}_{i} \subset \AA^{[n]}$ the line cut out by $t_j= 0$, for $j\neq i$.
 
For $n\geq 1$ and $1\le i\le n$,  we will  write $[i-1, i]\subset [n]$  for the subset $\{i-1, i\}$.
Thus  $\AA^{[n],[i-1, i]} \subset \AA^{[n]}$ denotes the codimension two subspace cut out by $t_{i-1} = t_i = 0$.

For $n\geq 0$ and $S\subset [n]$, set $C^{[n]}_{S}=C \times_{\AA^1} \AA^{[n]}_{S}$ and $C^{[n],S}  = C \times_{\AA^1} \AA^{[n],S} \subset C^{[n]}$  to be the corresponding subspace. Note 
\begin{equation*}
C^{[n], S}  \simeq \{c_0\} \times  \AA^{[n],S}, \quad \mbox{ if }S\ne \vn.
\end{equation*}
So in particular $C^{[n],i} = \{c_0\} \times \AA^{[n],i} \subset C^{[n]}$, and
$C^{[n],[i-1, i]} = \{c_0\} \times A^{[n],[i, i-1]} \subset C^{[n]}$.

Note the $S\Gm^{[n]}$-action  on $\AA^{[n]}$ preserves all of the above subspaces of  $\AA^{[n]}$, and in turn
the $S\Gm^{[n]}$-action  on $C^{[n]}$ preserves all of the above subschemes of  $C^{[n]}$.

For $n\geq 0$ and $S\subset [n]$, 
 set $\frY^{[n]}_{S} = \frY^{[n]}|_{C^{[n]}_{S}}$ and $\frY^{[n],S} = \frY^{[n]}|_{C^{[n],S}}$, and denote by 
 \begin{eqnarray*}
\tau^{[n]}_S :\frY^{[n]}_{S}\to  C^{[n]}_{S}, \quad \tau^{[n],S} :\frY^{[n],S}\to  C^{[n],S}
\end{eqnarray*}
the restrictions of $\t^{[n]}$.
%

For $n\geq 0$,  we denote the central fiber of $\frY^{[n]}$ by
\begin{equation*}
Y(n):= \frY^{[n]}|_{(c_0; 0, \cdots, 0)}.
\end{equation*}

\sss{Properties of the cosimplicial bubbling}\label{sss:prop bub}
Let $\t:\frY\to C$ be a rigidified separating nodal degeneration.

%
%
%
%
%

\begin{enumerate}

\item\label{bub:Y0} $\tau^{[0]}:\frY^{[0]}\to C^{[0]} \simeq C$ is the given separating nodal degeneration $\tau:\frY \to C$.

\item\label{bub:Gm} $\frY^\Delta$ is equipped with an $S\Gm^{\D}$-action, and $\t^{\D}:\frY^{\D}\to C^{\D}$ is $S\Gm^{\D}$-equivariant.

\item\label{bub:inj} The map  $\y_{\ph}:\frY^{[m]} \to \frY^{[n]}$ associated to an order-preserving {\em injection} $\ph:[m]\to [n]$ induces an isomorphism 
\begin{equation*}
\xymatrix{
\frY^{[m]}\ar[r]^-\sim & \frY^{[n]}\times_{C^{[n]}} C^{[m]}
}
\end{equation*}
where the fiber product is of the maps $\tau^{[n]}: \frY^{[n]} \to C^{[n]}$ and $\a_\ph:C^{[m]} \to C^{[n]}$.

\item\label{bub:sec} For $n\geq 0$ and $i\in [n]$, there are $S\Gm^{[n]}$-equivariant disjoint sections of $\tau^{n,i}:\frY^{[n],i} \to C^{[n],i}$ 
indexed by $R$. 
We denote these sections by 
\begin{equation*}
\xymatrix{
\s^{n,i}: R\times  C^{[n],i} \ar[r] & \frY^{[n],i}
}
\end{equation*}
and write $\cR^{[n],i} = \s^{n,i}(R\times C^{[n],i}) \subset \frY^{[n],i}$ for their image.

\item\label{bub:Y+-} For $n\geq 0$, there are $S\Gm^{[n]}$-equivariant closed embeddings 
\begin{equation*}
\xymatrix{
\io^n_{-}: Y_{-}\times C^{[n],0} \ar@{^(->}[r] & \frY^{[n],0}
&
\io^n_{+}: Y_{+}\times C^{[n],n} \ar@{^(->}[r] & \frY^{[n],n }
}
\end{equation*}
such that over $c^{n, 0} = (c_0; 0,1,1,\cdots, 1)\in C^{[n],0}$ (resp. $c^{n,n} = (c_0; 1,1,\cdots, 1,0)\in C^{[n],n}$),  where  \eqref{bub:inj} provides
 isomorphisms 
  \begin{equation*}
  \frY^{[n]}|_{c^{n,0}} \simeq Y \simeq \frY^{[n]}|_{c^{n,n}}\
  \end{equation*}
 the embedding $\io^{n}_-$ (resp. $\io^{n}_+$) is the given embedding of $Y_{-}$  (resp. $Y_{+}$) into $Y$. In particular, $\cR^{[n],0}$ is contained in the image of $\io^n_{-}$, and $\cR^{[n], n}$ is contained in the image of $\io^n_{+}$.

\item\label{bub:Y} 
For $n\geq 1$, and $i\in [n]$ with $i\geq 1$, there is an $S\Gm^{[n]}$-stable closed subscheme $\frZ^{[n],[i-1,i]}\subset \frY^{[n],[i-1,i]}$ containing $\cR^{[n],i-1}|_{C^{[n],[i-1,i]}}$ and $\cR^{[n], i}|_{C^{[n], [i-1,i]}}$, and an $S\Gm^{[n]}$-equivariant isomorphism
\begin{equation*}
\frZ^{[n],[i-1,i]}\simeq R\times \PP^{1}\times C^{[n],[i-1,i]}
\end{equation*}
over $C^{[n],[i-1,i]}$ (where on the right hand side, the action of $(t_{0},\cdots, t_{n})\in S\Gm^{[n]}$ is trivial on the factor $R$, is the scaling of the affine coordinate by $t_{i-1}$ on the factor $\PP^{1}$, and is the coordinate-wise scaling on the factor $C^{[n],[i-1,i]}$). Moreover,  under the above isomorphism, $\cR^{[n], i-1}|_{C^{[n], [i-1,i]}}$corresponds to $R\times \{0\}\times  C^{[n], [i-1,i]}$, and $\cR^{[n], i}|_{C^{[n], [i-1,i]}}$ corresponds to $R\times \{\infty\}\times C^{[n], [i-1,i]}$.

\item\label{bub:fib} For $n\ge1$, and $i \in [n]$ with $i\geq 1$, set $Y(n)_{[i-1,i]}=\frZ^{[n], [i-1,i]}|_{(c_0; 0, \cdots, 0)}\subset Y(n)$, and note  \eqref{bub:Y} provides an isomorphism 
\begin{equation*}
Y(n)_{[i-1,i]}\simeq R\times \PP^{1}
\end{equation*}
Set $Y(n)_{-}=\io^n_{-}(Y_{-}\times \{(c_0; 0, \cdots, 0)\})$, $Y(n)_{+}=\io^n_{+}(Y_{+}\times \{(c_0; 0, \cdots, 0)\})$,  
and $R(n)_{i}=\s^{n, i}(R\times \{(c_0; 0, \cdots, 0)\})$. Then $Y(n)$ is a reduced nodal curve
\begin{equation*}
\xymatrix{
Y(n)\simeq Y(n)_{-}\cup_{R(n)_{0}} Y(n)_{[0,1]}\cup_{R(n)_{1}} \cdots \cup Y(n)_{[n-1,n]}\cup_{R(n)_{n}} Y(n)_{+}
}
\end{equation*}
obtained by gluing $Y(n)_{-}, Y(n)_{[0,1]},\cdots, Y(n)_{[n-1,n]}$ and $Y(n)_{+}$ along the nodes $R(n)_{0},\cdots, R(n)_{n}$ (each $R(n)_{i}\simeq R$).

\item\label{bub:surj} For an 
order-preserving {\em surjection} $\ph:[m] \to [n]$, observe that $\y_{\ph}: \frY^{[m]}\to \frY^{[n]}$ restricts to a map  $c_{\ph}: Y(m)\to Y(n)$. Then $c_{\ph}$ contracts each component $Y(m)_{[i-1,i]}$ of $Y(m)$ such that $\ph(i-1)=\ph(i)$ to the corresponding point of $R(n)_{\ph(i)}\subset Y(n)$. Other components of $Y(m)$ are mapped isomorphically to their image in $Y(n)$.



\end{enumerate}

Below we verify some of the properties listed above, leaving the more obvious ones to the reader.

\sss{Property \eqref{bub:inj}} It is suffices to verify the $\Gm^{\WW}$-equivariant version of the same property for $\WW^{\D}$. When $\ph$ is injective, the  map $f_{a}$ is an isomorphism for each $0 \leq a \leq m$, so $\y_{\ph}$ induces a $\Gm^{\WW}$-equivariant open immersion 
\begin{equation*}
\xymatrix{
\wt w_{\ph}: \WW^{[m]}\ar@{^(->}[r] &  \WW^{[n]}\times_{\AA^{[n]}}\AA^{[m]}
}\end{equation*}
Moreover, over the locus $t_{j}\ne 0$,  $\WW^{[n]}$ is covered by $\{\UU^{[n]}_{j'}\}_{j'\ne j}$. Therefore the restriction
$\WW^{[n]}\times_{\AA^{[n]}}\AA^{[m]}$
 is covered by
$\{\UU^{[n]}_{\ph(a)}\}_{a\in [m]}$, thus  implying $\wt w_{\ph}$ is an isomorphism.

\sss{Property \eqref{bub:sec}}  Fix $n\geq 0$ and $i\in [n]$. Consider the subscheme $\Sig_{i}\subset \UU^{[n]}_{i}\subset \WW^{[n]}$ defined by $x_{i}=0$ and $y_{i}=0$. The projection $\Sig_{i}\to \AA^{[n]}$ is an isomorphism onto the hyperplane $\AA^{[n],i}$. Let $\cR^{[n],i}$ the preimage of $\Sig_{i}$ in $\frY^{[n]}$, then the projection $\cR^{[n],i}\to \frY$ lands in the nodes $R$, and the projection $\cR^{[n],i}\to C^{[n]}$ lands in $C^{[n],i}$. It is easy to see that the induced map $\cR^{[n],i}\to R\times C^{[n],i}$ is an isomorphism.


\sss{Property \eqref{bub:Y+-}} Fix $n\geq 0$. Let $\WW_{-}\subset \WW$ be the subscheme defined by $y=0$, and $\WW_{+}\subset \WW$ defined by $x=0$. Consider the locally closed subscheme of $\ov\WW^{[n]}$ given by $x\ne0, y=0, t_{0}=0, t_{i}\ne0$ ($i>0$), and let $\WW^{[n]}_{-}$ be the closure of its preimage in $\WW^{[n]}$. Concretely, $\WW^{[n]}_{-}$ is contained in $\UU^{[n]}_{0}$ and is cut out by $y_{0}=0$ there. From this description we get an isomorphism $\WW^{[n]}_{-}\simeq \WW_{-}\times \AA^{[n],0}$. Taking preimage in $\frY^{[n]}$ gives the desired embedding $\io^{n}_{-}: Y_{-}\times C^{[n,0]}\simeq \wt\fra^{[n],-1}(\WW^{[n]}_{-})\incl \frY^{[n]}$. The other embedding $\io^n_{+}$ comes by taking the preimage of $\WW^{[n]}_{+}\subset \UU^{[n]}_{n}$ defined by $x_{n}=0$.



\sss{Property \eqref{bub:Y}}  Fix $n\geq 1$ and $i\in [n]$ with $i\geq 1$.
Define $\cZ^{[n],[i-1,i]}\subset \WW^{[n]}$ as the union of the subscheme of $\UU^{[n]}_{i-1}$ defined by $x_{i-1}=0, t_{i-1}=t_{i}=0$, and the subscheme of $\UU^{[n]}_{i}$ defined by $y_{i}=0, t_{i-1}=t_{i}=0$.  We have a canonical isomorphism
\begin{equation*}
\cZ^{[n],[i-1,i]}\simeq \PP^{1}\times \AA^{[n],[i-1,i]}
\end{equation*}
whose projection on $\PP^{1}$ is given by $(x_{i-1},y_{i-1}, t_{0},\cdots, t_{n})\mapsto [1,y_{i-1}]$ on $\cZ^{[n],[i-1,i]}\cap \UU^{[n]}_{i-1}$ and by $(x_{i},y_{i}, t_{0},\cdots, t_{n})\mapsto [x_{i},1]$ on $\cZ^{[n],[i-1,i]}\cap \UU^{[n]}_{i}$.


Let $\frZ^{[n], [i-1,i]}$ be the preimage of $[\cZ^{[n],[i-1,i]}/\Gm^{\WW}]$ in $\frY^{[n]}$. Note that $\frZ^{[n], [i-1,i]}$ lies in the exceptional locus of $\frY^{[n]}$. By \eqref{exc locus}, we get an isomorphism (depending on the trivialization of $\om_{Y_{-}}|_{R}$)
\begin{equation}\label{Ya}
\xymatrix{
\frZ^{[n],[i-1,i]}\simeq R\times \cZ^{[n],[i-1,i]}\simeq R\times \PP^{1}\times \AA^{[n],[i-1,i]} \simeq R\times \PP^{1}\times C^{[n],[i-1,i]}
}
\end{equation}
Under this isomorphism, the images of the sections $\s^{n,i-1}$ and $\s^{n,i}$ restricted to $C^{[n],[i-1,i]}$ are identified with $R\times \{[1,0]\}\times C^{[n],[i-1,i]}$ and $ R\times\{[0,1]\}\times C^{[n],[i-1,i]}$ respectively.

\sss{Property \eqref{bub:fib}}\label{sss:fiber} 
Fix $n\ge1$ and $i \in [n]$ with $i\geq 1$.
We have already constructed the components $Y(n)_{\pm}$ and $Y(n)_{[i-1,i]}$ by taking the special fibers of $\io^{n}_{\pm}$ and $\frZ^{[n],[i-1,i]}$.  By looking at the defining equations of $\UU^{[n]}_{i}$, we observe
\begin{eqnarray*}
\UU^{[n]}_{0}\cap Y(n) &=& Y(n)_{-}\cup_{R(n)_{0}}(Y(n)_{[0,1]}\bs R(n)_{1});\\
\UU^{[n]}_{i}\cap Y(n) &=& (Y(n)_{[i-1,i]}\bs R(n)_{i-1})\cup_{R(n)_{i}} (Y(n)_{[i,i+1]}\bs R(n)_{i+1}); \quad 1\le i\le n-1;\\
\UU^{[n]}_{n}\cap Y(n) &=& (Y(n)_{[n-1,n]}\bs R(n)_{n-1})\cup_{R(n)_{n}} Y(n)_{+}.
\end{eqnarray*}
Property \eqref{bub:fib} follows directly.

This completes the verification of the properties listed in \S\ref{sss:prop bub}. 

\sss{Variant--marked sections}\label{sss:more sections} If the original family $\tau:\frY\to C$  is equipped with a finite collection of disjoint sections $s_{a}: C\to \frY$ (for $a\in \Sigma$) that are necessarily avoiding the nodes $R$, then we obtain
a corresponding collection of disjoint sections $s^\Delta_{a}: C^\Delta\to \frY^\Delta$.
Namely, given such a section $s_{a}$, we first form its base change $\ov s^\Delta_{a}: C^\Delta = C \times_{\AA^1} \AA^\Delta \to \frY \times_C C^\Delta =  \ov\frY^\Delta$, and then take $s^\Delta_{a}$ to be
the proper transform of $\ov s^\Delta_{a}$.

\begin{exam}\label{ex:blowup section}
In the situation of Example \ref{ex: nodal degen}, take the proper transform of $R_{0}\times \AA^{1}$ under the blowup $\frY\to Z\times \AA^{1}$ we get sections $\s_{z}:\AA^{1}\to \frY$ indexed by $z\in R_{0}$. The section $\s_{z}$ meets the $\PP^{1}$ corresponding to $z$ in the special fiber.
\end{exam}

\subsection{Twisting construction}\label{ss:twist}

We recall here the twisting construction for curves that introduces orbifold points.
(We will later consider the effect of the construction on the moduli space of bundles.
We learned such constructions from P.~Solis's work \cite{Solis}.)

A comment about the generality of the constructions we present. We will make assumptions that allow us to work with Zariski as opposed to more general \'etale gluings. We expect such technical distinctions disappear if one works in the equivalent context of log-stacks.

Throughout this section, we fix a positive integer $k$.

\sss{Twisting of the standard bubbling}\label{sss:tw std bub}
We denote by $\wt\WW=\AA^{2}$ with coordinates $(u,v)$,  $\wt\AA^{1}=\AA^{1}$ with coordinate $\e$, and $\wt\t_{\WW}: \wt\WW\to \wt \AA^{1}$ the map $(u,v)\mapsto \e=uv$.  Let $\tGm^{\WW}=\Gm$ act on $\wt\WW$ hyperbolically: $\l\cdot (u,v)=(\l u, \l^{-1} v)$. The reason we add tilde to these objects is that there is a commutative diagram
\begin{equation}\label{wt WW}
\xymatrix{ \wt\WW \ar[d]^{\wt\t_{\WW}}\ar[r]^-{p_{\WW, k}} & \WW\ar[d]^{\t_{\WW}}\\
\wt\AA^{1} \ar[r]^{p_{\AA, k}} & \AA^{1}}
\end{equation}
Here $p_{\WW, k}(u,v)=(u^{k}, v^{k})$ and $p_{\AA, k}(\e)=\e^{k}$. Moreover, $p_{\WW, k}$ is equivariant under the $\tGm^{\WW}$ and $\Gm^{\WW}$ actions via the $k$th power map $\tGm^{\WW}\to \Gm^{\WW}$.

Let $\wt\t^{\D}_{\WW}: \wt\WW^{\D}\to \wt\AA^{\D}$ be the standard bubbling constructed in \S\ref{sss:std bub}, starting with $\wt\t^{[0]}_{\WW}=\wt\t_{\WW}: \wt\WW^{[0]}=\wt\WW\to \wt\AA^{1}=\wt\AA^{[0]}$. Then $S\tGm^{\D}$ acts on $\wt\AA^{\D}$ by coordinate-wise multiplication, and $\tGm^{\WW}\times S\tGm^{\D}$ acts on $\wt\WW^{\D}$.

We have a commutative diagram of cosimplicial schemes extending \eqref{wt WW}
\begin{equation*}
\xymatrix{ \wt\WW^{\D} \ar[d]^{\wt\t^{\D}_{\WW}}\ar[r]^-{p^{\D}_{\WW, k}} & \WW^{\D}\ar[d]^{\t_{\WW}^{\D}}\\
\wt\AA^{\D} \ar[r]^{p^{\D}_{\AA, k}} & \AA^{\D}}
\end{equation*}
Moreover, $p^{\D}_{\AA, k}$ intertwines the $S\tGm^{\D}$-action on $\wt\AA^{\D}$ and the $S\Gm^{\D}$-action on $\AA^{\D}$ via the coordinate-wise $k$th power map $S\tGm^{\D}\to S\Gm^{\D}$. The map $p^{\D}_{\WW, k}$ intertwines the $\tGm^{\WW}\times S\tGm^{\D}$-action on $\wt\WW^{\D}$ and the $\Gm^{\WW}\times S\Gm^{\D}$-action on $\WW^{\D}$ via the coordinate-wise $k$th power map $\tGm^{\WW}\times S\tGm^{\D}\to \Gm^{\WW}\times S\Gm^{\D}$.

\sss{Twisted bubbling}\label{sss:tw bub}
For a rigidified separating nodal degeneration $\t: \frY\to C$, we have constructed the cosimplicial bubbling $\t^{\D}:\frY^{\D}\to C^{\D}$ in \S\ref{sss:cons cosim}. Now define a cosimplicial scheme $\frX^{\D}$ by the Cartesian diagram
\begin{equation*}
\xymatrix{ \frX^{\D}\ar[d]^{p^{\D}_{k}} \ar[r]^-{\wt\fra^{\D}} & [\wt\WW^{\D}/\tGm^{\WW}]\ar[d]^{p^{\D}_{\WW,k}}\\
\frY^{\D} \ar[r]^-{\fra^{\D}} & [\WW^{\D}/\Gm^{\WW}]}
\end{equation*}
Alternatively, we may describe $\frX^{\D}$ as the fiber product
\begin{equation*}
\frX^{\D}\simeq \frY\times_{[\WW/\Gm^{\WW}]}[\wt\WW^{\D}/\tGm^{\WW}]
\end{equation*}
where we use the map $[\wt\WW^{[n]}/\tGm^{\WW}]\to[\wt\WW/\tGm^{\WW}] \xr{p_{\WW,k}} [\WW/\Gm^{\WW}]$ to form the fiber product. 

We define a new cosimplicial base $B^{\D}$ by
\begin{equation*}
B^{\D}=C\times_{\AA^{1}, \Pi_{k}}\wt\AA^{\D}
\end{equation*}
where the map $\Pi_{k}: \AA^{[n]}\to \AA^1$ is given by $(t_0, \ldots, t_n) \mapsto t_0^k \cdots t_n^k$. By construction we have a map of cosimplicial schemes
\begin{equation*}
\pi^{\D}: \frX^{\D}\to B^{\D}
\end{equation*}
covering $\t^{\D}$. 

\begin{defn} \label{def:tw bub} For a fixed positive integer $k$, the cosimplicial scheme $\frX^{\D}$ over $B^{\D}$ is called the {\em $k$-th twisted cosimplicial bubbling} of the separating nodal degeneration $\t:\frY\to C$.
\end{defn}

\sss{New notation}\label{sss:tw new notation}
Denote
\begin{equation*}
\frX:=\frX^{[0]}=\frY\times_{[\WW/\Gm^{\WW}]}[\wt\WW/\tGm^{\WW}], \quad B:=B^{[0]}=C\times_{\AA^{1}}\wt\AA^{1}, \quad \pi=\pi^{[0]}: \frX\to B.
\end{equation*}
Let $b_{0}\in B$ be the unique preimage of $c_{0}\in C$, and $B^{\times}:=B\bs\{b_{0}\}$.

We introduce the notation
\begin{equation}\label{An SG}
A^{[n]}:=\wt\AA^{[n]}, \quad S\Gm^{[n]}:=S\tGm^{[n]}.
\end{equation}
We denote the coordinates on $A^{[n]}$ by $(\e_{0},\cdots, \e_{n})$. For an order-preserving map $\ph:[m]\to [n]$, we denote the corresponding maps in $\frX^{\D}$ and $B^{\D}$ by 
\begin{equation*}
\xi_{\ph}: \frX^{[m]}\to \frX^{[n]}; \quad \b_{\ph}: B^{[m]}\to B^{[n]}.
\end{equation*}

Again we denote a point in $B^{[n]}=C\times_{\AA^{1}, \Pi_{k}}A^{[n]}$ by $(c;\e_{0},\cdots, \e_{n})$ where $c\in C$ and $(\e_{0},\cdots, \e_{n})\in A^{[n]}$ such that $g(c)=\e^{k}_{0}\cdots\e_{n}^{k}$.  In particular, $b_{0}=(c_{0}; 0)$. Denote the special fiber of $\frX\to B$ by 
\begin{equation*}
X := \frX|_{b_{0}}\simeq Y\times_{\AA^{1}}[\wt\WW/\tGm^{\WW}].
\end{equation*}
where $Y$ maps to $0\in \AA^{1}$. Then $X\to  Y$ realizes $Y$ as the coarse moduli space of $X$. We will add an underline to the notation for a stack to denote its coarse moduli. Then
\begin{equation*}
X=X_{-}\cup_{Q} X_{+}, \quad \mbox{ with }\un X_{-}=Y_{-}, \quad \un X_{+}=Y_{+}, \quad \un Q=R.
\end{equation*}
Note that $Q=R\times \BB\mu_{k}$.  


For $n\ge0$, denote the  special fiber of $\frX^{[n]}$  and its twisted nodes by 
\begin{equation*}
\xymatrix{
X(n) = \frX^{[n]}|_{(c_0; 0, \ldots, 0)} & Q(n)_i \subset X(n), \, i \in [n]
}
\end{equation*}
Note that $Q(n)_{i}\simeq Q$.  The map $p^{[n]}_{k}$ induces  isomorphisms on coarse moduli
\begin{equation*}
\xymatrix{
\un \frX^{[n]}  \ar[r]^-\sim & \frY^{[n]} \times_{C^{[n]}} B^{[n]} & 
\un X(n) \ar[r]^-\sim & Y(n) &  \un Q(n)_i  \ar[r]^-\sim & R(n)_i, \, i \in[n].
}\end{equation*}

Denote by $\cQ^{[n],i}\subset \frX^{[n]}$  the preimage of the nodal locus $\cR^{[n],i}\subset \frY^{[n]}$ defined in Property \ref{bub:sec} of \S\ref{sss:prop bub}. 

\sss{Local geometry of $\frX^{[n]}$} To understand the local geometry of $\frX^{[n]}$, we note that $\frY^{[n]}$ and $\WW^{[n]}$ are \'etale locally isomorphic, hence $\frX^{[n]}$ is \'etale locally isomorphic to $[\wt\WW^{[n]}/\mu^{\WW}_{k}]$ which fits into the Cartesian diagram
\begin{equation}\label{W muk}
\xymatrix{ [\wt\WW^{[n]}/\mu^{\WW}_{k}] \ar[d]\ar[r] & [\wt\WW^{[n]}/\tGm^{\WW}]\ar[d]^{p^{[n]}_{\WW,k}}\\
\WW^{[n]} \ar[r] & [\WW^{[n]}/\Gm^{\WW}]
}
\end{equation}
Here $\mu^{\WW}_{k}$ denotes the $k$-torsion in the torus $\Gm^{\WW}$.

For each affine subset $\wt\UU^{[n]}_{i}\subset \wt\WW^{[n]}$ with coordinates $(u_{i},v_{i},\e_{0},\cdots, \e_{n})$ and defining equation $u_{i}v_{i}=\e_{i}$, the action of $\z\in\mu_{k}^{\WW}$ is $\z\cdot (u_{i}, v_{i},\e_{0},\cdots, \e_{n})=(\z u_{i}, \z^{-1} v_{i}, \e_{0},\cdots, \e_n)$. Therefore $\wt\UU^{[n]}_{i}/\mu^{\WW}_{k}$ is a scheme away from the locus $u_{i}=v_{i}=\e_{i}=0$, along which the automorphism group is $\mu_{k}^{\WW}$.  

Transporting the above discussion from $\wt\WW^{[n]}/\mu^{\WW}_{k}$ back to $\frX^{[n]}$ by \'etale local charts, we see that $\frX^{[n]}$ has nontrivial automorphism group $\mu_{k}$ precisely along $\cQ^{[n],i}$  (for all $i\in [n]$), and is a scheme elsewhere.

\sss{Exceptional loci}\label{sss:tw exc loci}
Parallel to the discussion in \S\ref{sss:exc loci}, we denote by $\frX^{[n]}_{\t}$ the preimage of $Q\subset \frX$ in $\frX^{[n]}$. Let $\wt\WW^{[n]}_{\t}\subset \wt\WW^{[n]}$ be the preimage of $0\in \wt\WW$. From the definition we get
\begin{equation*}
\frX^{[n]}_{\t}\simeq R\times_{[\{0\}/\Gm^{\WW}]}[\wt\WW^{[n]}_{\t}/\wt\Gm^{\WW}].
\end{equation*}
Using the rigidification of $\om_{Y_-}|_R$, we get a $S\Gm^{[n]}$-equivariant isomorphism
\begin{equation}\label{tw exc locus}
\frX^{[n]}_{\t}\simeq R\times [\wt\WW^{[n]}_{\t}/\mu^{\WW}_{k}].
\end{equation}

\sss{Properties of twisted bubbling}\label{sss:prop tw bub}
Below we list properties of the twisted cosimplicial bubbling, indexed to match the corresponding properties in \S\ref{sss:prop bub}. We adopt the convention of \S\ref{sss:subspace} for coordinate subspaces such as $B^{[n], S}$ and $\frX^{[n],S}$ ($S\subset [n]$).

\begin{enumerate}

\item[($2'$)]\label{bub:tw Gm} Both $\frX^{\D}$ and $B^{\D}$ are equipped with an action of the cosimplicial torus $S\Gm^{\D}$, and the map $\pi^{\D}$ is $S\Gm^{\D}$-equivariant.

\item[($3'$)]\label{bub:tw inj} The structure map  $\xi_{\ph}:\frX^{[m]} \to \frX^{[n]}$ associated to an order-preserving {\em injection} $\ph:[m]\to [n]$ induces an isomorphism 
\begin{equation*}
\xymatrix{
\frX^{[m]}\ar[r]^-\sim & \frX^{[n]}\times_{B^{[n]}} B^{[m]}
}
\end{equation*}

\item[($4'$)]\label{bub:tw sec} For $n\geq 0$ and $i\in [n]$, there are $S\Gm^{[n]}$-equivariant closed embeddings over $B^{[n],i}$
\begin{equation*}
\xymatrix{
\s^{n}_{i}: Q\times  B^{[n],i} \ar[r] & \frX^{[n], i}
}
\end{equation*}
with image $\cQ^{[n], i}$. The fiber of $\cQ^{[n],i}$ over $(c_{0};0,\cdots, 0)$ is $Q(n)_{i}\subset X(n)$.

\item[($5'$)]\label{bub:tw X+-} For $n\geq 0$, there are $S\Gm^{[n]}$-equivariant closed embeddings 
\begin{equation*}
\xymatrix{
\io^n_{-}: X_{-}\times B^{[n],0} \ar@{^(->}[r] & \frX^{[n],0}
&
\io^n_{+}: X_{+}\times B^{[n],n} \ar@{^(->}[r] & \frX^{[n],n }
}
\end{equation*}
such that over $b^{n,0} = (c_0; 0,1,1,\cdots, 1)\in B^{[n], 0}$ (resp. $b^{n, n} = (c_0; 1,1,\cdots, 1,0)\in B^{[n], n}$),  where  \eqref{bub:tw inj} provides
 isomorphisms 
  \begin{equation*}
  \frX^{[n]}|_{b^{n,0}} \simeq X \simeq \frX^{[n]}|_{b^{n,n}}\
  \end{equation*}
 the embedding $\io^{n}_-$ (resp. $\io^{n}_+$) is the given embedding of $X_{-}$  (resp. $X_{+}$) into $X$. In particular, $\cQ^{[n],0}$ is contained in the image of $\io^n_{-}$, and $\cQ^{[n], n}$ is contained in the image of $\io^n_{+}$.

\item[($6'$)]\label{bub:tw X} 
For $n\geq 1$, and $i\in [n]$ with $i\geq 1$, there is an $S\Gm^{[n]}$-stable closed substack $\frZ^{[n],[i-1,i]}\subset \frX^{[n],[i-1,i]}$ containing $\cQ^{[n],i-1}|_{B^{[n],[i-1,i]}}$ and $\cQ^{[n], i}|_{B^{[n], [i-1,i]}}$, and an $S\Gm^{[n]}$-equivariant isomorphism
\begin{equation*}
\frZ^{[n],[i-1,i]}\simeq  \un Q\times (\PP^{1}/\mu_k) \times B^{[n],[i-1,i]}
\end{equation*}
over $B^{[n],[i-1,i]}$ (where on the right hand side, the action of $(t_{0},\cdots, t_{n})\in S\Gm^{[n]}$ is trivial on the factor $\un Q$, is the scaling of the affine coordinate by $t_{i-1}$ on the factor $\PP^{1}/\mu_k$, and is the coordinate-wise scaling on the factor $B^{[n],[i-1,i]}$). Moreover,  under the above isomorphism, $\cQ^{[n], i-1}|_{B^{[n], [i-1,i]}}$corresponds to $\un Q\times \{0\}/\mu_k\times  B^{[n], [i-1,i]}$, and $\cQ^{[n], i}|_{B^{[n], [i-1,i]}}$ corresponds to $\un Q\times \{\infty\}/\mu_k\times B^{[n], [i-1,i]}$.

\item[($7'$)]\label{bub:tw fib} For $n\ge1$, and $i \in [n]$ with $i\geq 1$, set $X(n)_{ [i-1,i]}=\frZ^{[n], [i-1,i]}|_{(c_0; 0, \cdots, 0)}$, and note  \eqref{bub:tw X} provides an isomorphism 
\begin{equation*}
X(n)_{[i-1,i]}\simeq \un Q\times (\PP^{1}/\mu_k)
\end{equation*}
Set $X(n)_{-}=\io^n_{-}(X_{-}\times \{(c_0; 0, \cdots, 0)\})$ and $X(n)_{+}=\io^n_{+}(X_{n+}\times \{(c_0; 0, \cdots, 0)\})$. Then $X(n)$ is a twisted nodal curve
\begin{equation*}
\xymatrix{
X(n)\simeq X(n)_{-}\cup_{Q(n)_{ 0}} X(n)_{ [0,1]}\cup_{Q(n)_{1}} \cdots \cup X(n)_{ [n-1,n]}\cup_{Q(n)_{n}} X(n)_{+}
}
\end{equation*}
obtained by gluing $X(n)_{-}, X(n)_{ [0,1]},\cdots, X(n)_{[n-1,n]}$ and $X(n)_{+}$ along the twisted nodes $Q(n)_{0},\cdots, Q(n)_{ n}$. Moreover, \'etale locally near each twisted node, $X(n)$ is isomorphic to the quotient stack $(\Spec \CC[u,v]/(uv))/\mu_{k}$ where $\z\in\mu_{k}$ acts on $\CC[u,v]/(uv)$ by $(u,v)\mapsto (\z u,\z^{-1}v)$.
  
\item[($8'$)]\label{bub:tw surj} For an 
order-preserving {\em surjection} $\ph:[m] \to [n]$,  $\xi_{\ph}: \frX^{[m]}\to \frX^{[n]}$ restricts to a map  $c_{\ph}: X(m)\to X(n)$. Then $c_{\ph}$ contracts each component $X(m)_{[i-1,i]}$ of $X(m)$ such that $\ph(i-1)=\ph(i)$ to the corresponding point of $Q(n)_{\ph(i)}\subset X(n)$. Other components of $X(m)$ are mapped isomorphically to their image in $X(n)$. 

%
\end{enumerate}

\sss{Variant -- marked sections}\label{sss:more sections tw} If we have additional disjoint sections $s_{a}: C\to \frY$ ($a\in \Sigma$) as in \S\ref{sss:more sections}, the sections $s^{\D}_{a}: C^{\D}\to \frY^{\D}$ induce sections $\s^{\D}_{a}: B^{\D}\to \frX^{\D}$ that land in the smooth loci of fibers.

\section{Bundles on twisted curves}\label{s:Bun tw}
The fibers of the twisted bubbling of a separating nodal degeneration $\t: \frY\to C$ are DM curves with twisted nodes  (i.e., nodes with automorphism groups). In this section, we describe the moduli stack of $G$-bundles on such twisted nodal curves in terms of the moduli stack of $G$-bundles on its coarse curve with level structures.

\subsection{Smooth case}

\sss{$G$-bundles over $\BB\mu_{k}$}\label{sss:Yk} 
Let $k\in \NN$. Let $\YY_{k}$ be the set of $G$-conjugacy classes of homomorphisms $\y: \mu_{k}\to G$.  If we choose a maximal torus $T\subset G$ with Weyl group $W$, then any $\y:\mu_{k}\to G$ can be conjugated to some $\wt\y: \mu_{k}\to T$, which is classified by an element in $\frac{1}{k}\xcoch(T)/\xcoch(T)$. Therefore we can identify $\YY_{k}$ with the set of orbits of $\frac{1}{k}\xcoch(T)$ under the action of the extended affine Weyl group $\tilW=\xcoch(T)\rtimes W$:
\begin{equation*}
\YY_{k}\simeq \frac{1}{k}\xcoch(T)/\xcoch(T)\rtimes W.
\end{equation*}

Let $\cY_{k}:=\Bun_{G}(\BB \mu_{k})$ be the moduli stack classifying $G$-bundles over $\BB\mu_{k}=[\pt/\mu_{k}]$. Then
\begin{equation*}
\cY_{k}=\un\Hom(\mu_{k}, G)/G
\end{equation*}
where the action of $G$ on the homomorphism scheme $\un\Hom(\mu_{k},G)$ is by conjugation. Therefore the isomorphism classes of points in $\cY_{k}$ are indexed by $\YY_{k}$. For $\y\in \YY_{k}$, let $\cY_{k,\y}\subset \cY_{k}$ be the corresponding component. If we choose a representative $\wt\y: \mu_{k}\to G$ for $\y$, then $\cY_{k,\y}\simeq \BB C_{G}(\wt\y)$.

\sss{Smooth DM curve}\label{sss:DM X} Let $X$ be a smooth DM curve with isolated orbifold points indexed by $R$. We identify $R$ with a set of closed points in the coarse curve $\un X$. For $p\in R$,  the automorphism group $\Aut_{X}(p)$ acts on the tangent line $T_{p}X$. This action  gives a {\em canonical} isomorphism $\Aut_{X}(p)\isom \mu_{k_{p}}$ for some $k_{p}\in \NN$. Therefore we get {\em canonical} embeddings $\BB\mu_{k_{p}}\incl X$ for each $p\in R$.


Let $\Bun_{G}(X)$ be the moduli stack of $G$-bundles over $X$.  For each $G$-torsor $\cE$ over $X$, restricting $\cE$ to $p\in R$ gives a $G$-torsor over $\BB\mu_{k_{p}}$. Therefore we get a map
\begin{equation}\label{res tw nodes}
\Bun_{G}(X)\to \prod_{p\in R}\cY_{k_{p}}
\end{equation}
In particular, $\Bun_{G}(X)$ maps to the discrete set $\prod_{p\in R}\YY_{k_{p}}$. For $\cE\in\Bun_{G}(X)$, its image in $\YY_{k_{p}}$ is called the {\em type} of $\cE$ at $p\in R$. 
According to the types $\y=(\y_{p})_{p\in R}$,  $\Bun_{G}(X)$ decomposes into open-closed substacks
\begin{equation*}
\Bun_{G}(X)=\coprod_{\y\in\prod_{p\in R}\YY_{k_{p}}}\Bun_{G}(X)_{\y}.
\end{equation*}
The map \eqref{res tw nodes} induces
\begin{equation}\label{eta res tw nodes}
\Bun_{G}(X)_{\y}\to \prod_{p\in R}\cY_{k_{p}, \y_{p}}
\end{equation}
where we recall $\cY_{k_{p},\y_{p}}$ is a gerbe with group isomorphic to $C_{G}(\wt\y_{p})$.

\sss{Level structure on the coarse curve}\label{sss:parah}
Let $\nu:X\to \un X$ be the natural map. Choose a maximal torus $T\subset G$ and $\wt\y_{p}\in \frac{1}{k_{p}}\xcoch(T)$ for each $p\in R$. Let $\y_{p}$ be the image of $\wt\y_{p}$ in $\YY_{k_{p}}$, and $\y=(\y_{p})_{p\in R}$.

For $p\in R$, let $\cK_{p}$ be the local field of $\un X$ at $p$. then $\wt\y_{p}\in \xcoch(T)_{\QQ}$ defines a point in the apartment of the building of $G(\cK_{p})$ corresponding to $T$.  Let $\bP'_{\wt\y_{p}}\subset G(\cK_{p})$ be the stabilizer of  $\wt\y_{p}$ under $G(\cK_{p})$. Then the neutral component  $\bP_{\wt\y_{p}}\subset \bP'_{\wt\y_{p}}$ is a parahoric subgroup of $G(\cK_{p})$.

Let $\bP'_{\wt\y}$ denote the collection $\{\bP'_{\wt\y_{p}}\}_{p\in R}$. Let $\Bun_{G, \bP'_{\wt\y}}(\un X, R)$ be the moduli stack of $G$-torsors over $\un X$ with $\bP'_{\wt\y_{p}}$-level structure at each $p\in R$. More precisely, the collection $\bP'_{\wt\y}$ gives a group scheme $\cG_{\wt\y}$ over $\un X$  (by the construction of Bruhat-Tits) whose generic fiber is $G$, and $\Bun_{G, \bP'_{\wt\y}}(\un X, R)$ is the moduli stack of $\cG_{\wt\y}$-torsors over $\un X$.

We claim that  the stack $\Bun_{G, \bP'_{\wt\y}}(\un X, R)$ is canonically independent of the choice of $T$ and the lifting $\wt\y$ of $\y$. Indeed, for another maximal torus $T'\subset G$ and liftings $\wt\y'_{p}:\mu_{k_{p}}\to T'$ of $\y_{p}$ for $p\in R$, let $g_{p}\in G(\cK_{p})$ be such that $g_{p}[\wt\y_{p}]=[\wt\y'_{p}]$ (where $[\wt\y_{p}]$ denotes the point in the building of $G(\cK_{p})$ given by $\wt\y_{p}$). Then $g_{p}\bP'_{\wt\y_{p}}=\bP'_{\wt\y'_{p}}$. Using uniformization $\prod_{p\in R}G(\cK_{p})/\bP'_{\wt\y_{p}}\to \Bun_{G, \bP'_{\wt\y}}(\un X, R)$,  we see that right multiplication by $g=(g_{p})_{g\in R}$ induces an isomorphism $\Bun_{G, \bP'_{\wt\y'}}(\un X, R)\isom \Bun_{G, \bP'_{\wt\y}}(\un X, R)$ where $\wt\y'=(\wt\y'_{p})_{p\in R}$. Moreover, a different choice $g'_{p}$ give the same isomorphism because the double coset $\bP'_{\wt\y'_{p}}g_{p}\bP'_{\wt\y_{p}}$  is well-defined.  The above discussion shows that there is a canonically defined stack $\Bun_{G, \y}(\un X, R)$ that depends only on $\y=(\y_{p})_{p\in R}$ and not on $T$ and $\wt\y$.

Let $L'_{\wt\y_{p}}$ be the reductive quotient of $\bP'_{\wt\y_{p}}$, which is not necessarily connected. Then we have a canonical isomorphism $L'_{\wt\y_{p}}\simeq C_{G}(\wt\y_{p})$. Since $L'_{\wt\y_{p}}$ is a quotient of $\bP'_{\wt\y_{p}}$, 
there is a canonical map $\Bun_{G,\bP'_{\wt\y}}(\un X, R)\to \prod_{p\in R}\BB L'_{\wt\y_{p}}$. Consider the composition
\begin{equation}\label{level to cY}
\Bun_{G,\y}(\un X, R)\simeq \Bun_{G,\bP'_{\wt\y}}(\un X, R)\to \prod_{p\in R}\BB L'_{\wt\y_{p}}\simeq \prod_{p\in R}\BB C_{G}(\wt\y_{p})\simeq \prod_{p\in R}\cY_{k_{p},\y_{p}}.
\end{equation} 
The same argument as above shows that the above map is independent of the choices of $T$ and $\wt\y$, and depends only on $\y$.

\begin{remark} We contrast the canonicity of $\Bun_{G,\y}(\un X,R)$ with the situation of the moduli stack of $G$-bundles with parahoric level structures. Given parahoric subgroups $\bP_{p}\subset G(\cK_{p})$ at $p\in R$, we have the moduli stack $\Bun_{G, (\bP_{p})}(\un X,R)$. Although the isomorphism type of  $\Bun_{G, (\bP_{p})}(\un X,R)$  depends only on the conjugacy classes of $\bP_{p}$ under $G(\cK_{p})$, the isomorphisms for different choices of $\bP_{p}$ cannot be made canonical in general. For example,   when $G$ is not simply-connected and $\bP_{p}$ is an Iwahori subgroup, the group $N_{G(\cK_{p})}(\bP_{p})/\bP_{p}$ then acts nontrivially on $\Bun_{G, (\bP_{p})}(\un X,R)$.
\end{remark}


\begin{prop}\label{p:Bun on tw} For any $\y\in (\y_{p})_{p\in R}\in \prod_{p\in R}\YY_{k_{p}}$, there is a canonical isomorphism of stacks over $\prod_{p\in R}\cY_{k_{p},\y_{p}}$
\begin{equation*}
\Bun_{G,\y}(\un X, R)\isom\Bun_{G}(X)_{\y}.
\end{equation*}
Here the map from the left side to $\prod_{p\in R}\cY_{k_{p},\y_{p}}$ is given by \eqref{level to cY}, and the similar map from the right side is given by \eqref{eta res tw nodes}.
\end{prop}

\sss{Uniformization for bundles on a twisted curve}\label{sss:unif tw}
Before giving the proof of Proposition~\ref{p:Bun on tw}, we remark on the uniformization of $\Bun_{G}(X)$ which will be used in the proof.

We first consider one orbifold point  $p\in R$ with automorphism group $\mu_{k}$. Let $\cO_{p}\subset \cK_{p}$ be the completed local ring of $\un X$ at $p$. Over the formal disc $\Spec \cO_{p}\to\un X$, the map $\nu:X\to \un X$ takes the form $(\Spec \cO^{(k)}_{p})/\mu_{k}\to \Spec \cO_{p}$, where $\cO^{k}_{p}$ is the ring of integers in the unique degree $k$ extension $\cK^{(k)}_{p}$ or $\cK_{p}$. Note $\Gal(\cK^{(k)}_{p}/\cK_{p})=\mu_{k}$. Consider the affine Grassmannian $\Gr^{(k)}_{p}:=G(\cK^{(k)}_{p})/G(\cO^{(k)}_{p})$. Let $\mu_{k}$ act on $\Gr^{(k)}_{p}$ via its Galois action on $\cK^{(k)}_{p}$. Then the fixed points $(\Gr^{(k)}_{p})^{\mu_{k}}$ parametrize $G$-torsors on $X$ with a trivialization on $X\bs\{p\}$.  

Replacing $G$ by a maximal torus $T\subset G$ we have the affine Grassmannian $\Gr^{(k)}_{T,p}$ attached to the local field $\cK^{(k)}_{p}$ whose reduced structure is the discrete set $\frac{1}{k}\xcoch(T)$: for $\l\in\frac{1}{k}\xcoch(T)$ the corresponding point in $\Gr^{(k)}_{T,p}$ is $(k\l)(\e)$, where $\e\in \cK^{(k)}_{p}$ is any uniformizer and $k\l$ is now a cocharacter $\Gm\to T$ (the point $(k\l)(\e)\in \Gr^{(k)}_{T,p}$ is independent of the choice of $\e$). The natural embedding $\Gr^{(k)}_{T,p}\incl \Gr^{(k)}_{p}$ restricted to the reduced structure must land in $\mu_{k}$-fixed locus, giving an embedding
\begin{equation*}
e_{\Gr,p}: \frac{1}{k}\xcoch(T)\incl (\Gr^{(k)}_{p})^{\mu_{k}}.
\end{equation*}

The loop group $G(\cK_{p})$ acts on $(\Gr^{(k)}_{p})^{\mu_{k}}$ by left translation. The $G(\cK_p)$-orbits on $(\Gr^{(k)}_{p})^{\mu_{k}}$ are naturally indexed by $\YY_{k}$: each $G(\cK_p)$-orbit contains $e_{\Gr,p}(\l)$ for a unique $\tilW$-orbit of $\l\in \frac{1}{k}\xcoch(T)$. For $\l\in \frac{1}{k}\xcoch(T)$, the stabilizer of $G(\cK_p)$ at $e_{\Gr,p}(\l)$ is $\bP'_{\l}\subset G(\cK_p)$, the stabilizer of the point $\l$ on the standard apartment of $G(\cK_p)$  corresponding to $T$. For $\y\in \YY_{k}$, let $(\Gr^{(k)}_{p})^{\mu_{k}}_{\y}$ be the $G(\cK_p)$-orbit corresponding to $\y$.  By choosing a representative $\wt\y\in \frac{1}{k}\xcoch(T)$ for each $\y\in \YY_{k}$, we get an $G(\cK_p)$-equivariant  isomorphism
\begin{equation}\label{Gr fixed pt}
 (\Gr^{(k)}_{p})^{\mu_{k}}=\coprod_{\y\in \YY_{k}}(\Gr^{(k)}_{p})^{\mu_{k}}_{\y}\simeq \coprod_{\y\in \YY_{k}}G(\cK_p)/\bP'_{\wt\y}.
\end{equation}
Therefore $ (\Gr^{(k)}_{p})^{\mu_{k}}$ is a finite disjoint union of affine partial flag varieties of $G(\cK_p)$.

The above discussion can be generalized to simultaneously uniformization at all $p\in R$ in an obvious way.  Let $U=\un X\bs R$ which is identified with its preimage in $X$.  The product  $\prod_{p\in R}(\Gr^{(k_{p})}_{p})^{\mu_{k_{p}}}$ parametrizes $G$-torsors on $X$ with a trivialization on $U$. We have an embedding
\begin{equation}\label{eGr}
e_{\Gr}: \op_{p\in R}\frac{1}{k_{p}}\xcoch(T)\to \prod_{p\in R}(\Gr^{(k_{p})}_{p})^{\mu_{k_{p}}}.
\end{equation}
Using the $G(\cK_{p})$-orbits on $(\Gr^{(k_{p})}_{p})^{\mu_{k_{p}}}$, we get a uniformization of $\Bun_{G}(X)_{\y}$ by
\begin{equation*}
\prod_{p\in R}(\Gr^{(k_{p})}_{p})^{\mu_{k_{p}}}_{\y_{p}}
\end{equation*}
which parametrizes an object in $\Bun_{G}(X)_{\y}$ together with a trivialization over $U$.

\sss{Proof of Proposition~\ref{p:Bun on tw}}  Choose a maximal torus $T\subset G$ and a lifting $\wt\y_{p}\in \frac{1}{k_{p}}\xcoch(T)$ of $\y_{p}$. Let $\cG_{\wt\y}$ be the Bruhat-Tits group scheme over $\un X$ which is the constant group $G$ over $U$ and $\cG_{\wt\y}(\cO_{p})=\bP'_{\wt\y_{p}}$ at $p\in R$. 

To define a map
\begin{equation}\label{isom Bun}
\ph: \Bun_{G,\bP'_{\wt\y}}(\un X,  R )\to\Bun_{G}(X)_{\y}
\end{equation}
it suffices to construct a right $G$-torsor $\cE_{\wt\y}$ on $X$ with type $\y_x$ at $p_{i}$, together with a commuting left action of $\nu^{*}\cG_{\wt\y}$. Indeed, once we have such a $\cE_{\wt\y}$, for any $\cF\in \Bun_{G,\bP'_{\wt\y}}(\un X)$, viewed as a $\cG_{\wt\y}$-torsor over $\un X$, we form the right $G$-torsor $\cE_{\cF}:=\nu^{*}\cF\twtimes{\nu^{*}\cG_{\wt\y}}\cE_{\wt\y}$ on $X$. It is easy to see that $\cE_{\cF}$ has the same type as $\cE_{\wt\y}$ at $p_{i}$.

Recall the map $e_{\Gr}$ from \eqref{eGr}. We now take $\cE_{\wt\y}$ to be image of $e_{\Gr}(\wt\y)$ in $\Bun_{G}(X)_{\y}$. Under the left translation action of $\prod_{p\in R}G(\cK_{p})$, the stabilizer of $e_{\Gr}(\wt\y)$ is exactly $\prod\bP'_{\wt\y_{p}}$. Therefore the tautological left action of $G$ on the trivial $G$-torsor $\cE_{\wt\y}|_{U}$ extends to an action of $\nu^{*}\cG_{\wt\y}$.

%

This completes the construction of the map $\ph$. To show it is an isomorphism, we may as well show that both sides are uniformized by the same affine partial flag varieties. On the one hand, $\prod_{p\in R}G(\cK_{p})/\bP'_{\wt\y_{p}}$ parametrizes objects in $\Bun_{G,\bP'_{\wt\y}}(\un X, R)$ with trivialization over $U$. On the other hand, $\prod_{p\in R}(\Gr^{(k_{p})}_{p})^{\mu_{k_{p}}}_{\y_{p}}$ parametrizes objects in $\Bun_{G}(X)_{\y}$ with trivialization over $U$. The isomorphism \eqref{Gr fixed pt} gives an isomorphism 
\begin{equation*}
\prod_{p\in R}G(\cK_{p})/\bP'_{\wt\y_{p}}\simeq \prod_{p\in R}(\Gr^{(k_{p})}_{p})^{\mu_{k_{p}}}_{\y_{p}}
\end{equation*}
sending the base point of the left side to the point $e_{\Gr}(\wt\y)$ on the right. The construction of $\cE_{\wt\y}$ guarantees that this isomorphism descends to the map $\ph$. Therefore $\ph$ is also an isomorphism. 

Finally one checks that for a different choice of $T'$ and $\wt\y'$, the canonical isomorphism $\Bun_{G,\bP'_{\wt\y}}(\un X, R)\simeq \Bun_{G,\bP'_{\wt\y'}}(\un X, R)$ constructed in \S\ref{sss:parah} is compatible with the isomorphisms $\ph$ and $\ph'$ defined using the old and new choices. Therefore we get a canonical isomorphism $\Bun_{G,\y}(\un X, R)\simeq \Bun_{G}(X)_{\y}$. \qed

\subsection{Twisted nodal case}
Now we consider the case $X$ has twisted nodes.

\sss{The curve $X$}\label{sss:tw nodal X}
Let $X$ be a proper DM curve over $\CC$ whose orbifold points $R$ are exactly the nodal singularities. More precisely,  for each $p\in R$, the formal neighborhood of $p$ in $X$ is isomorphic to $(\Spec \CC\tl{u,v}/(uv))/\mu_{k_{p}}$ where $\z\in \mu_{k_{p}}$ acts by $\z\cdot (u,v)=(\z u, \z^{-1}v)$. For each $p\in R$ the tangent space $T_{p}X$ is the direct sum of two lines $T^{+}_{p}X\op T^{-}_{p}X$ corresponding to the two analytic branches, such that $\Aut(p)\simeq \mu_{k_{p}}$ acts on $T^{+}_{p}X$ by the tautological character $\mu_{k_{p}}\incl \Gm$ and acts on $T^{-}_{p}X$ by the inverse of the  tautological character.  The choice of the $+$-local branch at $p$ determines an isomorphism $\Aut(p)\simeq \mu_{k_{p}}$. {\em We fix a choice of a $+$-local branch at each $p$ in the sequel.}


For each $G$-torsor $\cE$ over $X$, restricting $\cE$ to $p\in R$ gives a point in $\cY_{k_{p}}$, hence an element $\y_{p}\in \YY_{k_{p}}$ called it type at $p$. The moduli stack of $G$-bundles then decomposes into open-closed substacks according the types $\y=(\y_{p})\in \prod_{p\in R}\YY_{k_{p}}$
\begin{equation*}
\Bun_{G}(X)=\coprod_{\y\in\prod_{p\in R}\YY_{k_{p}}}\Bun_{G}(X)_{\y}.
\end{equation*}


\sss{Level structure on the coarse nodal curve} We use the notations $\un X, \nu: X\to \un X$ and $U=\un X-R$ from \S\ref{sss:parah}. Let $\un\t: \wt {\un X}\to \un X$ be the normalization map.  We also identify $U$ with its preimage in $\wt{\un X}$. We name the preimages of $p\in R$ in $\wt{\un X}$ by $p^{+}$ and $p^{-}$ according to the $\pm$-analytic branches we fixed before.  Let $\cO_{p}^{\pm}$ be the formal completion of $\wt{\un X}$ at $p^{\pm}$, with fraction field $\cK^{\pm}_{p}$. 

Let $\y=(\y_{p})\in \prod_{p\in R}\YY_{k_{p}}$. Choose a maximal torus $T\subset G$ and liftings $\wt\y_{p}\in \frac{1}{k_{p}}\xcoch(T)$ of $\y_{p}$. As in \S\ref{sss:parah} we define $\bP'^{+}_{\wt\y_{p}}\subset G(\cK^{+}_{p})$ and $\bP'^{-}_{-\wt\y_{p}}\subset G(\cK^{-}_{p})$. Note there is a canonical isomorphism between the reductive quotients of $\bP'^{+}_{\wt\y_{p}}$ and $\bP'^{-}_{-\wt\y_{p}}$, which we denote by $L'_{\wt\y_{p}}\simeq L'_{-\wt\y_{p}}$. We form the subgroup
\begin{equation*}
\bP'_{\pm\wt\y}:=\bP'^{+}_{\wt\y_{p}}\times_{L'_{\wt\y_{p}}}\bP'^{-}_{-\wt\y_{p}}\subset G(\cK^{+}_{p})\times G(\cK^{-}_{p}).
\end{equation*}

Let $\bP'_{\pm\wt\y}$ denote the collection $\{\bP'_{\pm\wt\y_{p}}\}_{p\in R}$.   Let $\Bun_{G, \bP'_{\pm\wt\y}}(\un X,R)$ be the moduli stack of $G$-torsors on $\un X$ with level structure $\bP'_{\pm\wt\y_{p}}$ at $p$. It can be described using the normalization $\wt{\un X}$ as follows: $\Bun_{G, \bP'_{\pm\wt\y}}(\un X,R)$ classifies $G$-bundles over $\wt{\un X}$ with $\bP'^{+}_{\wt\y_{p}}$-level structure at $p^{+}$, $\bP'^{-}_{-\wt\y_{p}}$-level structure at $p^{-}$ and an isomorphism between the induced $L'_{\wt\y_{p}}\simeq L'_{-\wt\y_{p}}$-torsors at $p^{+}$ and $p^{-}$ (for each $p\in R$).

The same argument as in \S\ref{sss:parah} shows that $\Bun_{G, \bP'_{\pm\wt\y}}(\un X,R)$ depends only on $\y$ and not on the choices of $T$ and $\wt\y$. We denote this canonical moduli stack by $\Bun_{G, \pm\wt\y}(\un X,R)$. Moreover, there is a canonical map
\begin{equation*}
\Bun_{G, \pm\wt\y}(\un X,R)\to \prod_{p\in R}\cY_{k_{p},\y_{p}}.
\end{equation*}

\begin{lemma}\label{l:Bun on tw nodal} Let $\y=(\y_{p})\in \prod_{p\in R}\YY_{k_{p}}$. Then there is a canonical isomorphism of stacks over $\prod_{p\in R}\cY_{k_{p},\y_{p}}$
\begin{equation}\label{Bun on tw nodal}
\Bun_{G,\pm\wt\y}(\un X,R)\isom\Bun_{G}(X)_{\y}.
\end{equation}
\end{lemma}
\begin{proof}
Let $\t: \wt X\to X$ be the normalization. There is a natural map $\wt\nu: \wt X\to \wt{\un X}$ realizing $\wt{\un X}$ as the coarse moduli space of $\wt X$. The set of nodes of $\wt {\un X}$ is $\wt R=R^{+}\coprod R^{-}$ where $R^{+}=\{p^{+}; p\in R\}$ and $R^{-}=\{p^{-}; p\in R\}$. Let $\pm\y$ denote the element of $\prod_{p^{\pm}\in \wt R}\YY_{k_{p}}$ that assigns $\y_{p}$ to  $p^{+}$ and $-\y_{p}$ to $p^{-}$. Then $\Bun_{G, \pm\y}(\wt{\un X}, \wt R)$ is defined according to \S\ref{sss:parah}.

By Proposition~\ref{p:Bun on tw}, we have a canonical isomorphism over $\prod_{p\in R}(\cY_{k_{p}, \y_{p}}\times \cY_{k_{p}, -\y_{p}})$
\begin{equation}\label{Bun wt X}
\Bun_{G, \pm\y}(\wt{\un X}, \wt R)\isom \Bun_{G}(\wt X)_{\pm\y}
\end{equation}

Note there is a canonical isomorphism $\cY_{k_{p}, \y_{p}}\simeq\cY_{k_{p}, -\y_{p}}$ given by the automorphism $\z\mt\z^{-1}$ of $\mu_{k_{p}}$. Let $\D^{-}_{p}:\cY_{k_{p}, \y_{p}}\to \cY_{k_{p}, \y_{p}}\times \cY_{k_{p}, -\y_{p}}$ be the graph of this isomorphism. By definition, both sides of \eqref{Bun on tw nodal} are obtained from the corresponding sides of \eqref{Bun wt X} via base change along
\begin{equation*}
(\D^{-}_{p})_{p\in R}: \prod_{p\in R}\cY_{k_{p}, \y_{p}}\to \prod_{p\in R}(\cY_{k_{p}, \y_{p}}\times \cY_{k_{p}, -\y_{p}}).
\end{equation*}
We get the desired isomorphism \eqref{Bun on tw nodal} from \eqref{Bun wt X}.

\end{proof}

\subsection{Family version}
The results from previous subsections can be generalized to families of curves. We state a version that allows both  smooth orbifolds and nodal orbifold points. 

\sss{Twisted nodal family} Let $\pi: \frX\to B$ be a proper relatively Deligne-Mumford map of algebraic stacks. Suppose $\pi$ is flat of relative dimension $1$. Let $\frX^{sm}$ be the locus where $\pi$ is smooth. Let $\Sigma$ be a finite collection of closed embeddings $\s: B_{\s}=B\times \BB\mu_{k_{\s}}\incl \frX^{sm}$ (over $B$) for some $k_{\s}\in \NN$. Suppose $\frX-\frX^{sm}$ is a finite disjoint union $\coprod_{r\in R}B_{r}$ where $B_{r}\cong B\times \BB\mu_{k_{r}}$ for some $k_{r}\in \NN$. Assume the formal completion of $\frX$ along $B_{r}$ is locally isomorphic to $(\Spec A\tl{u,v})/\mu_{k_{r}}\to \Spec A$ (where $\z\in \mu_{k_{r}}$ acts by $\z\cdot (u,v)=(\z u,\z^{-1}v)$).  We can canonically identify $B_{\s}$ with $B\times \BB\mu_{k_{\s}}$ for $\s\in \Sigma$ by using the relative tangent bundle of $\pi$ (as we did in \S\ref{sss:DM X}); similarly, by choose a $+$-local branch along $B_{r}$, we  can canonically identify $B_{r}$ with $B\times \BB\mu_{k_{r}}$ for $r\in R$ (as we did in \S\ref{sss:tw nodal X}).

Let $\y_{\Sigma}=(\y_{\s})_{\s\in \Sigma}\in \prod_{\s\in \Sigma}\YY_{k_{\s}}$ and $\y_{R}=(\y_{r})_{r\in R}\in \prod_{r\in R}\YY_{k_{r}}$. We have the moduli stack $\Bun_{G}(\pi)$ classifying $G$-bundles along the fibers of  $\pi$. We also have its open-closed substack $\Bun_{G}(\pi)_{\y_{\Sigma},\y_{R}}$ by fixing the types of the bundle along $B_{\s}$ and $B_{r}$ to be $\y_{\s}$ and $\y_{r}$.

\sss{Level structure on the coarse family} Let $\un\pi: \un \frX\to B$ be the relative coarse moduli space (i.e., base change along any map $B'\to B$ with $B'$ a scheme, $\un\frX_{B'}$ is the coarse moduli space of $\frX_{B'}$). Let $\wt{\un\pi}: \wt{\un \frX}\to B$ be the fiberwise normalization of $\un\frX$. Let $\un B_{\s}$ and $\un B_{r}$ be the image of $B_{\s}$ and $B_{r}$  in $\un\frX$, and $\un B^{+}_{r}\coprod\un B^{-}_{r}$ be the preimages of $\un B_{r}$ in $\wt{\un \frX}$ for $r\in R$.

By choosing a maximal torus $T\subset G$ and liftings $\wt\y_{\Sigma}\in \op_{\s\in \Sigma}\frac{1}{k_{\s}}\xcoch(T)$ and $\wt\y_{R}\in \op_{r\in R}\frac{1}{k_{r}}\xcoch(T)$, we may form the moduli stack (over $B$)
\begin{equation*}
\Bun_{G, \bP'_{\wt\y_{\Sigma}}, \bP'_{\pm\wt\y_{R}}}(\un \pi, \Sigma, R)
\end{equation*}
classifying $G$-bundles along the fibers of $\wt{\un\pi}: \wt{\un \frX}\to B$ with $\bP'_{\wt\y_{\s}}$-level structure along  $\un B_{\s}$ for $\s\in \Sigma$, $\bP'_{\pm\wt\y_{r}}$-level structure along  $\un B^{\pm}_{r}$ for $\s\in \Sigma$, and an isomorphism between the induced $L'_{\wt\y_{r}}$-torsors along $\un B^{+}_{r}$ and $\un B^{-}_{r}$ (both isomorphic to $\un B_{r}\cong B$). Again this moduli stack is canonically independent of the choices of $T, \wt\y_{\Sigma}$ and $\wt\y_{R}$, and we get a canonical moduli stack
\begin{equation*}
\Bun_{G, \y_{\Sigma}, \pm\y_{R}}(\un \pi, \Sigma, R)
\end{equation*}
over $\prod_{\s\in \Sigma}\cY_{k_{\s},\y_{\s}}\times \prod_{r\in R}\cY_{k_{r},\y_{r}}$.

\begin{prop}\label{p:Bun tw family}
For any $\y_{\Sigma}\in \prod_{\s\in \Sigma}\YY_{k_{\s}}$ and $\y_{R}\in \prod_{r\in R}\YY_{k_{r}}$, there is a canonical isomorphisms of stacks over $\prod_{\s\in \Sigma}\cY_{k_{\s},\y_{\s}}\times \prod_{r\in R}\cY_{k_{r},\y_{r}}$
\begin{equation*}
\Bun_{G, \y_{\Sigma}, \pm\y_{R}}(\un \pi, \Sigma, R)\isom\Bun_{G}(\pi)_{\y_{\Sigma},\y_{R}}.
\end{equation*}
\end{prop}
The proof is essentially the same as that of Proposition \ref{p:Bun on tw} and Lemma \ref{l:Bun on tw nodal}, using uniformization along the sections $\un B_{\s}$ and $\un B^{\pm}_{r}$.


\section{Universal nilpotent cone} 

The automorphic category  in Betti geometric Langlands consists  of sheaves on $\Bun_{G}(X)$   with singular support in the global nilpotent cone $\cN \subset T^*\Bun_G(X)$, i.e. the zero fiber of the Hitchin map. Since we will consider families of curves in this paper, we will need a family version of the global nilpotent cone. 

Suppose $\pi: \frX\to S$ is a smooth projective family of curves over a smooth base stack $S$. Let $\Pi: \Bun_{G}(\pi)\to S$ be the corresponding moduli stack of $G$-bundles along fibers of $\pi$. In this section, we will define a closed conic Lagrangian $\wt\cN_{\Pi}\subset T^{*}\Bun_{G}(\pi)$, called the {\em universal nilpotent cone},  such that its restriction to each fiber $T^{*}\Bun_{G}(\frX_{s})$ ($s\in S$) is the global nilpotent cone of $\Bun_{G}(\frX_{s})$. 


\subsection{Hitchin system and global nilpotent cone}

We recall here some well-known Lie theory ``over a point", its traditional generalization over a  smooth  curve, and then its natural  generalization  over  a  relative DM curve.


\subsubsection{Adjoint quotient} Let $G$ be a reductive group, $B\subset G$ a Borel subgroup, with unipotent radical $N = [B, B]$, and Cartan quotient $H = B/N$. Let $\frg$, $\frb$, $\frn$, and $\frh$ denote the respective Lie algebras. Let $W$ denote the Weyl group, and $\frc = \frh\quot W =  \Spec \cO(\frh)^W$.

We have  the characteristic polynomial map
$$
\xymatrix{
\chi:\frg/(G\times \Gm)  \ar[r] & (\frg\quot G)/\Gm \simeq (\frh\quot W)/\Gm=\frc/\Gm
}
$$ 
where the $G$-action on $\frg$ is the adjoint action, and the $\Gm$-action  is by  scaling. 
Note the $\Gm$-action on $\frh$ is also by  scaling, but 
 the weights of the induced $\Gm$-action on the affine space $\frc$ are the degrees of $W$.
 
Let $\wt \frg \to\frg$ be the Grothendieck-Springer alteration, and recall the canonical isomorphism
$\wt \frg/G \simeq \frb/B$ of adjoint quotients. We have the ordered eigenvalue map 
$$
\xymatrix{
\wt \chi:\wt \frg/(G \times \Gm) \simeq \frb/(B\times \Gm)  \ar[r] & \frh/\Gm 
}
$$ 
induced by $\frb\to \frh = \frb/\frn$.

%
%
%
%
%


\subsubsection{Marked smooth curves}\label{sss:glob nilp cone}

Let $X$ be a smooth  projective curve, and $\omega_X$  the canonical bundle of $X$.

Let $\Sigma\subset X$ be a finite set of closed points, and $\omega_X(\Sigma)$  the sheaf of $1$-forms with possibly simple poles at $\Sigma$. We will use the canonical isomorphism $\Res_{\Sigma}: \omega_X(\Sigma)|_\Sigma \simeq \cO_\Sigma$ given by taking the residues.

For a $G$-bundle $\cE$ over $X$ and a representation $V$ of $G$,  let $\cE(V)$ be the associated vector bundle on $X$.

For a scheme $Y$ with $\Gm$-action, let $Y_{\om_{X}(\Sigma)}\to X$ denote the twist of $Y$ using the $\Gm$-torsor associated to $\om_{X}(\Sigma)$. Then the restriction of $Y_{\om_{X}(\Sigma)}$ over $\Sigma$ has a canonical trivialization $Y\times \Sigma$ using the residue map.

Consider the moduli stack
$$
\xymatrix{ 
\Bun_{G, N}(X, \Sigma) := \Maps((X, \Sigma), (\BB G, \BB N))
}
$$
classifying pairs $(\cE, \cE_{\Sigma,N})$ of a $G$-bundle $\cE$ on $X$ with an $N$-reduction $\cE_{\Sigma,N}$ along $\Sigma$.

 Consider the moduli stack of Higgs bundles
$$
\xymatrix{
 \operatorname{Higgs}_{G, N}(X, \Sigma) = \Maps ((X, \Sigma), (\frg_{\om_{X}(\Sigma)}/(G\times \Gm),
  \frb/N)) 
}$$
classifying data $(\cE, \cE_{\Sigma,N}, \phi, \phi_{\Sigma})$ of a pair $(\cE, \cE_{\Sigma,N}) \in \Bun_{G, N}(X, \Sigma)$ along with a Higgs field 
$$
\xymatrix{
\phi\in H^0(X, \cE(\frg) \otimes \omega_X(\Sigma))
}
$$
and 
$$
\xymatrix{
\phi_\Sigma=\Res_{\Sigma}\phi  \in H^0(\Sigma, \cE_{\Sigma,N}(\frb))
}
$$
Note that the datum of $\phi_{\Sigma}$ is determined by $\phi$, and the existence of $\phi_{\Sigma}$ imposes a linear condition on $\Res_{\Sigma}\phi$.

The Killing form on $\frg$ gives a $G$-invariant identification $\frg^*\simeq \frg$, with induced identification $(\frg/\frn)^* \simeq \frb$. Serre duality provides  
 an isomorphism
$$
\xymatrix{
 T^*\Bun_{G, N}(X, \Sigma)  \simeq \operatorname{Higgs}_{G, N}(X, \Sigma) 
}$$
 
Consider the Hitchin base
$$
\xymatrix{
A_{G, N}(X, \Sigma)= 
\Maps((X, \Sigma), (\frc_{\omega_X(\Sigma)}, \frh))
}
$$ 
classifying pairs $(a, a_\Sigma)$ of a section
$a\in H^0(X,  \frc_{\omega_X(\Sigma)})$ with a lift of its residue $a_\Sigma\in H^0(\Sigma, \frh)$.

 Applying the characteristic polynomial map $\chi$ and ordered eigenvalue map $\wt \chi$ to Higgs fields provides the Hitchin system
 $$
\xymatrix{
H:T^*\Bun_{G, N}(X, \Sigma)  \simeq  \operatorname{Higgs}_{G, N}(X, \Sigma)  \ar[r] &  A_{G, N}(X, \Sigma)
}
$$ 

The  {\em global nilpotent cone} is defined to be the zero-fiber  
$$
\cN =\cN_{X,\Sigma}^{N}= H^{-1}(0) \subset T^*\Bun_{G, N}(X, \Sigma)
$$ 

When $\Sigma=\vn$, it is proved by Laumon \cite{Lau} (in type $A$) and Ginzburg \cite{Gin} (in general) that $\cN$ is a closed, conic Lagrangian.  The same is true when $\Sigma\ne\vn$, as we will show in Lemma \ref{l:two cones B} and Corollary \ref{c:NN}.

%

%
%


\subsubsection{Marked relative DM curves}\label{sss:rel Hitchin}

Let $S$ be a smooth base stack, and $\pi:\frX\to S$   a flat proper Gorenstein relative DM curve. 
Assume the locus $\frX^0 \subset \frX$ where both $\pi$ is  schematic and $\pi$ is smooth is dense in each fiber of $\pi$. Let $\sigma_a\subset \frX^0$, $a\in \Sigma$, be a finite collection of disjoint closed substacks \'etale over $S$. We will treat all  sections at once and so set $\sigma = \cup_{a\in \Sigma} \sigma_a$.
Let $\omega_\pi$ be the relative dualizing sheaf of $\pi$ (which is a line bundle by the Gorenstein assumption), and $\omega_\pi(\sigma)$ the twisting allowing simple poles along $\s$.

Given a map of stacks $s:S'\to  S$, with base change
$\pi_s:\frX_s = \frX \times_S S'\to S'$, $\sigma_s = \sigma \times_S S'$, we have  canonical isomorphisms $\omega_{\pi}|_{\frX_s} \simeq \omega_{\pi_s }$,
 $\omega_{\pi}(\sigma)|_{\frX_s} \simeq \omega_{\pi_s}(\sigma_s)$.
 Note also the canonical isomorphism $\omega_{\pi}(\sigma)|_\sigma \simeq \cO_\sigma$.

Consider the moduli stack $\Bun_{G, N}(\pi, \sigma)$ whose $S'$-points 
classifies triples $(s, \cE, \cE_{\sigma, N})$ where $s:S'\to S$, $\cE$ is a $G$-bundle on $\frX_s = \frX \times_S S'$,
and $\cE_{\sigma, N}$ is an $N$-reduction of $\cE$ along $\sigma_s = \sigma \times_S S'$.

Consider the Higgs  moduli
$
 \operatorname{Higgs}_{G, N}(\pi, \sigma)
 $
classifying data $(s, \cE, \cE_{\sigma, N}, \ph, \ph_\sigma)$ of a triple $(s, \cE, \cE_{\sigma, N}) \in \Bun_{G, N}(\pi, \sigma)$ along with a Higgs field 
$$
\xymatrix{
\ph\in H^0(\frX_s, \cE(\frg)\otimes \omega_{\pi_s}(\sigma_s))
}
$$
and a lift of its residue
$$
\xymatrix{
\ph_\sigma =\Res_{\s}\ph \in H^0(\sigma_s, \cE_{\sigma, N}(\frb)).
}
$$

Let $\Pi: \Bun_{G, N}(\pi, \sigma)\to S$  denote the natural projection. Consider projection to the relative cotangent bundle
$$
\xymatrix{
P_\Pi:T^*\Bun_{G, N}(\pi, \sigma) \ar[r] & T^*_\Pi = T^*\Bun_{G, N}(\pi, \sigma)/\Pi^*(T^*S)
}
$$ 
Note $T^*\Bun_{G, N}(\pi, \sigma)$ and  $T^*_\Pi $ are not in general vector bundles, but  the vector bundle $\Pi^*(T^*S)$ acts freely on $T^*\Bun_{G, N}(\pi, \sigma)$, and $T^*_\Pi $ is the corresponding quotient.
Using Killing form to identify $(\frg/\frn)^*$ with $\frb$, and Grothendieck-Serre  duality provides  
 an isomorphism
$$
\xymatrix{
 T^*_\Pi   \simeq \operatorname{Higgs}_{G, N}(\pi, \sigma) 
}$$

Let $\frc_{\omega_{\pi}(\sigma)}$ denote the $\omega_\pi(\sigma)$-twist of $\frc$ (a scheme over $\frX$), and  consider the Hitchin base $ A_{G, N}(\pi, \sigma) $ 
whose $S'$-points classify triples $(s, a, a_\sigma)$ where $s:S'\to S$, and
$a\in H^0(\frX_s,  \frc_{\omega_{\pi_s}(\sigma_s)})$ with a lift of its residue $a_\sigma\in H^0(\sigma_s, \frh)$.

 Applying the characteristic polynomial map $\chi$ and ordered eigenvalue map $\wt \chi$ to Higgs fields provides the Hitchin system
 $$
\xymatrix{
H_\Pi: T^*_\Pi   \simeq  \operatorname{Higgs}_{G, N}(\pi, \sigma)  \ar[r] &  A_{G, N}(\pi, \sigma)
}
$$

The {\em relative global nilpotent cone} is defined to be the zero-fiber  in the relative cotangent bundle
$$
 \cN_\Pi = H_\Pi^{-1}(0) \subset T^*_\Pi
$$ 
Its fibers over closed points $s\in S$ recover the global nilpotent cone for $\Bun_{G,N}(X_{s},\s_{s})$ defined in \S\ref{sss:glob nilp cone}.

The {\em total global nilpotent cone} is defined to be the inverse image in the cotangent bundle
$$
\xymatrix{
\cN^+_{\Pi}= P^{-1}_\Pi(\cN_{\Pi}) \subset T^*\Bun_{G, N}(\pi, \sigma).
}
$$
It is  closed, conic, and coisotropic.

\subsection{Eisenstein cone}\label{ss:Eis}
Suppose $\pi: \frX\to S$ is a smooth projective family of curves over a smooth base stack $S$. Let $\Pi: \Bun_{G}(\pi)\to S$ be the corresponding family of relative bundles,  $T^{*}_{\Pi} = T^*\Bun_{G}(\pi)/ \Pi^*(T^*S)$ the relative cotangent bundle of $\Pi$, and $\cN_\Pi = H^{-1}(0) \subset T^{*}_{\Pi} $ the global relative nilpotent cone. We would like to define a closed conic Lagrangian $\wt\cN_{\Pi}\subset T^{*}\Bun_{G}(\pi)$ such that $\wt \cN_{\Pi}$ maps isomorphically to  $\cN_{\Pi}$
 under the natural projection
\begin{equation}\label{ppi}
\xymatrix{
P_{\Pi}:  T^{*}\Bun_{G}(\pi)\ar[r] &  T^{*}_{\Pi}.
}
\end{equation}

In the discussion below, we will freely use notations introduced in Appendix \ref{sss:Lag mfd} on transportation of Lagrangians.


\sss{A single curve} 

To motivate the definition of $ \wt\cN_{\Pi}$, we first reformulate the global nilpotent cone $\cN$ for a single curve.

Consider first a smooth connected projective curve $X$ over $\CC$. Let $ p: \Bun_{B}(X)\to \Bun_{G}(X)$ be the  induction map  and consider the associated Lagrangian correspondence
\begin{equation*}
\xymatrix{T^{*}\Bun_{B}(X) &  \Bun_{B}(X)\times_{ \Bun_{G}(X)}T^{*} \Bun_{G}(X)\ar[l]_-{dp}\ar[r]^-{p^{\na}} & T^{*}\Bun_{G}(X) }
\end{equation*}


\begin{defn} Define  the Eisenstein cone $\cN^{\Eis}_{X}\subset T^{*}\Bun_{G}(X)$ to be
\begin{equation*}
\cN^{\Eis}_{X}=\orr{p}(0_{\Bun_{B}(X)}) = p^{\na}((dp)^{-1}(0_{\Bun_{B}(X)})).
\end{equation*}
\end{defn}

\begin{lemma}\label{l:two cones} The subset $\cN^{\Eis}_{X}$ of $T^{*}\Bun_{G}(X)$ coincides with the global nilpotent cone $\cN_{X}$,  i.e.~the zero-fiber of the Hitchin map. In particular, $\cN^{\Eis}_{X}$ is a closed conic Lagrangian of $T^{*}\Bun_{G}(X)$.
\end{lemma}
\begin{proof}
This is essentially in the proof of the main theorem of Ginzburg \cite{Gin}. 

First we show that $\cN^{\Eis}_{X}\subset\cN_{X}$. Let $\cE_{B}$ be a $B$-torsor over $X$ and let  $\cE$ be the induced $G$-torsor. The cotangent space of $\Bun_{B}(X)$ at $\cE_{B}$ is $\cohog{0}{X, \cE_{B}(\frb^{*})\ot\om_{X}}$. The differential of the induction map $p: \Bun_{B}(X)\to \Bun_{G}(X)$ at $\cE$ is induced by the restriction map $\cE(\frg^{*})=\cE_{B}(\frg^{*})\to \cE_{B}(\frb^{*})$ 
\begin{equation*}
\xymatrix{  dp:    \cohog{0}{X, \cE(\frg^{*})\ot\om_{X}} \ar[r] &     \cohog{0}{X, \cE_{B}(\frb^{*})\ot\om_{X}}
}
\end{equation*}
Under the Killing form on $\frg$, $\ker(\frg^{*}\to \frb^{*})$ is identified with 
$\frn$, Therefore $\ker(dp)_{\cE} = \cohog{0}{X, \cE_{B}(\frn)\ot\om_{X}}$, which visibly consists of nilpotent Higgs fields.

Now we prove the other inclusion $\cN_{X}\subset\cN^{\Eis}_{X}$. Let $(\cE,\ph)\in \cN_{X}\subset T^{*}\Bun_{G}$, where $\ph\in\cohog{0}{X,\cE(\frg^{*})\ot\om_{X}}$ is a nilpotent Higgs field. Restricting to the geometric generic point $\ov\eta$ of $X$, and choosing trivializations of $\cE_{\ov\eta}$ and $\om_{X,\ov\eta}$, $\ph_{\ov\eta}$ gives a nilpotent element in the Lie algebra $\frg_{\ov\eta}$. By the Jacobson-Morozov theorem, there is a canonical parabolic subgroup $P$ of $G$ defined over the residue field $k(\ov\eta)$, such that $\ph_{\ov\eta}$ lies in its nilpotent radical. The canonicity of $P$  insures one can descend its conjugacy class to the generic point of $X$ and give a canonical reduction $\cE_{P}$ of $\cE$ (first at the generic point then extended to the whole curve uniquely) such that
\begin{equation*}
\ph\in\cohog{0}{X,\cE_{P}(\frn_{P})\ot\om_{X}}
\end{equation*}
Here $\frn_{P}$ is the nilpotent radical of $\Lie(P)$. Now let $\cE_{B}$ be an arbitrary reduction of $\cE_{P}$ to $B$, then $\ph$ lies in $\cohog{0}{X,\cE_{B}(\frn)\ot\om_{X}}$ as well, and hence $(\cE,\ph)\in \cN^{\Eis}_{X}$.
\end{proof}

\sss{A family of curves} 
Let $\pi: \frX\to S$ be a smooth projective family of curves over a smooth base $S$. Let $\frp: \Bun_{B}(\pi)\to \Bun_{G}(\pi)$ be the induction map. Note that $\Bun_{B}(\pi)\to S$ is also smooth. We have the  correspondence
\begin{equation}\label{Lag frp}
\xymatrix{T^{*}\Bun_{B}(\pi) & \Bun_{B}(\pi)\times_{\Bun_{G}(\pi)}T^{*}\Bun_{G}(\pi)\ar[l]_-{d\frp}\ar[r]^-{\frp^{\na}} & T^{*}\Bun_{G}(\pi)
}
\end{equation}

\begin{defn}\label{d:univ nilp cone} Define the universal nilpotent cone $\wt\cN_{\Pi}\subset T^{*}\Bun_{G}(\pi)$ to be
\begin{equation*}
\wt\cN_{\Pi}=\orr{\frp}(0_{\Bun_{B}(\pi)})=\frp^{\na}(d\frp)^{-1}(0_{\Bun_{B}(\pi)}).
\end{equation*}
\end{defn}

\begin{theorem}\label{th:Eis cone closed} The universal nilpotent cone $\wt\cN_{\Pi}$ is a closed conic Lagrangian inside $T^{*}\Bun_{G}(\pi)$.  The natural projection
\begin{equation*}
\xymatrix{
P_{\Pi}:  T^{*}\Bun_{G}(\pi)\ar[r] &  T^{*}_{\Pi} =  T^{*}\Bun_{G}(\pi)/\frp^*(T^*S)
}
\end{equation*}
  maps $\wt\cN_{\Pi}$  bijectively to the relative nilpotent cone $\cN_{\Pi}$.
\end{theorem}
\begin{proof}
First let's check that  $P_{\Pi}$ induces a bijection $\wt\cN_{\Pi}\isom \cN_{\Pi}$. This can be checked on geometric fibers over $S$.  For $b\in S(\CC)$, let $\frX_{b}$ denote the fiber $\pi^{-1}(b)$. We show that $P_{\Pi}$ induces a bijection
\begin{equation}\label{ppi b}
\xymatrix{
P_{\Pi, b}: \wt\cN_{\Pi}|_{\Bun_{G}(\frX_{b})}\ar[r]^-\sim &  \cN^{\Eis}_{\frX_{b}}
}\end{equation}
Both $\Bun_{G}(\pi)$ and $\Bun_{B}(\pi)$ are smooth over $S$. When restricted to a point $\cE_{B}\in \Bun_{B}(\frX_{b})$, the differential $d\frp$ fits into exact sequences
\begin{equation*}
\xymatrix{      0\ar[r] & T^{*}_{b}S\ar@{=}[d] \ar[r] & T^{*}_{\cE}\Bun_{G}(\frX/S) \ar[d]^{d\frp} \ar[r] &  T^{*}_{\cE}\Bun_{G}(\frX_{b})\ar[d]^{dp_{b}}\ar[r] & 0\\
0\ar[r] & T^{*}_{b}S \ar[r] & T^{*}_{\cE_{B}}\Bun_{B}(\frX/S) \ar[r] &  T^{*}_{\cE_{B}}\Bun_{B}(\frX_{b})\ar[r] & 0}
\end{equation*}
Therefore $d\frp$ and $dp_{b}$ have the same kernel. This implies \eqref{ppi b}. Since $\cN^{\Eis}_{\frX_{b}}=\cN_{\frX_{b}}$ by Lemma \ref{l:two cones}, we conclude that $P_{\Pi}$ induces a bijection $\wt\cN_{\Pi}\isom \cN_{\Pi}$ fiber by fiber.

Next we show that   $\wt\cN_{\Pi}$ is closed. For this, fix a coweight $\l \in \XX_*(H)$, and consider the family version of the Drinfeld relative compactification $\ov\Bun^{\l}_{B}(\pi)$ of $B$-bundles whose induced $H$-bundle has degree $\l$.
By definition, 
 $\ov\Bun^{\l}_{B}(\pi)$ classifies $(b,\cE, \cL, i_{\mu}: \cL(\mu)\to \cE(V_{\mu}))$ where $b\in S$ is a point, $\cE$ is a $G$-bundle on $\frX_{b}$, $\cL$ is an $H$-bundle over $\frX_{b}$ of degree $\l$, and $i_{\mu}$ is a family of injective maps $\cL(\mu)\to \cE(V_{\mu})$, where $\cL(\mu)$ is the line bundle associated to $\cL$ and a dominant weight $\mu\in \xch(H)$, and $\cE(V_{\mu})$ is the vector bundle on $\frX_{b}$ associated to $\cE$ and the irreducible representation $V_{\mu}$ of $G$ with highest weight $\mu$. The maps $\{i_{\mu}\}$ are required to satisfy the Pl\"ucker relations. By the same argument of \cite[Proposition 1.2.2]{BG}, the natural map
\begin{equation*}
\xymatrix{
\ov\frp^{\l}: \ov\Bun^{\l}_{B}(\pi)\ar[r] &  \Bun_{G}(\pi)
}
\end{equation*}
is representable and proper.

We next define  a closed subset $\cW^{\l}\subset \ov\Bun^{\l}_{B}(\pi)\times_{\Bun_{G}(\pi)}T^{*}\Bun_{G}(\pi)$. Let $2\r$ be the sum of the positive roots, and $n$  the number of positive roots of $G$. Then $\wedge^{n}(\frg)$ has highest weight $2\r$ with multiplicity one, hence it contains $V_{2\r}$ with multiplicity one, i.e., a canonical up to scalar map $\io: V_{2\r}\incl \wedge^{n}(\frg)$. For a point $(b,\cE, \cL, \{i_{\mu}\})$,  consider the composition
\begin{equation*}
\xymatrix{\t: \cL(2\r)\ar[r]^-{i_{2\r}} & \cE(V_{2\r})\ar[r]^-{\io} & \cE(\wedge^{n}(\frg))=\wedge^{n}\cE(\frg)}
\end{equation*}
We define $\cW^{\l}$ to consist of those $(b,\cE, \cL, \{i_{\mu}\}, \ph)\in \ov\Bun^{\l}_{B}(\pi)\times_{\Bun_{G}(\pi)}T^{*}\Bun_{G}(\pi)$, so $\ph\in \cohog{0}{\frX_{b}, \cE(\frg)\ot\om_{\frX_{b}}}$ is a Higgs field, such that the map
\begin{equation}\label{wedge ph}
\xymatrix{\om_{\frX_{b}}^{-1}\ot \cL(2\r)\ar[r]^-{\ph\ot \t} & \cE(\frg)\ot \wedge^{n}\cE(\frg)\ar[r]^-{\wedge} &  \wedge^{n+1}\cE(\frg)}
\end{equation}
is zero. This is clearly closed in $\ov\Bun^{\l}_{B}(\pi)\times_{\Bun_{G}(\pi)}T^{*}\Bun_{G}(\pi)$. Since $\ov\frp^{\l}$ is proper, the image of $\cW^{\l}$ under
\begin{equation*}
\xymatrix{
\ov\frp^{\l,\na}: \ov\Bun^{\l}_{B}(\pi)\times_{\Bun_{G}(\pi)}T^{*}\Bun_{G}(\pi)\ar[r] &  T^{*}\Bun_{G}(\pi)
}\end{equation*}
is also closed.  Denote the image of $\cW^{\l}$ under $\ov\frp^{\l,\na}$ by $\ov\cW^{\l}$. 

We claim that $\wt\cN_{\Pi}$ is the union of $\ov\cW^{\l}$ for all $\l\in \xcoch(H)$. 

First we show $\wt\cN_{\Pi}\subset \cup_{\l}\ov \cW^{\l}$. Let $\cW^{\l,\c}\subset \Bun^{\l}_{B}(\pi)\times_{\Bun_{G}(\pi)}T^{*}\Bun_{G}(\pi)$ be the restriction of $\cW^{\l}$ to $\Bun^{\l}_{B}(\pi)$. For a point $(b,\cE_{B}, \ph)\in \cW^{\l,\c}$ (so that $\cL$ is the $H$-bundle $\cE_{B}/N$), the map $\t:\cL(2\r)\to \wedge^{n}\cE(\frg)$ has image equal to $\cE_{B}(\frn)$. Therefore the condition that \eqref{wedge ph} be zero is equivalent to that $\ph\in \cohog{0}{\frX_{b},\cE_{B}(\frn)\ot\om_{\frX_{b}}}$, i.e., $d\frp(\ph)=0$.  Let $\ov\cW^{\l,\c}=\frp^{\na}\cW^{\l,\c}$. Then  $\wt\cN_{\Pi}=\cup_{\l}\ov \cW^{\l,\c}$ by  definition, hence $\wt\cN_{\Pi}\subset \cup_{\l}\ov \cW^{\l}$.

Conversely, we show that $\ov \cW^{\l}\subset \wt\cN_{\Pi}$ for every $\l$. If $(b,\cE, \cL, \{i_{\mu}\}, \ph)\in \cW^{\l}$,  then $\{i_{\mu}\}$ gives a $B$-reduction $\cE_{\y, B}$ of the generic fiber $\cE_{\y}$ ($\y\in \frX_{b}$ is the generic point), such that the image of $i_{2\r}$ at the generic point is $\wedge^{n}(\cE_{\y, B}(\frn))\subset \cE_{\y}(\frg)$. Let $\cE_{B}$ be the unique $B$-reduction of $\cE$ that is equal to $\cE_{\y,B}$ at the generic point of $\frX_{b}$. Therefore the vanishing of \eqref{wedge ph} at the generic point implies that the image of $\ph_{\y}$ lies in $\cE_{B}(\frn)\ot\om_{\frX_{b}}$, hence $(b,\cE,\ph)\in \wt\cN_{\Pi}$ by the definition of the latter. This proves the asserted containment.

For any finite type open subset $\cU\subset \Bun_{G}(\pi)$, the restriction $\cN_{\Pi}|_{\cU}$ is constructible, hence covered by the image of finitely many $\ov \cW^{\l}$. Since $\wt\cN_{\Pi}|_{\cU}$ maps bijectively to $\cN_{\Pi}|_{\cU}$, $\wt\cN_{\Pi}|_{\cU}$ is covered by  finitely many $\ov \cW^{\l}$. Each $\ov\cW^{\l}|_{\cU}$ is closed in $T^{*}\cU$, therefore $\wt\cN_{\Pi}|_{\cU}$ is closed in $T^{*}\cU$. This being true for all finite type $\cU$, we conclude that $\wt\cN_{\Pi}$ is closed.

Finally, since $\wt\cN_{\Pi}$  is the transport of the zero section of $T^{*}\Bun_{B}(\pi)$ under the Lagrangian correspondence \eqref{Lag frp}, it is conic and isotropic in $T^{*}\Bun_{G}(\pi)$ by Lemma \ref{l:pres iso}. Moreover, $\cN_{\Pi}$ is of pure dimension equal to $\dim\Bun_{G}(\pi)$. Since $\wt\cN_{\Pi}\to \cN_{\Pi}$ is bijective, we conclude that $\dim\wt\cN_{\Pi}=\dim\cN_{\Pi}=\dim\Bun_{G}(\pi)$.  By Lemma \ref{l:Lag dim},  $\wt\cN_{\Pi}$ is Lagrangian. This finishes the proof.
\end{proof}

\begin{cor}\label{c:univ cone non-char} The universal nilpotent cone $\wt\cN_{\Pi}$ is non-characteristic with respect to the projection $\Pi:\Bun_{G}(\pi)\to S$, i.e. its intersection with  $\Pi^*(T^*S) \subset T^* \Bun_{G}(\pi)$ lies in the zero-section.
\end{cor}

\begin{proof}
By the theorem, $\wt\cN_{\Pi}$ projects isomorphically to $\cN_{\Pi}$, in particular maps injectively, under the quotient by $\Pi^*(T^*S)$.
\end{proof}

\subsection{Eisenstein cone with level structure}\label{ss:Eis cone level}
In this subsection we extend the construction of the universal nilpotent cone to the moduli of bundles with level structures. As before, it helps to first describe the global nilpotent cone for a single curve $X$ in the case of Iwahori level structure.

\sss{Iwahori level structure on a single curve}\label{sss:univ cone B}
Let $X$ be a smooth projective curve, and $\Sigma\subset X$ finitely many closed points where we impose $B$ or $N$ reductions on $G$-bundles.
 
Let $\Bun_{G,1}(X,\Sigma)$ (resp. $\Bun_{B,1}(X,\Sigma)$) be the moduli of $G$-bundles (resp. $B$-bundles) on $X$ with a trivialization at $\Sigma$. Since $\Bun_{G,1}(X,\Sigma)\to \Bun_{G}(X)$ (resp. $\Bun_{B,1}(X,\Sigma)\to \Bun_{B}(X)$) is a $G^\Sigma$-torsor (resp. $B^{\Sigma}$-torsor), the cotangent bundle $T^{*}\Bun_{G,1}(X,\Sigma)$ (resp. $T^{*}\Bun_{B,1}(X,\Sigma)$) is $G^\Sigma$-equivariant (resp. $B^{\Sigma}$-equivariant) hence descends to a vector bundle $\Om_{G, X,\Sigma}$ (resp. $\Om_{B, X,\Sigma}$) over $\Bun_{G}(X)$ (resp. $\Bun_{B}(X)$), whose fiber at $\cE_G$  (resp. $\cE_{B}$) is 
 $\cohog{0}{X, \cE_{G}(\frg^{*})\ot\om_{X}(\Sigma)}$  (resp. $\cohog{0}{X, \cE_{B}(\frb^{*})\ot\om_{X}(\Sigma)}$).

 Let $p_{\Sigma}: \Bun_{B,1}(X,\Sigma)\to \Bun_{G,1}(X,\Sigma)$ be the induction map, giving rise to  the correspondence
\begin{equation}\label{Om}
\xymatrix{\Om_{B,X,\Sigma} & \Bun_{B}(X)\times_{\Bun_{G}(X)}\Om_{G,X,\Sigma}\ar[l]_-{dp_{\Sigma}}\ar[r]^-{p^{\na}_{\Sig}} & \Om_{G,X,\Sigma}}
\end{equation}
Define 
\begin{equation*}
\un\cN^{\Eis}_{X,\Sigma}=p^{\na}_{\Sig}((dp_{\Sigma}^{-1}(0_{\Bun_{B}(X)})))\subset \Om_{G,X,\Sigma}
\end{equation*}
We have the natural map $q: T^{*}\Bun_{G,B}(X,\Sigma)\to \Om_{G, X,\Sigma}$ given by the differential of the projection
$\Bun_{G,1}(X,\Sigma)\to \Bun_{G,B}(X,\Sigma)$.
Set
\begin{equation*}
\cN^{\Eis, B}_{X,\Sigma}=q^{-1}(\un\cN^{\Eis}_{X,\Sigma})
\end{equation*}

Now we have the following generalization of Lemma \ref{l:two cones}.

\begin{lemma}\label{l:two cones B} 
The subset $\cN^{\Eis}_{X, \Sigma}$ of $T^{*}\Bun_{G, B}(X, \Sigma)$ coincides with the global nilpotent cone $\cN^B_{X, \Sigma}$, i.e.~the zero-fiber of the Hitchin map. Moreover, $\cN^{\Eis}_{X}$ is a closed conic Lagrangian of $T^{*}\Bun_{G, B}(X, \Sigma)$.
\end{lemma}
\begin{proof}
The equality $\cN^{\Eis,B}_{X,\Sigma}=\cN^{B}_{X,\Sigma}$ is proved by the same argument as Lemma \ref{l:two cones}. We leave further details to the reader.  In particular, this shows $\cN^{\Eis,B}_{X,\Sigma}$ is closed and conic.

Next we prove that $\cN^{\Eis,B}_{X,\Sigma}$ is isotropic. We may reformulate the definition of $\cN^{\Eis,B}_{X,\Sigma}$ as follows. We have an isomorphism
\begin{equation*}
\Bun_{G,B}(X,\Sigma)\times_{\Bun_{G}(X)}\Bun_{B}(X)\simeq \Bun_{B,1}(X,\Sig)\twtimes{B^{\Sigma}}(G/B)^{\Sigma}.
\end{equation*}
Let $\a: \Bun_{B,1}(X,\Sig)\twtimes{B^{\Sigma}}(G/B)^{\Sigma}\to \Bun_{G,B}(X,\Sigma)$ be the resulting projection.  Consider the commutative diagram
\begin{equation}\label{sq Om}
\xymatrix{T^{*}(\Bun_{B,1}(X,\Sig)\twtimes{B^{\Sigma}}(G/B)^{\Sigma})\ar[d]^{q_{B}} & \Bun_{B}(X)\times_{\Bun_{G}(X)}T^{*}\Bun_{G,B}(X,\Sigma)\ar[d]^{\wt q}\ar[l]_-{d\a}\ar[r]^-{\a^{\na}} & T^{*}\Bun_{G,B}(X,\Sigma)\ar[d]^{q}\\
\Om_{B,X,\Sigma} & \Bun_{B}(X)\times_{\Bun_{G}(X)}\Om_{G,X,\Sigma}\ar[l]_-{dp_{\Sigma}}\ar[r]^-{p^{\na}_{\Sig}} & \Om_{G,X,\Sigma}}
\end{equation}
Here the top row is the Lagrangian correspondence for $p^{B}_{\Sig}$ and the bottom row is \eqref{Om}. The map $q_{B}$ is induced from the differential of the map $\Bun_{B,1}(X,\Sig)\to \Bun_{B,1}(X,\Sig)\twtimes{B^{\Sigma}}(G/B)^{\Sigma}$ (setting the  $(G/B)^{\Sig}$ factor to be $1$). Note the right square of \eqref{sq Om} is Cartesian. By Lemma \ref{l:Cart Lag} we have
\begin{equation*}
\cN^{\Eis,B}_{X,\Sig}=q^{-1}p^{\na}_{\Sig}dp_{\Sig}^{-1}(0_{\Bun_{B}(X)})=\a^{\na}\wt q^{-1} dp_{\Sig}^{-1}(0_{\Bun_{B}(X)})=\a^{\na}d\a^{-1}q_{B}^{-1}(0_{\Bun_{B}(X)})=\orr{\a}q_{B}^{-1}(0_{\Bun_{B}(X)}).
\end{equation*}
We have another projection
\begin{equation*}
\b: \Bun_{B,1}(X,\Sigma)\twtimes{B^{\Sigma}}(G/B)^{\Sigma}\to (B\bs G/B)^{\Sigma}
\end{equation*}
and 
\begin{equation*}
q_{B}^{-1}(0_{\Bun_{B}(X)})=\oll{\b}(T^{*}(B\bs G/B)^{\Sigma}).
\end{equation*}
Therefore
\begin{equation*}
\cN^{\Eis, B}_{X,\Sigma}=\orr{\a}\oll{\b}(T^{*}(B\bs G/B)^{\Sigma}).
\end{equation*}
The entire classical cotangent bundle $T^{*}(B\bs G/B)^{\Sigma}$ is isotropic,  so $\cN^{\Eis, B}_{X,\Sigma}$ is also isotropic by Lemma \ref{l:pres iso}. 

Finally let us use the equality $\cN^{\Eis,B}_{X,\Sigma}=\cN^{B}_{X,\Sigma}$ to show that $\cN^{\Eis,B}_{X,\Sigma}$ is Lagrangian. In view of Lemma \ref{l:Lag dim}, it suffices to show that $\cN^{B}_{X,\Sigma}$ contains a dense open with  dimension  $\ge\dim\Bun_{G,B}(X,\Sigma)$ (then equality must hold since $\cN^{B}_{X,\Sigma}$ is isotropic). This follows from the fact that the Hitchin base has the same dimension as $\Bun_{G,B}(X,\Sigma)$.
\end{proof}

\begin{cor}\label{c:NN} The global nilpotent cone $\cN=\cN^{N}_{X,\Sigma}$ defined in \S\ref{sss:glob nilp cone} is a closed conic Lagrangian in $T^{*}\Bun_{G,N}(X,\Sigma)$.
\end{cor}
\begin{proof}
Let $p: \Bun_{G,N}(X,\Sigma)\to \Bun_{G,B}(X,\Sigma)$ be the projection. We have $\cN^{N}_{X,\Sigma}=\oll{p}\cN^{B}_{X,\Sigma}$. Since $p$ is  smooth, $\oll{p}$ takes a closed conic Lagrangian to a closed conic Lagrangian. 
\end{proof}

\sss{Family version}
Now consider the family version. Let $\s=\{\s_{a}\}_{a\in \Sigma}$ be a finite collections of sections of $\pi:\frX\to S$. As above, introduce the rigidified moduli  $\Bun_{G,1}(\pi, \s)$ (resp. $\Bun_{B,1}(\pi, \s)$), and let 
$\Om_{G,\Pi,\s}$  (resp. $\Om_{B,\Pi,\s}$) denote the descent of its cotangent bundle to $\Bun_{G}(\pi)$ (resp. $\Bun_{B}(\pi)$). 
Let $\frp_{\s}: \Bun_{B,1}(\pi, \s)\to \Bun_{G,1}(\pi, \s)$ be the induction map, giving rise to the correspondence
\begin{equation*}
\xymatrix{ \Om_{B,\Pi,\s} & \Bun_{B}(\pi)\times_{\Bun_{G}(\pi)}\Om_{G,\Pi,\s}\ar[l]_-{d\frp_{\s}}\ar[r]^-{\frp^{\na}}  & \Om_{G,\Pi,\s} 
}
\end{equation*}
Define
\begin{equation*}
\wt{\un\cN}_{\Pi,\s}=\frp^{\na}((d\frp_{\s}^{-1}(0_{\Bun_{B}(\pi)})))\subset \Om_{G,\Pi,\s}.
\end{equation*}
We have the natural maps
\begin{equation*}
\xymatrix{
\frq: T^{*}\Bun_{G,B}(\pi, \sigma)\ar[r] &  \Om_{G, \Pi,\sigma}
&
\frq': T^{*}\Bun_{G,N}(\pi, \sigma)\ar[r] &  \Om_{G, \Pi,\sigma}
}\end{equation*}
 given by the differentials of the respective projections
$\Bun_{G,1}(\pi,\sigma)\to \Bun_{G,B}(\pi,\sigma)$
and $\Bun_{G,1}(\pi,\sigma)\to \Bun_{G,N}(\pi,\sigma)$.

\begin{defn}\label{def:univ cone Eis}
\begin{enumerate}
\item The  universal nilpotent cone $\wt\cN^{B}_{\Pi, \s}\subset T^{*}\Bun_{G,B}(\pi,\s)$ is defined to be $\frq^{-1}(\wt{\un\cN}_{\Pi,\s})$.
\item The universal nilpotent cone $\wt\cN^{N}_{\Pi, \s}\subset T^{*}\Bun_{G,N}(\pi,\s)$ 
is defined to be $(\frq')^{-1}(\wt{\un\cN}_{\Pi,\s})$.
%
\end{enumerate}
\end{defn}

\begin{remark}\label{rem: cone B vs N}
It is elementary to check: for the natural $H^\Sigma$-torsor $q:\Bun_{G,N}(\pi,\sigma)\to \Bun_{G,B}(\pi,\sigma)$
and $(G/B)^\Sigma$-fibration $r:\Bun_{G,B}(\pi,\sigma)\to \Bun_{G}(\pi)$, we have 
$\oll{q}(\wt\cN^{B}_{\Pi, \s}) = \wt\cN^{N}_{\Pi, \s}$,
$\orr{q}(\wt\cN^{N}_{\Pi, \s}) = \wt\cN^{B}_{\Pi, \s}$,
$\orr{r}(\wt\cN^{B}_{\Pi, \s}) = \wt\cN_{\Pi}$.
\end{remark}

Let $T^{*}_{\Pi,\s}$ be the relative cotangent bundle of $\Pi:\Bun_{G,B}(\pi,\s)\to S$, and $P_{\Pi,\s}: T^{*}\Bun_{G,B}(\pi,\s)\to T^{*}_{\Pi,\s}$ the natural map.

\begin{theorem}\label{th:univ cone Iw} The universal nilpotent cone $\wt\cN^{B}_{\Pi, \s}$ is a closed conic Lagrangian inside $T^{*}\Bun_{G,B}(\pi,\s)$. The map $P_{\Pi,\s}$ restricts to a bijection $\wt\cN^{B}_{\Pi, \s}\isom \cN_{\Pi,\s}$. Analogous  statements are true for $\wt\cN^{N}_{\Pi,\s}$.
\end{theorem}
\begin{proof}The proof is similar to that of Theorem \ref{th:Eis cone closed}. We will prove the case of $B$; the $N$ version can be proved similarly or deduced from Remark~\ref{rem: cone B vs N}.

First, Lemma \ref{l:two cones B} implies $P_{\Pi,\s}$ restricts to a bijection $\wt\cN^{B}_{\Pi, \s}\isom \cN_{\Pi,\s}$. 

Next, we show that $\wt\cN^{B}_{\Pi, \s}$ is closed. It suffices to show that $\wt{\un\cN}_{\Pi,\s}$ is closed in $\Om_{G,\Pi,\s}$. For this, we consider the analogous closed subset $\cW^{\l}_{\s}\subset \ov\Bun_{B}(\pi)\times_{\Bun_{G}(\pi)}\Om_{G,\Pi,\s}$ (only changing $\om_{\frX_{b}}$ to $\om_{\frX_{b}}(\s_{b})$) and note its image $\ov\cW^{\l}_{\s}\subset \Om_{G,\Pi,\s}$ is  closed. Since $\wt{\un\cN}_{\Pi,\s}$ is the union of the $\ov\cW^{\l}_{\s}$, the same argument as in Theorem \ref{th:Eis cone closed} proves that $\wt{\un\cN}_{\Pi,\s}$ is closed.

Finally, to show that  $\wt\cN^{B}_{\Pi, \s}$ is Lagrangian, we use the same argument as in Lemma \ref{l:two cones B} to rewrite $\wt\cN^{B}_{\Pi, \s}$ using the maps
\begin{equation*}
\xymatrix{(B\bs G/B)^{\Sigma} & \Bun_{B,1}(\pi,\s)\twtimes{B^{\Sigma}}(G/B)^{\Sigma} \ar[l]_-{\b}\ar[r]^-{\a} & \Bun_{G,B}(\pi,\s)}
\end{equation*}
and we have
\begin{equation}\label{alt NB}
\wt\cN^{B}_{\Pi, \s}=\orr{\a}\oll{\b}(T^{*}(B\bs G/B)^{\Sigma}).
\end{equation}
Lemma \ref{l:pres iso} shows that $\wt\cN^{B}_{\Pi, \s}$ is isotropic. Since $\dim \wt\cN^{B}_{\Pi, \s}=\dim \cN^{B}_{\Pi, \s}=\dim\Bun_{G,B}(\pi,\s)$, we conclude that $\wt\cN^{B}_{\Pi, \s}$ is Lagrangian in $T^{*}\Bun_{G,B}(\pi,\s)$ by Lemma \ref{l:Lag dim}. 
\end{proof}

\sss{Compatibility with base change}
Let $\th: S'\to S$ be an arbitrary map from another smooth stacks $S'$ and let $\pi':\frX'=\frX\times_{S}S'\to S'$ be the base change of $\pi$. Let $\s=\{\s_{a}\}_{a\in \Sigma}$ be a collection of disjoint sections of $\pi: \frX\to S$, and let $\s'$ be its base change to $S'$. Note the natural identification
$\Bun_{G,B}(\pi,\s)\times_{S}S'\simeq\Bun_{G,B}(\pi',\s')$. Let
\begin{equation*}
\xymatrix{
\th_{G,B}: \Bun_{G,B}(\pi',\s')\ar[r] &  \Bun_{G,B}(\pi,\s)
}\end{equation*}
be the natural projection. 


\begin{prop}\label{p:univ cone base change} In the above situation, we have
\begin{equation*}
\wt\cN^{B}_{\Pi',\s'}=\oll{\th_{G,B}}(\wt\cN^{B}_{\Pi,\s}).
\end{equation*}
A similar equality holds when $B$ is replaced by $N$. 
\end{prop}
\begin{proof} We will prove the case of $B$; the $N$ version can be proved similarly or deduced from Remark~\ref{rem: cone B vs N}

When there are no sections $\s$, the assertion follows by applying Lemma \ref{l:Cart Lag} to the Cartesian diagram
\begin{equation*}
\xymatrix{\Bun_{B}(\pi')\ar[r]^{\frp'}\ar[d]^{\th_B} & \Bun_{G}(\pi')\ar[d]^{\th_G}\\
\Bun_{B}(\pi)\ar[r]^{\frp} & \Bun_{G}(\pi)}
\end{equation*}
When there are sections $\s$, adding trivializations along the sections in the above diagram, 
so replacing $\Bun_{B}(\pi)$ by $\Bun_{B,1}(\pi,\s)$, $\Bun_{G}(\pi)$ by $\Bun_{G,1}(\pi,\s)$, etc.,
we get
\begin{equation*}
\wt{\un\cN}_{\Pi',\s'}=\th^{\na}_{\Om}(d\th_{\Om})^{-1}(\wt{\un\cN}_{\Pi,\s}).
\end{equation*}
where 
\begin{equation*}
\xymatrix{\Om_{G,\Pi',\s'} & S'\times_{S}\Om_{G,\Pi,\s} \ar[r]^-{d\th_{\Om}}\ar[l]_-{\th_{\Om}^{\na}} & \Om_{G,\Pi,\s}}.
\end{equation*}
We have a commutative diagram
\begin{equation}\label{OmT*}
\xymatrix{ T^{*}\Bun_{G,B}(\pi',\s')\ar[d]^{\frq'} & S'\times_{S}T^{*}\Bun_{G,B}(\pi,\s) \ar[r]^-{\th_{G,B}^{\na}}\ar[l]_-{d\th_{G,B}} \ar[d]^{\frq''} & T^{*}\Bun_{G,B}(\pi,\s)\ar[d]^{\frq}\\
\Om_{G,\Pi',\s'} & S'\times_{S}\Om_{G,\Pi,\s} \ar[r]^-{d\th_{\Om}}\ar[l]_-{\th_{\Om}^{\na}}  & \Om_{G,\Pi,\s} }
\end{equation}
Therefore
\begin{eqnarray*}
\wt\cN^{B}_{\Pi',\s'}&=&\frq'^{-1}(\wt{\un\cN}_{\Pi',\s'})=\frq'^{-1}\th^{\na}_{\Om}(d\th_{\Om})^{-1}(\wt{\un\cN}_{\Pi,\s})\\
&=&\th_{G,B}^{\na}\frq''^{-1}(d\th_{\Om})^{-1}(\wt{\un\cN}_{\Pi,\s})=\th_{G,B}^{\na}(d\th_{G,B})^{-1}\frq^{-1}(\wt{\un\cN}_{\Pi,\s})\\
&=&\oll{\th_{G,B}}(\wt\cN^{B}_{\Pi,\s})
\end{eqnarray*}
where we use that the left square in \eqref{OmT*} is Cartesian to get the third equality. \end{proof}

\sss{Real variant}\label{sss:real univ cone} 
Elsewhere in the paper, we will base change families of curves and hence bundles to real analytic bases. We comment here how the preceding constructions naturally extend to this setting.

We start as before with a smooth projective  family $\pi:\frX\to S$ of curves over a smooth base stack $S$.  Suppose $S'\in\frR_{/S}$ is a smooth real analytic space over $S$. Let $\pi': \frX_{S'}\to S'$  be the base change of $\pi$ and $\Pi': \Bun_{G}(\pi')\to S'$  the base change of $\Pi: \Bun_{G}(\pi)\to S$ along $\th$. More precisely, using notation of \S\ref{sss:real bc},  we define $\frX_{S'}=\pi^{\#}S'\in\frR_{/\frX}$, and $\Bun_{G}(\pi')=\Pi^{\#}S'\in\frR_{/\Bun_{G}(\pi)}$.

Observe that Definition \ref{d:univ nilp cone} naturally extends to the setting of  $\Pi'$: we define $\wt\cN_{\Pi'}$ to be the transport of the zero section of $T^{*}\Bun_{B}(\pi')$ to $T^{*}\Bun_{G}(\pi')$. The same argument as Proposition \ref{p:univ cone base change} shows that $\wt\cN_{\Pi'}$ is also the transport of $\wt\cN_{\Pi}$ under the Lagrangian correspondence attached to $\th_{G}: \Bun_{G}(\pi')\to \Bun_{G}(\pi)$.  
Moreover, the same argument as Theorem \ref{th:Eis cone closed} shows that $\wt\cN_{\Pi'}$ is a closed conic Lagrangian in $T^{*}\Bun_{G}(\pi')$ and the projection to the relative global nilpotent cone $\cN_{\Pi}$ is a bijection. 

Similar remarks apply in the case of $B$ or $N$-reductions along sections $\s = \{\s_{a}\}_{a\in \Sigma}$. 

In conclusion,  all results from this section apply to the real analytic family $\pi': \frX'\to S'$ obtained from an algebraic family by base change.

\subsection{Nilpotent sheaves and Hecke stability}
The goal of this subsection is to show that sheaves with singular support in the universal nilpotent cone are stable under Hecke functors. This is a generalization of \cite[Theorem 5.2.1]{NY}.

\sss{Nilpotent sheaves}\label{sss:two real bases} Consider a family of curves $\pi:\frX\to S$ as in the beginning of the section. Let $M',M\in\frR_{/S}$ be smooth real analytic spaces over $S$ with a map $\th$
\begin{equation*}
\xymatrix{M' \ar[r]^{\th} &M\ar[r] & S}
\end{equation*}
Let $\pi_{M}:\frX_{M}:=\pi^{\#}M\to M$, $\pi_{M'}: \frX_{M'}:=\pi^{\#}M' \to M'$ be the base changes of $\pi$, and $\Pi_{M}: \Bun_{G}(\pi_{M})=\Pi^{\#}M\to M$ and $\Pi_{M'}: \Bun_{G}(\pi_{M'})=\Pi^{\#}M'\to M'$  be the base changes of $\Pi$. For notations see \S\ref{sss:real bc}.

Let $\Th: \Bun_{G}(\pi_{M'})\to \Bun_{G}(\pi_{M})$ be the map in $\frR_{/\Bun_{G}(\pi)}$ induced from $\th$.

We use abbreviated notation
\begin{equation*}
\Sh_{\cN}(\Bun_{G}(\pi_{M})):=\Sh_{\wt\cN_{\Pi_{M}}}(\Bun_{G}(\pi_{M}))
\end{equation*}
where $\wt\cN_{\Pi_{M}}$ is the real version of the universal nilpotent cone for $\pi_{M}: \frX_{M}\to M$ (see  \S\ref{sss:real univ cone}). Similarly define $\Sh_{\cN}(\Bun_{G}(\pi_{M'}))$.

\begin{lemma}\label{l:nilp pullback} In the above situation,  $\Th^{*}$ sends $\Sh_{\cN}(\Bun_{G}(\pi_{M}))$ to $\Sh_{\cN}(\Bun_{G}(\pi_{M'}))$.
\end{lemma}
\begin{proof}By passing to smooth covers in $\frR_{/S}$ of $M$ (e.g., $M_{V,v}$ for $(V,v)\in\Sch^{sm}_{/S}$), we reduce to the case where $M'$ and $M$ are both real manifolds. 

We claim that $\Th$ is non-characteristic with respect to $\wt\cN_{\Pi_{M}}$ in the sense of \cite[Definition 5.4.12]{KS}. Indeed, in the Lagrangian correspondence
\begin{equation*}
\xymatrix{T^{*}\Bun_{G}(\pi_{M'}) & (T^{*}\Bun_{G}(\pi_M))\times_{M}M'\ar[l]_-{d\Th}\ar[r]^-{\Th^{\na}} & T^{*}\Bun_{G}(\pi_M)}
\end{equation*}
The kernel of $d\Th$ is contained in the pullback of $T^{*}M$. By Corollary \ref{c:univ cone non-char}, $\wt\cN_{\Pi_{M}}\cap \Pi_{M}^{*}(T^{*}M)$ is the zero section, hence $\Th^{\na,-1}\wt\cN_{\Pi_{M}}$ also intersects the pullback of $T^{*}M'$ via $d\Th$ in the zero section. Therefore $\Th$ is non-characteristic with respect to $\wt\cN_{\Pi_{M}}$.

By \cite[Proposition 5.4.13]{KS}, for $\cF\in \Sh_{\cN}(\Bun_{G}(\pi_{M'}))$, $\ssupp(\Th^{*}\cF)$ is contained in $\oll{\Th}(ss(\cF))\subset \oll{\Th}(\wt\cN_{\Pi_{M}})$, which is $\wt\cN_{\Pi_{M'}}$ by the real version of Proposition \ref{p:univ cone base change}.
\end{proof}

In particular, take $M'$ to be a point $\{s\}$ in $M$, we get a restriction functor
\begin{equation}\label{res Ms}
i_{s}^{*}: \Sh_{\cN}(\Bun_{G}(\pi_{M}))\to \Sh_{\cN}(\Bun_{G}(\frX_{s}))
\end{equation}
where on the right side $\cN$ denotes the usual global nilpotent cone of $\Bun_{G}(\frX_{s})$. 

\begin{conj}\label{c:res eq} Suppose $M$ is a contractible real manifold, then the restriction functor \eqref{res Ms} is an equivalence.
\end{conj}

\begin{remark}Intuitively, this conjecture says that the categories $\Sh_{\cN}(\Bun_{G}(\frX_{s}))$ vary locally constantly with $s$. This is expected from the Betti geometric Langlands conjecture because the spectral side of the conjecture depends only on the underlying topological surface.
\end{remark}

\sss{Spherical Hecke functor} 
Now suppose $S'$ is another stack with a map $\th_{S}: S'\to S$ that lifts to $\frX$ 
\begin{equation*}
\xymatrix{ & \frX\ar[d]^{\pi}\\
S'\ar[r]^{\th_{S}}\ar[ur]^{\s} & S}
\end{equation*}
Attached to this datum there is a Hecke correspondence 
\begin{equation*}
\xymatrix{  & \Hk^{\Sph}_{\s} \ar[rr]^-{\inv^{\Sph}_{\s}}\ar[dr]^{p_{2}}\ar[dl]_{p_{1}}  && \frac{G\tl{z}\bs G\lr{z}/G\tl{z}}{\Aut(\CC\tl{t})}\\
\Bun_{G}(\pi) &  & \Bun_{G}(\pi')  
}
\end{equation*}
which classifies $(s', \cE_{1}, \cE_{2}, \t)$ where $s'\in S'$, $\cE_{1},\cE_{2}$ are $G$-torsors over $X_{\th(s')}$ and  $\t$ is an isomorphism between  $\cE_{1}$ and $\cE_{2}$ over $X_{\th(s')}\bs\{\s(s')\}$.  The map $\inv^{\Sph}_{\s}$ records the relative position of $\cE_{1}$ and $\cE_{2}$ near $\s$.

Consider the spherical Hecke category 
\begin{equation*}
\cH^{\Sph}=\Sh\left(\frac{G\tl{z}\bs G\lr{z}/G\tl{z}}{\Aut(\CC\tl{t})}\right)
\end{equation*}
For $\cK\in \cH^{\Sph}$, we have the Hecke functor
\begin{equation*}
H^{\Sph}_{\s,\cK}=p_{2!}(p^{*}_{1}(-)\ot \inv^{\Sph,*}_{\s}\cK): \Sh(\Bun_{G}(\pi))\to \Sh(\Bun_{G}(\pi'))
\end{equation*}

\begin{theorem}\label{th:sph Hk pres} For any $\cK\in \cH^{\Sph}$, the Hecke functor $H^{\Sph}_{\s,\cK}$ sends $\Sh_{\cN}(\Bun_{G}(\pi))$ to $\Sh_{\cN}(\Bun_{G}(\pi'))$. 
\end{theorem}

\begin{remark} We explain how Theorem \ref{th:sph Hk pres} recovers \cite[Theorem 5.2.1]{NY}. The latter says that for a single curve $X$ and Hecke modification $H_{\cK}: \Sh(\Bun_{G}(X))\to \Sh(X\times \Bun_{G}(X))$ for a moving point in $X$, then  $H_{\cK}$ sends $\Sh_{\cN}(\Bun_{G}(X))$ to $\Sh_{X\times \cN}(X\times \Bun_{G}(X))$ (i.e., singular support of the result is contained in the product of the zero section of $T^{*}X$ and $\cN\subset T^{*}\Bun_{G}(X)$).

To recover this statement, we only need to apply Theorem \ref{th:sph Hk pres} to the constant curve $\frX=X$ over $S=\pt$ and $S'=X, \s=\id_{X}:S'=X\to X$. Note that $\Bun_{G}(\pi')=X\times \Bun_{G}(X)$ and its universal nilpotent cone is $X\times \cN$ by the base change property in Proposition \ref{p:univ cone base change}.
\end{remark}

\sss{Real version}
Theorem \ref{th:sph Hk pres} has a real analytic version as follows. Let $M'\in\frR_{/S'}$, $M\in\frR_{/S}$ and $(\th_{S}, \th_{M}): (S', M')\to (S, M)$  be a map in $\frR_{\Sta}$. Then the map $\s:S'\to \frX$ induces a map $\s_{M}: M'\to \frX_{M}$. Define $\Hk^{\Sph}_{\s_{M}}$ to be the base change of $M'$ along $\Hk^{\Sph}_{\s}\to S'$. Then Hecke functors 
\begin{equation*}
H_{\s_{M},\cK}^{\Sph}: \Sh(\Bun_{G}(\pi_{M}))\to \Sh(\Bun_{G}(\pi_{M'}))
\end{equation*}
are defined. 

\begin{theorem}\label{th:sph Hk pres real} In the above situation, for any $\cK\in \cH^{\Sph}$, the Hecke functor $H^{\Sph}_{\s_{M},\cK}$ sends $\Sh_{\cN}(\Bun_{G}(\pi_{M}))$ to $\Sh_{\cN}(\Bun_{G}(\pi_{M'}))$. 
\end{theorem}
If we take $\cK$ to be the monoidal unit in $\cH^{\Sph}$, $H_{\s_{M},\cK}^{\Sph}$ is the same as $\Th^{*}$ considered in \S\ref{sss:two real bases}. Therefore this theorem generalizes Lemma \ref{l:nilp pullback}. 

We prove this theorem by first reducing to the Iwahori version (in \S\ref{sss:red to Iw}) and then prove the Iwahori version in \S\ref{sss:proof Iw}.

\sss{Affine Hecke category with universal monodromy}\label{sss:aff Hk}
We shall state a version of Theorem \ref{th:sph Hk pres} with Iwahori level structure. For this we need to define a version of the monodromic affine Hecke category where the monodromy is allowed to vary.

Let $G\lr{z}$ be the loop group of $G$ and $\bI$ (resp. $\bI^{\circ}$) be the preimage of $B\subset G$ (resp. $N\subset G$) under the reduction mod $z$ map $G\tl{z}\to G$. 

Recall the affine Hecke category is the monoidal category
\begin{equation*}
\cH=\Sh_{\bimon}(\bI^{\circ}\bs G\lr{z}/\bI^{\circ})
\end{equation*}
under convolution. Here the subscript $\bimon$ means we are taking only sheaves that are monodromic under both the left and the right translation under $H=B/N$ on $\bI^{\circ}\bs G\lr{z}/\bI^{\circ}$. The monoidal structure is defined by the usual convolution diagram
\begin{equation*}
\xymatrix{ & \bI^{\circ}\bs G\lr{z}\times^{\bI^{\circ}} G\lr{z}/\bI^{\circ} \ar[dl]_{p_{1}}\ar[dr]^{p_{2}}\ar[rr]^-{m}&& \bI^{\circ}\bs G\lr{z}/\bI^{\circ}\\
\bI^{\circ}\bs G\lr{z}/\bI^{\circ} && \bI^{\circ}\bs G\lr{z}/\bI^{\circ}}
\end{equation*}
with the formula
\begin{equation*}
\cK_{1}\star\cK_{2}:=m_{!}(p_{1}^{*}\cK_{1}\otimes p_{2}^{*}\cK_{2})[2\dim H].
\end{equation*}
Recall from \S\ref{sss:Sh0H} the universal local system $\cL_{\univ}$ on $H$. Let $\d: H\simeq \bI/\bI^{\c}\incl G\lr{z}/\bI^{\circ}$ be the natural map. Then $\d_{*}\cL_{\univ}\in \cH$ is an identity object for the monoidal structure on $\cH$. 

The monoidal category $\Sh_{0}(H)$ of local systems on $H$ can be identified with a full monoidal subcategory of $\cH$ of objects supported on $\bI$. Left and right convolution with $\Sh_{0}(H)$ gives $\cH$ the structure of a $\Sh_{0}(H)\ot \Sh_{0}(H)$-module.

For each $w\in \tilW$, fix a lifting $\dot w\in N_{G\lr{z}}(T)$. Let $i_{w}: \bI^{\c}\bs \bI \dot{w} \bI/\bI^{\c}\incl \bI^{\c}\bs G\lr{z}/\bI^{\circ}$ be the inclusion. The map $h_{\dot w}: H\to\bI^{\c}\bs \bI \dot{w} \bI/\bI^{\c}$ mapping $x\mapsto x\dot w$ is a gerbe for a pro-unipotent group. Therefore the universal local system $\cL_{\univ}$ descends to a local system $\cL_{\univ}(w)$ on $\bI^{\c}\bs \bI \dot{w} \bI/\bI^{\c}$. Let $\wh\D(w):=i_{w!}\cL_{\univ}( w)[\ell(w)]$. We have $\d\cong \wh\D(1)$. The objects $\{\wh\D(w)\}_{w\in W}$ generate $\cH$ under colimits.

\sss{Iwahori version}\label{sss:aff Hk action}
Assume $\s: S\to \frX$ is a section, and we consider $\Bun_{G,N}(\pi,\s)\to S$ with $N$-reductions along the section.  Again we simply write
\begin{equation*}
\Sh_{\cN}(\Bun_{G,N}(\pi,\s)):=\Sh_{\wt\cN^{N}_{\Pi,\s}}(\Bun_{G,N}(\pi,\s)).
\end{equation*}
Suppose we are given a trivialization of the formal neighborhood of $\frX$ along $\s$:
\begin{equation*}
\frX^{\wedge}_{/\s}\isom \Spf(\CC\tl{z})\wh\times S. 
\end{equation*}

Consider the diagram for affine Hecke modification along $\s$
\begin{equation}\label{aff Hk corr}
\xymatrix{  & \Hk_{\s} \ar[rr]^-{\inv_{\s}}\ar[dr]^{q_{2}}\ar[dl]_{q_{1}}  && \bI^{\c}\bs G\lr{z}/\bI^{\c}\\
\Bun_{G,N}(\pi,\s) &  & \Bun_{G,N}(\pi,\s)  
}
\end{equation}
For $\cK\in \cH$, its action on $\Sh_{\cN}(\Bun_{G,N}(\pi,\s))$ is given by (using the diagram \eqref{aff Hk corr})
\begin{equation*}
H_{\s,\cK}=q_{2!}(q_{1}^{*}(-)\ot \inv^{*}_{\s}\cK): \Sh(\Bun_{G,N}(\pi,\s))\to \Sh(\Bun_{G,N}(\pi,\s)).
\end{equation*}

If $M\in\frR_{/S}$ is smooth, all the above constructions can be base changed to $M$. We denote the resulting Hecke stack by $\Hk_{\s_{M}}$, and Hecke functors by
$$H_{\s_{M},\cK}: \Sh(\Bun_{G,N}(\pi_{M},\s_{M}))\to \Sh(\Bun_{G,N}(\pi_{M},\s_{M})).$$

\begin{theorem}\label{th:aff Hk pres} In the above situation, for any $\cK\in \cH$, the functor $H_{\s_{M},\cK}$ preserves $\Sh_{\cN}(\Bun_{G,N}(\pi_{M},\s_{M}))$.
\end{theorem}

Here we stated the theorem in the case of Iwahori level structure along a single section. There is an obvious extension of the theorem (with the same proof) for finitely many disjoint sections, where the affine Hecke functors are acting by modifying along one of the sections.

\sss{Theorem \ref{th:aff Hk pres} implies Theorem \ref{th:sph Hk pres} and Theorem \ref{th:sph Hk pres real}}\label{sss:red to Iw} To save notation we will deduce Theorem \ref{th:sph Hk pres}  from Theorem \ref{th:aff Hk pres}. The argument works equally well for the real version Theorem \ref{th:sph Hk pres real}.

Let $\Th: \Bun_{G}(\pi')\to\Bun_{G}(\pi)$ be the map induced from $\th_{S}$. We may factor $p_{1}$ as the composition
\begin{equation*}
p_{1}: \Hk_{\s}^{\Sph}\xr{p'_{1}} \Bun_{G}(\pi')\xr{\Th} \Bun_{G}(\pi)
\end{equation*}
By Lemma \ref{l:nilp pullback}, for $\cF\in \Sh_{\cN}(\Bun_{G}(\pi))$, $\th^{*}_{G}\cF\in \Sh_{\cN}(\Bun_{G}(\pi'))$. If we work with the family $\pi': \frX'\to S'$ together with the section $\s'=(\s,\th):S'\to \frX'$, and the associated Hecke functor $H^{\Sph}_{\s',\cK}$ for this family, we have $H^{\Sph}_{\s,\cK}(\cF)=H^{\Sph}_{\s',\cK}(\th_{G}^{*}(\cF))$. Therefore we may reduce the general case to the case $S'=S$ and $\s$ is a section of $\pi$. Zariski locally we may trivialize the formal neighborhood of $\s$ in $\frX$. We are in exactly the same setting as the Iwahori version \S\ref{sss:aff Hk action}. 

Let $r: \Bun_{G,N}(\pi,\s)\to \Bun_{G}(\pi)$ be the projection. For $\cK\in\cH^{\Sph}$, let $\wt\cK\in \cH$ be the pullback via the projection $\bI^{\c}\bs G\lr{z}/\bI^{\c}\to (G\tl{z}\bs G\lr{z}/G\tl{z})/\Aut(\CC\tl{z})$. Then for $\cF\in \Sh(\Bun_{G}(\pi))$ 
we have
\begin{equation*}
H_{\s,\wt\cK}(r^{*}\cF) \simeq r^{*}H^{\Sph}_{\s, \cK}(\cF) \otimes H^*_c(G/N)
\end{equation*}
Since $v$ is smooth, we have $\ssupp(r^{*}\cF)=\oll{r}(\ssupp(\cF))$. By the construction of $\wt\cN_{\Pi,\s}$ as the transform of $\wt{\un\cN}_{\Pi,\s}\subset \Om_{G,\Pi,\s}$, we have $\oll{r}\wt\cN_{\Pi}\subset \oll{r}\wt{\un\cN}_{\Pi,\s}=\wt\cN^{N}_{\Pi,\s}$ since $\wt\cN_{\Pi}\subset \wt{\un\cN}_{\Pi,\s}$. Therefore, if $\ssupp(\cF)\subset \wt\cN_{\Pi}$, then $\ssupp(r^{*}\cF)\subset \wt\cN^{N}_{\Pi,\s}$. By Theorem \ref{th:aff Hk pres}, $\ssupp(H_{\s,\wt\cK}(r^{*}\cF))\subset \wt\cN^{N}_{\Pi,\s}$, hence $\oll{r}\ssupp(H^{\Sph}_{\s, \cK}(\cF))=\ssupp(r^{*}H^{\Sph}_{\s, \cK}(\cF))\subset \wt\cN^{N}_{\Pi,\s}$. Since $r$ is surjective, $\ssupp(H^{\Sph}_{\s, \cK}(\cF))=\orr{r}\oll{r}\ssupp(H^{\Sph}_{\s, \cK}(\cF))\subset \orr{r}\wt\cN^{N}_{\Pi,\s}=\wt\cN_{\Pi}$ (the last equality is contained in Remark \ref{rem: cone B vs N}).  We have shown that $H^{\Sph}_{\s, \cK}$ preserves nilpotent sheaves.

\sss{Proof of Theorem \ref{th:aff Hk pres}}\label{sss:proof Iw} To simplify notation, we will write down the argument for $M=S$. The argument works equally well for any smooth $M\in\frR_{/S}$. Also we will write $\Bun_{G,B}$ for $\Bun_{G,B}(\pi,\s)$, etc.

We first reduce the proof to a calculation of transport of Lagrangians. Since $\wh\D(w)$ generate $\cH$ for various $w$, and for $\ell(w_{1})+\ell(w_{2})=\ell(w_{1}w_{2})$ it is easy to see that $\wh\D(w_{1})\star \wh\D(w_{2})\cong\wh\D(w_{1}w_{2})$, the objects $\wh\D(s)$ for simple reflections $s\in W_{\aff}$ and $\wh\D(\om)$ for $\ell(\om)=0$ generate $\cH$ monoidally. Therefore it suffices to check $H_{\s, \wh\D(s)}$ and  $H_{\s,\wh\D(\om)}$ preserve nilpotent sheaves. 

The action of $H_{\s,\wh\D(s)}$ ($s$ is a simple reflection in $\tilW$) can be expressed as follows. Let $\bP_{s}\subset G\lr{z}$ be the standard parahoric corresponding to $s$. Consider the moduli stack $\Bun_{G,\bP_{s}}$ classifying $G$-bundles along the fibers of $\pi$ together with $\bP_{s}$-level structures along $\s$. Let $\Hk_{\le s}:=\Bun_{G,N}\times_{\Bun_{G,\bP_{s}}}\Bun_{G,N}$ be the preimage of $\bI^{\c}\bs \bP_{s}/\bI^{\c}$ in $\Hk_{\s}$ under $\inv_{\s}$. 

Note that $H$ acts on $\Bun_{G,N}$ by changing the $N$-reductions along $\s$. Let $H_{>0}$ be the identity component of $H(\RR)$. Let $\Bun_{G,NH_{>0}}=\Bun_{G,N}(\s,\pi)/H_{>0}$ as a real analytic stack. Let $\ov\Hk_{\le s}=\Bun_{G,NH_{>0}}\times_{\Bun_{G,\bP_{s}}}\Bun_{G,NH_{>0}}$. We get a diagram
\begin{equation*}
\xymatrix{& \ov\Hk_{\le s}\ar[rr]^-{\inv_{\le s}}\ar[dr]^{ q_{2,\le s}}\ar[dl]_{  q_{1,\le s}}  && \bI^{\c}H_{>0}\bs \bP_{s}/\bI^{\c}H_{>0}\\
\Bun_{G,NH_{>0}} &  & \Bun_{G,NH_{>0}}  
}
\end{equation*}
Note that $ q_{1,\le s}$ and $q_{2,\le s}$ are now proper with fibers $\bP_{s}/\bI^{\c}H_{>0}$ (a fiber bundle over $\PP^{1}$ with fibers the compact form of $H(\CC)$). We have $ \ov\Hk_{\le s}=\ov \Hk_{s}\cup \ov \Hk_{1}$ according to the decomposition $\bP_{s}=\bI s\bI\cup \bI$. Let $i: \ov \Hk_{1}\incl \ov \Hk_{\le s}$ and $j: \ov \Hk_{s}\incl \ov \Hk_{\le s}$ be the closed and open embeddings. Let $q_{i,1}$ and $q_{i,s}$ be the restriction of $q_{i,\le s}$ to $\ov\Hk_{1}$ and $\ov \Hk_{s}$ respectively, for $i=1,2$.

All objects $\cK\in \cH$ descend to $\bI^{\c}H_{>0}\bs G\lr{z}/\bI^{\c}H_{>0}$ by monodromicity. We denote the descent still by $\cK$. Also, since all objects  $\cF\in \Sh_{\cN}(\Bun_{G,NH_{>0}} )$ are $H$-monodromic, they descend to $\Bun_{G,NH_{>0}}$ which we still denote by $\cF$. If $\cK$ is supported on $\bI^{\c}\bs \bP_{s}/\bI^{\c}$, then  its singular support is contained in the union of the zero section and the conormal of the unit coset $\bI^{\c}\bs\bI/\bI^{\c}$. Since $q_{1,s}$ and $q_{1,s}$ are both smooth, $q_{1,\le s}^{*}\cF\ot\inv_{\le s}^{*}\cK$ is contained in the union 
\begin{equation*}
\oll{q}_{1,s}(\wt\cN^{N}_{\Pi,\s})\cup \orr{i}\oll{q}_{1,1}(\wt\cN^{N}_{\Pi,\s}).
\end{equation*}
Since $q_{2,\le s}$ is proper, $H_{\cK}(\cF)=q_{2,\le s *}(q_{1,\le s}^{*}\cF\ot\inv_{\le s}^{*}\cK)$ is contained in
\begin{equation*}
\orr{q}_{2,\le s}\oll{q}_{1,\le s}(\wt\cN^{N}_{\Pi,\s})\cup \orr{q}_{2,\le s}\orr{i}\oll{q}_{1,1}(\wt\cN^{N}_{\Pi,\s})\subset T^{*}\Bun_{G,NH_{>0}}.
\end{equation*}
Note $q_{2,1}=q_{2,\le s}\c i$, hence $\orr{q}_{2,\le s}\orr{i}=\orr{q}_{2,1}$. Therefore, to show the above is contained in $\wt\cN^{N}_{\Pi,\s}$,  it suffices to show
\begin{eqnarray}
\label{Hk s Lag}\orr{q}_{2,\le s}\oll{q}_{1,\le s}(\wt\cN^{N}_{\Pi,\s})\subset \wt\cN^{N}_{\Pi,\s}\\
\label{Hk 1 Lag}\orr{q}_{2,1}\oll{q}_{1,1}(\wt\cN^{N}_{\Pi,\s})\subset \wt\cN^{N}_{\Pi,\s}.
\end{eqnarray}

The same argument applies to a length zero element $\om\in \tilW$ instead of $s$, and it implies that, $H_{\s,\wh\D(\om)}$ preserves nilpotent sheaves if we can show
\begin{equation}\label{Hk om Lag}
\orr{q}_{2,\om}\oll{q}_{1,\om}(\wt\cN^{N}_{\Pi,\s})\subset \wt\cN^{N}_{\Pi,\s}.
\end{equation}

Now we show \eqref{Hk s Lag}, and the arguments for \eqref{Hk 1 Lag} and \eqref{Hk om Lag} are simpler and will be omitted.  Since $\wt\cN^{N}_{\Pi,\s}$ is the transport of $\wt\cN^{B}_{\Pi,\s}$ from $T^{*}\Bun_{G,B}$, it suffices to prove the analogue of \eqref{Hk s Lag} for $\Bun_{G,B}$. Namely in the diagram
\begin{equation*}
\xymatrix{\Bun_{G,B} & \Bun_{G,B}\times_{\Bun_{G,\bP_{s}}}\Bun_{G,B}\ar[l]_-{r_{1}}\ar[r]^-{r_{2}} & \Bun_{G,B}}
\end{equation*}
we need to show
\begin{equation}\label{Hk s NB}
\orr{r}_{2}\oll{r}_{1}(\wt\cN^{B}_{\Pi,\s})\subset \wt\cN^{B}_{\Pi,\s}.
\end{equation}

To compute $\orr{r}_{2}\oll{r}_{1}(\wt\cN^{B}_{\Pi,\s})$ we will use the description of $\wt\cN^{B}_{\Pi,\s}$ in the proof of Theorem \ref{th:univ cone Iw}, i.e., the equality \eqref{alt NB}. For $w\in W$, let $\Bun_{B,1}(w)=\Bun_{B,1}\times^{B}(BwB/B)\subset \Bun_{B,1}\times^{B}G/B$. Since $T^{*}(B\bs G/B)$ is the union of conormals of $B\bs BwB/B$,  $\oll{\b}T^{*}(B\bs G/B)\subset T^{*}(\Bun_{B,1}\times^{B}G/B)$ is the union of the conormals of $\Bun_{B,1}(w)$ for $w\in W$. Hence, if we write $\a_{w}: \Bun_{B,1}(w)\to \Bun_{G,B}$ for the restriction of $\a$, we can rewrite \eqref{alt NB} as
\begin{equation}\label{NB aw}
\wt\cN^{B}_{\Pi,\s}=\bigcup_{w\in W}\orr{\a}_{w}(0_{\Bun_{B,1}(w)}).
\end{equation}
Now fix $w_{1}$ and let $w_{2}\in W$. Let $\wt C(w_{1},w_{2})$ be the moduli of  $(y,\cF_{1},\cF_{2}, \t, \cE_{1}, \cE_{2}, \cE_{1,B}, \cE_{2,B})$ where $y\in S$, $\cF_{i}$ are $B$-bundles over $\frX_{y}$, $\cE_{i}$ the $G$-bundle associated to $\cF_{i}$, $\cE_{i,B}$ a $B$-reduction of  $\cE_{i}|_{\s(y)}$ that is in relative position  $w_{i}$ with $\cF_{i}|_{\s(y)}$, and $\t$ is an isomorphism between $\cF_{1}|_{\frX_{y}\bs \s(y)}$ and  $\cF_{2}|_{\frX_{y}\bs \s(y)}$ such that the induced rational isomorphism of $(\cE_{1}, \cE_{1,B})$ and $(\cE_{2}, \cE_{2,B})$ near $\s(y)$ has relative position $s$. We have a natural map $\wt c_{1,w_{2}}: \wt C(w_{1},w_{2})\to \Bun_{B,1}(w_{1})$ by recording $(y,\cF_{1},\cE_{1},\cE_{1,B})$. Similarly we have $\wt c_{2,w_{2}}: \wt C(w_{1},w_{2})\to \Bun_{B,1}(w_{2})$. Let $\wt c_{1}: \coprod_{w_{2}\in W}\wt C(w_{1},w_{2})\to \Bun_{B,1}(w_{1})$  be the disjoint union of $\wt c_{1,w_{2}}$ and similarly define $\wt c_{2}$. We consider the following commutative diagram
\begin{equation*}
\xymatrix{ & \coprod_{w_{2}\in W}\wt C(w_{1},w_{2})\ar[d]^{\nu}\ar[dl]_{\wt c_{1}}\ar[dr]^{\wt c_{2}=\coprod_{}\wt c_{2,w_{2}}}\\
\Bun_{B,1}(w_{1}) \ar[d]^{\a_{w_{1}}}& C(w_{1}) \ar[d]^{\b_{w_{1}}}\ar[l]_-{c_{1}} &  \coprod_{w_{2}\in W}\Bun_{B,1}(w_{2})\ar[d]^{\a':=\coprod\a_{w_{2}}}\\
\Bun_{G,B} & \Bun_{G,B}\times_{\Bun_{G,\bP_{s}}}\Bun_{G,B} \ar[l]_-{r_{1}}\ar[r]^-{r_{2}}& \Bun_{G,B}
}
\end{equation*}
Here $C(w_{1})$ is defined to make the lower left square Cartesian. Since $r_{1}$ is a $\PP^{1}$-fibration, so is $c_{1}$. The map $\nu$ is injective on $\CC$-points  (on the level of $\CC$-points, given $(y,\cF_{1}, \cE_{1}, \cE_{1,B}, \cE_{2}, \cE_{2,B})$ and an isomorphism of $\cE_{1}$ and $\cE_{2}$ over $\frX_{y}\bs \s(y)$, $\cF_{1}$ then gives a $B$-reduction of $\cE_{2}|_{\frX_{y}\bs \s(y)}$, which saturates to a unique $B$-reduction of $\cE_{2}$). Moreover, the images of $\nu$ for various $w_{2}$ form a partition of $C(w_{1})$. From this we see that the fibers of $\wt c_{1,w_{2}}$ and $\wt c_{2,w_{2}}$ are locally closed subsets of $\PP^{1}$; moreover, both $\wt c_{1}$ and $\wt c_{2}$ are \'etale locally trivial fibrations. Therefore, up to passing to reduced structures of $\wt C(w_{1},w_{2})$, both maps $\wt c_{1,w_{2}}$ and $\wt c_{2,w_{2}}$ are smooth of relative dimension $1$ or $0$. Let us denote the reduced structure of  $\wt C(w_{1},w_{2})$ by the same notation, so they are smooth.

Using \eqref{NB aw} and Lemma \ref{l:Cart Lag}, we have
\begin{equation*}
\orr{r}_{2}\oll{r}_{1}(\wt\cN^{B}_{\Pi,\s})=\orr{r}_{2}\oll{r}_{1}\orr{\a}_{w_{1}}(0_{\Bun_{B,1}(w_{1})})=\orr{r}_{2}\orr{\b}_{w_{1}}\oll{c}_{1}(0_{\Bun_{B,1}(w_{1})})=\orr{r}_{2}\orr{\b}_{w_{1}}(0_{C(w_{1})}).
\end{equation*}
Since $\nu$ is surjective, the zero section of $T^{*}C(w_{1})$ is contained in $\orr{\nu}(\coprod 0_{\wt C(w_{1},w_{2})})$ (union over $w_{2}\in W$). Therefore
\begin{eqnarray*}
\orr{r}_{2}\oll{r}_{1}(\wt\cN^{B}_{\Pi,\s})&=&\orr{r}_{2}\orr{\b}_{w_{1}}(0_{C(w_{1})})\subset\orr{r}_{2}\orr{\b}_{w_{1}}\orr{\nu}(\coprod 0_{\wt C(w_{1},w_{2})})\\
&=&\orr{\a'}\orr{\wt c}_{2}(\coprod_{w_{2}}0_{\wt C(w_{1},w_{2})})=\bigcup_{w_{2}\in W}\orr{\a}_{w_{2}}\orr{\wt c}_{2,w_{2}}(0_{\wt C(w_{1},w_{2})}).
\end{eqnarray*}
Since $\wt c_{2,w_{2}}$ is smooth, $\orr{\wt c}_{2,w_{2}}(0_{\wt C(w_{1},w_{2})})$ is the zero section of $T^{*}\Bun_{B,1}(w_{2})$, hence the right side above is contained in union of $\orr{\a}_{w_{2}}(0_{\Bun_{B,1}(w_{2})})$, which is $\wt\cN^{B}_{\Pi,\s}$ by \eqref{NB aw}. This proves \eqref{Hk s NB}, and finishes the proof of Theorem \ref{th:aff Hk pres}.
\qed



\section{Hecke actions via bubbling: diagrams}\label{s:diagram}


In this section we define a semi-cosimplicial category that encodes the sheaf categories on $\Bun_{G}$ of various fibers of a cosimplicial bubbling and functors between them. To do this we use the formalism of correspondence of Gaitsgory and Rozenblyum \cite{GR} to first define a diagram of categories parametrized by a combinatorially defined $2$-category of correspondences, and then restrict the structure to get the desired semi-cosimplicial category. 


\subsection{Diagram category}\label{ss:diag cat}

\sss{The category $\cC$}
For $n\geq -1$, set  $[n] =\{0, \ldots, n\}$, so by convention $[-1] = \vn$.

We define a category $\cC$ as follows. Objects are given by pairs $(n, S)$ of an integer $n\geq -1$ and a subset $S\subset [n]$. Morphisms $\varphi:(n, S) \to (n', S')$ are order-preserving inclusions $\varphi:[n]\to [n']$ such that (i) $\ph(S) \subset S'$ and (ii) $[n'] \setminus \ph([n]) \subset S'$. Compositions are the evident compositions of maps.

The motivation for introducing $\cC$ is to encode natural maps between coordinate subspaces of various affine spaces $\AA^{[n]}$. See \S\ref{sss:alg base}.

\subsubsection{Open and closed morphisms}\label{sss:open closed morphisms}

A morphism $\varphi: (n, S) \to (n', S')$ in $\cC$ is called {\em closed} if $n'=n$ (so that the $S\subset S'\subset [n]$). The morphism $\varphi$ is called {\em open} if $\ph^{-1}(S')=S$ (i.e., $S'=\ph(S)\cup ([n'] \setminus \varphi([n]))$, so all of the change in $n$ is accounted for by the change in $S$). Let $\frO$ (resp. $\frC$) be the set of open (resp. closed) morphisms in $\cC$. The notation stands for $\frC$ = ``closed" and $\frO$ = ``open" as will be relevant  below.

 
 Note any morphism $\varphi:(n, S) \to (n', S')$ in $\cC$ can be factored 
$$
\xymatrix{
\varphi:(n, S) \ar[r]^-{\varphi_{\frC}} & (n, \varphi^{-1}(S')) \ar[r]^-{\varphi_\frO} &  (n', S')
}
$$
where $\varphi_\frC \in \frC$ and $\varphi_\frO \in \frO$. 
 
Moreover, 
every morphism in $\frC$ and in $\frO$ can be written as the composition of basic morphisms of the same type.
For $\frC$, we have the basic morphisms $\varphi:(n, S) \to (n, S')$ where $S' \setminus S$ is a singleton.
For $\frO$, we have the basic morphisms $\varphi:(n, S) \to (n+1, S')$ where $S' \setminus S$ is a singleton.

%

Finally, we record the following elementary observation whose proof is left to the reader.
It confirms basic properties of the set $\frC$ needed to consider certain correspondences in $\cC$.

\begin{lemma}\label{l:pushout}
\begin{enumerate}
\item For any $\varphi: (n, S) \to (n, S')$ in the set $\frC$ and arbitrary morphism $\psi: (n, S) \to (n'', S'')$, we have a pushout diagram
\begin{equation}\label{pushout C1}
\xymatrix{
\ar[d]_-\psi (n, S) \ar[r]^-\varphi &  (n, S')\ar[d]^-{\wt \psi} \\
(n'', S'') \ar[r]^-{\wt \varphi} &  (n'', T)
}
\end{equation}
where $T = \psi(S' )\cup_{\psi(S)} S''$, and $\wt \varphi \in \frC$.   Moreover,  if $\psi \in \frC$, then $\wt \psi \in \frC$.


\item The class of morphisms $\frC$ satisfies the ``2 out of 3'' property. More precisely, for any $\varphi: (n, S) \to (n', S')$, $\varphi': (n', S') \to (n'', S'')$, if $\varphi'$, $\varphi' \circ \varphi \in \frC$, so that $n'' = n$ and $n' = n$, then $n' = n$, so that $\varphi \in \frC$.
\end{enumerate}

\end{lemma}

\subsection{Geometric functors}\label{ss:geom functors}
We have engineered $\cC$ in order to  define certain functors with domain $\cC$. In this section, we introduce some functors valued in stacks and real analytic stacks.

In this subsection we fix a positive integer $k$, a base $(C,g,c_{0})$ as in Definition \ref{def:base}, and a rigidified separating nodal degeneration of curves $\t:\frY\to C$ as in Definitions \ref{def:sep nodal curve} and \ref{def:rig sep nodal curve}.

\subsubsection{Algebraic base scheme}\label{sss:alg base} Let $\dot\cC\subset \cC$ be the full subcategory where the objects are $(n,S)$ for $n\ge0$.

We begin with a simple example of a functor $A:\dot\cC \to \Sch$. 

For $n\geq 0$, we work with affine spaces $A^{[n]} = \AA^{n+1}$ with coordinates $(\e_0, \ldots, \e_n)$. In applications $A^{[n]}$ will be the same one that appears in \eqref{An SG} (obtained by taking $k$-th roots of the standard $\AA^{[n]}$).   For $(n, S)\in \dot\cC$ (hence $n\geq 0$), set $A_{(n, S)} \subset \AA^{[n]}$ 
to be the coordinate subspace  where $\e_i = 0$, for $i\not \in S$.
So in particular $A_{(n,\vn)} = \{0\}$ and $A_{(n,[n])} = \AA^{[n]}$. 


To a morphism $\varphi: (n, S) \to (n', S')$, associate the  locally closed embedding 
$$
\xymatrix{
A_\varphi: A_{(n, S)}   \ar@{^(->}[r] & A_{(n', S')}}
$$ 
given on coordinates  $(\e_i)_{i\in [n]} \mapsto (\e'_i)_{i \in [n']}$ by $\e'_{\varphi(i)} = \e_i$, and $\e'_i  = 1$, for $i \not \in \varphi([n])$.

\sss{Algebraic base stack}\label{sss:alg base stack}
Next we define a functor $\sfA: \cC\to \Sta$ by taking a quotient of $A$. Recall the torus $S\Gm^{[n]}$ that acts on $A^{[n]}$ (see \eqref{An SG}) that preserves each $A_{(n,S)}$ for any $S\subset [n]$. Let $S\Gm^{[-1]}=\{1\}$. The tori $S\Gm^{[n]}$ organize into an augmented  cosimplicial  group $S\Gm^{\D}$ as in \S\ref{sss:cosim prel}.

%

Set $\sfA_{(-1, \vn)}=\Gm$. For $(n, S)\in \cC$ with $n\geq 0$, set $\sfA_{(n, S)} = A_{(n, S)}/ S\Gm^{[n]}$. So in particular  $\sfA_{(n, \vn)} = \{0\}/S\Gm^{[n]}$ and $\sfA_{(n, [n])} = A^{[n]}/S\Gm^{[n]}$, for $n\geq 0$.

To the unique morphism $\varphi: (-1,\vn) \to (n', S')$, with necessarily $S' = [n']$, corresponds the open embedding
$$
\xymatrix{
\sfA_\varphi:\sfA_{(-1, \vn)} = \Gm \ar@{^(->}[r] &   \AA^{[n']}/S\Gm^{[n']} = \sfA_{(n', [n'])} 
}
$$ 
with image  $\e_i \not = 0$, for  $i\in [n']$. Note this $\varphi$ is in the set $\frO$ of ``open" morphisms,
and indeed $\sfA_\varphi$ is an open embedding.

For a morphism $\ph: (n,S)\to (n',S')$ in $\cC$ with $n\ge0$,  the map $A_{\ph}$ is equivariant under the natural actions of $S\Gm^{[n]}$ and $S\Gm^{[n']}$ on $A_{(n,S)}$ and $A_{(n',S')}$ respectively. Passing to quotients, we get a locally closed embedding 
$$
\xymatrix{
\sfA_\varphi: \sfA_{(n, S)} = A_{(n, S)}/ S\Gm^{[n]}  \ar@{^(->}[r] & A_{(n', S')}/ S\Gm^{[n']}  = \sfA_{(n', S')}.}
$$ 


If a morphism $\varphi: (n, S) \to (n, S')$ is  in the set $\frC$ of ``closed" morphisms, the map $\sfA_\varphi$ is the closed embedding  induced by the inclusion of a coordinate subspace in another.
If a morphism $\varphi: (n, S) \to (n', S')$  is in the set $\frO$ of ``open" morphisms, so $S' = S \cup ([n'] \setminus \varphi([n]))$, the map $\sfA_\varphi$ is  the natural open embedding with image $\e_i \not = 0$, for $i \in [n'] \setminus \varphi([n])$. In general, $\sfA_{\ph}$ is always schematic.


\begin{remark}
The stacks assigned by $\sfA$ have the following simple property. The assignments  $\sfA_{(-1, \vn)} = \Gm$ and $\sfA_{(n, [n])} \simeq A^{[n]}/S\Gm^{[n]}$, for $n\geq 0$, contain an open  substack $\Gm \simeq    \Gm^{[n]}/S\Gm^{[n]}$ whose complement consists of finitely many isomorphism classes.
All of its other assignments $\sfA_{(n, S)} =  A_{(n, S)}/ S\Gm^{[n]}$, so with $[n] \setminus S$ nonempty, consist of  finitely many isomorphism classes.

Although we can define $A_{(-1,\vn)}=\Gm$, this assignment does not extend $A$ to a functor $\cC\to \Sch$.
\end{remark}

\subsubsection{Real and positive bases}\label{sss:A plus}
We have the notion of a relative real analytic spaces over an algebraic stack over $\CC$ as introduced in \S\ref{sss:real stacks}, and the notion of subanalytic subsets in relative real analytic spaces, see \S\ref{sss:subset}. Recall from \S\ref{sss:real st cat} the category $\frR_{\Sta}$ consists of pairs $(X,M)$ where $X\in\Sta$ and $M\in\frR_{/X}$ is a real analytic space over $X$.  The subcategory $\frR^{\Sch}_{\Sta}\subset\frR_{\Sta}$ restricts morphisms $(f,f_{M}): (X,M)\to (Y,N)$ to schematic maps $f$. The category $\frR_{\Sta,\supset}$ consists of triples  $(X,M\supset M')$ where $X\in\Sta$, $M\in\frR_{/X}$ and $M'\subset M$ is a subanalytic subset. We also have $\frR^{\Sch}_{\Sta,\supset}$ with the same objects as $\frR_{\Sta,\supset}$ and schematic maps.

For each $(n,S)\in \dot\cC$ let $A^{\RR}_{(n,S)}=\RR^{[n]}\cap A_{(n,S)}$. Let  $A^{+}_{(n,S)}\subset A^{\RR}_{(n,S)}$ be the semi-algebraic quadrant defined by $\e_i \in \RR_{\geq 0}$, for $i \in [n]$, and $\e_i = 0$, for $i\not \in S$. 
So in particular $A^{+}_{(n,\vn)} = \{0\}$ and $A^{+}_{(n,[n])} = \RR^{[n]}_{\geq 0}$. For any morphism $\ph: (n,S)\to (n',S')$ in $\dot\cC$, $A_{\ph}$ sends $A^{+}_{(n,S)}$ into $A^{+}_{(n',S')}$. Therefore the assignment $(n,S)\mapsto A^{\RR}_{(n,S)}$ defines a functor $\dot\cC\to \frR$. The assignment $(n,S)\mapsto (A_{(n,S)}, A^{\RR}_{(n,S)}\supset A^{+}_{(n,S)})$ defines a functor $\dot\cC\to \frR^{\Sch}_{\Sta,\supset}$.


Similarly, we extend $\sfA:\cC\to \Sta$ to a functor $(\sfA, \sfA^{\RR}\supset\sfA^{+}): \cC\to \frR^{\Sch}_{\Sta,\supset}$ as follows.

%


Set $\sfA^{\RR}_{(-1, \vn)}=\RR^{\times}\subset A^{[0]}$, $\sfA^{+}_{(-1, \vn)} = \RR_{>0}\subset A^{[0]}$. For $(n, S)\in \cC$ with $n\geq 0$, set $\sfA^{\RR}_{(n, S)} = A^{\RR}_{(n, S)}/ S\RR_{> 0}^{[n]}$ and $\sfA^{+}_{(n, S)} = A^{+}_{(n, S)}/ S\RR_{> 0}^{[n]}$ where $S\RR_{> 0}^{[n]} \subset S\Gm^{[n]}(\RR)$ is the multiplicative group of positive real numbers $(\e_0, \ldots, \e_n) \in \RR_{>0}^{[n]}$ with $\e_0 \cdots \e_n = 1$.  The quotients are made sense of by Example \ref{ex:real quot}: viewing $S\RR_{> 0}^{[n]}$ as a subgroup of $S\Gm^{[n]}$ , hence $\sfA^{\RR}_{(n,S)}$ is a real analytic space over $\sfA_{(n,S)}=A_{(n,S)}/S\Gm^{[n]}$, and $\sfA^{+}_{(n,S)}$ is a subanalytic subset of $\sfA^{\RR}_{(n,S)}$.  In particular  $\sfA^{+}_{(n, \vn)} = \{0\}/S\RR_{> 0}^{[n]}$ and $\sfA^{+}_{(n, [n])} = \RR_{>0}^{[n]}/S\RR_{> 0}^{[n]}$, for $n\geq 0$.

To the unique morphism $\varphi: (-1,\vn) \to (n', S')$, with necessarily $S' = [n']$, associate the open embedding
$$
\xymatrix{
\sfA^{\RR}_\varphi:\sfA^{\RR}_{(-1, \vn)} = \RR\ar@{^(->}[r] &  
 \RR^{[n']}/S\RR_{> 0}^{[n']} = \sfA^{\RR}_{(n', [n'])}
}
$$ 
with image  $\e_i \not = 0$, for  $i\in [n']$. It maps $\sfA^{+}_{(-1, \vn)}$ to $\sfA^{+}_{(n', [n'])}$.

To a morphism $\varphi: (n, S) \to (n', S')$ with $n\geq 0$, associate the  locally closed embedding 
$$
\xymatrix{
\sfA^{\RR}_\varphi: \sfA^{\RR}_{(n, S)}  = A^{\RR}_{(n, S)}/ S\RR_{> 0}^{[n]} \ar@{^(->}[r] & A^{\RR}_{(n', S')}/ S\RR_{> 0}^{[n']} =  \sfA^{\RR}_{(n', S')}}
$$ 
given on coordinates  $(\e_i)_{i\in [n]} \mapsto (\e'_i)_{i \in [n']}$ by $\e'_{\varphi(i)} = \e_i$, and $\e'_i  = 1$, for $i \not \in \varphi([n])$.  It maps $\sfA^{+}_{(n, S)}$ to $\sfA^{+}_{(n', S')}$. Since $\sfA_{\ph}$ is a schematic map,  we get a functor 
$$(\sfA,\sfA^{\RR}\supset\sfA^{+}):\cC\to \frR_{\Sta,\supset}^{\Sch}.$$

\begin{remark}\label{rem:sfA}
The stacks assigned by $\sfA^+$ have the following simple property. The assignments  $\sfA^+_{(-1, \vn)} = \RR_{>0}$ and $\sfA^+_{(n, [n])} \simeq  \RR_{\geq 0}^{[n]}/S\RR_{>0}^{[n]}$, for $n\geq 0$, contain an open  $\RR_{>0} \simeq    \RR_{>0}^{[n]}/S\RR_{>0}^{[n]}$ whose complement consists of finitely many isomorphism classes.
All of its other assignments $\sfA^+_{(n, S)} =  A^+_{(n, S)}/ S\RR_{>0}^{[n]}$, so with $[n] \setminus S$ nonempty, have finitely many isomorphism classes.
Moreover, the automorphism groups of all objects  are
 contractible in all cases.
\end{remark}

\sss{The bases $B, B^{\RR}$ and $B^{+}$}
Now we introduce variants of $A, A^{\RR}$ and $A^{+}$ by incorporating the base $(C,g,c_{0})$. Recall we have fixed a positive integer $k>0$. 

Regard  $\AA^1$ as a constant functor $\cC\to  \Sta$. Consider  the natural map of functors $\Pi_k:A_{(n,S)}\subset \AA^{[n]}\to  \AA^1$ induced object-wise by the $k$th power of the product of coordinates.
Let $B: \dot\cC\to \Sta$ be the  functor  defined by the fiber product
 \begin{equation*}
  B_{(n,S)} =  A_{(n,S)} \times_{\Pi_{k}, \AA^1} C, \quad (n,S)\in\cC.
 \end{equation*} 
Thus we have  $B_{(0, \{0\})}$ is the base change of $C$ along the $k$th power map $A^{[0]}=\wt\AA^{1}\to \AA^{1}$.  For $(n, S)\in \dot\cC$, we have $B_{(n, S)} = A_{(n, S)} \times_{\AA^1} B_{(0,\{0\})}$ with an action of $S\Gm^{[n]}$ acts on the $A_{(n, S)}$ factor. 

Also let $B^\times=\AA^{1}\times_{\Pi_{k}, \AA^{1}, g}C^{\times}\subset B_{(0,\{0\})}$ be the preimage of $C^{\times}=C\bs\{c_{0}\}$. 

Let $C_{\RR}\subset g^{-1}(\RR)$ be the connected component of the real $1$-manifold $g^{-1}(\RR)$ containing $c_{0}$.  Let $C_{\ge0}= C_{\RR}\cap g^{-1}(\RR_{\ge0})$. Let $g_{\RR}: C_{\RR}\to \RR$ and $g_{\ge0}: C_{\ge0}\to \RR_{\ge0}$ be the restriction of $g$. Since $g$ is \'etale, $g_{\ge0}$ is an homeomorphism from $C_{\ge0}$ onto an interval $[0,b)\subset \RR_{\ge0}$. Let $C_{>0}=C_{\ge0}\bs\{c_{0}\}=C_{\ge0}\cap g^{-1}(\RR_{>0})$.

%

For $(n,S)\in \dot\cC$, define subanalytic subsets $B^{+}_{(n,S)}\subset B^{\RR}_{(n,S)}\subset B_{(n,S)}$ as the  fiber product
\begin{equation*}
B^{\RR}_{(n,S)}= A^{\RR}_{(n,S)} \times_{\Pi_k,\RR} C_{\RR}, \quad B^+ _{(n,S)}= A^+_{(n,S)} \times_{\Pi_k,\RR_{\geq 0}} C_{\geq 0}. 
 \end{equation*} 
Then we have $B_{\ge0}:=B^{+}_{(0,\{0\})}$ is the base change of $C_{\ge0}$ along the $k$-th power map $\RR_{\geq 0}\to \RR_{\geq 0}$. We have $B^+_{(n, S)} = A^+_{(n, S)} \times_{ \RR_{\geq 0}} B_{\geq 0}$ and $S\RR_{> 0}^{[n]}$ acts on the $A^+_{(n, S)}$ factor.

Set $B^{\RR}_{(-1,\vn)}=:B_{\RR^{\times}}\subset B_{(0,\{0\})}$ to be the preimage of $C_{\RR}\bs \{c_{0}\}$;  set $B^+_{(-1, \vn)}:=B_{>0}$ to be the preimage of $C_{>0}$. 
The assignment $(n,S)\mapsto (B_{(n,S)}, B^{\RR}_{(n,S)}\supset B^{+}_{(n,S)})$ defines a  functor 
$ \dot\cC\to \frR^{\Sch}_{\Sta,\supset}$. This functor depends on the choice of the positive integer $k>0$ though our notation does not reflect this.

\sss{The bases $\sfB, \sfB^{\RR}$ and $\sfB^{+}$}\label{sss:base B} 
Now we introduce variants of $\sfA, \sfA^{\RR}$ and $\sfA^{+}$ by incorporating the base $(C,g,c_{0})$. Recall we have fixed a positive integer $k>0$. 

Regard  $\AA^1$ as a constant functor $\cC\to  \Sta$. Consider  the natural map of functors $\Pi_k:\sfA\to  \AA^1$ induced object-wise by the $k$th power of the product of coordinates.
Let $\sfB$ be the  fiber product
 \begin{equation*}
 \xymatrix{
 \sfB =  \sfA \times_{\Pi_{k}, \AA^1} C: \cC\ar[r] &   \Sta
 }
 \end{equation*} 

Thus we have $\sfB_{(n, S)} = B_{(n, S)}/ S\Gm^{[n]}$ and $S\Gm^{[n]}$ acts on the $A_{(n, S)}$ factor in $B_{(n, S)} = A_{(n, S)} \times_{\AA^1} B_{(0,\{0\})}$.

For $(n,S)\in\cC$, define
 \begin{eqnarray*}
  \sfB^{\RR} _{(n,S)}=  \sfA^{\RR}_{(n,S)} \times_{\Pi_k,\RR} C_{\RR}\in\frR_{/\sfB_{(n,S)}},\\
 \sfB^+ _{(n,S)}=  \sfA^+_{(n,S)} \times_{\Pi_k,\RR_{\geq 0}} C_{\geq 0}\in\frR_{/\sfB_{(n,S)}}.
 \end{eqnarray*} 
Then we have $\sfB^{\RR}_{(0,\{0\})}=:B_{\RR}$,  $\sfB^{+}_{(0,\{0\})}=:B_{\ge0}$, $\sfB^{\RR}_{(-1,\vn)}=B^{\RR}_{(-1,\vn)}=B_{\RR^{\times}}$, $\sfB^+_{(-1, \vn)} = B^{+}_{(-1,\vn)}=B_{>0}$.  
For $(n, S)\in \cC$ with $n\geq 0$, we have
 $\sfB^+_{(n, S)} = B^+_{(n, S)}/ S\RR_{> 0}^{[n]}$ where $S\RR_{> 0}^{[n]}$ acts on the $A^+_{(n, S)}$ factor of $B^+_{(n, S)} = A^+_{(n, S)} \times_{ \RR_{\geq 0}}B_{\geq 0}$.

The assignments $(n,S)\mapsto (\sfB_{(n,S)}, \sfB^{\RR}_{(n,S)}\supset\sfB^{+}_{(n,S)})$ thus define a functor 
$$
\xymatrix{
(\sfB, \sfB^{\RR}\supset\sfB^{+}): \cC\ar[r] &  \frR^{\Sch}_{\Sta,\supset}.
}$$
Again this functor depends on the choice of the positive integer $k>0$.

\subsubsection{Families of curves}\label{sss:sfX} 
Let $\tau: \frY \to C$ be a separating nodal degeneration,
and  $  \pi^\Delta: \frX^\Delta \to  B^\Delta$
the associated twisted cosimplicial bubbling introduced in \S\ref{sss:tw bub}.

We will  construct a functor
$$
\xymatrix{
(\frX, \frX^{\RR}\supset\frX^{+}):\dot\cC \ar[r] &  \frR^{\Sch}_{\Sta,\supset}
}$$ fitting into a  diagram of  functors with Cartesian squares
\begin{equation*}
\xymatrix{
  \ar[d]^{\pi^{+}} \frX^+ \ar@{^{(}->}[r] & \frX^{\RR}\ar@{^{(}->}[r]\ar[d]^{\pi^{\RR}} & \frX \ar[d]^{\pi} \\
 B^+ \ar@{^{(}->}[r] & B^{\RR}\ar@{^{(}->}[r]&  B
}
\end{equation*} 
In other words, in the notation of \S\ref{sss:real bc}, $\frX^{\RR}_{(n,S)}=\pi^{\#}B^{\RR}_{(n,S)}$ and $\frX^{+}_{(n,S)}$ is the preimage of $B^{+}_{(n,S)}$ under $\pi^{\RR}_{(n,S)}$. Thus it suffices to construct the map of functors $\frX\to B: \dot\cC\to \Sta$.

Set $\frX_{(-1, \vn)} = \frX^{[0]}|_{ B^\times}$ with the  map to $B_{(-1, \vn)} = B^\times =\Gm \times_{\AA^1} B$ given by $\pi^{[0]}|_{B^\times}$.

For $(n, S)\in \cC$ with $n\geq 0$, set
$\frX_{(n, S)} = \frX^{[n]}|_{ B_{(n, S)}}$
with the  map to 
$B_{(n, S)}$
given by $\pi^{[n]}|_{B_{(n, S)}}$.
So in particular  $\frX_{(n, \vn)} = X(n)$ and $\frX_{(n, [n])} =  \frX^{[n]}$, for $n\geq 0$.

To the unique morphism $\varphi: (-1,\vn) \to (n', S')$, with necessarily $S' = [n']$, 
the cosimplicial structure of  $  \pi^\Delta: \frX^\Delta \to  B^\Delta$ 
provides  a locally closed  embedding
\begin{equation*}
\xymatrix{
\frX_\varphi:\frX_{(-1, \vn)} =  \frX^{[0]}|_{ B^\times}   
\ar@{^(->}[r] &  
\frX^{[n']}=\frX_{(n', [n'])} 
}
\end{equation*}
covering the embedding $B_\varphi: B^\times\incl B_{(n',[n'])}$.

To a morphism $\varphi: (n, S) \to (n', S')$ with $n\geq 0$, the cosimplicial structure of  $  \pi^\Delta: \frX^\Delta \to  B^\Delta$ provides a  locally closed embedding 
$$
\xymatrix{
\frX_\varphi: \frX_{(n, S)} = \frX^{[n]}|_{ B_{(n, S)}}  \ar@{^(->}[r] &  \frX^{[n']}|_{ B_{(n', S')}}=  \frX_{(n', S')}}
$$ 
covering the locally closed embedding $B_\varphi: B_{(n, S)}\incl B_{(n', S')}$. This defines a map of functors $\pi: \frX\to B:\cC\to \Sta$.

We then construct a functor 
$$
\xymatrix{
(\sfX,\sfX^{\RR}\supset\sfX^{+}):\cC \ar[r] &  \frR^{\Sch}_{\Sta,\supset}
}
$$ fitting into a  diagram of  functors with Cartesian squares (in the sense of \S\ref{sss:real bc})
\begin{equation*}
\xymatrix{
  \ar[d] \sfX^+ \ar@{^{(}->}[r]& \sfX^{\RR} \ar[d]\ar[r] & \sfX \ar[d] \\
 \sfB^+ \ar@{^{(}->}[r]& \sfB^{\RR}\ar[r] &  \sfB
}
\end{equation*} 
Thus it suffices to construct the map of functors $\sfX\to \sfB: \cC\to \Sta$. 

Set $\sfX_{(-1, \vn)} = \frX^{[0]}|_{ B^\times}$ with the  map to $\sfB_{(-1, \vn)} = B^\times$ given by $\pi^{[0]}|_{B^\times}$.

For $(n, S)\in \cC$ with $n\geq 0$, set
$\sfX_{(n, S)} = \frX^{[n]}|_{ B_{(n, S)}}/ S\Gm^{[n]}$
with the  map to 
$\sfB_{(n, S)} = B_{(n, S)}/ S\Gm^{[n]}$
given by $\pi^{[n]}|_{B_{(n, S)}}$.
So in particular  $\sfX_{(n, \vn)} = X(n)/S\Gm^{[n]}$ and $\sfX_{(n, [n])} =  \frX^{[n]}/S\Gm^{[n]}$, for $n\geq 0$.

To the unique morphism $\varphi: (-1,\vn) \to (n', S')$, with necessarily $S' = [n']$, 
the cosimplicial structure of  $  \pi^\Delta: \frX^\Delta \to  B^\Delta$ 
provides  an  open embedding
\begin{equation*}
\xymatrix{
\sfX_\varphi:\sfX_{(-1, \vn)} =  \frX^{[0]}|_{ B^\times}   
\ar@{^(->}[r] &  
\frX^{[n']}/ S\Gm^{[n']} =  \sfX_{(n', [n'])} 
}
\end{equation*}
covering the open embedding $ \sfB_\varphi$
with image  $\e_i \not = 0$, for  $i\in [n']$.

To a morphism $\varphi: (n, S) \to (n', S')$ with $n\geq 0$, the cosimplicial structure of  $  \pi^\Delta: \frX^\Delta \to  B^\Delta$ provides a  locally closed embedding 
$$
\xymatrix{
\sfX_\varphi: \sfX_{(n, S)} = \frX^{[n]}|_{ B_{(n, S)}}/ S \Gm^{[n]}  \ar@{^(->}[r] &  \frX^{[n']}|_{ B_{(n', S')}}/ S \Gm^{[n']} =  \sfX_{(n', S')}}
$$ 
covering the locally closed embedding $\sfB_\varphi$
given on coordinates  $(\e_i)_{i\in [n]} \mapsto (t'_i)_{i \in [n']}$ by $t'_{\varphi(i)} = \e_i$, and $t'_i \not =  0$, for $i \not \in \varphi([n])$.  

Note the compatibility of the above assignments gives a map of functors $\sfX\to \sfB$.


\subsubsection{Moduli of bundles}\label{sss:sfBun}

Continuing with the above setup, we will  construct here functors
\begin{eqnarray*}
(\Bun, \Bun^{\RR}\supset \Bun^{+}): \dot\cC\longrightarrow \frR^{\Sch}_{\Sta,\supset}\\
(\sfBun, \sfBun^{\RR}\supset\sfBun^{+}):\cC \longrightarrow \frR^{\Sch}_{\Sta,\supset}
\end{eqnarray*}
fitting into  Cartesian squares
\begin{equation*}
\xymatrix{\ar[d] \Bun^+ \ar@{^{(}->}[r] & \ar[d] \Bun^{\RR}\ar[r] & \Bun \ar[d] & 
\ar[d] \sfBun^+ \ar@{^{(}->}[r] & \sfBun^{\RR}\ar[d]\ar[r] & \sfBun \ar[d] \\
B^+ \ar@{^{(}->}[r] &  B^{\RR}\ar[r] & B &  \sfB^+ \ar@{^{(}->}[r] & \sfB^{\RR}\ar[r] &  \sfB
}
\end{equation*} 
Thus it suffices to construct the map of functors $\Bun\to B: \dot\cC\to\Sta$ and $\sfBun\to \sfB:\cC\to \Sta$. 

Said succinctly, we will take $\sfBun$ to assign moduli stacks of relative $G$-bundles along the fibers of $\sfX\to \sfB$. We spell this out in the below assignments. 

Set $\Bun_{(-1,\vn)}=\sfBun_{(-1, \vn)} = \Bun_G( \pi^{[0]}|_{ B^\times})$ with the evident map to $B_{(-1,\vn)}=\sfB_{(-1, \vn)} = B^\times$. In terms of the family $\t:\frY\to C$, we have
\begin{equation*}
\Bun_{(-1,\vn)}=\sfBun_{(-1, \vn)}\simeq \Gm\times_{\Pi_{k}, \Gm}\Bun_{G}(\t|_{C^{\times}}).
\end{equation*}

For $(n, S)\in \cC$ with $n\geq 0$, set
$$\Bun_{(n,S)}=\Bun_G( \pi^{[n]}|_{ B_{(n, S)}}), \quad \sfBun_{(n, S)} = \Bun_{(n,S)}/ S \Gm^{[n]}$$
with the evident maps to $B_{(n,S)}$ and $\sfB_{(n,S)}$ respectively. 
So in particular  $\sfBun_{(n, \vn)} = \Bun_G( X(n))/S\Gm^{[n]}$ and $\sfBun_{(n, [n])} = \Bun_G( \pi^{[n]})/S\Gm^{[n]}$, for $n\geq 0$.

To the unique morphism $\varphi: (-1,\vn) \to (n', S')$, with necessarily $S' = [n']$, associate the natural open embedding
\begin{equation*}
\xymatrix{
\sfBun_\varphi:\sfBun_{(-1, \vn)} =  \Bun_G( \pi^{[0]}|_{ B^\times})   
\ar@{^(->}[r] &  
\Bun_G( \pi^{[n']})/ S \Gm^{[n']} =  \sfBun_{(n', [n'])} 
}
\end{equation*}
covering the open embedding $ \sfB_\varphi$
with image  $\e_i \not = 0$, for  $i\in [n']$.

To a morphism $\varphi: (n, S) \to (n', S')$ with $n\geq 0$, associate the natural locally closed embedding 
$$
\xymatrix{
\sfBun_\varphi: \sfBun_{(n, S)} = \Bun_G( \pi^{[n]}|_{ B_{(n, S)}})/ S \Gm^{[n]}  \ar@{^(->}[r] & \Bun_G( \pi^{[n']}|_{ B_{(n', S')}})/ S \Gm^{[n']} =   \sfBun_{(n', S')}}
$$ 
covering the locally closed embedding $\sfB_\varphi$
given on coordinates  $(\e_i)_{i\in [n]} \mapsto (t'_i)_{i \in [n']}$ by $t'_{\varphi(i)} = \e_i$, and $t'_i \not =  0$, for $i \not \in \varphi([n])$.  

Note the compatibility of the above assignments gives a map of functors $\sfBun\to \sfB$.


\subsection{The correspondence category} 
We will review the main result of \cite[Ch.7, Sect. 3]{GR} on extending functors from a category (which we will take to be the category $\cC$ defined in \S\ref{ss:diag cat}) to its correspondence $2$-category. It will be be applied to prove Prop. \ref{p:sh functor from C}.

\sss{$2$-category of correspondences}
Consider the opposite category $\cC^{op}$. We will view the set $\frC$ of ``closed" morphisms just as well as morphisms in $\cC^{op}$.

We will consider certain correspondences in $\cC^{op}$ and specifically adopt the notation from \cite{GR}.

Define $\hor$ to be the set of all morphisms in $\cC^{op}$.
Define $\ver = \adm$ to be the distinguished set $\frC$ of ``closed" morphisms.
Let $\cC^{op}_\ver$ be the corresponding (non-full) subcategory of $\cC^{op}$.

We consider the correspondence 2-category $\Corr^\adm_{\ver;\hor}(\cC^{op})$ as defined in~\cite[Ch.7, Sect. 1.1]{GR}.   Thanks to Lemma~\ref{l:pushout}, the classes of morphisms $\ver,\hor$ and $\adm$ satisfy all of the hypotheses of~\cite[Sect. 1.1.1]{GR}. Note in this case  $\cC^{op}$ is discrete, 
and so  $\Corr^\adm_{\ver;\hor}(\cC^{op})$ is elementary to construct. Its objects are the same as objects of $\cC$, namely $(n,S)$. A $1$-morphism from $(n,S)$ to $(n',S')$ in $\Corr^\adm_{\ver;\hor}(\cC^{op})$ is a quadruple $(n', T, \ph,\psi)$ represented by a diagram of maps in $\cC$
\begin{equation}\label{mor corr}
\xymatrix{ (n',T) &  (n,S)\ar[l]_{\psi}\\
(n',S')\ar[u]^{\ph}}
\end{equation}
where $\ph\in \frC$. A $2$-morphism  $(n',T,\ph,\psi)\to (n', T', \ph',\psi')$ in $\Corr^\adm_{\ver;\hor}(\cC^{op})$  is a (closed) map $\g: (n',T')\to (n,T)$ such that $\g\c\ph'=\ph, \g\c\psi'=\psi$. We refer to ~\cite[Sect. 1.1]{GR} for more details.

We have canonical functor $\cC= (\cC^{op})^{op} \to  \Corr^\adm_{\ver;\hor}(\cC^{op})$ that is the identity map on objects, and sends a morphism $\psi: (n,S)\to (n',S')$ in $\cC$ to the following $1$-morphism in $\Corr^\adm_{\ver;\hor}(\cC^{op})$
\begin{equation*}
\xymatrix{ (n',S')\ar@{=}[d] & (n,S)\ar[l]_-{\psi}\\
(n',S')}
\end{equation*}

\sss{Functors out of correspondences}\label{sss:Fun Corr}
Now by~\cite[Theorem 3.2.2(b)]{GR}, for any $(\infty, 2)$-category $\SS$, restriction along the canonical functor $\cC\to  \Corr^\adm_{\ver;\hor}(\cC^{op})$ defines an equivalence between the space of
2-functors
$$
\xymatrix{
\Phi_{\Corr}: 
\Corr^\adm_{\ver;\hor}(\cC^{op}) \ar[r] & \SS
}
$$
and the subspace of 1-functors
$$
\xymatrix{
\Phi_*: \cC = (\cC^{op})^{op}\ar[r] & \SS
}
$$
that satisfy the right Beck-Chevalley condition with respect to $\ver=\frC$. 
Moreover, the resulting functor
$
 \Phi_{\Corr}|_{\cC^{op}_\ver}$ is obtained from $\Phi_*|_{ \cC_\ver}$ by passing to left adjoints. 

Here the right Beck-Chevalley condition~\cite[Definition 3.1.5]{GR} with respect
to $\ver$ is the following requirement on $\Phi_*: \cC = (\cC^{op})^{op}  \to \SS$. For every 1-morphism $\beta: c\to c'$ in $\cC^{op}$ and $\beta\in \ver$, the corresponding 1-morphism
$
\Phi_*(\beta)
$
admits a left adjoint, denoted by $\Phi^{*}(\beta)$, such that for every Cartesian diagram in $\cC^{op}$
\begin{equation*}
\xymatrix{
\ar[d]_-{\beta_1} c_{0, 1} \ar[r]^-{\alpha_0} &  c_{0, 0}\ar[d]^-{\beta_0} \\
c_{1, 1} \ar[r]^-{\alpha_1} & c_{1, 0}
}
\end{equation*}
so equivalently cocartesian diagram in $\cC$
$$
\xymatrix{
 c_{0, 1} & \ar[l]_-{\alpha_0}   c_{0, 0} \\
\ar[u]^-{\beta_1} c_{1, 1} & \ar[l]_-{\alpha_1} c_{1, 0}\ar[u]_-{\beta_0}
}
$$
with $\alpha_0, \alpha_1\in \hor$ and $\beta_0, \beta_1 \in \ver$, the natural  2-morphism 
\begin{equation*}
\xymatrix{
\Phi^{*}(\beta_1) \circ \Phi_*(\alpha_0) \ar[r] & \Phi_*(\alpha_1) \circ \Phi^{*}(\beta_0) 
}
\end{equation*}
arising by adjunction from the isomorphism $\Phi_*(\alpha_0) \c\Phi_{*}(\beta_0) \simeq\Phi_{*}(\beta_1) \c\Phi_*(\alpha_1)$ is an isomorphism.

For a morphism in $\Corr^\adm_{\ver;\hor}(\cC^{op})$ as in \eqref{mor corr}, its value under $\Phi_{\Corr}$ is the composition
\begin{equation*}
\Phi^{*}(\ph)\c\Phi_{*}(\psi): \Phi_{*}(n,S)\to \Phi_{*}(n',S').
\end{equation*}

\subsection{Functor of nilpotent sheaves}\label{ss: nilp shvs}
Let $\St$ be the $(\infty,2)$-category of stable $\infty$-categories. Now we will define  a functor $\Phi_*:\cC \to \St$ by passing to nilpotent sheaves and $*$-pushforwards on 
 $\sfBun^{+}_{(n,S)}$.

\sss{Nilpotent singular support conditions} 
Recall the real subanalytic spaces  $\sfBun^{+}_{(n,S)}\subset\sfBun^{\RR}_{(n,S)}\in\frR_{/\sfBun_{(n,S)}}$. Note that $\sfBun^{\RR}_{(n,S)}$ is smooth over $\sfB^{\RR}_{(n,S)}$, hence smooth. Consider the natural maps
 \begin{equation*}
\xymatrix{
 \sfp_{(n,S)}: \sfBun_{(n, S)}\ar[r] & \sfB_{(n, S)} &\sfp^{\RR}_{(n,S)}: \sfBun^\RR_{(n, S)}\ar[r] & \sfB^\RR_{(n, S)} & \sfp^{+}_{(n, S)}: \sfBun^+_{(n, S)}\ar[r] & \sfB^+_{(n, S)}
} \end{equation*}

   
Consider the  cotangent bundle $T^*  \sfBun_{(n, S)}$, the relative cotangent bundle $T^* \sfp_{(n, S)}$,
and the natural projection
\begin{equation}\label{proj nS}
\xymatrix{
\Pi_{(n,S)}:T^*  \sfBun_{(n, S)}  \ar[r] & T^* \sfp_{(n, S)} .
}
\end{equation}
It has a real version
\begin{equation}\label{proj rel nS}
\xymatrix{
\Pi^{\RR}_{(n,S)}:T^*  \sfBun^\RR_{(n, S)}  \ar[r] & T^* \sfp^{\RR}_{(n, S)} 
}
\end{equation}

Recall from \S\ref{sss:rel Hitchin} that we have the relative nilpotent cones $\cN_{(n,S)}\subset T^{*}\sfp_{(n,S)}$ and $\cN^{\RR}_{(n,S)}\subset  T^* \sfp^{\RR}_{(n, S)}$. Recall the notion of relative singular support in \S\ref{sss:rel sing supp stacks}.

\begin{defn}
We say a sheaf $\cF\in \Sh ( \sfBun^\RR_{(n, S)})$  (resp. $\cF\in \Sh ( \sfBun_{(n, S)})$) is {\em relatively nilpotent} if $\Pi^{\RR}_{(n,S)}(ss(\cF)) \subset \cN^{\RR}_{(n,S)}$ (resp. $\Pi_{(n,S)}(ss(\cF)) \subset \cN_{(n,S)}$).
\end{defn}

Next, consider $(n, [n]) \in \cC$. The unique map $(-1,\vn)\to (n,[n])$ in $\cC$ gives  open embeddings
\begin{equation*}
\sfB_{(-1,\vn)}\incl \sfB_{(n, [n])}, \quad \sfB^{\RR}_{(-1,\vn)}\incl \sfB^{\RR}_{(n, [n])}. 
\end{equation*}
Denote the images by  $\sfB^{\times}_{(n,[n])}$ and $\sfB^{\RR^{\times}}_{(n, [n])}$. Denote the preimage of $\sfB^{\RR^{\times}}_{(n, [n])}$ in $\sfBun^{\RR}_{(n,[n])}$ by $\sfBun^{\RR^{\times}}_{(n,[n])}\cong \sfBun^{\RR}_{(-1,\vn)}$. Similarly define $\sfBun^{\times}_{(n,[n])}\subset \sfBun_{(n,[n])}$. Define $\sfB^{>0}_{(n,[n])}$ to be the image of the open embedding
\begin{equation}\label{open stratum sfBn}
\sfB^{+}_{(-1,\vn)}\incl \sfB^+_{(n, [n])}.
\end{equation}
Denote $\sfBun^{>0}_{(n, [n])}=\sfBun^+_{(n, [n])}|_{\sfB^{>0}_{(n, [n])}}$, which is isomorphic to $\Bun_{G}(\pi|_{B_{>0}})\cong \Bun_{G}(\t|_{C_{>0}})\times_{\RR_{>0},\Pi_{k}}\RR_{>0}$ (inside $\frR_{\Sta}$).


Using the universal nilpotent cone in Definition \ref{d:univ nilp cone} and its real variant in \S\ref{sss:real univ cone},  we have the universal nilpotent cones $\cN^{\times}_{(n,[n])}\subset T^*  \sfBun^{\times}_{(n, [n])}$ $\cN^{\RR^{\times}}_{(n,[n])}\subset T^*  \sfBun^{\RR^{\times}}_{(n, [n])}$, defined to be the transport of the universal nilpotent cones
\begin{equation*}
\wt\cN_{\sfp_{(-1,\vn)}}\subset   T^{*}\sfBun_{(-1,\vn)}, \quad \wt\cN_{\sfp^{\RR}_{(-1,\vn)}}\subset   T^{*}\sfBun^{\RR}_{(-1,\vn)}
\end{equation*}
under the isomorphisms $\sfBun_{(-1,\vn)}\simeq \sfBun^{\times}_{(n, [n])}$ and $\sfBun^{\RR}_{(-1,\vn)}\simeq \sfBun^{\RR^{\times}}_{(n, [n])}$.

\begin{defn}\label{def:gen nilp}
We say a sheaf $\cF\in \Sh ( \sfBun^\RR_{(n, [n])} )$ (resp.  $\cF\in \Sh ( \sfBun_{(n, [n])} )$) is {\em generically nilpotent} if $\ssupp(\cF)|_{\sfBun^{\RR^{\times}}_{(n, [n])}}  \subset \cN^{\RR^{\times}}_{(n,[n])}$ (resp. $\ssupp(\cF)|_{\sfBun^{\times}_{(n, [n])}}  \subset \cN^{\times}_{(n,[n])}$).
\end{defn}

Now we are ready to define the functor $\Phi_*:\cC \to \St$.

\begin{defn}\label{def:ShN}
For any $(n, S) \in \cC$, we set $\Sh_{\cN}(\sfBun_{(n, S)})\subset \Sh(\sfBun_{(n, S)})$ to be the full subcategory of sheaves  that are both relatively nilpotent and generically nilpotent (which is relevant only if $S=[n]$).

Similarly, we set $\Sh_{\cN}(\sfBun^+_{(n, S)})\subset \Sh(\sfBun^\RR_{(n, S)})$ to be the full subcategory of sheaves supported on $\sfBun^+_{(n, S)}$ that are both relatively nilpotent and generically nilpotent (which is relevant only if $S=[n]$).
\end{defn} 
 

\begin{lemma}\label{l:ShN}
For any $(n,S)\in \cC$, we have $\Sh_{\cN}(\sfBun^+_{(n, S)})=\Sh_{\L_{(n,S)}^{+}}(\sfBun^\RR_{(n, S)})$ for a closed $\RR_{>0}$-conic Lagrangian $\L^{+}_{(n,S)}\subset T^{*}\sfBun^{\RR}_{(n, S)}$, and is compactly generated. Similarly, $\Sh_{\cN}(\sfBun_{(n, S)})$ is  compactly generated.
\end{lemma}
\begin{proof} We prove the real version, and the complex version is proved similarly. Once we construct $\L^{+}_{(n,S)}$,  Lemma \ref{lem:stack comp gens} then implies that $\Sh_{\L^{+}}(\sfBun^\RR_{(n, S)})$, hence $\Sh_{\cN}(\sfBun^+_{(n, S)})$, is compactly generated. 

Recall the map \eqref{proj rel nS} to the relative cotangent bundle. The preimage of the relative global nilpotent cone $(\Pi^{\RR}_{(n,S)})^{-1}(\cN^{\RR}_{(n,S)})$ is closed coisotropic.

If $S\ne [n]$, let $\L_{(n,S)}= (\Pi^{\RR}_{(n,S)})^{-1}(\cN^{\RR}_{(n,S)})$. We claim that $\L$ is a Lagrangian. Indeed, by Remark \ref{rem:sfA}, $\sfB^{\RR}_{(n,S)}\cong \sfA^{\RR}_{(n,S)}$ is stratified into finitely many points of the form $\sfB^{\RR}_{(n,S), \a}=B^{\RR}_{(n,S),\a}/ S\RR^{[n]}_{>0}$, where $B^{\RR}_{(n,S),\a}$ are the $S\RR^{[n]}_{>0}$-orbits on $A^{\RR}_{(n,S)}$. Let $i_{\a}: \sfBun^{\RR}_{(n,S),\a}\incl \sfBun^{\RR}_{(n,S)}$ be the preimage of $\sfB^{\RR}_{(n,S), \a}$. Let $\cN^{\RR}_{(n,S),\a}$ be the relative nilpotent cone for $\sfp^{\RR}_{(n,S),\a}: \sfBun^{\RR}_{(n,S), \a}\to \sfB^{\RR}_{(n,S), \a}$. Since $\sfB^{\RR}_{(n,S), \a}$ has only one point, the relative cotangent of $\sfp^{\RR}_{(n,S),\a}$ is the same as the cotangent of $\sfBun^{\RR}_{(n,S), \a}$, and $\cN^{\RR}_{(n,S),\a}$ is a Lagrangian in $T^{*}\sfBun^{\RR}_{(n,S), \a}$. Now $\L$ is the union of $\orr{i_{\a}}(\cN^{\RR}_{(n,S),\a})$, hence a Lagrangian by Lemma \ref{l:trans Lag}.  

If $S=[n]$, let $\L_{(n,S)}\subset (\Pi^{\RR}_{(n,S)})^{-1}(\cN^{\RR}_{(n,S)})$ be the subset that is $\cN^{\RR^{\times}}_{(n,[n])}$ over $\sfBun^{\RR^{\times}}_{(n,[n])}$, and equal to $(\Pi^{\RR}_{(n,S)})^{-1}(\cN^{\RR}_{(n,S)})$ in the complement. Since $\cN^{\RR^{\times}}_{(n,[n])}$ is closed in $T^{*}\sfBun^{\RR^{\times}}_{(n,[n])}$,  $\L$ is closed in $\sfBun^{\RR}_{(n,[n])}$.  We claim that $\L$ is a Lagrangian. Now $\L$ is clearly Lagrangian when restricted over the open $\sfBun^{\RR^{\times}}_{(n,[n])}$. The complement $\sfB^{\RR}_{(n,[n])}-\sfB^{\RR^{\times}}_{(n,[n])}$ is a union of finitely many points $\sfB^{\RR}_{(n,[n],\a)}$ (with automorphisms). Using the same notation in the previous paragraph, we see that the restriction of $\L$ to the complement of $\sfBun^{\RR^{\times}}_{(n,[n])}$ is the union of $\orr{i_{\a}}(\cN^{\RR}_{(n,S),\a})$, hence a Lagrangian by Lemma \ref{l:trans Lag}.

Finally, let $\L^{+}_{(n,S)}$ be the restriction of $\L_{(n,S)}$ over $\sfBun^{+}_{(n,S)}$, which is again a closed conic Lagrangian in $T^{*}\sfBun^{\RR}_{(n, S)}$. It is easy to see that $\Sh_{\cN}(\sfBun^+_{(n, S)})=\Sh_{\L^{+}_{(n,S)}}(\sfBun^\RR_{(n, S)})$.
\end{proof}
 
 \sss{Sheaf functors}\label{sss:def Phi} We shall define a functor
 \begin{equation}\label{Phi*}
\Phi_{*}: \cC\to \St.
\end{equation}
For any $(n,S)\in \cC$, we set
\begin{equation*}
\Phi_*(n, S):=\Sh_{\cN}(\sfBun^+_{(n, S)})
\end{equation*}
where the right side is defined in Definition \ref{def:ShN}.

To a morphism $\varphi:(n, S)\to (n', S')$, recall we have a 
 Cartesian square
\beq\label{eq:morph of rel bundles}
\xymatrix{
\ar[d]_-{\sfp^{+}_{(n, S)}} \sfBun^+_{(n, S)}  
\ar[r]^-{\sfBun^+_\varphi} &  \sfBun^+_{(n', S')} \ar[d]^-{\sfp^{+}_{(n', S')}}  \\
\sfB^+_{(n, S)}  \ar[r]^-{\sfB^+_\varphi} & 
 \sfB^+_{(n', S')}
}
\eeq
where the horizontal maps are locally closed embeddings.

To the  morphism $\varphi$, consider the $*$-pushforward (as defined in \S\ref{sss:real sh functors})
\begin{equation}\label{Bun push}
(\sfBun^\RR_{\varphi})_*: \Sh(\sfBun^\RR_{(n, S)})\to \Sh(\sfBun^\RR_{(n', S')}).
\end{equation}
Assertion (1) of the following proposition confirms that $(\sfBun^+_{\varphi})_*$ respects the specified support and singular support conditions and restricts to a functor
\begin{equation}\label{Phi*phi}
\Phi_*(\varphi)=(\sfBun^\RR_{\varphi})_*|_{\Phi_{*}(n,S)}:\Phi_{*}(n,S)\to \Phi_{*}(n',S'). 
\end{equation}
Therefore $\Phi_{*}$ gives a functor as in \eqref{Phi*}.

Assertion (2) below then allows us to apply to $\Phi_*$ the discussion of \S\ref{sss:Fun Corr} to obtain a functor
$$
\xymatrix{
\Phi_{\Corr}
: \Corr^\adm_{\ver;\hor}(\cC^{op}) \ar[r] & \St
}
$$
Finally, assertion (3) below then allows us to pass to left-adjoints to obtain a functor
$$
\xymatrix{
\Phi_{\Corr}^L
: \Corr^\adm_{\ver;\hor}(\cC^{op})^{op} \ar[r] & \St
}
$$

\begin{prop}\label{p:sh functor from C}
\begin{enumerate}
\item 
For any  morphism $\varphi:(n, S) \to (n', S')$ in $\cC$, the $*$-pushforward 
$(\sfBun^\RR_{\varphi})_*$ in \eqref{Bun push} sends the full subcategory  $\Phi_{*}(n,S)=\Sh_{\cN}(\sfBun^+_{(n, S)})$ to $\Phi_{*}(n',S')=\Sh_{\cN}(\sfBun^+_{(n', S')})$. 

\item For any  morphism $\varphi$ in $\cC$, the  functor $\Phi_*(\varphi)$ in \eqref{Phi*phi} admits a left adjoint.
Moreover, the functor $\Phi_*$  satisfies the right Beck-Chevalley condition with respect to the distinguished class of ``closed" morphisms $\ver = \frC$.

\item For any morphism $\g: (n,S)\to (n',S')$ in $\Corr^\adm_{\ver;\hor}(\cC^{op})$, $\Phi_{\Corr}(\g)$  admits a left adjoint.
\end{enumerate}
\end{prop}

\begin{proof} 
(1) Consider a morphism  $\varphi:(n, S) \to (n', S')$. Since $\sfBun^{\RR}_{\ph}$ sends $\sfBun^{+}_{(n,S)}$ to $\sfBun^{+}_{(n',S')}$ which is closed in the $\sfBun^{\RR}_{(n',S')}$, $(\sfBun^\RR_{\varphi})_*$ sends sheaves supported on $\sfBun^{+}_{(n,S)}$ to sheaves supported on $\sfBun^{+}_{(n',S')}$.

The  generically nilpotent condition is only relevant when $S' = [n']$ and over the generic locus
$\sfB^{>0}_{(n', S')}$. Moreover, the generic locus $\sfB^{>0}_{(n', [n'])}$ is only in the image of the map $\sfB^+_\varphi$ when $S = [n]$. In this case, its pre-image is  $\sfB^{>0}_{(n, [n])}$, and $\sfB^+_{\varphi}$ restricts to an isomorphism $\sfB^{>0}_{(n, [n])}\simeq\sfB^{>0}_{(n', [n'])}$.
Hence  the same properties hold for its base-change $\sfBun^\RR_{\varphi}$, and so  $(\sfBun^\RR_{\varphi})_*$ evidently respects the  generically nilpotent condition.

On the other hand, Lemma~\ref{lem:rel ss stacks} and the fact the relative global nilpotent cone $\cN_{(n,S)}$ is closed together imply $(\sfBun^\RR_{\varphi})_*$ respects the relatively nilpotent condition. 


(2)
First, the left adjoint to the
$*$-pushforward $(\sfBun^\RR_{\varphi})_*$ is the $*$-pullback $ (\sfBun^\RR_{\varphi})^*$.
As argued in (1), using Lemma~\ref{lem:rel ss stacks}, $ (\sfBun^\RR_{\varphi})^*$ respects the support condition and the generically and relatively nilpotent conditions. Hence its restriction to $\Phi_{*}(n',S')$ lands in $\Phi_{*}(n,S)$ and gives a left adjoint to $\Phi_{*}(\ph)$.

Next, we check the right Beck-Chevalley condition for the closed morphisms.  Consider a pushout diagram in $\cC$ of the form \eqref{pushout C1}
\begin{equation}\label{pushout C2}
\xymatrix{ (n'', \psi(S') \cup_{\psi(S)} S'') &  (n, S')\ar[l]_-{\wt\psi}\\
 (n'', S'') \ar[u]^{\wt\ph} & (n, S) \ar[u]^{\ph}\ar[l]_-{\psi}}
\end{equation}

%

where $\varphi,\wt\ph\in\frC$ and $\psi$ is arbitrary. We must see the natural morphism
$$
\xymatrix{
\bc(\ph,\psi): \Phi^*(\wt\ph)\c\Phi_{*}(\wt \psi)\ar[r] &\Phi_{*}(\psi)\c\Phi^{*}(\varphi)}
$$
is an isomorphism as functors $\Phi_{*}(n,S')\to \Phi_{*}(n'',S'')$.

Recall the factorization 
$$
\xymatrix{
\psi:(n, S) \ar[r]^-{\psi_\frC} & (n, \psi^{-1}(S'')) \ar[r]^-{\psi_\frO} &  (n'', S'')
}
$$
where $\psi_\frC \in \frC$,  $\psi_\frO \in \frO$.
Thus it suffices to check $\bc(\ph,\psi)$ is an isomorphism when $\psi$ itself is either in $\frC$ or in $\frO$.

 When  $\psi\in\frC$, the maps $\sfB^+_{\psi}$, $\sfB^+_{\wt \psi}$ and hence their base-changes   $\sfBun^+_\psi$, 
 $\sfBun^+_{\wt \psi}$ are closed embeddings. Thus in this case $\Phi_{*}(\psi)$ and $\Phi_{*}(\wt\psi)$ coincide with their $!$-pushforward versions, and $\bc(\ph,\psi)$ is an isomorphism by proper base change.
 

Now consider $\psi\in \frO$ so that
 $S'' = \psi(S) \cup ([n''] \setminus \psi([n]))$. We prove that $\bc(\ph,\psi)$ is an isomorphism by induction on $n''$. When  $n''=-1$ then both $\psi$ and $\ph$ are isomorphisms, in which case the statement trivially holds.
 
Now suppose $\bc(\ph,\psi)$ is an isomorphism for diagrams \eqref{pushout C2} with $n''<N$. Consider the case $n''=N$. Since all relevant substacks of $\sfB^{+}_{(n'',[n''])}$
 are contained in $\sfB^{+}_{(n'',\psi(S')\cup_{\psi(S)} S'')}$, we may restrict the situation to $\sfB^{+}_{(n'',S''\cup_{S} S')}$ 
%
 
which is isomorphic to  $\sfB^{+}_{(n''',[n'''])}$ for $n'''=\#(\psi(S')\cup_{\psi(S)} S'')-1$. We thus reduce to the same statement for the situation of relative $\Bun_{G}$ for the family of curves $\sfX_{(n'',S''\cup_{S} S')}$ over $\sfB_{(n'',S''\cup_{S} S')}$. If $n'''<n''$ we are done by the induction hypothesis. Thus we reduce to the case $n'''=n''$, hence $\psi(S')\cup_{\psi(S)} S''=[n'']$, i.e., $S'=[n]$. In this case, we have a partition of $[n'']$ into three disjoint subsets
\begin{equation*}
I_{1}=\psi(S), I_{2}=\psi([n]\setminus S), I_{3}=[n''] \setminus \psi([n])=S''\setminus \psi(S)
\end{equation*}
so that $\psi([n])=I_{1}\sqcup I_{2}$, $S''=I_{1}\sqcup I_{3}$.

For $J\subset I\subset [n'']$, let $\sfB^{+}(I,J)\subset \sfB^{+}_{(n'',[n''])}$ denote the subspace where $\e_{i}>0$ for $i\in J$, and $\e_{i}=0$ for $i\notin I$. We have open embeddings $\sfB^{+}(I,J)\incl \sfB^{+}(I,J')$ if $J'\subset J$, and closed embeddings $\sfB^{+}(I,J)\incl \sfB^{+}(I',J)$ if $I\subset I'$. Also $\sfB^{+}([n''], \vn)=\sfB^{+}_{(n'',[n''])}$. Under this notation, \eqref{pushout C2} corresponds to the inclusions of subspaces of $\sfB^{+}_{(n'',[n''])}$:
\begin{equation*}
\xymatrix{ \sfB^{+}([n''],\vn) &   \sfB^{+}([n''], I_{3})\ar[l]_-{\wt\psi} \\
 \sfB^{+}(I_{1}\sqcup I_{3}, \vn) \ar[u]^{\wt\ph} & \sfB^{+}(I_{1}\sqcup I_{3}, I_{3}) \ar[u]^{\ph}\ar[l]_-{\psi}
}
\end{equation*}
We abuse the notation to denote these inclusions of subspaces of $\sfB^{+}_{(n'', [n''])}$ by the same letters as in \eqref{pushout C2}. For $J\subset I\subset [n'']$, let $\Phi(I,J)$ be the category of generically nilpotent and relatively nilpotent sheaves on $\sfBun^{+}_{(n'',[n''])}|_{\sfB^{+}(I,J)}$.  For an inclusion $\io: \sfB^{+}(I,J)\incl \sfB^{+}(I',J')$, we denote by $\Phi_{*}(\io):\Phi(I,J)\to \Phi(I',J') $ and $\Phi^{*}(\io): \Phi(I',J')\to \Phi(I,J)$ the corresponding $*$-pushforward and $*$-pullback functors. Then $\bc(\ph,\psi)$ is a map between $\Phi^*(\wt\ph)\c\Phi_{*}(\wt \psi)$ and $\Phi_{*}(\psi)\c\Phi^{*}(\ph)$ as functors $\Phi([n''], I_{3})\to \Phi(I_{1}\sqcup I_{3}, \vn)$.

If $I_{2}=\vn$, then $\ph$ and $\wt\ph$ are isomorphisms, and the statement holds trivially;  if $I_{3}=\vn$, then $\psi$ and $\wt\psi$ are isomorphisms, and the statement again holds trivially. Below we assume $I_{2}$ and $I_{3}$ are non-empty.

Note that $\sfB^{+}([n''], I_{3})=\sqcup_{J} \sfB^{+}(J,J)$ for $I_{3}\subset J\subset [n'']$. It suffices to check that $\bc(\ph,\psi)$ is an isomorphism when applied to the image of $\Phi_{*}(\t): \Phi(J,J)\to \Phi([n''], I_{3})$,  where $\t: \sfB^{+}(J,J)\incl \sfB^{+}([n''], I_{3})$ is the inclusion for $I_{3}\subset J\subset [n'']$. For $\cF\in \Phi(J,J)$, the sheaf functors involved in $\bc(\ph,\psi)$ will take $\cF$ to sheaves supported over the closed substack $\sfB^{+}(J,\vn)\subset \sfB^{+}([n''], \vn)$. Since $*$-pushforward from a closed substack commutes with both $*$-pushforward and $*$-pullback, we reduce to  checking the similar base change isomorphism where we replace $\sfB^{+}_{(n'',[n''])}$ by  $\sfB^{+}(J,\vn)$. If $J\ne [n'']$ we are done by the inductive hypothesis.  Thus we reduce to the case $J=[n'']$ and consider the extended diagram
\begin{equation*}
\xymatrix{\sfB^{+}([n''],\vn) &   \sfB^{+}([n''], I_{3})\ar[l]_-{\wt\psi} & \sfB^{+}([n''],[n''])\ar[l]_{\t}\\
 \sfB^{+}(I_{1}\sqcup I_{3}, \vn) \ar[u]^{\wt\ph} & \sfB^{+}(I_{1}\sqcup I_{3}, I_{3}) \ar[u]^{\ph}\ar[l]_-{\psi}}
\end{equation*}
We need to check that
\begin{equation*}
\bc'(\ph,\psi): \Phi^{*}(\wt\ph)\c\Phi_{*}(\wt\psi\c\t)\to \Phi_{*}(\psi)\c\Phi^{*}(\ph)\c\Phi_{*}(\t)
\end{equation*}
is an isomorphism as functors $\Phi([n''],[n''])\to \Phi(I_{1}\sqcup I_{3}, \vn)$.

Similarly, $\sfB^{+}(I_{1}\sqcup I_{3}, \vn)=\sqcup \sfB^{+}(J,J)$ for $J\subset I_{1}\cup I_{3}$. It suffices to check that $\bc'(\ph,\psi)$ is an isomorphism after further restriction over each such $\sfB^{+}(J,J)$. If $J\ne\vn$,  we may then restrict the entire situation over the open substack $\sfB^{+}([n''], J)$ which is isomorphic to $\sfB^{+}_{(m,[m])}$ for $m=n''-\#J<n''$. Since open restriction commutes with both $*$-pushforward and $*$-pullback,  we again  reduce to the case of a smaller $n''$ and we are done by the induction hypothesis. Thus we only need to consider the case $J=\vn$, i.e., showing that the restriction of $\bc'(\ph,\psi)$ to the origin of $\sfB^{+}_{(n'',[n''])}$ is an isomorphism.  At this point we have the further expanded diagram
\begin{equation*}
\xymatrix{\sfB^{+}([n''],\vn) &   \sfB^{+}([n''], I_{3})\ar[l]_-{\wt\psi} & \sfB^{+}([n''],[n''])\ar[l]_{\t}\\
 \sfB^{+}(I_{1}\sqcup I_{3}, \vn) \ar[u]^{\wt\ph} & \sfB^{+}(I_{1}\sqcup I_{3}, I_{3}) \ar[u]^{\ph}\ar[l]_-{\psi}\\
\sfB^{+}(\vn,\vn) \ar[u]^-{\y} }
\end{equation*}
We need to check that 
\begin{equation*}
\bc''(\ph,\psi): \Phi^{*}( \wt\ph\c\y)\c\Phi_{*}(\wt\psi\c\t)\to \Phi^{*}(\y)\c\Phi_{*}(\psi)\c\Phi^{*}(\ph)\c\Phi_{*}(\t)
\end{equation*}
is an isomorphism as functors $\Phi([n''],[n''])\to \Phi(\vn,\vn)$.

We expand the above diagram one more time to get
\begin{equation*}
\xymatrix{ \sfB^{+}([n''],\vn) &   \sfB^{+}([n''], I_{3})\ar[l]_-{\wt\psi} \ar @{} [dr] |{\hs}& \sfB^{+}([n''], I_{1}\sqcup I_{3}) \ar[l]_-{\wt\a} &   \sfB^{+}([n''], [n''])\ar[l]_{\wt\t}\\
 \sfB^{+}(I_{1}\sqcup I_{3}, \vn) \ar[u]^{\wt\ph}\ar @{} [dr] |{\ds} & \sfB^{+}(I_{1}\sqcup I_{3}, I_{3}) \ar[u]^{\ph}\ar[l]_-{\psi} & \sfB^{+}(I_{1}\sqcup I_{3}, I_{1}\sqcup I_{3})\ar[l]_{\a}\ar[u]^{\un\ph}\\
 \sfB^{+}(I_{3},\vn)\ar[u]^{\wt\b} & \sfB^{+}(I_{3}, I_{3})\ar[l]_{\un\psi} \ar[u]^{\b}\\
\ar[u]^-{\wt\y} \sfB^{+}(\vn,\vn) 
}
\end{equation*}
The square marked $\hs$ gives a map of functors
\begin{equation*}
\bc(\un\ph,\a): \Phi^{*}(\ph)\c\Phi_{*}(\wt\a)\to \Phi_{*}(\a)\c\Phi^{*}(\un \ph)
\end{equation*}
which is another situation of \eqref{pushout C2}. Since the $I_{3}$-coordinates are all nonzero in the relevant substacks, the square marked $\hs$ is isomorphic to
\begin{equation*}
\xymatrix{\sfB^{+}_{(m'',[m''])} & \sfB^{+}_{(m,[m])}\ar[l]\\
\sfB^{+}_{(m'', T'')}\ar[u] & \sfB^{+}_{(m, \vn)}\ar[u]\ar[l]}
\end{equation*}
where $m''=n''-\#I_{3}$, $m=n-\#I_{3}$ and $T''=[m'']\setminus \psi'([m])$ has the same cardinality as $I_{1}$. Since $I_{3}\ne\vn$, $m''<n''$, so by the inductive hypothesis $\bc(\un\ph,\a)$ is an isomorphism. 

Similarly, the square marked $\ds$ gives a map of functors
\begin{equation*}
\bc(\b, \un\psi): \Phi^{*}(\wt\b)\c\Phi_{*}(\psi)\to \Phi_{*}(\un\psi)\c\Phi^{*}(\b).
\end{equation*}
A similar argument as above, using that $I_{2}\ne\vn$ and the inductive hypothesis, shows that $\bc(\b, \un\psi)$ is an isomorphism.

Using the isomorphisms $\bc(\un\ph,\a)$ and $\bc(\b, \un\psi)$, to prove $\bc''(\ph,\psi)$ is an isomorphism, it suffices to prove that the composition
\begin{eqnarray*}
&r: &\Phi^{*}(\wt\ph\wt\b\wt\y)\c\Phi_{*}(\wt\psi\wt\a\wt\t)\\
&\xr{\bc''(\ph,\psi)}&\Phi^{*}(\wt\b\wt\y)\c\Phi_{*}(\psi)\c\Phi^{*}(\ph)\c\Phi_{*}(\wt\a\wt\t)\\
&\xr{\bc(\b, \un\psi)\c\bc(\un\ph,\a)}&\Phi^{*}(\wt\y)\c \Phi_{*}(\un\psi)\c \Phi^{*}(\b)\c \Phi_{*}(\a)\c\Phi^{*}(\un\ph)\c\Phi_{*}(\t)
\end{eqnarray*}
is an isomorphism. To check $r$ is an isomorphism, it suffices to check its pullback to $B^{+}_{(n'',[n''])}$ is an isomorphism. Thus we are ready to apply  Theorem~\ref{th:comm nc stack} to the family $p: \Bun_{G}(\pi^{[n'']})\to B^{[n'']}$ and its base change to $B^{+}_{(n'',[n''])}$. The preimage of $\sfB^{+}(J,J)$ in $B^{+}_{(n'',[n''])}$ is denoted by $B^{+}_{J}$ in \S\ref{sss:notation nc stack}.



Let $a_{\bu}$ be the flag $\vn\subset I_{3}\subset I_{1}\sqcup I_{3}\subset [n'']$. The pullback of $r$ to the central point $\{b_{0}\}=\sfB^{+}(\vn,\vn)$ is the map 
\begin{equation*}
r_{a_{\bu}}: \psi\to \psi(a_{\bu})
\end{equation*}
defined in \S\ref{sss:notation nc stack}. To apply Theorem~\ref{th:comm nc stack}, we need to verify the singular support condition for the pullback of any $\cF\in \Phi([n''],[n''])$ to $\Bun_{G}(\pi^{[n'']})|_{B^{+}_{[n'']}}$. It suffices to verify the same conditions for the singular support of $\cF$. By definition, any $\cF\in\Phi([n''],[n''])$ is generically nilpotent in the sense of Definition \ref{def:gen nilp}, hence $ss(\cF)\subset \L:=\cN^{\RR^{\times}}_{(n'',[n''])}|_{\sfB^{+}([n''],[n''])}$, the universal nilpotent cone of the family $\sfX^{+}_{(n'',[n''])}|_{\sfB^{+}([n''],[n''])}$. In particular, by Corollary \ref{c:univ cone non-char}, $\L$ is non-characteristic with respect to the projection $\sfBun_{(n'',[n''])}|_{\sfB^{+}([n''],[n''])}\to \sfB^{+}([n''],[n''])$. Let $\Pi: T^{*}\sfBun_{(n'',[n''])}\to T^{*}_{\sfp_{(n'',[n''])}}$ be the projection to the relative cotangent bundle of $\sfp_{(n'',[n''])}: \sfBun_{(n'',[n''])}\to \sfB_{(n'',[n''])}$. Then $\ol{\Pi(\L)}$ is contained in the relative nilpotent cone in $T^{*}_{\sfp_{(n'',[n''])}}$, hence its fiber over the central point $b_{0}\in \sfB_{(n'',[n''])}$ is contained in the global nilpotent cone of the central fiber $X(n'')$ of $\sfX_{(n'',[n''])}$, and in particular is isotropic. This verifies the two conditions in Theorem~\ref{th:comm nc stack}, therefore $r_{a_{\bu}}$, hence $r$, is an isomorphism. This finishes the proof of (2).

(3) Let $\g: (n,S)\to (n',S')$ be a morphism in $\Corr^\adm_{\ver;\hor}(\cC^{op})$ given by the diagram \eqref{mor corr}.  Then $\Phi_{\Corr}(\g)=\Phi^{*}(\ph)\Phi_{*}(\psi)$. Since $\Phi_{*}(\psi)$ has a left adjoint $\Phi^{*}(\psi)$, it remains to show that $\Phi^{*}(\ph)$ has a left adjoint. Note that $\Phi^{*}(\ph)$ is the $*$-pullback along the closed embedding $(\sfBun^{\RR}_{\varphi})^*$, $\Phi_{*}(n,S)$ and $\Phi_{*}(n',S')$ are full subcategories of $\Sh(\sfBun^{\RR}_{(n,S)})$ and $\Sh(\sfBun^{\RR}_{(n',S')})$ cut out by singular support conditions by Lemma \ref{l:ShN}.   
Thus Lemma~\ref{lem:rest adj st} of Appendix B provides the required left adjoint.
\end{proof}

%
%


\subsection{Augmented semi-cosimplicial category}\label{ss:cosim cat}

We are now ready to record the main output of our prior constructions. We continue with the setup of the prior section.

 Let $\D_{\inj, +}$ be the augmented semi-simplex category whose objects are the ordered sets  $[n]=\{0,\cdots, n\}$ for $n\ge -1$, and morphisms $\ph:[m] \to [n]$ are order-preserving injections. 
By definition, an {\em augmented semi-cosimplicial object} of a category $\cD$ is a functor $\Delta_{\inj, +} \to \cD$. 

We will define an augmented semi-cosimplicial object 
$$
\xymatrix{
c^\bullet: \Delta_{\inj, +} \ar[r] &  \Corr^\adm_{\ver;\hor}(\cC^{op})
}$$
On objects, we take $c^n = (n, \vn)$ for $n\geq -1$. On morphisms, 
for $\varphi_i:[n] \to [n+1]$ the inclusion with $i \not \in \varphi([n])$, we take the correspondence
$$
\xymatrix{ (n+1, \{i\})  & \ar[l] (n, \vn) =c^{n}\\
  (n+1, \vn) = c^{n+1}\ar[u]
}
$$
where we take the convention of diagram \eqref{mor corr} (the arrows are in $\cC$, and the vertical arrow is in $\frC$). It is elementary to check the augmented semi-cosimplicial identities hold.
 
Applying the 2-functor $\Phi_{\Corr}$ induced by
$\Phi_*$ defined in \S\ref{sss:def Phi}, we obtain an augmented semi-cosimplicial category 
$$
\xymatrix{
\cS^\bullet: \Delta_{\inj, +}\ar[r]^-{c^\bullet} &  \Corr^\adm_{\ver;\hor}(\cC^{op})  \ar[r]^-{\Phi_{\Corr}} & \St
}$$

We summarize the features of $\cS^{\bu}$ in the following proposition.
\begin{prop} The augmented semi-cosimplicial category $\cS^{\bu}$ satisfies:
\begin{enumerate}
\item There are canonical equivalences
\begin{eqnarray}
\notag \cS^{-1} = \Sh_{\wt\cN}(\Bun_G(  \t|_{C_{>0}}))\\
\label{Sn Xn}\cS^n = \Sh_\cN(\Bun_G( X(n))), \quad n\geq 0. 
\end{eqnarray}
Here $\wt\cN$ denotes the universal nilpotent cone for $\Bun_G(  \t|_{C_{>0}})$, and $\cN$ in \eqref{Sn Xn} denotes the global nilpotent cone for $\Bun_{G}(X(n))$.

\item The augmentation map of  $\cS^\bullet$  is the nearby cycles functor for the real $1$-dimensional family $\Bun_G(  \t|_{C_{\ge0}})\to C_{\ge0}$, i.e., it is the composition
$$
\xymatrix{
\Sh_{\wt\cN}(\Bun_G(  \t|_{C_{>0}}))
\ar[r]^-{J_*} & \Sh_{\wt\cN,\cN}(\Bun_G(  \t|_{C_{\ge0}}))
\ar[r]^-{I^*} &
 \Sh_\cN(\Bun_G( X(0))) 
}
$$
where $J: \Bun_G(  \t|_{C_{>0}})\incl \Bun_G(\t|_{C_{\ge0}})$ is the open inclusion and $I: \Bun_G( X(0))\incl \Bun_G(  \t|_{C_{\ge0}})$  is the closed complement. The singular support condition in the middle term means generically nilpotent and relatively nilpotent. 

\item For $n\geq 0$ and $i\in[n+1]$, its $i$th face map is the nearby cycles functor for the real $1$-dimensional family $\Bun_G(  \pi^{[n+1]}|_{B^+_{(n+1, i)}})\to \pi_{n+1}|_{B^+_{(n+1, i)}}$, where $B^+_{(n+1, i)}$ is the $i$th non-negative coordinate ray in $B^{+}_{(n+1,[n+1])}$. I.e., it is  the composition 
$$
\xymatrix{
\Sh_\cN(\Bun_G( X(n)))
\ar[r]^-{J_{i*}} &
 \Sh_\cN(\Bun_G(  \pi^{[n+1]}|_{B^+_{(n+1, i)}})/\RR_{>0})
\ar[d]^{I_i^*}\\
&
\Sh_\cN(\Bun_G( X({n+1}))/\RR_{>0}) 
\ar[r]^-\sim & 
\Sh_\cN(\Bun_G( X({n+1}))) 
}
$$
Here, $J_i$ and $I_i$ are the natural inclusions induced by the open and closed points of $B^+_{(n+1, i)}/\RR_{>0}$
respectively. The singular support condition $\cN$ in the second term means relatively nilpotent.
\end{enumerate}
\end{prop}

\begin{proof} Most of the statements directly follow by unwinding definitions. Only \eqref{Sn Xn} requires argument. By definition, $\cS^n = \Sh_\cN(\Bun_G( X(n))/S \RR_{>0}^{[n]})$. The map $\nu: \Bun_G( X(n))\to \Bun_G( X(n))/S \RR_{>0}^{[n]}$ is a surjective submersion in $\frR_{/\sfBun_{(n,\vn)}}$. By Proposition \ref{p:real sm des}, $\Sh_\cN(\Bun_G( X(n))/S \RR_{>0}^{[n]})$ is the limit of $\Sh_{\cN_{i}}(M_{i})$, where $M_{i}=(S \RR_{>0}^{[n]})^{i}\times \Bun_G( X(n))$ for $i=0,1,\cdots$, with singular support $\cN_{i}$ given by $0$ on the first factor and $\cN$ on the second.  Since $S \RR_{>0}^{[n]}$ is contractible,  pullback  induces an equivalence $\Sh_{\cN_{0}} (M_{0})\isom \Sh_{\cN_{i}}(M_{i})$ so that the transition functors in the cosimplicial category $\Sh_{\cN_{\bu}}(M_{\bu})$ are identity functors for $\Sh_{\cN_{0}} (M_{0})$. Therefore the pullback
$$\xymatrix{
\Sh_\cN(\Bun_G( X(n))/S \RR_{>0}^{[n]}) =\lim_{i\ge0}\Sh_{\cN_{i}}(M_{i})\ar[r]^-\sim & 
\Sh_{\cN_{0}}(M_{0})=\Sh_\cN(\Bun_G( X(n))) 
}
$$
is an equivalence. This gives the canonical isomorphism \eqref{Sn Xn}.
\end{proof}

%
%


\subsection{Fixing the twisting type}\label{ss:fix type}

\sss{Local branches}\label{sss:loc br}
Recall from \S\ref{sss:Yk} that $\YY_{k}$ is the set of $G$-orbits of homomorphisms $\mu_{k}\to G$, which is the set of isomorphism classes of the stack $\cY_{k}=\Bun_{G}(\BB\mu_{k})$. Now fix $\y\in \YY_{k}$.

Recall from \S\ref{sss:prop tw bub} Property \eqref{bub:tw sec} that we have a map
\begin{equation*}
\s^{n}_{i}: Q\times B^{[n],i}=Q\times B_{(n,i^{c})}\isom \cQ^{[n],i}\incl \frX^{[n],i}.
\end{equation*}
Let $b^{n,i}=(c_{0}; 1,\cdots , 0, \cdots, 1)\in B^{[n],i}$, where only the $i$-th coordinate is zero. Then $\frX^{[n]}|_{b^{n,i}}\cong X(0)=X_{-}\cup X_{+}$. Now $\frX^{[n],i}$ has two local branches along each of the sections $Q\times B^{[n],i}$. To talk about types of $G$-bundles along $\s^{n}_{i}$, we need to name one of the local branches the $+$-branch (see \S\ref{sss:tw nodal X}). Our convention is that we always name the local branch that contains $X_{+}$ over $b^{n,i}$  to be the $+$-branch. 

\sss{Fixing the type}
For $(n,S)\in \cC$ and any $i\in [n]\bs S$,  restricting $\cQ^{[n],i}$ to  $B_{(n,S)}$ gives the $i$-th twisted nodal locus
\begin{equation}\label{sfQ}
\sfQ_{(n,S)}^{i}:=(\cQ^{[n],i}|_{B_{(n,S)}})/S\Gm^{[n]}\incl \sfX_{(n,S)}.
\end{equation}
Note that $\sfQ_{(n,S)}^{i}\simeq Q\times \sfB_{(n,S)}\simeq R\times \BB\mu_{k}\times \sfB_{(n,S)}$.  Restricting $G$-bundles to $\sfQ_{(n,S)}^{i}$ induces a map
\begin{equation*}
r_{(n,S)}^{i}: \sfBun_{(n,S)}\to \Bun_{G}(\sfQ_{(n,S)}^{i}/\sfB_{(n,S)})\simeq \cY_{k}^{R}.
\end{equation*}
For $\y\in \YY_{k}$, let 
\begin{equation*}
\sfBun^{(i;\y)}_{(n,S)}\subset \sfBun_{(n,S)}
\end{equation*}
denote the preimage of the component of $\cY_{k}^{R}$ indexed by $\y^{R}$ under $r_{(n,S)}^{i}$. Then $\sfBun^{(i;\y)}_{(n,S)}$ is open and closed in $\sfBun_{(n,S)}$.

For any $(n,S)\in \cC$, we then define the open substack
\begin{equation*}
\sfBun^{\y}_{(n,S)}\subset \sfBun_{(n,S)}
\end{equation*}
to be the complement of the union of the following closed substacks for each $i\in [n]$
\begin{equation*}
\sfBun_{(n,S\bs \{i\})}-\sfBun^{(i;\y)}_{(n,S\bs \{i\})}.
\end{equation*}
Concretely, $\sfBun^{\y}_{(n,S)}$ is the substack of $\sfBun_{(n,S)}$ parametrizing bundles on $\sfX_{(n,S)}$ whose restriction to each node of each fiber has type $\y$.

For a morphism  $\ph:(n,S)\to (n',S')$ in $\cC$, $\sfBun_{\ph}$ restricts to a map $\sfBun^{\y}_{(n,S)}\to \sfBun^{\y}_{(n',S')}$. Thus we have a functor $\sfBun^{\y}: \cC\to \Sta$. Similarly we extend $\sfBun^{\y}$ to a functor  $(\sfBun^{\y}, \sfBun^{\RR,\y}\supset\sfBun^{+,\y}):\cC\to \frR^{\Sch}_{\Sta,\supset}$.

For $(n,S)\in \cC$, set
\begin{equation*}
\Phi^{\y}_{*}(n,S):=\Sh_{\cN}(\sfBun^{+,\y}_{(n,S)})
\end{equation*}
where the subscript $\cN$ still means generically nilpotent and relative nilpotent. This extends to a functor $\Phi^{\y}_{*}:\cC\to \St$ in an obvious way. The same construction  as in \S\ref{ss: nilp shvs} gives a functor
\begin{equation*}
\Phi^{\y}_{\Corr}: \Corr^\adm_{\ver;\hor}(\cC^{op}) \to \St.
\end{equation*}
Precomposing with $c^{\bu}$ we get an augmented semi-cosimplicial category
\begin{equation}\label{S eta}
\xymatrix{
\cS^{\bu,\y}:  \Delta_{\inj, +}\ar[r]^-{c^\bullet} &  \Corr^\adm_{\ver;\hor}(\cC^{op})  \ar[r]^-{\Phi^{\y}_{\Corr}} & \St.}
\end{equation}
We have
\begin{eqnarray*}
 \cS^{-1,\y} = \cS^{-1}=\Sh_{\wt\cN}(\Bun_G(  \t|_{C_{>0}}))\\
\cS^{n,\y} = \Sh_{\cN}(\Bun_G( X(n))_{\y}), \quad n\geq 0. 
\end{eqnarray*}
Here $\Bun_G( X(n))_{\y}$ is the open-closed substack of $\Bun_{G}(X(n))$ whose restriction to each twisted node has type $\y$.

\subsection{Affine Hecke symmetry}
Here we consider the situation where the initial family $\t:\frY\to C$ is equipped with a collection of disjoint sections. In this situation, we upgrade the augmented semi-cosimplicial category $\cS^{\bu}$ into an augmented semi-cosimplicial object in modules over the affine Hecke category. 

\sss{Variant -- marked sections}\label{sss:cosim cat sections}
Suppose the initial separating nodal degeneration  $\t:\frY\to C$ is equipped with disjoint sections $s_{a}: C\to \frY$ ($a\in \Sigma$) that are necessarily avoiding the nodes $R$. As in \S\ref{sss:more sections tw}, each $s_{a}$ induces a section  $ \s^{\D}_{a}:   B^\Delta \to  \frX^{\D}$. Let $\s^{\D}$ be the union of $\{\s^{\D}_{a}\}_{a\in \Sigma}$.

We can then
define $\sfBun_{\s}:\cC \to \Sta$ to  include
pro-unipotent Iwahori level-structure along these marked points. More precisely, for $(n,S)\in \cC$, define
\begin{equation*}
\sfBun_{\s,(n,S)}:= \Bun_{G,N}( \pi^{[n]}|_{ B_{(n, S)}}, \s^{[n]}|_{B_{(n, S)}})/ S\Gm^{[n]}
\end{equation*}
using notation from \S\ref{sss:rel Hitchin}.  Similarly we introduce $\sfBun^{+}_{\s,(n,S)}\subset \sfBun^{\RR}_{\s,(n,S)}\to\sfBun_{\s,(n,S)}$ by base change along $\sfB^{+}\subset \sfB^{\RR}\to\sfB$.

Using the notion of relative global nilpotent cone in \S\ref{sss:rel Hitchin} and the universal nilpotent cone with pro-unipotent Iwahori level structures in Definition \ref{def:univ cone Eis}, we have the notion of relative nilpotent sheaves on $\sfBun_{\s,(n,S)}$ and $\sfBun^{\RR}_{\s,(n,S)}$, and generically nilpotent sheaves on $\sfBun_{\s,(n,[n])}$ and $\sfBun^{\RR^{\times}}_{\s,(n,[n])}$.  Now we define
\begin{equation*}
\Sh_{\cN}(\sfBun_{\s,(n,S)})\subset\Sh(\sfBun_{\s,(n,S)})
\end{equation*}
to be the full subcategory of sheaves that are both relatively nilpotent and generically nilpotent (if $S=[n]$). Similarly, define
\begin{equation*}
\Phi_{\s, *}(n,S):=\Sh_{\cN}(\sfBun^{+}_{\s,(n,S)})\subset \Sh(\sfBun^{\RR}_{\s,(n,S)})
\end{equation*}
to be the full subcategory of sheaves supported on $\sfBun^{+}_{\s,(n,S)}$ that are both relatively nilpotent and generically nilpotent (if $S=[n]$). We extend this assignment to a functor $\Phi_{\s,*}:\cC \to \St$ by using $*$-pushforwards along $\sfBun_{\s,\ph}$ for morphisms $\ph$ in $\cC$.

We obtain an augmented semi-cosimplicial category 
\begin{equation}\label{cS marked section}
\xymatrix{
\cS^\bullet_{\s}: \Delta_{\inj, +}\ar[r]^-{c^\bullet} &  \Corr^\adm_{\ver;\hor}(\cC^{op})  \ar[r]^-{\Phi_{\s,\Corr}}  & \St.}
\end{equation}

\subsubsection{Affine Hecke symmetries}\label{sss:aff Hk sym} 

For each $a\in \Sigma$, let $D_{\s_{a}}$ be the formal completion of $\frY$ along $\s_{a}(C)$. Let $D^{[n]}_{\s_{a}}$ be the formal completion of $\frX^{[n]}$ along the section $\s^{[n]}_{a}: B^{[n]}\to \frX^{[n]}$ induced by $\s_{a}$. By construction $D^{[n]}_{\s_{a}}\simeq D_{\s_{a}}\wh\times_{C}B^{[n]}$. 

Suppose we are given a trivialization of the formal neighborhood
\begin{equation}\label{triv nbhd}
\io_{a}: D_{\s_{a}}\simeq \Spf (\CC\tl{z}) \wh\times C.
\end{equation}
It induces a trivialization
\begin{equation*}
\io^{[n]}_{a}: D^{[n]}_{\s_{a}}\simeq \Spf (\CC\tl{z}) \wh\times B^{[n]}
\end{equation*}
for any $n\ge-1$. Consider the Hecke modifications of bundles along $\cup_{a}\s^{[n]}_{a}$
\begin{equation*}
\xymatrix{ & \Hk^{[n]}_{\s}\ar[dl]_{h_{1}}\ar[dr]^{h_{2}}\ar[rr]^-{(\ev_{a})_{a\in \Sigma}} && \prod_{a\in \Sigma} \bI^{\circ}\bs G\lr{z}/\bI^{\circ} \\
\Bun_{G,N}(\pi^{[n]},\s^{[n]}) & & \Bun_{G,N}(\pi^{[n]},\s^{[n]})}
\end{equation*}
Here $\Hk_{\s}$ classifies $(\cE,\cE', \t)$ where $\cE,\cE'\in \Bun_{G,N}(\pi^{[n]},\s^{[n]})$ and $\t$ is an isomorphism of $\cE$ and $\cE'$ on the complement of $\cup_{a}\s^{[n]}_{a}$. The map $\ev_{a}$ sends $(\cE,\cE', \t)$ to their restriction to $D^{[n]}_{\s_{a}}$, and we use the trivialization $\io^{[n]}_{a}$ to identify the loop group along $\s^{[n]}_{a}$ with the loop group $G\lr{z}$. 

Recall the affine Hecke category with universal monodromy defined in \S\ref{sss:aff Hk}. Then $\cH^{\ot \Sigma}$ acts on $\Sh(\Bun_{G}(\pi^{[n]}, \s^{[n]}))$ by the functor
\begin{eqnarray*}
\cH^{\ot \Sigma}\ot \Sh(\Bun_{G}(\pi^{[n]}, \s^{[n]}))&\to& \Sh(\Bun_{G}(\pi^{[n]}, \s^{[n]}))\\
(\boxtimes_{a\in \Sigma}\cK_{a})\boxtimes\cF&\mt& (\boxtimes_{a\in \Sigma}\cK_{a})\star\cF=h_{2!}(h^{*}_{1}\cF\ot \bigotimes_{a}\ev_{a}^{*}\cK_{a}).
\end{eqnarray*}
Restricting the construction to $B_{(n,S)}\subset B^{[n]}$ and passing to $S\Gm^{[n]}$-equivariant sheaves, we get an action of $\cH^{\ot \Sigma}$ on $\Sh(\sfBun_{\s,(n,S)})$. Similarly we get have an action of $\cH^{\ot \Sigma}$ on $\Sh(\sfBun^{+}_{\s,(n,S)})$.

\begin{lemma}\label{c:Haff pres N}
The action of $\cH^{\ot \Sigma}$ on $\Sh(\sfBun_{\s,(n,S)})$ (resp. $\Sh(\sfBun^{+}_{\s,(n,S)})$) respect the nilpotent sheaves $\Sh_{\cN}(\sfBun_{\s,(n,S)})$ (resp. $\Sh_{\cN}(\sfBun^{+}_{\s,(n,S)})$). 
\end{lemma}
\begin{proof}
Let us prove the version for $\Sh_{\cN}(\sfBun_{\s,(n,S)})$, and the real version is similar. 

Theorem \ref{th:aff Hk pres} implies that $\cH^{\ot \Sigma}$ preserves generically nilpotent sheaves (for $S=[n]$). 

On the other hand, to check that relative nilpotency is preserved, by Corollary \ref{c:check rel ss}, it suffices to check the same statement over each stratum of $\sfB_{(n,S)}$ that is a single point with automorphisms.  This reduces to check that for a single DM nodal curve $X$ (a fiber of $\frX^{[n]}$) and a  smooth non-orbifold point $x\in X$, the action of $\cH$ on $\Sh(\Bun_{G,N}(X,x))$ preserves the full subcategory $\Sh_{\cN}(\Bun_{G,N}(X,x))$ of nilpotent sheaves. To check this statement, the argument of Theorem \ref{th:aff Hk pres} (the case of a single curve) reduces to checking that the global nilpotent cone for $\Bun_{G,N}(X,x)$ is preserved under transporting by the Hecke correspondence at $x$, which is obvious because the transportation of covectors preserves Higgs fields away from $x$. 
\end{proof}

\begin{prop}\label{p:cosimp cat marked sections}  With the choice of the trivializations $(\io_{a})_{a\in \Sigma}$ in \eqref{triv nbhd}, the functor $\Phi_{\s,*}:\cC \to \St$ canonically lifts to a functor $\Phi_{\s,*}:\cC \to \cH^{\ot \Sigma}\module$. The induced functor $\Phi_{\s,\Corr}:\Corr^\adm_{\ver;\hor}(\cC^{op}) \to \St$ canonically lifts to $\wt\Phi_{\s,\Corr}:\Corr^\adm_{\ver;\hor}(\cC^{op}) \to \cH^{\ot \Sigma}\module$.

In particular, the augmented semi-cosimplicial category $\cS^{\bu}_{\s}$ in \eqref{cS marked section} canonically lifts to  an augmented semi-cosimplicial category in $\cH^{\ot \Sigma}\module$:
$$
\xymatrix{
\wt\cS^\bullet_{\s}: \Delta_{\inj, +}\ar[r] & \cH^{\ot \Sigma}\module.
}$$

\end{prop}
\begin{proof}
We need to check for any morphism $\ph: (n,S)\to (n',S')$ in $\cC$ and any $\cK\in \cH^{\ot\Sig}$ there is a natural isomorphism for $\cF\in \Sh_{\cN}(\sfBun^{+}_{\s,(n,S)})$
\begin{equation}\label{conv K push}
\cK\star(\sfBun^{+}_{\ph})_{*}(\cF)\simeq (\sfBun^{+}_{\ph})_{*}(\cK\star\cF).
\end{equation}
Let $\sfHk^{+}_{\s,(n,S)}=\Hk^{[n]}_{\s}|_{B^{+}_{(n,S)}}/S\RR^{[n]}_{>0}$ (over $\sfB^{+}_{(n,S)}$) be the affine Hecke correspondence for $\sfBun^{+}_{\s, (n,S)}$. We have a commutative diagram
\begin{equation}\label{Hk nn'}
\xymatrix{ \sfBun^{+}_{\s, (n,S)}\ar[d]^{\sfBun^{+}_{\ph}} & \sfHk^{+}_{\s, (n,S)}\ar[r]^{h_{2}}\ar[l]_-{h_{1}}\ar[d]^{\sfHk^{+}_{\ph}} &  \sfBun^{+}_{\s, (n,S)} \ar[d]^{\sfBun^{+}_{\ph}}\\
\sfBun^{+}_{\s, (n',S')} & \sfHk^{+}_{\s, (n',S')}\ar[r]^-{h'_{2}}\ar[l]_-{h'_{1}} &  \sfBun^{+}_{\s, (n',S')} 
}
\end{equation}
Let $\ev: \sfHk^{+}_{\s, (n,S)}\to (\bI^{\c}\bs G\lr{z}/\bI^{\c})^{\Sig}$ and $\ev': \sfHk^{+}_{\s, (n',S')}\to (\bI^{\c}\bs G\lr{z}/\bI^{\c})^{\Sig}$ be the maps that record the relative positions of the modifications along $\s$. We have a natural map
\begin{equation*}
h'^{*}_{1}(\sfBun^{+}_{\ph})_{*}(\cF)\ot\ev^{*}\cK\to (\sfHk^{+}_{\ph})_{*}(h_{1}^{*}(\cF)\ot \ev^{*}\cK)
\end{equation*}
Since the left square is Cartesian and $h'_{1}$ is a locally trivial fibration (at least when we restrict to the inverse image of finitely many $\bI$-double cosets in $\sfHk^{+}_{\s,(n',S')}$) with fibers isomorphic to $G\lr{z}/\bI^{\c}$,  the natural map above is an isomorphism.  Applying $h'_{2!}$ to both sides we get a natural isomorphism
\begin{equation*}
\cK\star(\sfBun^{+}_{\ph})_{*}(\cF)=h'_{2!}(h'^{*}_{1}(\sfBun^{+}_{\ph})_{*}(\cF)\ot\ev^{*}\cK)\isom h'_{2!}(\sfHk^{+}_{\ph})_{*}(h_{1}^{*}(\cF)\ot \ev^{*}\cK)
\end{equation*}
On the other hand, 
\begin{equation*}
 (\sfBun^{+}_{\ph})_{*}(\cK\star\cF)= (\sfBun^{+}_{\ph})_{*}h_{2!}(h_{1}^{*}(\cF)\ot \ev^{*}\cK).
\end{equation*}
Again we have a natural transformation
\begin{equation}\label{Hk !*}
h'_{2!}(\sfHk^{+}_{\ph})_{*}\to (\sfBun^{+}_{\ph})_{*}h_{2!}.
\end{equation}
It remains to show that this natural transformation is an isomorphism when applied to $h_{1}^{*}(\cF)\ot \ev^{*}\cK$. There are two $H$-actions on $\sfHk^{+}_{\s,(n,S)}$  by modifying the $N$-reductions of $\cE_{1}$ and $\cE_{2}$ respectively. We have $h_{1}$ is equivariant under the first $H$-action and $h_{2}$ is equivariant under the second, and $\ev$ is equivariant under both. Since $\cF$  has nilpotent singular support, it is monodromic under $H$, hence $h_{1}^{*}\cF$ is monodromic under the first $H$-action; similarly $\ev^{*}\cK$ is monodromic under both $H$-actions. Therefore $h_{1}^{*}(\cF)\ot \ev^{*}\cK$ is monodromic under the first $H$-action.  We show \eqref{Hk !*} is an isomorphism when applied to any sheaf $\cG$ on $\sfHk^{+}_{\s,(n,S)}$ that is monodromic under the first action of $H$ with support in the preimage of finitely many $\bI$-double cosets.

Let us factor $H$ as a topological group $H(\CC) \simeq H_c \times H_{>0}$ where $H_c$ is the maximal compact torus and $H_{>0}=H(\RR_{>0})$.  Now $h_{1}^{*}(\cF)\ot \ev^{*}\cK$ is $H_{>0}$-equivariant under the first action. The maps $h_{2}$ and $h'_{2}$ are invariant under the first $H$-action, hence we can rewrite the right square of \eqref{Hk nn'} as
\begin{equation*}
\xymatrix{\sfHk^{+}_{\s, (n,S)}\ar[r]\ar[d]^{\sfHk^{+}_{\ph}} & \sfHk^{+}_{\s, (n,S)}/H_{>0}\ar[d]^{\ov\sfHk^{+}_{\ph}} \ar[r]^{\ov h_{2}}&  \sfBun^{+}_{\s, (n,S)} \ar[d]^{\sfBun^{+}_{\ph}}\\
 \sfHk^{+}_{\s, (n',S')}\ar[r] &  \sfHk^{+}_{\s, (n',S')}/H_{>0} \ar[r]^{\ov h'_{2}} & \sfBun^{+}_{\s, (n',S')} 
}
\end{equation*}
Here the quotients use the first actions of $H$. For any $H$-monodromic sheaf $\cG$ on $\sfHk^{+}_{\s,(n,S)}$ as above, it has a canonical descent $\ov\cG$ to $\sfHk^{+}_{\s,(n,S)}/H_{>0}$.  Now $\ov h_{2}$ and $\ov h'_{2}$ are ind-proper. The natural map \eqref{Hk !*} applied to $\cG$ can be written as the composition (where $d=\dim_{\CC}H=\dim_{\RR}H_{>0}$)
\begin{equation*}
h'_{2!}(\sfHk^{+}_{\ph})_{*}\cG\cong \ov h'_{2!}(\ov \sfHk^{+}_{\ph})_{*}\ov \cG[-d]\isom (\sfBun^{+}_{\ph})_{*}\ov h_{2!}\ov \cG [-d]\cong (\sfBun^{+}_{\ph})_{*}h_{2!}\cG.
\end{equation*}
\end{proof}
%
%
%


\begin{remark}
The spectral action of  $\Perf(\Loc_{G\vee}(X))$ on $\Sh_{\cN}(\Bun_{G}(X))$ constructed in \cite{NY} also has a family version and a bubbling version that upgrades the functor $\cS^{\bu}$ to be a functor valued in modules over a spectral category. This will be constructed in another paper.
\end{remark}



\section{Hecke actions via bubbling: algebra}


\subsection{Monoidal diagram category}

We continue to work with the category $\cC$ introduced in \S\ref{ss:diag cat}.

\sss{The monoidal category $\cC_{\star}$}
Define the full subcategory $\cC_\star \subset \cC$  with
objects given by pairs $(n, S)$ of an integer $n\geq 0$ and a subset $S\subset [n]$ disjoint from the endpoints $\{0, n\}$. 

\begin{remark}\label{rem:active}
Recall for $(n, S), (n', S') \in \cC$, a morphism $\ph:(n, S) \to (n', S')$ in $\cC$ is 
an order-preserving inclusions $\ph:[n]\to [n']$ such that (i) $\ph(S) \subset S'$ and (ii) $[n'] \setminus \ph([n]) \subset S'$. Note for $(n, S), (n', S') \in \cC_\star$,  condition (ii) implies $\ph$ takes $0 \mapsto 0$ and $n\mapsto n'$. Following established terminology, one could use the term ``active" to refer to morphisms with this property.  
\end{remark} 

We define a monoidal structure on $\cC_\star$ as follows. The monoidal product is given by
\begin{equation}\label{uplus}
(n_1, S_1) \star (n_2, S_2) = (n_1 + n_2, S_1 \uplus S_2)
\end{equation}
where $S_1 \uplus  S_2\subset [n_{1}+n_{2}]$ is the union of the disjoint subsets $S_1$ and $n_1 + S_2$ inside $[n_1] \cup_{\{n_1\}} (n_1+ [n_2])=[n_{1}+n_{2}]$. For morphisms $\ph_{i}: (n_{i},S_{i})\to (n'_{i}, T_{i})$ ($i=1,2$), their product $\ph=\ph_{1}\star\ph_{2}: [n_{1}+n_{2}]\to [n_{1}'+n'_{2}]$ is defined by 
\begin{equation*}
\ph(i)=\begin{cases} \ph_{1}(i), & 0\le i\le n_{1}; \\ \ph_{2}(i-n_{1})+n'_{1}, & n_{1}\le i\le n_{1}+n_{2}. \end{cases}
\end{equation*}
When $i=n_{1}$, $\ph_{1}(n_{1})=n'_{1}$ and $\ph_{2}(i-n_{1})+n'_{1}=\ph_{2}(0)+n'_{1}=n'_{1}$ by Remark \ref{rem:active}, so the above definition is consistent.  The monoidal unit is $(0, \vn)$.

\sss{The module categories $\cC_{L}$ and $\cC_{R}$}
Define the full subcategories $\cC_{L} \subset \cC$ (resp. $\cC_R \subset \cC$) with
objects given by pairs $(n, S)$ of an integer $n\geq 0$ and a subset $S\subset [n]$ disjoint from the endpoint $0$ (resp. $n$). 

The formulas \eqref{uplus} make $\cC_L$ (resp. $\cC_R$) a left (resp. right) $\cC_\star$-module.

\subsubsection{Coalgebra object in correspondences}\label{sss:coalg}
Recall  the correspondence 2-category $\Corr^\adm_{\ver;\hor}(\cC^{op})$ defined above.

Following Remark~\ref{rem:active}, the pushouts of Lemma~\ref{l:pushout} preserve the full subcategory
$\cC_\star \subset \cC$. Thus we can form the
 2-category 
$\Corr^\adm_{\ver;\hor}(\cC_\star^{op}) \subset \Corr^\adm_{\ver;\hor}(\cC^{op})$ by restricting to objects in $\cC_\star \subset \cC$.  Moreover, the monoidal structure on 
$\cC_\star$ naturally induces a monoidal structure on  $\Corr^\adm_{\ver;\hor}(\cC_\star^{op})$, given by the same formula \eqref{uplus} on objects. For two morphisms in $\Corr^\adm_{\ver;\hor}(\cC_\star^{op})$ given as follows
\begin{equation}\label{two Corr}
\xymatrix{ (n'_{1},T_{1}) & \ar[l]_-{\psi_{1}}(n_{1},S_{1}) & (n'_{2}, T_{2}) & \ar[l]_-{\psi_{2}}(n_{2},S_{2})\\
\ar[u]^{\ph_{1}} (n_{1}', S_{1}') && \ar[u]^{\ph_{2}} (n_{2}',S_{2}')   
}
\end{equation}
their product under the monoidal structure is
\begin{equation}\label{monoidal Corr}
\xymatrix{ (n'_{1}+n'_{2},T_{1}\uplus T_{2})  & \ar[l]_-{\psi=\psi_{1}\star\psi_{2}} (n_{1}+n_{2}, S_{1}\uplus S_{2})\\ 
(n'_{1}+n'_{2}, S'_{1}\uplus S'_{2})\ar[u]^{\ph=\ph_{1}\star\ph_{2}}}
\end{equation}

Recall $c^{n}=(n,\vn)\in \cC_{\star}$, and $c^{n}$ is the $n$-th tensor power of $c^{1}$ under monoidal structure on $\cC_{\star}$.  Viewed as an object in $\Corr^\adm_{\ver;\hor}(\cC_\star^{op})$, $c^1 $ is naturally a (non-counital) coalgebra object with comultiplication given by 
 the correspondence
$$
\xymatrix{(2, \{1\})  & \ar[l] c^1 = (1, \vn)\\
\ar[u]   (2, \vn) = c^{2} = c^1 \star c^1
}
$$
More generally, all possible iterated comultiplications are given by 
 the correspondence
$$
\xymatrix{(n, \{1, \ldots, n-1\})  & \ar[l]c^1 = (1, \vn) \\
   \ar[u] (n, \vn) = c^{n} = (c^1)^{\star n}
}
$$
exhibiting the coassociativity. Here the horizontal map is the unique one with image $\{0,n\}\subset [n]$.

Similarly, 
we can form the
 2-category 
$\Corr^\adm_{\ver;\hor}(\cC_L^{op}) \subset \Corr^\adm_{\ver;\hor}(\cC^{op})$ 
(resp. $\Corr^\adm_{\ver;\hor}(\cC_R^{op}) \subset \Corr^\adm_{\ver;\hor}(\cC^{op})$)
  by restricting to objects in $\cC_L \subset \cC$ (resp.
  $\cC_R \subset \cC$).  
  Moreover, the module structure on 
$\cC_L$ (resp. $\cC_R$) naturally induces a module structure on  $\Corr^\adm_{\ver;\hor}(\cC_L^{op})$
(resp.  $\Corr^\adm_{\ver;\hor}(\cC_R^{op})$).

The same formulas as above make $c_L^0 = (0, \vn)\in \Corr^\adm_{\ver;\hor}(\cC_L^{op})$ (resp. 
  $c_R^0 = (0, \vn)\in \Corr^\adm_{\ver;\hor}(\cC_R^{op})$) a left (resp. right) $c^1$-comodule.

\subsection{Monoidal structure on functors}\label{ss:mon str}

We have engineered $\cC_\star$ in order to  define certain monoidal functors with domain $\cC_\star$. Likewise,
we have engineered $\cC_L, \cC_R$ in order to  define corresponding  functors of modules.

In what follows, given a functor $F$ with domain $\cC$, we write $F_\star = F|_{\cC_\star}$, $F_L = F|_{\cC_L}$,
$F_R = F|_{\cC_R}$ for its respective restrictions.

\subsubsection{Bases} Let $(C,g,c_{0})$ be a base. Recall the  functors $\sfA, \sfB:\cC \to \Sta$ defined in~\S\ref{sss:alg base stack} and \S\ref{sss:base B}. Note the 
restrictions $\sfA_\star $ and  $\sfB_\star$ coincide
since $\sfB = \sfA \times_{\Pi_{k}, \AA^1} C$,  the assignments $\sfA_{(n, S)}$, for $[n] \setminus S$ nonempty, project to $0\in \AA^1$, and $g^{-1}(0) = c_0$ is a single point. 
Similarly, 
 the 
restrictions $\sfA_L $,  $\sfB_L$ (resp. $\sfA_R $,  $\sfB_R$) coincide.

Now $\sfA_\star =\sfB_\star$, we still use $\sfB_{\star}$ below.

Equip  $\Sta$ with  the monoidal structure given by Cartesian product over $\CC$.
 We will equip the restriction $\sfA_\star$ with a natural monoidal structure as follows.

For $(n, S)\in \cC$ with $n\geq 0$, recall $\sfB_{(n, S)} = A_{(n, S)}/S\Gm^{[n]}$ where $A_{(n, S)} \subset A^{[n]}$ is the coordinate subspace cut out by $\e_i = 0$, for $i\not \in S$. 
In particular, to the monoidal unit $(0, \vn) \in \cC_\star$, we assign a point $\sfB_{(0, \vn)} = \{0\}$ arising as the origin  in  $A^{[0]}\simeq \AA^1$.  

For $n\geq 0$, and  $i\in [n]$, set $i^{c}= [n] \setminus \{i\}$. 
Observe we have a canonical isomorphism
$
\sfB_{(n, i^{c})}  \simeq \AA^{n}/ \Gm^{n}
$
where $ \Gm^{n}$-acts on $\AA^n$ by coordinate-wise scaling. 
Therefore for $n_1, n_2\geq 0$, we have a canonical isomorphism
\beq\label{eq:basic mon isom}
\xymatrix{
\sfB_{(n_1, n_{1}^{c})} \times \sfB_{(n_2, 0^{c})} \simeq  \AA^{n_1}/ \Gm^{n_1}  \times \AA^{n_2}/ \Gm^{n_2}  
\simeq \sfB_{(n_1 + n_2, n_1^{c})}
}
\eeq
where we understand $(n_1,  n_{1}^{c})=(n,\vn)$ if $n_1=0$, and $(n_{2}, 0^{c}) =(n_{2}, \vn)$
if $n_2=0$.
Thus for $(n_1, S_1), (n_2, \cC_2)\in \cC_\star$, we have a canonical isomorphism
\begin{equation*}
\xymatrix{
\sfB_{(n_1, S_1)} \times \sfB_{(n_2, S_2)} \simeq  \sfB_{(n_1 + n_2, S_1 \uplus S_2)}
}
\end{equation*}
cut out of the identity \eqref{eq:basic mon isom} by the further equations $\e_i = 0$, for $i\not \in S_1 \cup S_2$. 
These isomorphisms provide a monoidal structure on objects, 
and its extension to morphisms is evident from the definitions.

The same formulas  make $\sfB_L $ (resp. $\sfB_R $) a functor of left (resp. right) $\cC_\star$-modules via $\sfB_\star$. 

As in~\S\ref{sss:A plus} and \S\ref{sss:base B}, we pass to real and positive versions of $\sfB_{L}, \sfB_{R}$ and $\sfB_{\star}$ to obtain functors 
\begin{eqnarray*}
(\sfB_{L}, \sfB^{\RR}_{L}\supset\sfB^{+}_{L})  : & \cC_{L}\to \frR^{\Sch}_{\Sta,\supset},\\
(\sfB_{R}, \sfB^{\RR}_{R}\supset\sfB^{+}_{R}):  &\cC_{R}\to \frR^{\Sch}_{\Sta,\supset},\\
(\sfB_{\star}, \sfB^{\RR}_{\star}\supset \sfB^{+}_{\star}): &\cC_{\star}\to \frR^{\Sch}_{\Sta,\supset}.
\end{eqnarray*}


%

The monoidal structure on $\sfB_\star$ extends to one on $(\sfB_{\star},\sfB^{\RR}_{\star}\supset\sfB^+_\star)$. Moreover,
$(\sfB_{L},\sfB^{\RR}_{L}\supset \sfB^+_L)$ (resp. $(\sfB_{R}, \sfB^{\RR}_{R}\supset\sfB^+_R)$) are functors of left (resp. right) $\cC_\star$-modules via $(\sfB_{\star}, \sfB^{\RR}_{\star}\supset\sfB^+_\star)$.

\subsubsection{Truncated curves}\label{sss:tr curves}
Let $\t:\frY \to C$ be a rigidified separating nodal degeneration, and $\pi: \frX^{\D}\to B^{\D}$ be its $k$-th twisted cosimplicial bubbling. Recall from \S\ref{sss:tw exc loci} that $\frX^{\D}_{\t}$ is the exceptional locus of $\frX^{\D}$, i.e., the preimage of $Q\subset \frX$ in $\frX^{\D}$.  Recall the functor  $\sfX:\cC \to \Sta$ defined in \S\ref{sss:sfX}. 

We define a functor  $\sfX_\tau:\cC_\star \to \Sta$ as follows
\begin{equation*}
\sfX_{\t,(n,S)}=\frX^{[n]}_{\t}|_{B_{(n, S)}}, \quad (n,S)\in \cC_{\star}.
\end{equation*}
Then we have an object-wise closed embedding $\sfX_\tau \incl \sfX_{\star}$ as follows. Informally, for $(n,S)\in \cC_{\star}$, each fiber of $\sfX_{(n,S)}\to \sfB_{(n,S)}$ has two extremal components isomorphic to $X_{-}$ and $X_{+}$, and  $\sfX_{\t,(n,S)}$ is obtained by removing the two constant extremal components while keeping the twisted nodes $Q_{-}\subset X_{-}$ and $Q_{+}\subset X_{+}$. By definition, for $(n,S)\in \cC_{\star}$ we have
\begin{equation*}
\sfX_{(n,S)}=(X_{-}\times \sfB_{(n,S)})\cup_{Q_{-}\times \sfB_{(n,S)}}\sfX_{\t,(n,S)} \cup_{Q_{+}\times \sfB_{(n,S)}}(X_{+}\times \sfB_{(n,S)}).
\end{equation*}
By convention, $\sfX_{\t,(0,\vn)}=Q=X_{+}\cap X_{-}$.

%


The functor $\sfX_\tau$  interacts with the monoidal structure of $\cC_{\star}$ in the following way. For $(n_{1},S_{1}), (n_{2}, S_{2})\in \cC_{\star}$, we have
\begin{equation}\label{X tau union}
\sfX_{\t, (n_{1}+n_{2}, S_{1}\uplus S_{2})}=(\sfX_{\t, (n_{1},S_{1})}\times \sfB_{(n_{2},S_{2})})\cup (\sfB_{(n_{1},S_{1})}\times \sfX_{\t, (n_{2},S_{2})})
\end{equation}
where the union identifies the nodal sections indexed by $n_{1}\in [n_{1}]$ of the first part with the nodal sections indexed by $0\in [n_{2}]$ of the second. 

\begin{remark}\label{r:X tau indep}
Starting with the family $\t_{\WW}: \WW\to \AA^{1}$ we have its $k$-th twisted bubbling $\wt\WW^{\D}/\mu_{k}\to A^{\D}$. We may similarly define the truncated version
\begin{equation}\label{sfW tau}
\wt\sfW_{\t}: \cC_{\star}\to \Sch
\end{equation}
that sends $(n,S)\in \cC_{\star}$ to $\wt\sfW_{\t,(n,S)}$ by removing the extremal components and keeping their  nodes. 
 
Note that $\sfX_{\t,(n,S)}$ always maps to the nodes $Q\subset \frX$. The isomorphism \eqref{tw exc locus} implies that 
\begin{equation*}
\sfX_{\t,(n,S)}\simeq R\times [\wt\sfW_{\t, (n,S)}/\mu_{k}], \quad (n,S)\in \cC_{\star}
\end{equation*}
where $\wt\sfW_{\t, (n,S)}$ is the counterpart of $\sfX_{\t,(n,S)}$ for the standard bubbling $\wt\WW^{\D}$. Note this isomorphism depends on the trivialization of $\om_{Y_{-}}|_{R}$ that fix throughout.
Therefore, except for the factor $R$, the monoidal functor $\sfX_\tau:\cC_\star\to \Sta$ is independent of the input data $\pi: \frX\to B$. 
In particular, it assigns $R\times(\PP^1/\mu_k)$ to $(1, \vn) \in \cC_\star$. 
\end{remark}

We define a functor  $\sfX_{\t, L}: \cC_L\to \Sta$  with an object-wise closed embedding $\sfX_{\tau, L}\to \sfX_{L}$ as follows. For $(n,S)\in \cC_{L}$, since $0 \not \in S$, the curve $\sfX_{(n, S)}$ has a constant extremal component $ X_- \times \sfB_{(n, S)}$. We take $ \sfX_{\tau, (n,S)}$ to be what results from $\sfX_{(n,S)}$ by removing the constant extremal component $X_-\times \sfB_{(n,S)}$ while keeping its twisted nodes.  Similarly we define a functor $\sfX_{\t, R}: \cC_R \to \Sta$ with an object-wise closed embedding $\sfX_{\tau, R}\to \sfX_{R}$, by removing the constant extremal component $X_+\times \sfB_{(n,S)}$ for $(n,S)\in \cC_{R}$ and keeping its nodes.

%


Moreover, for $(n_{1},S_{1})\in \cC_{L}$ and $(n_{2},S_{2})\in \cC_{\star}$, we have 
\begin{equation}\label{X tau L}
\sfX_{\t, L, (n_{1}+n_{2}, S_{1}\uplus S_{2})}=(\sfX_{\t, L, (n_{1},S_{1})}\times \sfB_{(n_{2},S_{2})})\cup (\sfB_{(n_{1},S_{1})}\times \sfX_{\t, (n_{2},S_{2})}).
\end{equation}
Similarly for  $\sfX_{\tau, R}$.

\begin{remark} Parellel to Remark \ref{r:X tau indep}, 
the  functor $\sfX_{\tau, L}:\cC_L\to \Sta$ (resp. 
$\sfX_{\tau, R}:\cC_R\to \Sta$) with its structure as a map of modules
 only depends on the DM curve $X_+$ (resp. $X_-$). 
In particular, it assigns $X_+$ (resp. $X_-$) to $(0, \vn) \in \cC_L$ (resp. $(0, \vn) \in \cC_R$).
\end{remark}

Recall the functor   $(\sfX, \sfX^{\RR}\supset\sfX^+):\cC \to \frR^{\Sch}_{\Sta}$  defined in \S\ref{sss:sfX}. Define its truncated versions
\begin{eqnarray*}
(\sfX_{\t}, \sfX^\RR_{\tau} = \sfX_\tau  \times_{\sfB_\star} \sfB^\RR_\star\supset \sfX^+_{\tau} = \sfX_\tau  \times_{\sfB_\star} \sfB^+_\star) :\cC_\star\to \frR^{\Sch}_{\Sta,\supset},\\
(\sfX_{\t,L}, \sfX^\RR_{\tau, L} = \sfX_{\tau, L}  \times_{\sfB_L} \sfB^\RR_L\supset \sfX^+_{\tau, L} = \sfX_{\tau, L}  \times_{\sfB_L} \sfB^+_L) : \cC_L\to \frR^{\Sch}_{\Sta,\supset}\\
(\sfX_{\t,R}, \sfX^\RR_{\tau, R} = \sfX_{\tau, R}  \times_{\sfB_R} \sfB^\RR_R\supset\sfX^+_{\tau, R} = \sfX_{\tau, R}  \times_{\sfB_R} \sfB^+_R) : \cC_R\to \frR^{\Sch}_{\Sta,\supset}.
\end{eqnarray*}
We have analogues of \eqref{X tau union} and \eqref{X tau L} for these extended functors.

\subsubsection{Moduli of bundles on truncated curves} 
Recall $\cY_{k}=\Bun_{G}(\BB\mu_{k})\simeq\coprod_{\y\in \YY_{k}}\BB G_{\y}$.  Denote by $\cY^{R}_{k}$ is Cartesian power of $\cY_{k}$ indexed by $R$, the set of nodes $Y_{-}\cap Y_{+}$. 

For a stack $S$, let $\Sta/S$ be the overcategory of $S$, i.e., it consists of stacks over $S$ and maps over $S$. 
Equip  $\Sta/\cY_{k}^{R} \times \cY_{k}^{R}$ with  the cartesian monoidal structure relative to $\cY_{k}^{R}$, i.e., $(S,T)\mapsto S\times_{\cY_{k}^{R}} T$ where the fiber product uses the second map $p_{S,2}: S\to \cY_{k}^{R}$ and the first map $p_{T,1}: T\to \cY_{k}^{R}$; the resulting stack is viewed as a stack over $\cY_{k}^{R} \times \cY_{k}^{R}$ via $(p_{S,1},p_{T,2}): S\times_{\cY_{k}^{R}} T\to \cY_{k}^{R} \times \cY_{k}^{R}$. Similarly,  $\Sta/\cY_{k}^{R}$ is equipped with the natural left and right $\Sta/\cY_{k}^{R}\times \cY_{k}^{R}$-module structures.

Recall from \S\ref{sss:sfBun} the functor $\sfBun:\cC\to \Sta$ obtained by passing to the moduli stacks of $G$-bundles along the fibers of $\sfX\to \sfB$. Set $\sfBun_\tau:\cC_\star\to \Sta$ to be the functors of moduli stacks of $G$-bundles along the fibers of $\sfX_\tau\to \sfB_\star$.   For $(n,S)\in \cC_{\star}$, $\sfX_{\t,(n,S)}$ is equipped with two collections of (twisted nodal) sections $\sfQ^{0}_{(n,S)}$ and $\sfQ^{n}_{(n,S)}$, each isomorphic to $Q\times \sfB_{(n,S)}$.  In \S\ref{sss:loc br} we have named the $+$-branch along each of $\sfQ^{i}_{(n,S)}$. Restricting $G$-bundles to these sections gives two maps
\begin{equation*}
r_{0}, r_{n}:\sfBun_{\t, (n,S)}=\Bun_{G}(\sfX_{\t, (n,S)}/\sfB_{(n,S)})\to \Bun_{G}(\BB\mu_{k})^{R }=\cY_{k}^{R}.
\end{equation*}
This allows us to lift the functor $\sfBun_{\t}$ to
\begin{equation*}
\sfBun_\tau:\cC_\star\to \Sta/\cY_{k}^{R} \times \cY_{k}^{R}.
\end{equation*}
Observe that $\sfBun_\tau$ inherits a monoidal structure defined as follows. For $(n_{1},S_{1}), (n_{2},S_{2})\in \cC_{\star}$, there is a canonical isomorphism over $\cY_{k}^{R} \times \cY_{k}^{R}$
\begin{equation}\label{Bun tau monoidal}
\sfBun_{\tau, (n_{1}+n_{2},S_{1}\uplus S_{2})}\simeq \sfBun_{\tau, (n_{1},S_{1})}\times_{\cY_{k}^{R} } \sfBun_{\t, (n_{2},S_{2})}
\end{equation}
compatible with the monoidal structure on $\sfB_{\star}$. This follows from \eqref{X tau union}. 

Similarly we define functors
\begin{eqnarray*}
\sfBun_{\tau, L}:\cC_L\to \Sta/\cY_{k}^{R},\\
\sfBun_{\tau, R} :\cC_R\to \Sta/\cY_{k}^{R}
\end{eqnarray*}
to be the functor of moduli stacks of $G$-bundles along the fibers of $\sfX_{\tau, L}\to \sfB_L$, $\sfX_{\tau, R}\to \sfB_{R}$ respectively. Moreover, $\sfBun_{\tau, L}$ (resp. $\sfBun_{\tau, R}$) carries the structure of a map of left (resp right) $\cC_\star$-module via  $\sfBun_\tau$. These structures follow from \eqref{X tau L} and its analogue for $\sfX_{\tau, R}$.


Define the real and positive versions of $\sfBun_\tau, \sfBun_{\t, L}$ and $\sfBun_{\t,R}$ by base change
\begin{eqnarray*}
(\sfBun_{\t}, \sfBun^\RR_\tau = \sfBun_\tau  \times_{\sfB_\star} \sfB^\RR_\star\supset \sfBun^+_\tau = \sfBun_\tau  \times_{\sfB_\star} \sfB^+_\star):\cC_\star\to \frR^{\Sch}_{\Sta/\cY_{k}^{R}\times\cY_{k}^{R},\supset},\\
(\sfBun_{\t,L},\sfBun^\RR_{\tau, L} = \sfBun_{\tau, L}  \times_{\sfB_L} \sfB^\RR_L\supset \sfBun^+_{\tau, L} = \sfBun_{\tau, L}  \times_{\sfB_L} \sfB^+_L): \cC_L\to \frR^{\Sch}_{\Sta/\cY_{k}^{R},\supset},\\
(\sfBun_{\t,R},\sfBun^\RR_{\tau, R} = \sfBun_{\tau, R}  \times_{\sfB_R} \sfB^\RR_R\supset \sfBun^+_{\tau, R} = \sfBun_{\tau, R}  \times_{\sfB_R} \sfB^+_R): \cC_{R}\to \frR^{\Sch}_{\Sta/\cY_{k}^{R},\supset}.
\end{eqnarray*}
They inherit  monoidal and module structures analogous to \eqref{Bun tau monoidal}. Here $\frR^{\Sch}_{\Sta/S,\supset}$ is the overcategory for $(S,S\supset S)\in \frR^{\Sch}_{\Sta,\supset}$.


\sss{Moduli of bundles with fixed twisting type}
Below we fix $\y\in \YY_{k}$ and a $G$-torsor $\EE_{\y}$ over $\BB\mu_{k}$  of type $\y$ such that $\Aut(\EE_{\y})$ is a maximal torus which we denote by $H$ (the existence of $\y$ requires $k$ to be not too small). Recall the functors $\sfBun^{\y}$ and $\sfBun^{+,\y}$ defined in \S\ref{ss:fix type} obtained as moduli stacks of $G$-bundles along the fibers of $\sfX\to \sfB$ with a fixed isomorphism type $\y$ when restricted to the twisted nodes. Replacing $\sfX$ by $\sfX_{\t}$ we obtain an open subfunctor of $\sfBun^+_\tau$ by fixing the twisting types at nodes to be $\y$:
\begin{equation*}
(\sfBun^{\y}_{\t}, \sfBun^{\RR, \y}_\tau\supset \sfBun^{+, \y}_\tau):\cC_\star \to \frR^{\Sch}_{\Sta/(\BB H)^{R}\times (\BB H)^{R},\supset}.
\end{equation*}
Here $(\BB H)^{R}$ appears as the component of $\cY_{k}^{R}$ corresponding to the isomorphism class $\y^{R}$.  Base changing along the map $\pt\to (\BB H)^{R}\times (\BB H)^{R}$ we obtain a functor 
\begin{equation*}
(\wh\sfBun^{\y}_{\t}, \wh\sfBun^{\RR, \y}_\tau\supset\wh\sfBun^{+, \y}_\tau):\cC_\star \to (\frR^{\Sch}_{\Sta,\supset})^{H^{R}\times H^{R}}.
\end{equation*}
Here $(\frR^{\Sch}_{\Sta,\supset})^{H^{R}\times H^{R}}$ means objects in $\frR^{\Sch}_{\Sta,\supset}$ with an action of $H^{R}\times H^{R}$. Concretely, $\wh\sfBun^{\y}_{\tau,(n,S)}$ is the moduli stack of $G$-bundles along the fibers of $\sfX_{\t,(n,S)}\to \sfB_{(n,S)}$, with an isomorphism with $\EE_{\y}$ along $\sfQ^{0}_{(n,S)}$ and $\sfQ^{n}_{(n,S)}$, and with isomorphism type $\y$ at all other twisted nodes.
 
For two objects $\cZ_{1}, \cZ_{2}\in (\frR^{\Sch}_{\Sta,\supset})^{H^{R}\times H^{R}}$, we may form the contracted product $\cZ_{1}\twtimes{H^{R}}\cZ_{2}\in (\frR^{\Sch}_{\Sta,\supset})^{H^{R}\times H^{R}}$ (the contracted product uses the second $H^{R}$-action on $\cZ_{1}$ and the first on $\cZ_{2}$). This defines a monoidal structure on $(\frR^{\Sch}_{\Sta,\supset})^{H^{R}\times H^{R}}$.   The functor $(\wh\sfBun^{\y}_\tau, \wh\sfBun^{\RR, \y}_\tau\supset\wh\sfBun^{+, \y}_\tau)$ inherits a monoidal structure from that of $\Bun_{\t}$ in \eqref{Bun tau monoidal}: for any $(n_{1},S_{2}), (n_{2},S_{2})\in \cC_{\star}$, there is a canonical isomorphism
\begin{equation}\label{Bun tau y monoidal}
\wh\sfBun^{+,\y}_{\tau, (n_{1}+n_{2}, S_{1}\uplus S_{2})}\simeq \wh\sfBun^{+, \y}_{\tau, (n_{1},S_{1})}\twtimes{H^{R}}\wh\sfBun^{+,\y}_{\tau, (n_{2},S_{2})}
\end{equation}
compatible with the monoidal structure of $\sfB_{\star}$. 

Similarly, we have the open subfunctors of $\sfBun^+_{\tau, L}$ and $\sfBun^+_{\tau, R}$ by fixing twisting types to $\y$:
\begin{eqnarray*}
(\sfBun^{\y}_{\tau, L}, \sfBun^{\RR, \y}_{\tau, L}\supset\sfBun^{+, \y}_{\tau, L}):\cC_L \to \frR^{\Sch}_{\Sta/(\BB H)^{R},\supset},\\
(\sfBun^{ \y}_{\tau, R}, \sfBun^{\RR, \y}_{\tau, R}\supset\sfBun^{+, \y}_{\tau, R}):\cC_R \to \frR^{\Sch}_{\Sta/(\BB H)^{R},\supset}.
\end{eqnarray*}
Base changing along the map $\pt\to (\BB H)^{R}$ we obtain functors
\begin{eqnarray*}
(\wh\sfBun^{\y}_{\tau, L},\wh\sfBun^{\RR, \y}_{\tau, L}\supset\wh\sfBun^{+, \y}_{\tau, L}):\cC_L \to (\frR^{\Sch}_{\Sta,\supset})^{H^{R}},\\
(\wh\sfBun^{\y}_{\tau, R}, \wh\sfBun^{\RR, \y}_{\tau, R}\supset\wh\sfBun^{+, \y}_{\tau, R}):\cC_R \to (\frR^{\Sch}_{\Sta,\supset})^{H^{R}}.
\end{eqnarray*}

\subsubsection{Nilpotent sheaves}\label{sss:nilp sh tau}
Recall from \S\ref{ss:fix type} the functor $\Phi^{\y}_*:\cC \to \St$  
and the induced
 functor
$$
\xymatrix{
\Phi^{\y}_{\Corr}
: \Corr^\adm_{\ver;\hor}(\cC^{op}) \ar[r] & \St
}
$$
obtained by passing to nilpotent sheaves on $\sfBun^{+,\y}_{(n,S)}$.

For $(n,S)\in \cC_{\star}$, we extend the notion of a relatively nilpotent sheaf to sheaves on $\wh\sfBun^{+,\y}_{\t, (n,S)}$. Namely, let $\cN^{+,\y}_{\t, (n,S)}\subset T^{*}\sfp^{+,\y}_{\t,(n,S)}$ be the relative nilpotent cone for $\sfp^{+,\y}_{\t,(n,S)}: \sfBun^{+,\y}_{\t, (n,S)}\to \sfB^{+}_{(n,S)}$. Let $\wh\sfp^{+,\y}_{\t,(n,S)}: \wh\sfBun^{+,\y}_{\t, (n,S)}\to \sfB^{+}_{(n,S)}$ be the projection. Then $T^{*}\sfp^{+,\y}_{\t,(n,S)}\incl T^{*}\wh\sfp^{+,\y}_{\t,(n,S)}$ is a closed embedding, hence we view $\cN^{+,\y}_{\t, (n,S)}$ as a closed substack of $T^{*}\wh\sfp^{+,\y}_{\t,(n,S)}$. We call $\cF\in \Sh(\wh\sfBun^{+,\y}_{\t, (n,S)})$ {\em relatively nilpotent} if the image of $ss(\cF)$ in $T^{*}\wh\sfp^{+,\y}_{\t,(n,S)}$  lies in $\cN^{+,\y}_{\t, (n,S)}$. Let
\begin{equation*}
\Phi^{\y}_{\tau, *}(n,S):=\Sh_{\cN}(\wh\sfBun^{+,\y}_{\t, (n,S)}).
\end{equation*}
be the category of relatively nilpotent sheaves on $\wh\sfBun^{+,\y}_{\t, (n,S)}$. In particular, sheaves in $\Sh_{\cN}(\wh\sfBun^{+,\y}_{\t, (n,S)})$ are monodromic under the action of $H^{R}\times H^{R}$. 

Let $\Sh_{0}(H^{R})$ be the category of sheaves on the torus $H^{R}$ with singular support in the zero section, i.e., with locally constant cohomology sheaves. This is a monoidal category under convolution. Then $\Phi^{\y}_{\tau, *}(n,S)$ carries an action of $\Sh_{0}(H^{R})\ot\Sh_0(H^{R})$ by convolution.  Therefore $(n,S)\mapsto \Phi^{\y}_{\tau, *}(n,S)$ gives a functor
\begin{equation*}
\Phi^{\y}_{\tau, *}:\cC_\star \to  \Sh_0(H^{R}) \otimes \Sh_0(H^{R})\module
\end{equation*}

The analogue of Proposition \ref{p:sh functor from C} holds for $\Phi^{\y}_{\tau, *}$,  and it induces a
 functor
$$
\xymatrix{
\Phi^{ \y}_{\tau,\Corr}
: \Corr^\adm_{\ver;\hor}(\cC_\star^{op}) \ar[r] & \Sh_0(H^{R}) \otimes \Sh_0(H^{R})\module
}
$$

Similarly, we have the functors $\Phi^{ \y}_{\tau, L, *}: \cC_{L}\to  \Sh_0(H^{R})\module$ and $\Phi^{ \y}_{\tau, R,*}: \cC_L \to  \Sh_0(H^{R})\module$  by passing to relatively nilpotent sheaves, where the $\Sh_0(H^{R})$-actions  come from $H^{R}$-actions on $\Sh_{\cN}(\wh\sfBun^{+,\y}_{\t, (n,S)})$ by changing the isomorphism with $\EE_{\y}$ along $\sfQ^{0}_{(n,S)}$ (in case $(n,S)\in \cC_{L}$) or $\sfQ^{n}_{(n,S)}$ (in case $(n,S)\in \cC_{R}$). In particular,
\begin{eqnarray}
\label{PhiL c0L}\Phi^{ \y}_{\tau, L, *}(0,\vn)=\Sh_{\cN}(\Bun_{G,1}(X_{+},Q_{+})_{\y});\\
\label{PhiR c0R} \Phi^{ \y}_{\tau, R, *}(0,\vn)=\Sh_{\cN}(\Bun_{G,1}(X_{-},Q_{-})_{\y}).
\end{eqnarray}
The functors $\Phi^{ \y}_{\tau, L, *}$ and $\Phi^{ \y}_{\tau, R, *}$ induce functors
\begin{equation*}
\xymatrix{
\Phi^{\y}_{\tau, L,\Corr}
: \Corr^\adm_{\ver;\hor}(\cC_L^{op}) \ar[r] &  \Sh_0(H^{R})\module
}
\end{equation*}
\begin{equation*}
\xymatrix{
\Phi^{\y}_{\tau, R, \Corr} 
: \Corr^\adm_{\ver;\hor}(\cC_R^{op}) \ar[r] &  \Sh_0(H^{R})\module
}
\end{equation*}

Equip $ \Sh_0(H^{R}) \otimes \Sh_0(H^{R})\module$ with the  monoidal structure given by tensoring over the middle $\Sh_0(H^{R})$, i.e., $(-)\ot_{\Sh_0(H^{R})}(-)$. 

\begin{lemma}\label{l:Phi eta tau Corr}
\begin{enumerate}
\item $\Phi^{\y}_{\tau, \Corr}$ is naturally monoidal.
\item  $\Phi^{ \y}_{\tau, L,\Corr}$ (resp. $\Phi^{ \y}_{\tau, R,\Corr}$) is naturally a map of left (resp. right) $\cC_\star$-modules via $\Phi^\y_{\tau,\Corr}$.
\item There are natural isomorphisms of functors $\Corr^\adm_{\ver;\hor}(\cC_\star^{op})\to \St$
\begin{equation*}
\xymatrix{
\Phi^\y_{\Corr}|_{\Corr^\adm_{\ver;\hor}(\cC_\star^{op})} \simeq \Sh_\cN(\Bun_{G,1}(X_-, Q_{-})_{\y}) \ot_{\Sh_0(H^R)} \Phi^{\y}_{\t,\Corr}
\ot_{\Sh_0(H^R)} \Sh_\cN(\Bun_{G,1}(X_+, Q_{+})_{\y})
}\end{equation*}
Similarly, we have natural isomorphisms of functors
\begin{equation*}
\xymatrix{
\Phi^\y_{*}|_{\cC_L} \simeq \Sh_\cN(\Bun_{G,1}(X_-, Q_{-})_{\y}) \ot_{\Sh_0(H^R)} \Phi_{\tau, L,*}^{\y}
}\end{equation*}
\begin{equation*}
\xymatrix{
\Phi^\y_{*}|_{\cC_R} \simeq  \Phi_{\tau, R,*}^{\y} \ot_{\Sh_0(H^R)} \Sh_\cN(\Bun_{G,1}(X_+, Q_{+})_{\y})
}
\end{equation*}
\end{enumerate}
\end{lemma}

\begin{proof} 
(1) Let $(n_{1},S_{1}), (n_{2},S_{2})\in \cC_{\star}$. The isomorphism \eqref{Bun tau y monoidal} together with Lemma~\ref{lem:monod gl} gives an equivalence
\begin{equation*}
\Sh_{\cN}(\wh\sfBun^{+,\y}_{\t, (n_{1}+n_{2}, S_{1}\uplus S_{2})})\simeq \Sh_{\cN}(\wh\sfBun^{+,\y}_{\t, (n_{1},S_{1})})\ot_{\Sh_0(H^R)}\Sh_{\cN}(\wh\sfBun^{+,\y}_{\t, (n_{2},S_{2})}).
\end{equation*}
This defines the desired monoidal structure of $\Phi^{\y}_{\t, \Corr}$ on objects. For two morphisms induced by two correspondences $(\ph_{1},\psi_{1})$ and $(\ph_{2}, \psi_{2})$ as in \eqref{two Corr} whose product is the correspondence $(\ph,\psi)$ in \eqref{monoidal Corr}, we need to give a canonical isomorphism between the  functors
\begin{eqnarray*}
&&\Phi^{\y}_{\t, *} (n_{1},S_{1})\ot_{\Sh_0(H^R)}\Phi^{\y}_{\t, *}(n_{2},S_{2})\xr{\Phi_{*}(\psi_{1})\ot \Phi_{*}(\psi_{2})}\Phi^{\y}_{\t, *} (n'_{1},T_{1})\ot_{\Sh_0(H^R)}\Phi^{\y}_{\t, *}(n'_{2},T_{2})\\
&&\quad\quad \xr{\Phi^{*}(\ph_{1})\ot \Phi^{*}(\ph_{2})}\Phi^{\y}_{\t, *} (n'_{1},S'_{1})\ot_{\Sh_0(H^R)}\Phi^{\y}_{\t, *}(n'_{2},S'_{2}),\\
&&\Phi^{\y}_{\t, *} (n_{1}+n_{2},S_{1}\uplus S_{2})\xr{\Phi_{*}(\psi)}\Phi^{\y}_{\t, *} (n'_{1}+n'_{2},T_{1}\uplus T_{2})\xr{\Phi^{*}(\ph)}\Phi^{\y}_{\t, *} (n'_{1}+n'_{2},S'_{1}\uplus S'_{2}).
\end{eqnarray*}
This is provided by Lemma \ref{l:monod gl func}.

The proofs of (2), (3) are similar to that of (1).
\end{proof}


\subsection{Bubbling Hecke category}\label{ss:bub Hk}
In this subsection we apply the previous discussion to a compactification of the standard $k$th twisted bubbling $\wt\WW^{\D}/\mu_{k}\to A^{\D}$, and obtain a {\em bubbling Hecke algebra} $\cH^{\bub}$. We prove that $\cH^{\bub}$ is canonically equivalent to the affine Hecke category $\cH$ as a monoidal category.

\subsubsection{Compactified $\WW$}\label{sss:comp W}
Let $\frW\subset \PP^1 \times \PP^1 \times \AA^1$ be the surface defined by 
\begin{equation*}
\frW=\{([x, x'], [y,y'], t)\in \PP^1 \times \PP^1 \times \AA^1| xy  - t x'y' = 0\}.
\end{equation*}
Then $\t_{\frW}: \frW\to \AA^{1}$ is a separating nodal degeneration. For $t\in \Gm\subset \AA^{1}$, $\frW|_{t}\isom \PP^{1}$ via either projection, and the special fiber $\frW|_{0}\simeq \PP^{1}\cup \PP^{1}$ where the two $\PP^{1}$ are glued at the point $[0,1]$. We also have the sections $\sigma_-, \sigma_+:\AA^1 \to \frW$ given by $\sigma_-(t) = ([1, 0 ], [0, 1], t)$, $\sigma_+(t) = ([0, 1], [1, 0], t)$.
 
Note that $\WW\subset \WW^{\bu}$ as the complement of the images of $\sigma_{-}$ and $\sigma_{+}$.

We apply the cosimplicial bubbling construction \S\ref{sss:cons cosim} to $\t_{\frW}:\frW\to \AA^{1}$ to get a cosimplicial family $\t^{\D}_{\frW}: \frW^{\D}\to \AA^{\D}$. We then apply the twisting construction \S\ref{sss:tw bub} to obtain the $k$th twisted cosimplicial bubbling
\begin{equation*}
\pi^{\D}_{\frV}: \frV^{\D}\to A^{\D}
\end{equation*}
equipped with sections $\s^{\D}_{-}$ and $\s^{\D}_{+}$ induced from $\s_{-}$ and $\s_{+}$.  The analogue of the diagram \eqref{W muk} shows that
\begin{equation}\label{VW muk}
\frV^{[n]}\simeq \wt\frW^{[n]}/\mu_{k}.
\end{equation}
Here $\wt\frW^{[n]}$ is the cosimplicial bubbling of $\wt\t_{\frW}: \wt\frW\to \wt\AA^{1}=A^{[0]}$ which is a compactification of $\wt\WW$ in \S\ref{sss:tw std bub}, i.e., taking the $k$th roots of the coordinates of $\frW$ and $\AA^{1}$. The action of $\mu_{k}=\mu^{\WW}_{k}$ on $\wt\WW^{[n]}$ extends to the compactification $\wt\frW^{[n]}$. 

The fiber of $\frV^{[n]}$ over $0\in A^{[n]}$ is of the form
\begin{equation*}
V(n)\simeq (\PP^{1}_{-}\cup \PP^{1}_{[0,1]}\cup\cdots\cup \PP^{1}_{[n-1,n]}\cup \PP^{1}_{+})/\mu_{k}.
\end{equation*}
Here each of $\PP^{1}_{-}, \PP^{1}_{[i-1,i]}$ and $\PP^{1}_{+}$ is canonically isomorphic to $\PP^{1}$,  the gluing identifies $\infty\in \PP^{1}_{-}$ with $0\in \PP^{1}_{[0,1]}$, etc, and the $\mu_{k}$-action is the scaling action on the standard affine coordinate on each $\PP^{1}$.  The sections $\s^{n}_{-}$ and $\s^{n}_{+}$ intersects $V(n)$ at $0\in \PP^{1}_{-}$ and $\infty\in \PP^{1}_{+}$ respectively.

\subsubsection{Bubbling Hecke category}
We then apply the procedure of \S\ref{ss:geom functors} to $\pi^{\D}_{\frV}: \frV^{\D}\to A^{\D}$ to obtain  a functor $\sfV: \cC\to \Sta$ and a map of functors $\sfV\to \sfA$. We also have the positive real versions $\sfV^{+}, \sfBun^{\sfV,+}$ and the truncated version $\sfV_{\t}\to \sfA_{\star}, \sfV_{\t}^{+}\to \sfA_{\star}$.  The zero fiber of $\sfV_{\t,(n,[n])}$ is 
\begin{equation*}
V(n)_{\t}\simeq (\PP^{1}_{[0,1]}\cup\cdots\cup \PP^{1}_{[n-1,n]})/\mu_{k}.
\end{equation*}

Recall the functor $\wt\sfW_{\t}:\cC_{\star}\to \Sch$ defined in   \eqref{sfW tau}. Then \eqref{VW muk} implies a canonical isomorphism of functors
\begin{equation*}
\sfV_{\t}\simeq [\wt\sfW_{\t}/\mu_{k}]: \cC_{\star}\to \Sta.
\end{equation*}

Then we pass to the moduli stack of $G$-bundles on fibers of $\sfV\to \sfA$ with $N$-reductions along the given sections to obtain $\sfBun^{\sfV}$. We have the truncated and real versions $(\sfBun^{\sfV}_{\t}, \sfBun^{\sfV,\RR}_{\t}\supset\sfBun^{\sfV,+}_{\t})$ as a functor $\cC_{\star}\to\frR^{\Sch}_{\Sta,\supset}$. We also have the version with fixed twisting type $\y$ and rigidifications at the two extreme nodes:
\begin{equation*}
(\wh\sfBun^{\sfV, \y}_{\t}, \wh\sfBun^{\sfV,\RR, \y}_{\t}\supset\wh\sfBun^{\sfV,+, \y}_{\t}):  \cC_{\star}\to (\frR^{\Sch}_{\Sta,\supset})^{H\times H}. 
\end{equation*}

Applying the construction of \S\ref{sss:nilp sh tau},  passing to relatively nilpotent sheaves on $\wh\sfBun^{\sfV,+,\y}_{\t}\to \sfA^{+}_{\star}$, we get a functor
\begin{equation*}
\Phi^{\sfV,\y}_{\t,*}: \cC_{\star}\to \Sh_{0}(H)\ot\Sh_{0}(H)\module.
\end{equation*}
Lemma \ref{l:Phi eta tau Corr} implies that $\Phi^{\sfV,\y}_{\t,*}$ extends to a monoidal functor 
$$
\xymatrix{
\Phi^{\sfV,  \y}_{\tau,\Corr}
: \Corr^\adm_{\ver;\hor}(\cC_\star^{op}) \ar[r] & \Sh_0(H) \otimes \Sh_0(H)\module
}
$$

Recall the (non-counital) coalgebra object $c^1 = (1, \vn) \in  \Corr^\adm_{\ver;\hor}(\cC_\star^{op})$ introduced in \S\ref{sss:coalg}. 
Transporting $c^1$ under the monoidal functor  $\Phi^{\sfV, \y}_\tau$ produces a (non-counital) coalgebra
 $
 \cH^\bub = \Phi^{\sfV,\y}_{\tau,\Corr}(c^1)\in \Sh_0(H) \otimes \Sh_0(H)\module.
 $
 
 Let ${}_kP_k = [\PP^1/\mu_k]$ with the scaling action of $\mu_{k}$ on the standard affine coordinate of $\PP^{1}$.  Denote the image of $0$ and  $\infty$ in ${}_{k}P_{k}$ by $0_{k}$ and $\infty_{k}$ to indicate they are orbifold points. According the convention in \S\ref{sss:DM X}, the action of $\Aut(0_{k})$ on the tangent line at $0_{k}$ gives the tautological identification $\Aut(0_{k})=\mu_{k}$, while  the action of $\Aut(\infty_{k})$ on the tangent line at $\infty_{k}$ gives the {\em inverse} of the tautological identification $\Aut(\infty_{k})=\mu_{k}$.  Let $\Bun_{G,1}(P_{k,k},\{0_{k},\infty_{k}\})_{\y,-\y}$ be the moduli of $G$-bundles on ${}_{k}P_{k}$ together with an isomorphism with $\EE_{\y}$ at $0_{k}$ and an isomorphism with $\EE_{-\y}$ at $\infty_{k}$. Then as a $\Sh_0(H) \otimes \Sh_0(H)\module$, we have 
\begin{equation*}
\xymatrix{
 \cH^\bub = \Sh_\cN(\Bun_{G,1}(P_{k,k},\{0_{k},\infty_{k}\})_{\y,-\y}).
}
\end{equation*}

By Proposition~\ref{p:sh functor from C}(3), the functors defining the coalgebra structure on $
 \cH^\bub$ admit left adjoints. We will pass to the left adjoints and thus view
$ \cH^\bub$ as a (non-unital) algebra object in $\Sh_0(H) \otimes \Sh_0(H)\module$, i.e.,  a monoidal category.

\begin{defn}\label{def:bub Hk}
The {\em bubbling Hecke category} is the (non-unital)   monoidal category 
 \begin{equation*}
 \cH^\bub = \Phi^{\sfV,\y}_{\tau,\Corr}(c^1) \in \Sh_0(H) \otimes \Sh_0(H)\module
 \end{equation*}
\end{defn}

\sss{The object $e \in \cH^\bub$}\label{sss:e} We have a projection map $p:{}_{k}P_{k}\to \BB\mu_{k}$, hence a pullback map $\cY_{k}=\Bun_{G}(\BB\mu_{k}) \to \Bun_{G}({}_{k}P_{k})_{\y,-\y}$ (the minus sign comes from the inverse identification $\Aut(\infty_{k})\simeq\mu_{k}$). The image $U$ of this map turns out to be open. Denote the preimage of $U$ in $\Bun_{G,1}({}_{k}P_{k}, \{0_{k},\infty_{k}\})_{\y,-\y}$ by $\wt U$. Then $\wh U$ classifies a $G$-bundle $\cE$ on ${}_{k}P_{k}$, isomorphic to $p^{*}\EE_{\y}$, together with isomorphisms $\t_{0}: \cE_{0_{k}}\simeq \EE_{\y}$ and $\t_{\infty}:\cE_{\infty_{k}}\simeq \EE_{-\y}$.  Using any global isomorphism $\cE\simeq p^{*}\EE_{\y}$ to identify $\cE_{\infty_{k}}$ and $\cE_{0_{k}}$, $\t_{0}\c\t^{-1}_{\infty}$ gives an element in $\Aut(\EE_{\y})=H$. Therefore we get a canonical isomorphism $\wh U\simeq H$.

Let $u:\wh U\incl \Bun_{G,1}({}_{k}P_{k}, \{0_{k},\infty_{k}\})_{\y,-\y}$ be the open inclusion, and let 
\begin{equation}\label{def e}
e=u_{!}\cL_{\univ}\in \cH^{\bub}.
\end{equation}
Here $\cL_{\univ}$ is the universal local system defined in \S\ref{sss:Sh0H}. We will see in Corollary \ref{c:unit} that $e$ is an identity object of $\cH^{\bub}$.


\subsubsection{Twisting bimodules}\label{sss:tw bimod}
We will construct here distinguished bimodules for the bubbling Hecke category $\cH^\bub$ and affine Hecke category $\cH$ from the family $\frV^{\D}$.

Let us apply our  constructions in \S\ref{ss:mon str} to obtain the half-truncated versions $\sfV_{\t, L}\to \sfA_{L}$ as functors $\cC_{L}\to \Sta$.  The central fiber of $\sfV_{\t,L,(n,[n])}$ is of the form
\begin{equation*}
V(n)_{\t, L}=(\PP_{[0,1]}\cup\cdots\cup \PP^{1}_{[n-1,n]}\cup \PP^{1}_{+})/\mu_{k}.
\end{equation*}
We have the real versions $\sfV^{\RR}_{\t, L}\supset\sfV^{+}_{\t, L}$. We then pass to moduli of $G$-bundles with twisting type $\y$ at all nodes, $N$-reduction along the section $\s^{\D}_{+}$ and rigidification at the extreme nodes to get $(\wh\sfBun^{\sfV,\y}_{\t,L}, \wh\sfBun^{\sfV, \RR,\y}_{\t,L}\supset\wh\sfBun^{\sfV, +,\y}_{\t,L}):\cC_{L}\to (\frR^{\Sch}_{\Sta,\supset})^{H\times H}$. The $H\times H$-action on $\wh\sfBun^{\sfV, \y}_{\t,L,(n,S)}$ are defined as follows: the first $H$ modifies the rigidification along the node $\sfQ^{0}_{(n,S)}$ (see \eqref{sfQ}), and the second $H$ modifies the $N$-reduction along $\s^{n}_{+}$.

Finally we pass to nilpotent sheaves to get the functor 
$$
\xymatrix{
\Phi^{ \sfV, \y}_{\tau, L}
: \Corr^\adm_{\ver;\hor}(\cC_L^{op}) \ar[r] &  \Sh_0(H)\ot \Sh_{0}(H)\module.
}
$$

Recall from the end of \S\ref{sss:coalg} the comodule object $c^0_L= (0, \vn) \in  \Corr^\adm_{\ver;\hor}(\cC_L^{op})$.
Transporting $c^0_{L}$ under the  functor  $\Phi^{ \sfV,\y}_{\tau, L}$ produces a left $ \Phi^{\sfV,\y}_\tau(c^1)\simeq \cH^{\bub}$-comodule
$\Phi^{\sfV,  \y}_{\tau, L}(c^0_L)$. Now we pass to left adjoints to view $\cH^\bub$ as a monoidal category.
By Proposition~\ref{p:sh functor from C}(3), the functors defining the comodule structure  on $\Phi^{\sfV,  \y}_{\tau, L}(c^0_L)$ likewise admit left adjoints. We will pass to the left adjoints and thus view $\Phi^{\sfV,  \y}_{\tau, L}(c^0_L)$  as a (left) module category for $\cH^{\bub}$.

On the other hand, by the discussion in \S\ref{sss:aff Hk sym}, the Hecke modifications along $\s^{\D}_{+}$ upgrades $\Phi^{\sfV,  \y}_{\tau, L}$ into a functor valued in $\cH$-modules.  In our setting it is natural to set up the $\cH$-action to be a right action: composing the $\cH$-action with the inverse $g\mapsto g^{-1}$ on the loop group passes from a left $\cH$-module to a right $\cH$-module. We use the notation $A\bimod B$ to denote bimodules that are $(A,B)$-bimodules, i.e., left $A$-modules and right $B$-modules. Thus we obtain the following distinguished bimodule.


\begin{defn}\label{d:left tw bimod}
The {\em left twisting bimodule} is the $(\cH^{\bub}, \cH)$-bimodule
\begin{equation*}
\cH'  = \Phi^{\sfV, \y}_{\tau, L}(c^0_L)\in \cH^\bub \bimod \cH.
\end{equation*}
\end{defn}

The same construction with $L$ replaced by $R$ gives
\begin{defn}
The {\em right twisting bimodule} is the  $(\cH,\cH^\bub)$-bimodule
\begin{equation*}
{}'\cH = \Phi^{\sfV,  \y}_{\tau, R}(c^0_R)\in \cH \bimod \cH^{\bub}.
\end{equation*}
\end{defn}

Let us describe $\cH'$, ${}'\cH$ as plain $\Sh_0(H) \otimes \Sh_0(H)$-modules. 
Let ${}_k P = [\AA^1/\mu_k] \coprod_{\GG_m} (\PP^{1}\bs \{0\})$ (resp. $P_k = \AA^1 \coprod_{\GG_m} [(\PP^{1}\bs \{0\})/\mu_k]$) denote the twisted version of $\PP^1$ where $0$ (resp. $\infty$) has $\mu_k$-symmetry.  We denote the image of $0$ in ${}_{k}P$ by $0_{k}$, and the image of $\infty$ in $P_{k}$ by $\infty_{k}$ to emphasize they are orbifold points. Then we have 
\begin{equation*}
\xymatrix{
\cH' = \Sh_\cN(\Bun_{G, 1,N}({}_k P, 0_{k},\infty)_\y) 
&
{}' \cH = \Sh_\cN(\Bun_{G, N,1}(P_k, 0,\infty_{k})_\y) 
}
\end{equation*}
Here, $\Bun_{G, 1,N}({}_k P, 0_{k},\infty)_\y$ is the moduli of $G$-bundles on ${}_{k}P$ with an isomorphism with $\EE_{\y}$ at $0_{k}$ and an $N$-reduction at $\infty$; similarly for  $\Bun_{G, N,1}(P_k, 0,\infty_{k})_{-\y}$.

\sss{Canonical generators $e'\in \cH'$, ${}' e\in {}' \cH$}
Let $\xi$ be the affine coordinate on the open part $\AA^{1}/\mu_{k}\subset {}_{k}P$. For each $a\in \frac{1}{k}\ZZ$ denote by $\cO(a)\in \Pic({}_{k}P)$ the line bundle glued from the trivial bundle on ${}_{k}P\bs\{\infty\}$ with global section $\CC[\xi^{-k}]$ and the $\mu_{k}$-equivariant $\CC[\xi]$-module $\xi^{-ak}\CC[\xi]$.  This gives a map $\frac{1}{k}\ZZ\to \Pic({}_{k}P)$ which is a bijection on isomorphism classes. Tensoring with $\xcoch(H)$ we get a map $\frac{1}{k}\xcoch(H)\to \Bun_{H}({}_{k}P)$, $\l\mapsto \cO(\l)$. Inducing $\cO(\l)$ to a $G$-bundle we get $\cE(\l)\in \Bun_{G}({}_{k}P)$.  The twisting type of $\cE(\l)$ at $0_{k}\in {}_{k}P$ is $\l\mod\xcoch(H)\in \YY_{k}$. 

Let $\wt\y\in \frac{1}{k}\xcoch(H)$ be the unique lift of $\y\in \YY_{k}$ contained in the fundamental alcove.  Then $\cE(\wt\y)\in \Bun_{G}({}_{k}P)_{\y}$, and its isomorphism classes defines an open substack $U''\subset \Bun_{G}({}_{k}P)_{\y}$. Let $U'\subset \Bun_{G,B}({}_{k}P,\infty)_{\y}$ be the preimage of $U''$, then
\begin{equation*}
U'\simeq \Aut(\cE(\wt\y))\bs G/B.
\end{equation*}
Since $\wt\y$ lies in the interior of the fundamental alcove, $\Aut(\cE(\wt\y))=B$. Therefore $U'\simeq B\bs G/B$.  Let $\wh U'$ be the preimage of $U'$ in $\Bun_{G,1,N}({}_{k}P, 0_{k},\infty)_{\y}$ then
\begin{equation*}
\wh U'\simeq N\bs G/N.
\end{equation*}
In particular, the open Bruhat cell gives an open embedding
\begin{equation*}
u': H\xr{\cdot \dot w_{0}}Hw_{0}\simeq N\bs Bw_{0}B/N\incl \wh U'\incl \Bun_{G,1,N}({}_{k}P, 0_{k},\infty)_{\y}.
\end{equation*}
Here $\dot w_{0}$ is a lifting of the longest element $w_{0}\in W$. 
Define
\begin{equation}\label{def e'}
e'=u'_{!}\cL_{\univ}\in \cH'.
\end{equation}

Similarly, we define an open embedding $'u:  H\isom \dot w_0   H\subset N\bs G/N\incl \Bun_{G, N,1}(P_k, 0,\infty_{k})_{-\y}$, and a canonical object 
\begin{equation}\label{def 'e}
'e={}' u_! \cL_{\univ}\in {}'\cH.
\end{equation}


\begin{lemma}\label{lem:free rank 1}
As a left (resp. right) $\cH$-module, ${}'\cH$ (resp. $\cH'$) is free of rank one on the generator ${}' e \in {}' \cH$ (resp. $e' \in \cH'$).
\end{lemma}

\begin{proof}
We will prove  $\cH'$  is  free  of rank one on $e'$ as a  right $\cH$-module. The  assertion for ${}'\cH$ is proved similarly. 

By Proposition~\ref{p:Bun on tw}, we have an isomorphism 
\begin{equation*}
 \xymatrix{
\Bun_{G, 1,N}({}_k P, 0_{k},\infty)_\y \simeq \Bun_{G, N}(\PP^1, \{0, \infty\}).
}
\end{equation*}
The  identification is local to $0$ so we have an equivalence of right $\cH$-modules
 \beq\label{cH' BunP1}
 \xymatrix{
  \cH' = \Sh_\cN(\Bun_{G, 1,N}({}_k P, 0_{k},\infty)_\y) \simeq \Sh_\cN(\Bun_{G, N}(\PP^1, \{0, \infty\}))
}
\eeq
Let $z$ be the affine coordinate of $\PP^{1}$ at $\infty$; let $\bJ\subset G[z^{-1}]$ be the preimage of the opposite Borel $B^{-}$ under $G[z^{-1}]\to G$ (evaluation at $z^{-1}=0$), and $\bJ^{\c}=\ker(\bJ\to H)$. Then using the same lifting $\dot w_{0}$ we have an isomorphism 
\begin{equation*}
\Bun_{G, N}(\PP^1, \{0, \infty\})\simeq \bJ^{\c}\bs G\lr{z}/\bI^{\c}.
\end{equation*}
It is well-known that the $(\bJ,\bI)$-double cosets of $G\lr{z}$ are indexed by the extended affine Weyl group $\tilW$. Let $j_{w}: \wh U_{w}=\bJ^{\c}\bs\bJ w \bI/\bI^{\c}\incl \Bun_{G, N}(\PP^1, \{0, \infty\})$ be the inclusion of the $(\bJ,\bI)$-coset of $w\in \tilW$. With the choice of a lifting $\dot w$, we have a map $H\isom H\dot w\to \wh U_{w}$ that induces an isomorphisms onto the coarse moduli space of $\wh U_{w}$.  Therefore we may view $\cL_{\univ}$ as a local system on $\wh U_{w}$, and define the universal monodromic standard sheaf $\wh\D(w)=j_{w!}\cL_{\univ}$.  

Under the equivalence \eqref{cH' BunP1},  $e'\in \cH'$ corresponds to $\wh\D(1)$. Similarly define the universal monodromic costandard sheaf $\wh\nb(w)\in \cH$ as the $*$-extension of $\cL_{\univ}$ from $\bI w \bI$ . The $\cH$-action on  $\Sh_\cN(\Bun_{G, N}(\PP^1, \{0, \infty\}))$ by modifying at $\infty$ has the property
\begin{equation}\label{Rad univ mono}
\wh\D(1)\star_{\infty} \wh\nb(w)\simeq \wh\D(w).
\end{equation}
Indeed, the functor
\begin{equation*}
\Rad_{!}(-)=: \wh\D(1)\star_{\infty}(-): \cH\to \Sh_\cN(\Bun_{G, N}(\PP^1, \{0, \infty\}))
\end{equation*}
is the Radon transform ($!$-averaging by left action of $\bJ^{\c}$). The proof of \eqref{Rad univ mono} is similar to that of \cite[Proposition 4.1.3(i)]{Y-tilt} using the contraction principle under a torus action. On the other hand, the right adjoint of $\Rad_{!}$ is given by the $*$-average under $\bI^{\c}$, and it sends $\wh\D(w)$ to $\wh\nb(w)$ up to a shift. Since $\{\wh\D(w);w\in \tilW\}$ compactly generate $\Sh_\cN(\Bun_{G, N}(\PP^1, \{0, \infty\}))$, and $\{\wh\nb(w); w\in \tilW\}$ compactly generate $\cH$,  we conclude that $\Rad_{!}$ is an equivalence. 

Transporting back to the $\cH$-action on $\cH'$ via \eqref{cH' BunP1}, we conclude that $\cH'$ is a free $\cH$-module with free generator $e'$.
\end{proof}

By the lemma, we have an equivalence $ \cH \simeq \End_{\rmod\cH}(\cH')$, $h\mapsto \alpha_h$ where $\alpha_h$ is characterized by $\alpha_h(e') = e' \star h$, $h\in \cH$.   Since $\cH^\bub$ acts on $\cH'$ as $\cH$-endomorphisms, we have a natural map of algebras
\begin{equation*}
\xymatrix{
\gamma:\cH^\bub \ar[r] & \End_{\rmod\cH}(\cH') \simeq \cH
}
\end{equation*}
characterized by a functorial multiplicative identity $h \star e' \simeq e' \star \g(h)$, $h\in \cH^\bub$.

Similarly, 
 we have an equivalence $ \cH^{op} \simeq \End_{\cH\module}({}' \cH)$, $h\mapsto \beta_h$ where $\beta_h$ is characterized by $\beta_h({}'e) = h\star {}'e$, $h\in \cH^{op}$.   Here we write $\cH^{op}$ for the monoidal opposite. Since $\cH^\bub$ acts on ${}'\cH$ on the right as $\cH$-endomorphisms, we have a natural map of algebras
\begin{equation*}
\xymatrix{
(\cH^\bub)^{op} \ar[r] & \End_{\cH\module}({}'\cH) \simeq \cH^{op}
}
\end{equation*}
so after passing to opposites a map of algebras
\begin{equation*}
\xymatrix{
\zeta:\cH^\bub  \ar[r] & \End_{\cH\module}({}'\cH)^{op} \simeq \cH
}
\end{equation*}
characterized by a functorial multiplicative identity $ {}' e \star h \simeq  \zeta(h) \star  {}' e$, $h\in \cH^\bub$.

\begin{theorem}\label{th:Hbub} The functors $\gamma:\cH^\bub\to \cH$,  $\zeta:\cH^\bub\to \cH$ are (non-unital) monoidal equivalences.
\end{theorem}
\begin{proof}
We will prove $\gamma$ is an equivalence. The  assertion for $\zeta$ is proved similarly. 

We must show the action map $\cH^\bub\to \cH'$, $h\mapsto h\star e'$ is an equivalence.

By Proposition~\ref{p:Bun on tw}, we have an isomorphism
\begin{equation*}
 \xymatrix{
 \Bun_{G, 1}({}_kP_k, \{0_{k},\infty_{k}\})_{\y,\y} \simeq \Bun_{G, 1, N}({}_k P, 0_{k}, \infty)_\y 
}
\end{equation*}
where we recall ${}_kP_k = \PP^1/\mu_k$.  
The  identification is local to $\infty$ so we have an equivalence of left $\cH^\bub$-modules
 \beq\label{Hbub H'}
 \xymatrix{
\cH^\bub = \Sh_\cN( \Bun_{G, 1}({}_kP_k, \{0_{k},\infty_{k}\})_{\y,-\y}) \simeq \Sh_\cN(\Bun_{G, 1,N}({}_k P, 0_{k},\infty)_\y) = \cH'
}
\eeq
Then $e'\in \cH'$ corresponds to $e\in \cH^{\bub}$ (see \eqref{def e}) under this equivalence.  Thus we must show the action map $\cH^\bub\to \cH^\bub$, $h\mapsto h\star e$ is an equivalence. Equivalently, passing to right adjoints, we must show the comultiplication $\mu: \cH^\bub \to \cH^\bub\otimes_{\Sh_{0}(H)} \cH^\bub$ followed by  restriction $u^!=u^{*}:\cH^\bub\to \Sh_{0}(\wh U)=\Sh_0(H)$ in the second factor is an equivalence (where $u: \wt U\incl \Bun_{G, 1}({}_kP_k, \{0_{k},\infty_{k}\})_{\y,-\y}$ is defined in \S\ref{sss:e}). 

Unraveling the definition, $\mu$ is given by the nearby cycles coming from the moduli of $G$-bundles for the family $\sfV_{\t,(2,\{1\})}\to \sfA_{(2,\{1\})}$. The latter is isomorphic to $v: [\frW/\mu^{\WW}_{k}\times \Gm]\to [\AA^{1}/\Gm]$, where $\frW$ is the compactification of $\WW$ introduced in \S\ref{sss:comp W}, and $s\in \Gm$ acts on $\WW$ by $(x,y)\mapsto (sx,y)$ and the action extends to $\frW$. Then $v^{-1}(0)=V_{-}\cup V_{+}$ with $V_{\pm}\simeq {}_{k}P_{k}$. Recall that $v$ admits two sections $\s_{-}$ and $\s_{+}$. Let $b:\frB:=\Bun_{G,1}(v; \{\s_{-}, \s_{+}\})_{\y,-\y}\to [\AA^{1}/\Gm]$ be the moduli stack of $G$-bundles along fibers of $v$ with twisting type $\y$ at the node of $v^{-1}(0)$, and isomorphisms with $\EE_{\pm\y}$ along $\s_{\pm}$. Then $\mu\simeq \Psi_{b}: \Sh_{\cN}(\frB_{1})\simeq\cH^{\bub}\to \cH^{\bub}\ot_{\Sh_{0}(H)}\cH^{\bub}\simeq\Sh_{\cN}(\frB_{0})$. Let $\frB'\subset \frB$ be the open substack parametrizing $G$-bundles whose restriction to $V_{+}$ is isomorphic to the pullback of $\EE_{\y}$. Then the special fiber $\frB'_{0}\simeq \Bun_{G,1}(V_{-}, \{0_{k},\infty_{k}\})_{\y,-\y}$.  Let $j:\frB'_{0}\subset \frB_{0}$ be the open embedding, then $u^{*}\c\mu\simeq \Psi_{b'}: \Sh_{\cN}(\frB'_{1})\simeq\cH^{\bub}\to \cH^{\bub}\simeq \Sh_{\cN}(\frB'_{0})$ is the nearby cycles functor for the family $b':\frB'\to[\AA^{1}/\Gm]$. However,  $b'$ is a trivial family: bundles parametrized by $\frB'$ are exactly those pulled back from the trivial family ${}_{k}P_{k}\times [\AA^{1}/\Gm]$ via the map $[\frW/\mu^{\WW}_k\times \Gm]\to {}_{k}P_{k}\times [\AA^{1}/\Gm]$ contracting $V_{+}$ to $\infty_{k}\in {}_{k}P_{k}$. Therefore $u^{*}\c\mu\simeq\Psi_{b'}\simeq \id$ as an endo-functor of $\cH^{\bub}$. This finishes the proof.
\end{proof}

\begin{cor}\label{c:unit} The object 
$e\in \cH^\bub$ defined in \eqref{def e} is an identity object for the monoidal structure of $\cH^{\bub}$.
\end{cor}
\begin{proof}
Under the non-unital equivalence $\g: \cH^{\bub}\isom \cH$, $e$ maps to the monoidal unit of $\cH$. Hence $e$ can be taken to be the monoidal unit of $\cH^{\bub}$.
\end{proof}

\begin{cor}\label{c:right modules} 
\begin{enumerate}
\item
The functor of right modules $\rmod\cH^\bub \to \rmod\cH$, $M \mapsto M\otimes_{\cH^\bub} \cH'$ is an equivalence. 
 
Moreover, the canonical equivalence $\theta':M \isom M\otimes_{\cH^\bub} \cH'$, 
$\theta'(m) = m \otimes e'$ intertwines the $\cH^\bub$-action on $M$ and the $\cH$-action on 
$M\otimes_{\cH^\bub} \cH'$ via the monoidal equivalence $\g: \cH^{\bub}\isom \cH$.

\item The functor of left modules $\cH \module \to \cH^\bub\module  $, $N \mapsto {}'\cH \otimes_{\cH^\bub} N$ is an equivalence.

Moreover, the canonical equivalence ${}'\theta:N \isom  {}' \cH \otimes_{\cH^\bub} N$, 
${}'\theta(n) = {}'e\otimes n$ intertwines the $\cH^\bub$-action on $N$ and the $\cH$-action on 
${}' \cH \otimes_{\cH^\bub} N$ via the monoidal equivalence $\zeta: \cH^{\bub}\isom \cH$.

\end{enumerate}
\end{cor}
\begin{proof}

We will prove (1). One proves (2) similarly.

We have seen $\cH'$ is free rank one as a left $\cH^\bub$-module and as a right $\cH$-module. This implies the equivalence on module categories. 
For the second assertion,  recall $\g$ is characterized by a functorial multiplicative identity $h 
\star e' \simeq e'\star  \g(h)$, for $h\in \cH^\bub$. Thus we have $\theta'(m h) = (m h)  \otimes e' \simeq m \otimes (h \star e') \simeq  m \otimes (e'\star  \g(h))$, so indeed $\theta'$ intertwines the actions via~$\g$.
\end{proof}


\subsection{The bar complex} Return to the set up of a separating nodal degeneration $\t: \frY \to C$ and its  $k$th twisted cosimplicial bubbling $\pi^{\D}: \frX^{\D}\to B^{\D}$.

\sss{The left and right bubbling Hecke modules} By Remark \ref{r:X tau indep}, we have an isomorphism of functors
\begin{equation*}
\sfX_{\t}\simeq R\times \sfV_{\t}: \cC_{\star}\to \Sta.
\end{equation*}
This induces a monoidal equivalence 
\begin{equation*}
\Phi^{\y}_{\t,\Corr}\simeq (\Phi^{\sfV, \y}_{\t,\Corr})^{\ot R}: \Corr^\adm_{\ver;\hor}(\cC^{op}_{\star})\to \Sh_{0}(H^{R})\ot \Sh_{0}(H^{R})\module. 
\end{equation*}
In particular, we have a monoidal equivalence
\begin{equation*}
\Phi^{\y}_{\t,\Corr}(c^{1})\simeq \cH^{\bub,\ot R}.
\end{equation*}

%
%

Recall the comodule object $c^0_L= (0, \vn) \in  \Corr^\adm_{\ver;\hor}(\cC_L^{op})$.
Transporting $c^0$ under the  functor  $\Phi^{\y}_{\tau, L,\Corr}$ produces a left $ \Phi^\y_{\tau,\Corr}(c^1)\simeq \cH^{\bub,\ot R}$-comodule $\cA_{+}:=\Phi^{ \y}_{\tau, L,\Corr}(c^0_L)$. By \eqref{PhiL c0L}, we have
\begin{equation*}
\cA_{+}\simeq \Sh_{\cN}(\Bun_{G,1}(X_{+},Q_{+})_{\y}).
\end{equation*}

Recall we pass to left adjoints to view $\cH^\bub = \Phi^{ \y}_{\tau}(c^1)$ as a monoidal category.
By Proposition~\ref{p:sh functor from C}(3), the functors defining the comodule structure  on $\cA_{+}$ likewise admit left adjoints. We will pass to the left adjoints and thus view $\cA_{+}$  as a left $\cH^{\bub,\ot R}$-module category. A priori, $\cA_{+}$ is a non-unital left $\cH^{\bub,\ot R}$-module.

\begin{lemma}\label{l:unit action}
The action of $e^{\ot R}\in \cH^{\bub,\ot R}$ on $\cA_{+}$ is canonically isomorphic to the identity.
\end{lemma}
\begin{proof}
The argument is the same as the last part of the proof of Theorem~\ref{th:Hbub}, using the one-parameter family $\sfX_{\t,L,(1,\{1\})}\to \sfA_{(1,\{1\})}\simeq[\AA^{1}/\Gm]$ (degenerating $X_{+}$ to $(R\times {}_{k}P_{k})\cup_{Q_{+}}X_{+}$) instead.
\end{proof} 

Similarly, the category
\begin{equation*}
\cA_{-}:=\Phi^{ \y}_{\tau, R,\Corr}(c^0_R)\simeq \Sh_{\cN}(\Bun_{G,1}(X_{-},Q_{-})_{-\y})
\end{equation*}
carries a right action of $\cH^{\bub,\ot R}$ induced from the comodule structure of $\Phi^{ \y}_{\tau, R,\Corr}(c^0_R)$. This action is also unital.

%
%
%
%
%

\sss{The bar complex}
Recall from \eqref{S eta} the augmented semi-cosimplicial object
$$
\xymatrix{
\cS^{\bu,\y}: \Delta_{\inj, +}\ar[r]^-{c^\bullet} &  \Corr^\adm_{\ver;\hor}(\cC^{op})  \ar[r]^-{\Phi^{\y}_{\Corr}} & \St
}$$ 
with term-wise assignments 
$$
\xymatrix{
\cS^{-1,\y} = \Sh_{\wt\cN}(\Bun_G(  \t|_{C_{>0}}))
&
\cS^{n,\y} = \Sh_\cN(\Bun_G( X(n))_{\y}), \; n\geq 0 
}
$$
By Lemma~\ref{l:Phi eta tau Corr}(3) we have
\begin{equation*}
\cS^{n,\y}\simeq \cA_{-}\ot_{\Sh_0(H^R)}\underbrace{\cH^{\bub, \ot R}\ot_{\Sh_0(H^R)}\cdots\ot_{\Sh_0(H^R)} \cH^{\bub, \ot R}}_{n}\ot_{\Sh_0(H^R)}\cA_{+}.
\end{equation*}

Let us pass to left adjoints to obtain an augmented semi-simplicial object
\begin{equation*}
\cS^{\y}_{\bu}: \Delta_{\inj, +}^{op}\to \St.
\end{equation*}
with $\cS^{\y}_{n}=\cS^{n,\y}$.  The above description of $S^{n,\y}$ can be summarized as follows. 

\begin{theorem}\label{th:bar cx bub} The (non-augmented) semi-simplicial category $\cS^{\y}_{\bu}|_{\Delta_{\inj}^{op}}$ is the bar complex calculating the tensor product 
\begin{equation*}
\xymatrix{
\cA_{-}\otimes_{\cH^{\bub,\ot R}} \cA_{+}.
}
\end{equation*} 
The augmentation map $\cS^{\y}_{0}\to \cS_{-1}^{\y}$ provides a natural functor
\begin{equation*}
\xymatrix{
\cA_{-} \otimes_{\cH^{\bub,\ot R}} \cA_{+}\ar[r] & 
\Sh_{\wt\cN}(\Bun_G(  \t|_{C_{>0}})).
}
\end{equation*} 
\end{theorem}
 
\sss{Variant: marked sections}
Consider the situation of \ref{sss:cosim cat sections} with a disjoint collection $s=\{s_{a}\}_{a\in \Sigma}$ of marked sections $s_{a}: C\to \frY$, and the induced sections  $\s^{\D}=\{\s^{\D}_{a}\}_{a\in \Sigma}$, where $ \s^{\D}_{a}:   B^\Delta \to  \frX^{\D}$. 
Let $\Sigma_{-}\subset\Sigma$ be the subset of $a\in \Sigma$ such that $s_{a}(c_{0})\in Y_{-}$. Similarly define $\Sigma_{+}$. Then $\Sigma=\Sigma_{-}\coprod \Sigma_{+}$. Let $\s_{-}$ (resp. $\s_{+}$) be those marked points of $X(0)$ that lie in $X_{-}$ (resp. $X_{+}$), so that we have bijections $\s_{\pm}\bij \Sigma_{\pm}$.

Let
\begin{equation*}
\cA_{-,\s_{-}}:=\Sh_{\cN}(\Bun_{G,1,N}(X_{-},Q_{-}, \s_{-})_{\y}), \quad \cA_{+,\s_{+}}:=\Sh_{\cN}(\Bun_{G,1,N}(X_{+},Q_{+}, \s_{+})_{\y}).
\end{equation*}
Now $\cH^{\ot \Sigma_{-}}$ acts on $\cA_{-,\s_{-}}$ by Hecke modifications at $\s_{-}$, so that $\cA_{-,\s_{-}}$ is an $\cH^{\ot \Sigma_{-}}$-module. Similarly, $\cA_{+,\s_{+}}$ is an $\cH^{\ot \Sigma_{+}}$-module.

Using Proposition \ref{p:cosimp cat marked sections} and passing to bundles with type $\y$ at the twisted nodes, we get an augmented semi-cosimplicial object
\begin{equation*}
\wt \cS^{\bu,\y}_{\s}: \D_{\inj, +}\to \cH^{\ot\Sigma}\module
\end{equation*}
with term-wise assignments
\begin{eqnarray*}
\wt \cS^{-1,\y}_{\s}&=& \Sh_{\wt\cN}(\Bun_{G,N}(\t|C_{>0}, s)),\\
\wt \cS^{n,\y}_{\s}&=&\cA_{-,\s_{-}}\ot_{\Sh_0(H^R)}\underbrace{\cH^{\bub, \ot R}\ot_{\Sh_0(H^R)}\cdots\ot_{\Sh_0(H^R)} \cH^{\bub, \ot R}}_{n}\ot_{\Sh_0(H^R)}\cA_{+,\s_{+}}.
\end{eqnarray*}
For $n\ge0$, $\wt \cS^{n,\y}_{\s}$ is an $\cH^{\ot \Sigma}$-module from the $\cH^{\ot \Sigma_{\pm}}$-action on $\cA_{\pm,\s_{\pm}}$; for $n=-1$,  $\cH^{\ot \Sigma}$ acts on $\wt \cS^{-1,\y}_{\s}$ by modifications along the marked sections $s$.

Passing to  left adjoints gives an augmented semi-simplicial object
\begin{equation*}
\wt \cS^{\y}_{\bu, \s}: \D^{op}_{\inj, +}\to \cH^{\ot\Sigma}\module.
\end{equation*}
Then  Theorem \ref{th:bar cx bub} has the following extension.

\begin{theorem}\label{th:bar cx bub marked sections} The (non-augmented) semi-simplicial $\cH^{\ot \Sigma}$-module $\wt \cS^{\y}_{\bu, \s}$ is the bar complex calculating the tensor product
\begin{equation*}
\cA_{-,\s_{-}}\ot_{\cH^{\bub,\ot R}}\cA_{+,\s_{+}}
\end{equation*}
in the category of $\cH^{\ot\Sigma}\module$. The augmentation map $\wt\cS^{\y}_{0,\s}\to \wt\cS_{-1,\s}^{\y}$ provides a natural functor
\begin{equation*}
\cA_{-,\s_{-}} \otimes_{\cH^{\bub,\ot R}} \cA_{+,\s_{+}}\to
\Sh_{\wt\cN}(\Bun_{G,N}(  \t|_{C_{>0}}, s)).
\end{equation*} 
\end{theorem}

\sss{Reformulation using coarse curves}
Finally, we would like to restate Theorem~\ref{th:bar cx bub} in terms of the affine Hecke category and the coarse curves $Y_{\pm}$. We will use Proposition~\ref{p:Bun on tw} to identify
\begin{equation}\label{Xpm Ypm}
\Bun_{G,1}(X_{\pm}, Q_{\pm})_{\y}\simeq\Bun_{G,N}(Y_{\pm}, R_{\pm})
\end{equation} 
and hence to obtain an equivalence
\begin{equation}\label{A-Y}
a_{\pm}: \cA_{\pm}\simeq\Sh_{\cN}(\Bun_{G,N}(Y_{\pm}, R_{\pm})).
\end{equation}
It remains to identify the Hecke actions on both sides.

Recall we view the minus versions as right modules and the plus versions as left modules.

\begin{theorem}\label{th:matching actions}
Under the equivalence \eqref{A-Y} in the minus (resp. plus) case, the action of $\cH^{\bub,\ot R}$ on the left hand side is canonically intertwined by the equivalence $\gamma$ (resp. $\zeta$) of Theorem~\ref{th:Hbub} with the action of $\cH^{\ot R}$ on the right hand side. 
\end{theorem}
\begin{proof} 
We will prove the minus case; the plus case is similar.

Write 
\begin{equation*}
\cA'_{-}:=\Sh_{\cN}(\Bun_{G,N}(Y_{-}, R_{-})).
\end{equation*}
Apply the procedure of Example \ref{ex: nodal degen} to the curve $Z=Y_{-}$  together with the finite subset $R=R_{-}\subset Y_{-}$ to obtain the separating nodal degeneration $\t_{-}:\frY_{-}\to \AA^{1}$. The nonzero fibers of $\frY_{-}$ are identified with $Y_{-}$ and the special fiber is $Y_{-}\cup_{\{0\}\times R} (\PP^{1}\times R)$.  Moreover, by Example \ref{ex:blowup section}, $\frY$ is equipped with sections $\{\s_{r}\}_{r\in R}$ that are proper transforms $R_{-}\times \AA^{1}$ under the blowup. 

We apply the $k$-th twisted bubbling construction to $\t_{-}:\frY_{-}\to \AA^{1}$ to obtain $\frX^{\D}_{-}\to A^{\D}$. The central fiber of $\frX^{[0]}_{-}\to A^{[0]}\cong \AA^{1}$ is $X_{-}(0)\cong X_{-}\cup ({}_{k}P\times R)$ (glued along $Q_{-}\cong 0_{k}\times R$). It is equipped with sections $\s^{\D}=\{\s_{r}^{\D}; r\in R\}$ (see \S\ref{sss:more sections tw}), such that $\s^{[0]}_{r}$ intersects $X_{-}(0)$ in $(\infty,r)\times {}_{k}P\times R$ for $r\in R$. 

Recall the left twisting $(\cH^{\bub}, \cH)$-bimodule $\cH'$ in Definition \ref{d:left tw bimod}. Its underlying category is $\Sh_{\cN}(\Bun_{G,1,N}({}_{k}P,0_{k}, \infty)_{\y})$,
and ${}_{k}P$ with  $\infty$ marked appear in $X_{-}(0)$. Now Theorem \ref{th:bar cx bub marked sections} gives a functor of (right) $\cH^{\ot R}$-modules
\begin{equation*}
\e': \cA_{-}\ot_{\cH^{\bub,\ot R}}\cH'^{\ot R}\to \cA'_{-}.
\end{equation*}
Here $\cH^{\ot R}$ acts on the left side by its right action on $\cH'^{\ot R}$, and acts on the right side by Hecke modification along $R_{-}\subset X_{-}$.

Consider the composition
\begin{equation*}
\cA_{-}\xr{\th'}\cA_{-}\ot_{\cH^{\bub,\ot R}}\cH'^{\ot R}\xr{\e'} \cA'_{-}.
\end{equation*}
Here $\th'(\cF)=\cF\ot e'^{\ot R}$, and $e'\in \cH'$ is the object defined in \eqref{def e'}, which is a free generator of $\cH'$  as a right $\cH$-module by Lemma \ref{lem:free rank 1}. 
The monoidal equivalence $\g: \cH^{\bub}\simeq \cH$ in Theorem \ref{th:Hbub} is constructed using $e'\in \cH'$, hence $\th'$ is an equivalence, and it intertwines the $\cH^{\bub,\ot R}$-action on $\cA_{-}$ and the $\cH^{\ot R}$-action on $\cA_{-}\ot_{\cH^{\bub,\ot R}}\cH'^{\ot R}$ via $\g^{\ot R}$. Therefore, the composition $\e'\c\th'$ intertwines the $\cH^{\bub,\ot R}$-action on $\cA_{-}$ and the $\cH^{\ot R}$-action on $\cA_{-}'$. 

It remains to show that $\e'\c\th'\simeq a_{-}$, the equivalence $\cA_{-}\simeq \cA'_{-}$ from \eqref{A-Y}. To see this, we  consider the family $\xi: \sfX_{\t, R,(1,\{0\})}\to \sfA_{(1,\{0\})}$ (see \S\ref{sss:tr curves}) obtained by restricting the family $\frX^{[1]}$ to the coordinate line $A_{(1,\{0\})}\subset A^{[1]}$, removing $X_{+}\bs Q_{+}$ from each fiber and quotient by $\Gm$. Then the family $\xi': [\frX_{-}^{[0]}/\Gm]\to [A^{[0]}/\Gm]$ is obtained from $\xi$ by taking the partial coarse space: only turning the twisted sections $\sfQ^{1}_{(1,\{0\})}$ of $\sfX_{\t, R,(1,\{0\})}$ into marked sections $\s^{[0]}$ of $\xi'$. Applying Proposition \ref{p:Bun tw family} to $\xi$ and $\xi'$ we get an isomorphism over $\sfA_{(1,\{0\})}\cong [\AA^{1}/\Gm]$
\begin{equation}\label{xi'}
\Bun_{G,N}(\xi', \s^{[0]})\simeq \Bun_{G, 1}(\xi, \sfQ^{1}_{(1,\{0\})})_{\y}.
\end{equation}
Under this isomorphism, we have a commutative diagram
\begin{equation*}
\xymatrix{\cA_{-}\ar[r]^-{\th}\ar@{=}[d]  & \cA_{-}\ot_{\cH^{\bub,\ot R}}\cH^{\bub,\ot R}\ar@{=}[r]\ar[d]_{\id\ot \g'^{\ot R}} &  \cA_{-}\ar[d]^{a_{-}}\\
\cA_{-}\ar[r]^-{\th'}& \cA_{-}\ot_{\cH^{\bub,\ot R}}\cH'^{\ot R}\ar[r]^-{\e'}  &  \cA'_{-}}
\end{equation*}
Here   $\th(\cF)=\cF\ot e^{\ot R}$, and the left square is commutative because $e$ corresponds to $e'$ under the equivalence $\g': \cH^{\bub}\simeq \cH'$ that comes from \eqref{Hbub H'} (which is compatible with the restriction of fiber of \eqref{xi'} over $0$). By Lemma \ref{l:unit action}, $e^{\ot R}$ acts on $\cA_{-}$ by the identity, therefore the upper row above is the identity. Therefore the bottom row is canonically isomorphic to $a_{-}$. This finishes the proof.
\end{proof}

Now let us give precedence to the monoidal equivalence $\gamma:\cH^\bub \isom \cH$ 
of Theorem~\ref{th:Hbub}. So by Theorem~\ref{th:matching actions},  the identification  \eqref{A-Y} in the minus case  intertwines the actions.
On the other hand, for the identification  \eqref{A-Y} to  intertwine actions  in the plus case, we should twist the $\cH$-action by $\tau:=\zeta \gamma^{-1}:\cH\isom \cH$ where  $\zeta:\cH^\bub \isom \cH$ 
is  the other monoidal equivalence 
of Theorem~\ref{th:Hbub}. Similarly to prior constructions, one can check $\tau$ is the automorphism of $\cH$ associated to the invertible $\cH$-bimodule $\Sh_{\cN}(\Bun_{G,N}(\PP^1, \{0, \infty\})) \simeq {}' \cH\otimes_{\cH^\bub} \cH'$ where $\cH$ acts by left modifications at $0$ and right modifications at $\infty$.

In summary, combining Theorems \ref{th:bar cx bub marked sections} and \ref{th:matching actions}, we obtain the following consequence as our main result of the paper.

\begin{theorem}\label{th:bar complex} Starting with a separating nodal degeneration $\t:\frY\to C$. Choose a trivialization of $\om_{Y_{-}}|_{R}$ and a regular $\y\in \YY_{k}$. These data give a canonical functor
\begin{equation*}
\Sh_{\cN}(\Bun_{G,N}(Y_{-}, R_{-}))\ot_{\cH^{\ot R}}\Sh_{\cN}(\Bun_{G,N}(Y_{+}, R_{+}))\to \Sh_{\wt\cN}(\Bun_G(  \t|_{C_{>0}})).
\end{equation*}
where $\cH^{\ot R}$ acts on $\Sh_{\cN}(\Bun_{G,N}(Y_{-}, R_{-}))$ by right modifications and on 
$\Sh_{\cN}(\Bun_{G,N}(Y_{+}, R_{+}))$ by $\tau$-twisted left modifications.
\end{theorem}

\appendix

\section{Sheaves with prescribed singular support}\label{app:A}


This appendix contains some background material called upon in \S\ref{ss:mon str}. We will work with complex algebraic stacks and sheaves on them with prescribed singular supports that are not necessarily algebraic. 
Our aim is to record some basic properties -- compact generation, dualizability, and tensor product identities -- enjoyed by sheaves with prescribed singular support.


\subsection{Sheaves on manifolds}\label{ss:A mfd}

\sss{Basics on Lagrangians}\label{sss:Lag mfd}


Recall from \cite[Definition 8.3.9]{KS} that for  a smooth (real) manifold $M$, an $\RR_{>0}$-conic subanalytic subset $\L\subset T^{*}M$ is called {\em isotropic} if the canonical $1$-form on $T^{*}M$ vanishes on the regular locus of $\L$. Equivalently $\L$ is isotropic if it is contained in the union of conormals of a locally finite collection of subanalytic submanifolds of $M$. If $M$ is equidimensional, an $\RR_{>0}$-conic subanalytic subset $\L\subset T^{*}M$ is called a Lagrangian if it is isotropic and has pure real dimension equal to $\dim_{\RR}M$.


Let $f:M_{1}\to M_{2}$ be a map of smooth manifolds. Consider the correspondence
\begin{equation*}
\xymatrix{ T^{*}M_{1} & M_{1}\times_{M_{2}}T^{*}M_{2}\ar[r]^-{f^{\na}}\ar[l]_-{df} & T^{*}M_{2}
}
\end{equation*}
For a subset $\L_{1}\subset T^{*}M_{1}$, let
\begin{equation*}
\orr{f}(\L_{1})=f^{\na}( df^{-1}(\L_{1}))\subset T^{*}M_{2}.
\end{equation*}
Similarly, for a subset $\L_{2}\subset T^{*}M_{2}$, let
\begin{equation*}
\oll{f}(\L_{2})=df( (f^{\na})^{-1}(\L_{2}))\subset T^{*}M_{1}.
\end{equation*}

\begin{lemma}\label{l:Cart Lag mfd} Consider a Cartesian diagram of manifolds
\begin{equation*}
\xymatrix{ M' \ar[d]^{f'}\ar[r]^{a} & M\ar[d]^{f}\\
N'\ar[r]^{b} & N}
\end{equation*}
Then for any  subset $\L\subset T^{*}M$ we have an equality of subsets of $T^{*}N'$
\begin{equation*}
\oll{b}\orr{f}(\L)=\orr{f'}\oll{a}(\L)\subset T^{*}N'.
\end{equation*}
\end{lemma}
\begin{proof} Let $p=f\c a=b\c f': M'\to N$. Let $Z=M'\times_{N}T^{*}N$. We have a natural map $\pi_{M}: Z\xr{a\times\id}M\times_{N}T^{*}N\xr{df}T^{*}M$. Similarly, we have a map $\pi_{N'}: Z\to T^{*}N'$. It is easy to see that
\begin{equation}\label{fa}
\orr{f'}\oll{a}(\L)=\pi_{N'}\pi_{M}^{-1}(\L).
\end{equation}
On the other hand, let $W=(M'\times_{N'}T^{*}N')\times_{T^{*}M'}(M'\times_{M}T^{*}M)$ (the maps to $T^{*}M'$ are $df'$ and $da$). Let $\nu_{M}: W\to T^{*}M$ and $\nu_{N'}: W\to T^{*}N'$ be the projections. Then one checks that
\begin{equation}\label{bf}
\oll{b}\orr{f}(\L)=\nu_{N'}\nu_{M}^{-1}(\L).
\end{equation}
Finally, we observe that the natural map $Z\to W$ (induced by $db$ and $df$) is an isomorphism, under which $\pi_{M}$ corresponds to $\nu_{M}$ and $\pi_{N'}$ corresponds to $\nu_{N'}$. Indeed this boils down to the fact that for any $x'\in M'$ with image $y',x$ and $y$ in $N', M$ and $N$ respectively, the map $(db,df): T^{*}_{y}N\to T^{*}_{y'}N'\op_{T^{*}_{x'}M'}T^{*}_{x}M$ is surjective. The desired equality then follows by combining \eqref{fa} and \eqref{bf}.
\end{proof}

\sss{Sheaves with singular support conditions}
Let $M$ be a manifold, and $\Lambda\subset T^*M$ a closed $\RR_{>0}$-conic Lagrangian. By
\cite[Corollary 8.3.22]{KS}, we may choose a $\mu$-stratification $\cS = \{S_\alpha\}$ of $M$ with connected strata so that $\Lambda\subset T^*_\cS M = \cup_\alpha T^*_{S_\alpha} M$.

Let $\Sh(M)$ be the stable $\QQ$-linear category of all complexes of $\QQ$-sheaves on $M$.
Let $\Sh_\cS(M) \subset \Sh(M)$ be the full subcategory of $\cS$-weakly constructible complexes.
Let $\Sh_\Lambda(M) \subset \Sh_\cS(M)$ be the full subcategory of complexes with singular support in $\Lambda$. Note Verdier duality preserves $\Sh_\cS(M)$ and preserves $\Sh_\Lambda(M)$ if $\Lambda$ is closed under the antipodal map so is $\RR^\times$-conic. 
Note the fully faithful inclusions $\Sh_\Lambda(M) \subset \Sh_\cS(M) \subset \Sh(M)$ preserve limits and colimits.

Recall for an object $\cF\in \Sh_\cS(M)$, we have the following characterization of when $\cF \in \Sh_\Lambda(M)$. For any smooth point $(x, \xi) \in T^*_\cS M$, consider a 
a function $f:M\to \RR$ such that $f(x) = 0$ and the graph $\Gamma_{df} \subset T^*B$ intersects $T^*_\cS M$ transversely at $(x, \xi)$. Choose
a sufficiently  small open ball $B\subset M$ around $x$, and let $h: B_{f\ge0}=\{x\in B|f(x)\geq 0\} \to B$ be the inclusion of the closed half-ball, and 
$p: B_{f\ge0}\to pt$ its map to a point.
We refer to  the functor 
\begin{equation*}
\xymatrix{
m_f:\Sh_\cS(M) \ar[r] &  \Sh(pt) &  m_f(\cF) = p_*h^!(\cF|_B) 
}
\end{equation*}
as an  $(x, \xi)$-based Morse functional.
One can show $m_f$  only depends on $f$ through the quadratic part of its Taylor expansion, i.e. the Lagrangian plane $T_{(x, \xi)} \Gamma_{df} \subset T _{(x, \xi)} T^*M$,  and in fact only depends on its linear part, i.e. the initial data $(x, \xi)\in T^*_\cS M$, up to tensoring with a graded line.
Finally, for  $\cF\in \Sh_\cS(M)$, we have $\cF \in Sh_\Lambda(M)$ if and only if   $m_f(\cF) \simeq 0$  whenever $(x, \xi) \in T^*_\cS M \setminus \Lambda$.

\subsubsection{Compact generation}

\begin{lemma}\label{lem:comp objs}
Fix a smooth point $(x, \xi) \in  \Lambda$.

Any $(x, \xi)$-based Morse  functional $m_f:\Sh_\Lambda(M)\to \Sh(pt)$  preserves limits and colimits, hence is corepresented by a compact object, to be denoted by $\cF_{f} \in \Sh_\Lambda(M)$.

\end{lemma}

\begin{proof}
We can express the Morse functional  in terms of the cone of the natural restriction map
\begin{equation*}
\xymatrix{
 m_f(\cF) [-1]\simeq q_* \textup{Cone}(\cF|_B \ar[r] & \cF|_{B_{f<0}})  \simeq i^* \textup{Cone}(\cF|_B \ar[r] & \cF|_{B_{f<0}}) 
}
\end{equation*}
Here $q:B \to pt$ is the map to the point and $i:x\to B$ is the inclusion of the point; the isomorphism $q_* \simeq i^*$ is standard for $B$ sufficiently small. 
Finally, restriction along open inclusions preserves limits and colimits; cones also preserve    limits and colimits since we are in a stable category; and $q_*$ preserves limits and $i^*$ preserves colimits.
\end{proof}

\begin{remark}
Since $m_f$ only depends on $(x, \xi)\in T^*_\cS M$ up to tensoring with a graded line, the corepresenting object $\cF_f$ likewise only depends on $(x, \xi)\in T^*_\cS M$ up to tensoring with a graded line.
\end{remark}

\begin{lemma}\label{lem:comp gens}
For each irreducible component $\Lambda_{\a}$ of $\L$, choose a 
smooth point $(x_{\a}, \xi_{\a})$ of $\L_{i}$
and
an $(x_{\a}, \xi_{\a})$-based Morse  functional $m_{f_{\a}}:\Sh_\Lambda(M)\to \Sh(pt)$.

 The corepresenting objects $\cF_{f_{\a}} \in \Sh_\Lambda(M)$
compactly generate
$\Sh_\Lambda(M)$, and hence $ \Sh_\Lambda(M)$ is dualizable.

\end{lemma}

\begin{proof}
By Lemma~\ref{lem:comp objs}, it remains to see the functors have no right orthogonal. But a sheaf in the right orthogonal has empty singular support so is zero.
\end{proof}

\subsubsection{Products}

For $i = 1, 2$, let $M_i$ be a manifold, and $\Lambda_i\subset T^*M_i$ a closed $\RR_{>0}$-conic Lagrangian.

\begin{lemma}\label{lem:ext tens}
External tensor product is an equivalence
\begin{equation*}
\xymatrix{
\Sh_{\Lambda_1}(M_1) \otimes \Sh_{\Lambda_2}(M_2) \ar[r]^-\sim & 
\Sh_{\Lambda_1\times \Lambda_2}(M_1\times M_2)
}
\end{equation*}
\end{lemma}

\begin{proof} 
By  results of Lurie~\cite[Theorem 7.3.3.9, Prop. 7.3.1.11]{Lhtt} and \cite[Prop. 4.8.1.17]{Lha},
  external tensor product is an equivalence 
\begin{equation*}
\xymatrix{
\t: \Sh_{}(M_1) \otimes \Sh_{}(M_2) \ar[r]^-\sim & 
\Sh_{}(M_1\times M_2)
}
\end{equation*}
We  also use the elementary fact that $\Sh_{}(M_i)$ is dualizable (see \cite[Remark A.2.4]{GKRV}).
 
Since $\Sh_{\Lambda_2}(M_2) \to \Sh(M_2)$ is fully faithful and $\Sh_{\Lambda_1}(M_1)$ is dualizable, the natural map
\begin{equation*}
\xymatrix{
\Sh_{\Lambda_1}(M_1) \otimes \Sh_{\Lambda_2}(M_2) \ar[r] & 
\Sh_{\Lambda_1}(M_1) \otimes \Sh(M_2)
}
\end{equation*}
is fully faithful.
 Indeed, we can rewrite the above map in the form
\begin{equation*}
\xymatrix{
\Fun( \Sh_{\Lambda_1}(M_1)^\vee, \Sh_{\Lambda_2}(M_2)) 
\ar[r] &
\Fun(  \Sh_{\Lambda_1}(M_1)^\vee, \Sh(M_2))
}
\end{equation*}
Similarly, since $\Sh_{\Lambda_1}(M_1) \to \Sh(M_1)$ is fully faithful and $\Sh_{}(M_2)$ is dualizable, the natural map
\begin{equation*}
\xymatrix{
\Sh_{\Lambda_1}(M_1) \otimes \Sh(M_2)
\ar[r] &
\Sh_{}(M_1) \otimes \Sh(M_2)
}
\end{equation*}
is fully faithful.

Thus both sides of the map $\t$ of the lemma admit fully faithful maps to $\Sh_{}(M_1) \otimes \Sh(M_2)$, and $\t$ is compatible with these  fully faithful embeddings, hence $\t$ itself is fully faithful.  It remains to see that it is essentially surjective. Let $\L_{1,\a}$ (resp. $\L_{2,\b}$) be irreducible components of $\L_{1}$ (resp. $\L_{2}$). Let $(x_{\a},\xi_{\a}, f_{\a}, \cF_{f_{\a}})$ and $(x_{\b},\xi_{\b}, f_{\b}, \cF_{f_{\b}})$ by the corresponding Morse data  and corepresenting objects of the Morse functionals as in Lemma~\ref{lem:comp gens}. Then $\cF_{f_{\a}}\bt \cF_{f_{\b}}$ corepresents the Morse functional attached to the function $(x_{1},x_{2})\mapsto f_{\a}(x_{1})+f_{\b}(x_{2})$ based at $((x_{\a},x_{\b}), (\xi_{\a},\xi_{\b}))$, a smooth point of the component $\L_{1,\a}\times \L_{2,\b}\subset\L_{1}\times \L_{2}$. By Lemma~\ref{lem:comp gens}, $\Sh_{\Lambda_1\times \Lambda_2}(M_1\times M_2)$ is compactly generated by these $\cF_{\a}\bt\cF_{\b}$, which are in the image of $\t$. This implies that $\t$ is essentially surjective. 
\end{proof}

\sss{Left adjoint to $*$-restriction}
Suppose $M$ is a manifold and 
$\Lambda\subset T^*M$ is a closed conic Lagrangian.

Suppose $i:M_0\to M$ is a closed submanifold and 
$\Lambda_0\subset T^*M_0$ is a closed conic Lagrangian.

Assume the restriction $i^*:\Sh(M) \to \Sh(M_0)$ takes $\Sh_{\Lambda}(M)$ to $\Sh_{\Lambda_0}(M_0)$.

\begin{lemma}\label{lem:rest adj}
The functor $i^*:\Sh_{\Lambda}(M)\to \Sh_{\Lambda_0}(M_0)$ preserves products, and hence admits a left adjoint. 
\end{lemma}

\begin{proof}
Choose a stratification $\cS$ of $M$ so that (i) $M_0$ is a union of strata $\cS_0 \subset \cS$, (ii) $\Lambda\subset T^*_\cS M$, and (iii) $\Lambda_0\subset T^*_{\cS_0} M_0$. Then we have a commutative diagram
\begin{equation*}
\xymatrix{
\ar[d]_-e\Sh_{\Lambda}(M)\ar[r]^-{i^*} & \Sh_{\Lambda_0}(M_0) \ar[d]_-{e_0}\\
\Sh_{\cS}(M)\ar[r]^-{i^*} & \Sh_{\cS_0}(M_0)\\
}
\end{equation*}
where the vertical arrows are the usual fully faithful inclusions. Moreover, the vertical maps preserve products since products respect the given conditions. 

Then for any set $A$ and $\cF_a \in \Sh_{\Lambda}(M)$ for each  $a \in A$,  we seek to show the natural map $i^*\prod_{a\in A} \cF_a\to \prod_{a\in A} i^*\cF_a$ is an equivalence. It suffices to show its composition $e_0 i^*\prod_{a\in A} \cF_a\to e_0 \prod_{a\in A} i^*\cF_a$ with $e_0$ is an equivalence. This we can rewrite as  
$e_0 i^*\prod_{a\in A} \cF_a \simeq i^* e \prod_{a\in A} \cF_a \simeq i^* \prod_{a\in A} e \cF_a \to  \prod_{a\in A} i^* e\cF_a \simeq 
 \prod_{a\in A} e_0 i^*\cF_a \simeq e_0 \prod_{a\in A}  i^*\cF_a  $ using that $e, e_0$ preserve products.

Now it suffices to show $i^*:\Sh_{\cS}(M)\to \Sh_{\cS_0}(M_0)$ preserves products. For this, we can identify  $\Sh_{\cS}(M)$ and $\Sh_{\cS_0}(M_0)$ with the category of  modules over the respective exit path categories $E_{\cS}(M)$ and  $E_{\cS_0}(M_0)$ (see~\cite[Theorem B.9]{J}). For such modules, products are calculated point-wise. Moreover, for such modules, $i^*$ is the forgetful functor of restriction to a subcategory. Thus it preserves products and we are done.
\end{proof}

%
%
%
%
%
%
%
%
%
%
%

\subsection{Relative real analytic spaces and Lagrangians}\label{ss:A real} 

We will need to consider stacks in the real subanalytic context in the main body of the paper. Instead of developing the general  theory, the following less intrinsic version suffices for this paper: we shall always consider a pair consisting of a complex algebraic stack and a ``relative real analytic space'' over it. Below we give more precise definitions, including the notion of Lagrangians for such stacks.

\sss{Real analytic spaces over stacks}\label{sss:real stacks}
Let $\Sta$ be the category of algebraic stacks locally of finite type over $\CC$ with schematic diagonals. For $X\in \Sta$, let $\Sch_{/X}$ be the category of pairs $(V,v)$ where $V$ is a scheme locally of finite type over $\CC$ and $v:V\to X$ is a map of stacks. Morphisms between  $(V,v)$ and $(V',v')$ in $\Sch_{/X}$ are maps  of schemes $\ph: V\to V'$ together with an isomorphism $v\simeq v'\c \ph$. Let  $\Sch^{sm}_{/X}\subset \Sch_{/X}$ be the (non-full) subcategory of pairs $(V,v)$ where $v$ is  smooth, and morphisms $\ph: (V,v)\to (V',v')$ are also required to be smooth.


Let $\frR$ be the category of real analytic spaces.

A {\em real analytic space $M$ over $X$} is a functor $M: \Sch_{/X}\to \frR$, written $(V,v)\mapsto M_{V,v}$, together with a natural transformation of functors $\xi_{V,v}: M_{V,v}\to V$ (map in $\frR$), such that
\begin{enumerate}
\item The functor $M$ commutes with coproducts (disjoint union).
\item For any map $\ph:  (V,v)\to (V',v')$ in $\Sch_{/X}$, the induced map $\ph_{M}: M_{V,v}\to M_{V',v'}$ and $\xi_{V,v}$ realizes $M_{V,v}$ as the fiber product $M_{V',v'}\times_{V'}V$.
\end{enumerate}
A real analytic space $M$ over $X$ is called {\em smooth} if $M_{V,v}$ is a smooth manifold for all $(V,v)\in\Sch^{sm}_{/X}$. 

A real analytic space $M$ over $X$ is of pure dimension $n$ if for any $(V,v)\in \Sch^{sm}_{/X}$ such that $v$ is of pure relative dimension $\dim_{\CC}(v)$, $\dim_{\RR}M_{V,v}=n+2\dim_{\CC}(v)$.

For $X\in\Sta$, let $\frR_{/X}$ be the category of real analytic spaces over $X$, with real analytic maps over $X$ as morphisms.

\begin{exam}[Global quotient]\label{ex:real quot} Let $M\in\frR_{/X}$. Let $G$ be an algebraic group acting on $X$, and $H\subset G(\CC)$ be a closed subgroup. It makes sense to talk about a $H$-action on $M$ compatible with the $G$-action on $X$: this means for each $(V,v)\in\Sch_{/X}$ with a compatible $G$-action, $M_{V,v}$ carries an action of $H$ making $\xi_{V,v}:M_{V,v}\to V$ equivariant under $H$, and the maps $\ph_{M}: M_{V,v}\to M_{V',v'}$ are $H$-equivariant for all $G$-equivariant morphisms $\ph$ in $\Sch_{/X}$. In such a situation, we can form the quotient $[M/H]\in\frR_{/[X/G]}$ as follows. Indeed, for $(V,v)\in \Sch_{/[X/G]}$, let $(U,u)\in\Sch_{/X}$ be the base change of $(V,v)$ to $X$ which carries a $G$-action making $U\to V$ a $G$-torsor. Define $[M/H]_{V,v}:=M_{U,u}/H$ (by definition $H$ acts on $M_{U,u}$).
\end{exam}

\begin{exam}[Underlying real analytic space]\label{ex:underlying}  Let $M\in\frR_{/X}$, and $\un M\in \frR$. We say that {\em $M$ has underlying space $\un M$} if for every $(V,v)\in \Sch_{/X}$ we are given a map of real analytic spaces $\om_{V,v}: M_{V,v}\to \un M$, compatible with morphisms in $\Sch_{/X}$, such that the following conditions hold:
\begin{enumerate}
\item For any $V,V'\in \Sch_{/X}$, the natural map $M_{V\times_{X}V', (v,v')}\to M_{V,v}\times_{\un M} M_{V',v'}$ is an isomorphism in $\frR$ (here we use $\om_{V,v}$ and $\om_{V',v'}$ to form the fiber product).
\item If $v:V\to X$ is smooth, $\om_{V,v}$ is submersive.
\item If $v:V\to X$ is surjective, so is $\om_{V,v}$.
\end{enumerate}
When $M\in\frR_{/X}$ has underlying space $\un M$, we should think of $M$ as ``$\un M$ equipped with a map to $X$''. This can be made precise but we do not do it here. 
\end{exam}

\sss{Change of bases}\label{sss:real bc}
For any map of stacks $f:X\to Y$, there is a base change functor $f^{\#}: \frR_{/Y}\to \frR_{/X}$ defined by the assignment $N\mapsto (\Sch_{/X}\ni(V,v)\mapsto N_{V, f\c v})$. When there is no confusion, we also use the notations
\begin{equation*}
N\times_{Y}X, \textup{ or } X\times_{Y}N
\end{equation*}
to denote $f^{\#}N$.

For a {\em schematic} map of stacks $f:X\to Y$, there is a forgetful functor $f_{\#}: \frR_{/X}\to \frR_{/X}$ defined by the assignment $M\mapsto (\Sch_{/Y}\ni (V,v)\mapsto M_{V\times_{Y}X, p_{X}})$.
 
\sss{The category $\frR_{\Sta}$}\label{sss:real st cat}
Let $\frR_{\Sta}$ be the category over $\Sta$ whose objects are pairs $(X,M)$ where $X\in\Sta$ and $M\in \frR_{/X}$, and a morphism $(f,f_{M}): (X,M)\to (Y,N)$ in $\frR_{\Sta}$ consists of a pair of maps $f:X\to Y$ and $f_{M}: M\to f^{\#}N$ (a morphism in $\frR_{/X}$). Concretely, for any $(U,u)\in\Sch_{/X}, (V,v)\in \Sch_{/Y}$ and $h:U\to V$ covering $f$, there is an induced map of real analytic spaces $g: M_{U,u}\to N_{V,v}$, functorial in $(U,u)$ and $(V,v)$.

Let $\frR^{\Sch}_{\Sta}$ be the category whose objects are pairs $(X,M)$ where $X\in\Sta$ and $M\in \frR_{/X}$, and a morphism $(f,f_{M}): (X,M)\to (Y,N)$ in $\frR^{\Sch}_{\Sta}$ consists of a {\em schematic} $f:X\to Y$ and $f_{M}: f_{\#}M\to N$ (as a morphism in $\frR_{/Y}$).  Note when $f$ is schematic, maps $f_{\#}M\to N$ in $\frR_{/Y}$ are in canonical bijection with maps $M\to f^{\#}N$ in $\frR_{/X}$. Therefore, $\frR^{\Sch}_{\Sta}$ is a subcategory of $\frR_{\Sta}$.

Let $P$ be a property of a morphism in $\frR$ that is local for submersions, i.e., $f:M\to N$ and $\wt M\to M$ and $\wt N\to N$ are surjective submersions, and $\wt f:\wt M\to \wt N$ covers $f$, then $f$ has property P if and only if $\wt f$ has property P.  Then we can define when a morphism $(f,f_{M}): (X,M)\to (Y,N)$ in $\frR_{\Sta}$ has property $P$: it has property P if for any $(U,u)\in \Sch^{sm}_{/X}$, $(V,v)\in \Sch^{sm}_{/Y}$ and $g:U\to V$ covering $f$, the induced map $g_{M}: M_{U,u}\to N_{U, f\c u} \to N_{V,v} $ as a morphism in $\frR$ has property P.  Using this we can define submersions and surjections in $\frR_{\Sta}$.

\begin{remark} It is possible to define more general morphisms $(X,M)\to (Y,N)$ for a pair of objects in $\frR_{\Sta}$ that do not cover a map $X\to Y$, i.e., those intrinsic to $M$ and $N$. It involves passing through an intermediate object $(X\times Y, \wt M)$ (where $\wt M$ is a lifting of $M$ to an object in $\frR_{/X\times Y}$). We do not elaborate on this point here.
\end{remark}

\sss{Subanalytic subsets}\label{sss:subset}
For a real analytic space $M$ there is the notion of subanalytic subsets of $M$, see \cite{H}. Now suppose $M\in\frR_{/X}$. We define a {\em subset} $M'$ of $M$ to be an assignment $\Sch_{/X}\ni (V,v)\mapsto M'_{V,v}\subset M_{V,v}$ such that for any morphism $\ph:(U,u)\to (V,v)$ in $\Sch_{/X}$, $\ph_{M}^{-1}(M'_{V,v})=M'_{U,u}$.  A subset $M'\subset M$ is called {\em subanalytic} if $M'_{V,v}$ is subanalytic for each $(V,v)\in \Sch_{/X}$.

Let $\frR_{\Sta,\supset}$ be the category of triples $(X,M\supset M')$ where $X\in\Sta$, $M\in\frR_{/X}$ and $M'\subset M$ is a subanalytic subset. Morphisms $(X,M\supset M')\to (Y,N\supset N')$ are morphisms $(X,M)\to (Y,N)$ in $\frR_{\Sta}$ that for any $(U,u)\in\Sch_{/X}, (V,v)\in \Sch_{/Y}$ and map $h:U\to V$ covering $f$, the induced map $g: M_{U,u}\to N_{V,v}$ sends $M'_{U,u}$ into $N'_{V,v}$.

Let $\frR^{\Sch}_{\Sta,\supset}$ be the category with the same objects as $\frR_{\Sta,\supset}$ but morphisms require that $f:X\to Y$  be schematic. 

%
%
%
%
%

\sss{Cones and Lagrangians}\label{sss:stack Lag}
Let $X$  be a stack and $M\in\frR_{/X}$ be smooth.  We define a {\em cone} $\L\subset T^{*}M$ (note $T^{*}M$ is not yet defined) as an assignment of a $\RR_{>0}$-conic subset $\L_{V,v}\subset T^{*}M_{V,v}$ for each $(V,v)\in \Sch^{sm}_{/X}$, such that for any morphism $\ph: (V,v)\to (V',v')$ in $\Sch^{sm}_{/X}$ inducing $\ph_{M}: M_{V,v}\to M_{V',v'}$, we have $\L_{V,v}=\oll{\ph_{M}}(\L_{V',v'})$. Similarly, we say that a cone $\L\subset T^{*}M$ is subanalytic (resp. closed, resp. isotropic, resp. Lagrangian) if $\L_{V,v}$ is subanalytic (resp. closed, resp.  isotropic, resp. Lagrangian) in $T^{*}M_{V,v}$. 

The datum of a  cone $\L\subset T^{*}M$ is determined by the assignment $\L_{V,v}$ for a single smooth surjective $v:V\to X$ from a scheme $V$; the pullback compatibility amounts to saying $\oll{\ph_{1,M}}\L_{V,v}=\oll{\ph_{2,M}}\L_{V,v}\subset T^{*}M_{V\times_{X}V,(v,v)}$ where $\ph_{1,M},\ph_{2,M}$ are the two projections $M_{V\times_{X}V,(v,v)}\to M_{V}$ induced by the projections $V\times_{X}V\to V$.

There is a partial order on the set of cones $\L$ in $T^{*}M$ by containment. Under this partial order there is a unique maximal element which we define as $T^{*}M$ itself.

We say a subanalytic cone $\L\subset T^{*}M$ has pure real dimension $n$ if for any $(V,v)\in \Sch^{sm}_{/X}$ such that $v$ has pure relative complex dimension $\dim_{\CC}(v)$, $\L_{V,v}\subset T^{*}M_{V,v}$ has pure real dimension $n+2\dim_{\CC}(v)$. Note that for any morphism $\ph: (V,v)\to (V',v')$ in $\Sch^{sm}_{/X}$ such that $v$ and $v'$ have pure relative dimension, the induced submersion $\ph_{M}: M_{V,v}\to M_{V',v'}$ has pure relative real dimension $2\dim_{\CC}(v)-2\dim_{\CC}(v')$. Therefore, if $\L_{V',v'}$ has pure real dimension $n+2\dim_{\CC}(v')$,  then $\L_{V,v}=\oll{\ph_{M}}(\L_{V',v'})$ has pure real dimension $n+2\dim_{\CC}(v)$. The converse is true if $\ph$ is surjective.

From this definition of dimension we see immediately:
\begin{lemma}\label{l:Lag dim} Let $X$ be a stack and $M\in\frR_{/X}$ be smooth of pure dimension. Let $\L\subset T^{*}M$ be  a subanalytic isotropic cone. Then $\L$ is a Lagrangian if and only if $\dim_{\RR} \L=\dim_{\RR} M$. 
\end{lemma}
%

\sss{Transport of cones}
Let $(f,f_{M}): (X,M)\to (Y,N)$ be a morphism in $\frR_{\Sta}$, with $M$ and $N$ smooth. Then for cones $\L_{M}\subset T^{*}M$ and $\L_{N}\subset T^{*}N$, one can define the transport $\oll{f_{M}}(\L_{N})\subset T^{*}M$ and $\orr{f_{M}}(\L_{M})\subset T^{*}N$ as follows.


For $\orr{f_{M}}(\L_{M})$, we need to assign to each $(V,v)\in\Sch^{sm}_{/Y}$ a cone $\L_{V,v}\subset T^{*}V$. Let $U\to X\times_{Y}V$ be a smooth surjective map from a scheme $U$ and let $u:U\to X$ be the projection to $X$. Note that $(U,u_{X})\in\Sch^{sm}_{/X}$,  and let $g: M_{U,u}\to N_{V,v}$ be the induced map from $f_{M}$. We then define $\L_{V,v}=\orr{g}(\L_{M, U,u})$. One checks that $\L_{V,v}$ is independent of the choice of $(U,u)$ and compatible under smooth pullbacks (using Lemma \ref{l:Cart Lag mfd}). This defines $\orr{f_{M}}(\L_{M})\subset T^{*}N$. 

To define $\oll{f_{M}}(\L_{N})\subset T^{*}M$, it suffices to define its pullback $\L_{U,u}\subset T^{*}U$ for a smooth surjective $u:U\to X$ for a scheme $U$. As above, we may assume there is a map $h:U\to V$ that covers $f:X\to Y$ where $(V,v)\in \Sch^{sm}_{/Y}$, which induces a map of manifolds $g: M_{U,u}\to N_{V,v}$. Then define $\L_{U,u}=\oll{h}(\L_{N,V,v})$. One checks that $\L_{U,u}$ is independent of the choice of $(V,v,h)$ and its two pullbacks to $M_{U\times_{X}U,(u,u)}$ are the same, hence defining a cone in $T^{*}M$.


Lemma \ref{l:Cart Lag mfd} generalizes to the case of stacks.
\begin{lemma}\label{l:Cart Lag} Consider a Cartesian diagram in the category $\frR_{\Sta}$
\begin{equation*}
\xymatrix{ (X',M') \ar[d]^{(f', f'_{M})}\ar[r]^{(a,a_{M})} & (X,M)\ar[d]^{(f,f_{M})}\\
(Y',N')\ar[r]^{(b,b_{N})} & (Y,N)}
\end{equation*}
where $M,M',N,N'$ are smooth. Then for any cone $\L\subset T^{*}M$ we have an equality of cones in  $T^{*}N'$
\begin{equation}\label{Cart Lag eq}
\oll{b_{N}}\orr{f_{M}}(\L)=\orr{f'_{M}}\oll{a_{M}}(\L)\subset T^{*}N'.
\end{equation}
\end{lemma}
\begin{proof} Choose $(U,u)\in \Sch^{sm}_{X}$, and $(V,v)\in \Sch^{sm}_{/Y}$, $(V',v')\in \Sch^{sm}_{/Y'}$ and maps $\ph:U\to V$, $\b: V'\to V$ covering $f$ and $b$. Then $(U'=U\times_{V}V',u')\in\Sch^{sm}_{/X'}$ and we have a Cartesian diagram of schemes
\begin{equation*}
\xymatrix{ U' \ar[d]^{\ph'}\ar[r]^{\a} & U\ar[d]^{\ph}\\
V'\ar[r]^{\b} & V
}
\end{equation*}
Then we have a Cartesian diagram of manifolds 
\begin{equation}\label{bc mfd}
\xymatrix{M'_{U',u'} \ar[d]^{\ph'_{M}}\ar[r]^{\a_{M}} & M_{U,u}\ar[d]^{\ph_{M}}\\
N'_{V',v'}\ar[r]^{\b_{N}} & N_{V,v}}
\end{equation}
By definition we have
\begin{eqnarray*}
(\oll{b_{N}}\orr{f_{M}}(\L))_{V',v'}=\oll{\b_{N}}\orr{\ph_{M}}(\L_{U,u}),\\
(\orr{f'_{M}}\oll{a_{M}}(\L))_{V',v'}=\orr{\ph'_{M}}\oll{\a_{M}}(\L_{U,u}).
\end{eqnarray*}
Applying Lemma \ref{l:Cart Lag mfd} to the diagram \eqref{bc mfd} and $\L_{U,u}\subset T^{*}M_{U,u}$ we conclude that both sides of \eqref{Cart Lag eq} are equal after transporting to $M_{V',v'}$. Since $(V',v')\in\Sch^{sm}_{/Y'}$ can be chosen such that $v'$ is surjective,  \eqref{Cart Lag eq} holds.
\end{proof}

\begin{lemma}\label{l:pres iso} 
Let $(f,f_{M}):(X,M)\to (Y,N)$ be a morphism in $\frR_{\Sta}$, with $M\in\frR_{/X}$ and $N\in\frR_{/Y}$  smooth. Let $\L_{1}\subset T^{*}M$, $\L_{2}\subset T^{*}N$ be subanalytic isotropic cones. Then
\begin{enumerate}
\item $\oll{f_{M}}(\L_{2})$ is a subanalytic isotropic cone in $T^{*}M$.
\item If $\orr{f_{M}}(\L_{1})$ is subanalytic (e.g., if $M$ and $N$ are semi-algebraic, and $M=f^{\#}N$), then it is also isotropic.
\end{enumerate}
\end{lemma}
\begin{proof}
Base changing along a smooth surjective $v: V\to Y$ by  $(V,v)\in \Sch^{sm}_{/Y}$, we reduce to the case where $Y$ is a smooth scheme and $N$ is a usual real manifold. Let $(U,u)\in \Sch^{sm}_{/X}$. Let $(f\c u, g): (U, M_{U,u})\to (Y, N)$ be the induced morphism.

(1) We have $\oll{f_{M}}(\L_{2})_{U,u}=\oll{g}(\L_{2})\subset T^{*}M_{U,u}$. To show $\oll{f_{M}}(\L_{2})$ is a subanalytic isotropic cone in $T^{*}M$ is equivalent to show the same for  $\oll{f_{M}}(\L_{2})_{U,u}$. Therefore we reduce to show that $\oll{g}(\L_{2})\subset T^{*}M_{U,u}$ is a subanalytic isotropic cone, which is \cite[Proposition 8.3.11(2)]{KS}. 

(2) Suppose $u$ is surjective. Let $(u, u_{M}): (U, M_{U,u})\to (X, M)$ be the morphism in $\frR_{\Sta}$.  Then $\L_{1}=\orr{u_{M}}(\L_{1,U,u})$. Therefore it suffices to show that $\orr{f_{M}}\orr{u_{M}}(\L_{1,U,u})=\orr{g}(\L_{1,U,u})$ is a subanalytic isotropic cone in $T^{*}N$, which is \cite[Proposition 8.3.11(1)]{KS} (assuming $\orr{g}(\L_{1,U,u})$ is subanalytic).

\end{proof}

\begin{lemma}\label{l:trans Lag} Let $M\in\frR_{/X}$ be smooth. 
\begin{enumerate}
\item Let  $i: N\incl M$ be a locally closed subanalytic subset that is also smooth. Then for any Lagrangian $\L_{N}\subset T^{*}N$, $\orr{i}(\L_{N})$ is a Lagrangian in $T^{*}M$.
\item Let $p: N\to M$ be a submersion. Then for any Lagrangian $\L_{M}\subset T^{*}M$, $\oll{p}(\L_{M})$ is a Lagrangian in $T^{*}N$.
\end{enumerate}
\end{lemma}
\begin{proof}
For $(V,v)\in\Sch^{sm}_{/X}$, we need to check the similar statements for $N_{V,v}\to M_{V,v}$ (either locally closed embedding of manifolds or submersion of manifolds) and Lagrangians $\L_{N,V,v}$ and $\L_{M,V,v}$, which are well-known. 
\end{proof}

\subsection{Sheaves on relative real analytic spaces}\label{ss:A stack}
Now we define the category of sheaves on objects in $\frR_{\Sta}$, and generalize results from \S\ref{ss:A mfd} from manifolds to objects in $\frR_{\Sta}$.
 
\sss{Sheaves} 
For a scheme $V$ of finite type over $\CC$, let $\Sh(V)$ be the stable $\QQ$-linear category of all complexes of $\QQ$-sheaves  on the topological space $V(\CC)$ with classical topology.

Following \cite[Appendix A.1]{GKRV}, we define the stable $\QQ$-linear category of sheaves on $X(\CC)$ to be the limit
\begin{equation*}
\Sh(X):=\lim_{(V,v)\in\Sch^{sm}_{/X}}\Sh(V).
\end{equation*}
For a morphism $\ph: (V,v)\to (V',v')$ in $\Sch^{sm}_{/X}$, the  transition functor is $\ph^{!}: \Sh(V')\to \Sh(V)$ .

For a morphism $f:X\to Y$ of stacks, the sheaf functors $f_{*}, f^{*}, f_{!}$ and $f^{!}$ between $\Sh(X)$ and $\Sh(Y)$ are defined and the usual adjunctions hold.

\sss{Sheaves on a relative real analytic space}
Let $X\in\Sta$ be a stack and $M\in\frR_{/X}$. We define
\begin{equation*}
\Sh(M/X)=\lim_{(V,v)\in\Sch^{sm}_{/X}}\Sh(M_{V,v}).
\end{equation*}

\begin{lemma}\label{l:underlying}
Suppose $M\in\frR_{/X}$ has an underlying real analytic space $\un M$, then there is a canonical equivalence $\Sh(\un M)\isom \Sh(M/X)$.
\end{lemma}
\begin{proof}
The $!$-pullbacks along $\om_{V,v}: M_{V,v}\to \un M$ give a functor $\a: \Sh(\un M)\to\Sh(M/X)$. On the other hand, let $v_{0}: V_{0}\to X$ be a smooth surjective cover by a scheme $V_{0}$ and let $V_{n}=(V_{0}/X)^{n+1}$, with map $v_{n}: V_{n}\to X$. Let $M_{n}=M_{V_{n}, v_{n}}$. Then by the definition of the underlying space $\un M$ in Example \ref{ex:underlying},  $M_{n}=(M_{0}/\un M)^{n+1}\in \frR$, and $M_{0}\to \un M$ is surjective and submersive. By usual smooth descent, we have $\Sh(\un M)\isom \lim_{n}\Sh(M_{n})$ via $!$-pullbacks. This gives a functor $\b: \Sh(M/X)=\lim\Sh(M_{V,v})\to \lim_{n}\Sh(M_{n})\simeq \Sh(\un M)$. One checks that $\a$ and $\b$ are inverse to each other.
\end{proof}

\begin{prop}\label{p:real sm des}
\begin{enumerate}
\item Let $M_{0}\to M$ be a surjective and submersive map in $\frR_{/X}$. Let $M_{n}=(M_{0}/M)^{n+1}$ (fiber product in $\frR_{/X}$). Then termwise $!$-pullback induces an equivalence $\Sh(M/X)\isom \lim_{n}\Sh(M_{n}/X)$. In other words, $\Sh(-/X)$ satisfies smooth descent.

\item Let $f:X\to Y$ be a schematic map in $\Sta$, and $M\in \frR_{/X}$. Then there is a canonical equivalence
\begin{equation*}
\Sh(M/X)\isom \Sh(f_{\#}M/Y).
\end{equation*}

\end{enumerate}
\end{prop}
\begin{proof}
(1) Unfold $\lim_{n}\Sh(M_{n}/X)$ into a double limit over $n$ and $(V,v)\in \Sch^{sm}_{/X}$ of $\Sh(M_{n,V,v})$. Note that $M_{n,V,v}=(M_{0,V,v}/M_{V,v})^{n+1}$. By the smooth descent for the covering $M_{0,V,v}\to M_{V,v}$ in $\frR$, we have $\Sh(M_{V,v})\isom \lim_{n}\Sh(M_{n,V,v})$. Taking limit over $(V,v)$ we get $\Sh(M/X)\isom \lim_{n,(V,v)}\Sh(M_{n,V,v})\simeq \lim_{n}\Sh(M_{n}/X)$.

(2) By definition, $\Sh(f_{\#}M/Y)=\lim_{(V,v)\in\Sch^{sm}_{/Y}}\Sh(M_{V\times_{Y}X, p_{X}})$. Therefore we have a canonical functor $\g: \Sh(M/X)\to \Sh(f_{\#}M/Y)$. To check this is an equivalence, we find a submersive surjection $M_{0}\to M$ in $\frR_{/X}$ such that $M_{0}$ has an underlying real analytic space $\un M_{0}$ (eg take $M_{0}=M_{U,u}$ for a smooth cover $u: U\to X$ by a scheme). Form the fiber powers $M_{n}=(M_{0}/M)^{n+1}\in\frR_{/X}$. Similarly $f_{\#}M_{n}=(f_{\#}M_{0}/f_{\#}M)^{n+1}\in\frR_{/Y}$. By (1), the functor $\g$ can be written as the limit of $\g_{n}: \Sh(M_{n}/X)\to \Sh(f_{\#}M_{n}/Y)$. Therefore it suffices to show $\g_{n}$ is an equivalence for $n\ge0$. However, both $M_{n}$ and $f_{\#}M_{n}$ have the same underlying space $\un M_{n}$, and both $\Sh(M_{n}/X)$ and $\Sh(f_{\#}M_{n}/Y)$ are identified with $\Sh(\un M_{n})$ by Lemma \ref{l:underlying} under which $\g$ becomes the identity functor. Therefore $\g$ is an equivalence. 
\end{proof}

By the above Proposition, it makes sense to say that $\Sh(M/X)$ is independent of the base $X$. In the sequel we will simply write $\Sh(M)$ instead of $\Sh(M/X)$.

\sss{Sheaf functors}\label{sss:real sh functors}
Let $(f,f_{M}): (X,M)\to (Y,N)$ be a morphism in $\frR^{\Sch}_{\Sta}$. Recall this means $f:X\to Y$ is schematic, and maps $f_{V,v}: M_{V\times_{Y}X, p_{X}}\to N_{V,v}$, functorial in $(V,v)\in \Sch_{/Y}$. We define $f_{M*}: \Sh(M)\to \Sh(N)$ as the composition
\begin{equation*}
\Sh(M)=\lim_{(U,u)\in \Sch^{sm}_{/X}}\Sh(M_{U,u})\to \lim_{(V,v)\in \Sch^{sm}_{/Y}}\Sh(M_{V\times_{Y}X}, p_{X})\xr{(f_{V,v*})} \lim_{(V,v)\in \Sch^{sm}_{/Y}}\Sh(N_{V,v})=\Sh(N).
\end{equation*}
The maps $f_{V,v*}$ in the middle pass to the limit because of smooth base change.

The assignment $(X,M)\mapsto \Sh(M)$ and $(f,f_{M})\mapsto f_{M*}$ defines a functor $\frR^{\Sch}_{\Sta}\to \St$. 

From the construction we see that $f_{M*}$ preserves limits, hence it admits a left adjoint $f_{M}^{*}: \Sh(N)\to \Sh(M)$.

On the other hand, for a submersion $(f,f_{M}): (X,M)\to (Y,N)$, $f^{!}_{M}: \Sh(N)\to \Sh(M)$ is defined and it admits a left adjoint $f_{M!}: \Sh(M)\to \Sh(N)$.

\sss{Singular support}
Let $\L\subset T^{*}X$ be a closed $\RR_{>0}$-conic Lagrangian in the sense of \S\ref{sss:stack Lag}.  Recall this means an assignment $\Sch^{sm}_{/X}\ni(V,v)\in \L_{V,v}\subset T^{*}V$ (a closed conic Lagrangian) that is compatible with pullbacks along smooth map $\ph: (V,v)\to (V',v')$ in $\Sch^{sm}_{/X}$. Note that in this case $\ph^{!}$ sends $\Sh_{\L_{V',v'}}(V')$ to $\Sh_{\L_{V,v}}(V)$. We define
\begin{equation*}
\Sh_{\L}(X):=\lim_{(V,v)\in \Sch^{sm}_{/X}}\Sh_{\L_{V,v}}(V).
\end{equation*}

Suppose $M\in\frR_{/ X}$  is smooth, and $\L\subset T^{*}M$ is a closed conic Lagrangian, then  we define the limit under $!$-pullbacks
\begin{equation*}
\Sh_{\L}(M/X)=\lim_{(V,v)\in\Sch^{sm}_{/X}}\Sh_{\L_{V,v}}(M_{V,v}),
\end{equation*}
where $M_{V,v}\subset V$ and $\L_{V,v}\subset T^{*}M_{V,v}$ are part of the data defining $M$ and $\L$. Note that for smooth $\ph:(V,v)\to (V',v')$ in $\Sch^{sm}_{/X}$, the induced map $\ph_{M}: M_{V,v}\to M_{V',v'}$ is submersive, and $\oll{\ph_{M}}(\L_{V',v'})=\L_{V,v}$ by definition, hence $\ph_{M}^{!}$ sends $\Sh_{\L_{V',v'}}(M_{V',v'})$ to $\Sh_{\L_{V,v}}(M_{V,v})$.

The analogues of Lemma \ref{l:underlying} and Proposition \ref{p:real sm des} hold for $\Sh_{\L}(M/X)$.
Therefore $\Sh_{\L}(M/X)$ is independent of the base $X$, and we will simply write it as $\Sh_{\L}(M)$.

\subsubsection{Compact generation} 

\begin{lemma}\label{lem:stack comp gens} Let $X\in\Sta$, $M\in\frR_{/X}$ be smooth  and $\L\subset T^{*}M$ be a closed conic Lagrangian. Then $\Sh_{\L}(M/X)$ is compactly generated, hence dualizable.
\end{lemma}
\begin{proof}
For any morphism $\ph: (V,v)\to (V',v')$ in $\Sch^{sm}_{/X}$,  $\ph^{!}_{M}: \Sh_{\L_{V',v'}}(M_{V',v'})\to \Sh_{\L_{V,v}}(M_{V,v})$ differs from $\ph^{*}_{M}$ by a twist, and therefore it preserves both limits and colimits. In particular,  $\ph^{!}_{M}$ admits a left adjoint $\ph^{\L}_{!}: \Sh_{\L_{V,v}}(M_{V,v})\to \Sh_{\L_{V',v'}}(M_{V',v'})$ that takes compact objects to compact objects.  (Here one should be careful to recognize that  $\ph^{\L}_{M!}$ is not in general the usual $!$-pushforward due to the singular support requirement.) We can then rewrite $\Sh_{\L}(M/X)$ as the colimit of $\Sh_{\L_{V,v}}(M_{V,v})$ under the transition functors $\ph^{\L}_{M!}$. Since each $\Sh_{\L_{V,v}}(M_{V,v})$ is compactly generated by Lemma \ref{lem:comp gens},  so is the colimit $\Sh_{\L}(M/X)$.
\end{proof}

For sheaves with prescribed singular support, smooth descent in Proposition \ref{p:real sm des}(1) can be stated in the form of a codescent. More precisely, let $M\in\frR_{/X}$ be smooth and $\L\subset T^{*}M$ be a closed conic Lagrangian.  Let $p:M_{0}\to M$ be a surjective submersion in $\frR_{/X}$. Form the fiber powers $M_{n}=(M_{0}/M)^{n+1}$ in $\frR_{/X}$. Let $\L_{n}\subset T^{*}M_{n}$ be the transport of $\L$. We have a cosimplicial category $\{\Sh_{\L_{n}}(M_{n})\}$ under $!$-pullbacks. The same reasoning as in the proof of Lemma \ref{lem:stack comp gens} shows that $!$-pushforward with singular support makes $\{\Sh_{\L_{n}}(M_{n})\}$ into a simplicial category with cocontinuous functors. 

\begin{lemma}\label{l:descent}
In the above situation, $p^{\L}_{!}: \Sh_{\L_{0}}(M_{0})\to \Sh_{\L}(M)$ ($!$-pushforward with singular support conditions) realizes $\Sh_{\L}(M)$  as the colimit of the simplicial category under $!$-pushforwards with singular support:
\begin{equation*}
\xymatrix{\Sh_{\L_{0}}(M_{0}) & \Sh_{\L_{1}}(M_{1})\ar[l]& \Sh_{\L_{2}}(M_{2})\ar@<.5ex>[l]\ar@<-0.5ex>[l]   & \cdots\ar@<.7ex>[l]\ar@<-0.7ex>[l]\ar[l] }
\end{equation*}
\end{lemma}
\begin{proof}
The same proof of Proposition \ref{p:real sm des}(1) shows that $\Sh_{\L}(M)$ is equivalent to the limit of the cosimplicial category $\{\Sh_{\L_{n}}(M_{n})\}$ under $!$-pullbacks. Passing to left adjoints we get the desired codescent. 
\end{proof}

\subsubsection{Products}\label{sss:A prod}
For $i = 1, 2$, let $X_{i}\in\Sta$ and $M_{i}\in\frR_{/X_{i}}$ be smooth.  Let $\L_{i}\subset T^{*}X_{i}$  a closed conic Lagrangian. The product $M_{1}\times M_{2}$ carries the closed conic Lagrangian $\L_{1}\times \L_{2}  \subset T^{*}M_{1}\times T^{*}M_{2}$.



\begin{lemma}\label{lem:stack ext tens}
External tensor product is an equivalence
\begin{equation*}
\xymatrix{
\Sh_{\Lambda_1}(M_1) \otimes \Sh_{\Lambda_2}(M_2) \ar[r]^-\sim & 
\Sh_{\Lambda_1 \times \Lambda_2}(M_1\times M_2)
}
\end{equation*}
\end{lemma}
\begin{proof} 
Let $(U,u)\in \Sch^{sm}_{/X_{1}}$ with $u$ surjective, and $(V,v)\in \Sch^{sm}_{/X_{2}}$ with $v$ surjective. Let $(U_{n},u_{n})=(U/X_{1})^{n+1}\in\Sch^{sm}_{/X_{1}}$ and $(V_{n},v_{n})=(V/X_{2})^{n+1}\in \Sch^{sm}_{/X_{2}}$. Define $M_{1,n}=M_{1,U_{n},u_{n}}$ and $M_{2,n}=M_{2,V_{n},v_{n}}$, with Lagrangians $\L_{1,n}=\L_{1,U_{n},u_{n}}$ and $\L_{2,n}=\L_{2,V_{n},v_{n}}$. By Lemma \ref{l:descent}, $\Sh_{\L_{1}}(M_{1})$ is equivalence to the colimit of the simplicial category by $!$-pushforwards with singular support
\begin{equation*}
\xymatrix{\Sh_{\L_{0}}(M_{1,0}) & \Sh_{\L_{1}}(M_{1,1})\ar[l]& \Sh_{\L_{2}}(M_{1,2})\ar@<.5ex>[l]\ar@<-0.5ex>[l]   & \cdots\ar@<.7ex>[l]\ar@<-0.7ex>[l]\ar[l] }
\end{equation*}
Similarly, $\Sh_{\L_{2}}(M_{2})$ is equivalence to the colimit of the simplicial category $\{\Sh_{\L_{2,n}}(M_{2,n})\}_{n\ge0}$, and $\Sh_{\Lambda_1 \times \Lambda_2}(M_1 \times M_2) $ is equivalence to the colimit of the simplicial category $\{\Sh_{\L_{1,n} \times \L_{2,n}}(M_{1,n}\times M_{2,n})\}_{n\ge0}$. Tensor product commutes with colimits so the assertion  follows from Lemma~\ref{lem:ext tens} applied term-wise.
\end{proof}



%

\subsubsection{Left adjoint to $*$-restriction}
Let $X\in\Sta$, $M\in\frR_{/X}$ be smooth, and $M_{0}\subset M$ be a smooth closed subspace.  Let $\L\subset T^{*}M$ and $\L_{0}\subset T^{*}M_{0}$ be closed conic Lagrangians. 


\begin{lemma}\label{lem:rest adj st} Let $i: M_{0}\incl M$ be the closed embedding. Assume the restriction $i^*:\Sh(M) \to \Sh(M_0)$ takes $\Sh_{\Lambda}(M)$ to $\Sh_{\Lambda_0}(M_0)$. Then the induced functor $i^*_{\L}:\Sh_{\Lambda}(M)\to \Sh_{\Lambda_0}(M_{0})$ admits a left adjoint. 
\end{lemma}
\begin{proof}
We need to check that $i_{\L}^{*}$ preserves products.  For $\cF_\a\in \Sh_{\Lambda}(M)$ ($\a \in I$),  we seek to show the natural map $i^*_{\L}\prod_{\a\in I} \cF_\a\to \prod_{\a\in I} i^*_{\L}\cF_\a$  is an isomorphism. By the definition of $\Sh_{\L_{0}}(M_{0})$ as a limit, it suffices to see its $!$-pullback along any $(V,v)\in\Sch^{sm}_{/X}$ is an isomorphism. Let $i_{N}: N_{0}:=M_{0,V,v}\incl M_{V,v}=:N$ be the corresponding closed embedding of manifolds, and $\nu_{0}: N_{0}\to M_{0}$, $\nu: N\to M$ be the submersions, and $L_{0}=\L_{0,V,v}\subset T^{*}N_{0}$ and $L=\L_{V,v}\subset T^{*}N$,  we need to show that
\begin{equation}\label{nu0 isom}
\nu_{0}^{!}i^*_{\L}\prod \cF_\a\to \nu^{!}_{0}\prod  i^*_{\L}\cF_\a
\end{equation}
is an isomorphism in $\Sh(N_{0})$. 

Note that $i^{*}_{N}$ does not necessarily send  $\Sh_{L}(N)$ to $\Sh_{L_{0}}(N_{0})$. Let $L'_{0}\subset T^{*}N_{0}$ be a closed conic Lagrangian containing $L_{0}$ such that $i_{N}^{*}$ sends $\Sh_{L}(V)$ to $\Sh_{L'_{0}}(V_{0})$. Let $i_{L}^{*}:\Sh_{L}(N)\to\Sh_{L'_{0}}(N_{0}) $ be the induced functor.

Using that $\nu_{0}^{*}$ and $\nu_{0}^{!}$, $\nu^{*}$ and $\nu^{!}$ differ by a twist, 
\begin{equation}\label{vi1}
\nu_{0}^{!}i^*_{\L}\prod  \cF_\a\simeq i^{*}_{L} \nu^{!}\prod \cF_\a.
\end{equation}
Since $\nu^{!}$ is a right adjoint, it preserves products. By Lemma~\ref{lem:rest adj}, $ i_{L}^{*}$ also preserves products, therefore
\begin{equation}\label{vi2}
i_{L}^{*}\nu^{!}\prod \cF_\a\simeq \prod i_{L}^{*}\nu^{!}\cF_{\a}\simeq \prod \nu_{0}^{!}i_{\L}^{*}\cF_{\a}
\end{equation}
(the last two products are taken in $\Sh_{L'_{0}}(V_{0})$). Since $\nu_{0}^{!}$ is a right adjoint, it preserves products , we have
\begin{equation}\label{vi3}
\prod \nu_{0}^{!}i_{\L}^{*}\cF_{\a}\simeq \nu_{0}^{!}\prod i_{\L}^{*}\cF_{\a}.
\end{equation}
Combining \eqref{vi1}, \eqref{vi2} and \eqref{vi3}, we get that \eqref{nu0 isom} is an isomorphism in $\Sh_{L'_{0}}(V_{0})$. We conclude that the natural map $i^*_{\L}\prod  \cF_\a\to \prod i_{\L}^{*}\cF_{\a}$ is an isomorphism in $\Sh_{\L_{0}}(X_{0})$.
\end{proof}


\subsection{Monodromic gluing}
\sss{Local system on abelian Lie groups}\label{sss:Sh0H}
Let $H$ be an abelian Lie group  with Lie algebra $\frh$. 
Let $\Sh_0(H)$ denote the category of locally constant sheaves on $H$.  Equip $\Sh_0(H)$ with its monoidal convolution product: the monoidal product  is the composition 
\begin{equation*}
\xymatrix{
\Sh_0(H) \otimes \Sh_0(H) \ar[r]^-{\boxtimes} & \Sh_0(H \times H) \ar[rr]^-{m_![\dim_{\RR}H]} && \Sh_0(H)
}
\end{equation*}
where  $m:H\times H\to H$ is the group multiplication. 

Let $\pi_{H}: H_{\univ}\to H$ be the universal cover of $H$ as a Lie group. Let
\begin{equation*}
\cL_{\univ}:=\pi_{H,!}\QQ\in \Sh_{0}(H)
\end{equation*}
be the {\em universal local system} on $H$. The stalk of $\cL_{\univ}$ at the identity $e\in H$ is canonically isomorphic to the group algebra $\QQ[\pi_{1}(H,e)]$. 

Then the monoidal unit of $\Sh_{0}(H)$ is $\cL_{\univ}$ (note the shift by $\dim_{\RR}H$ in the definition of the convolution product). Since $H$ is abelian, the monoidal structure is naturally symmetric.

\subsubsection{Monodromic gluing} We put ourselves in the setting of \S\ref{sss:A prod}. Let $H$ be an algebraic group acting on $X_{1}$ on the right and acting on $X_{2}$ on the left.  Let $M_{i}\in (\frR_{/X_{i}})^{H}$ be an $H$-equivariant object that is smooth, $i=1,2$. Then $[M_{1}/H]\in \frR_{/[X_{1}/H]}$ and $[H\bs M_{2}]\in\frR_{/[H\bs X_{1}]}$. Let $\un\L_{1}\subset T^{*}(M_{1}/H)$ and $\un\L_{2}\subset T^{*}(H\bs M_{2})$ be closed conic Lagrangians. Let $\L_{i}\subset T^{*}M_{i}$ be the pullback of $\un\L_{i}$.

Convolution makes $\Sh_{\Lambda_2}(M_2)$ into a left $\Sh_0(H)$-module: the action map
is the composition
\begin{equation*}
\xymatrix{
\Sh_0(H) \otimes \Sh_{\Lambda_2}(M_2) \ar[r]^-{\boxtimes} & \Sh_{0 \times \Lambda_2}(H \times M_2) \ar[rr]^-{a_{2!}[\dim_{\RR} H]} && \Sh_{\Lambda_2}(M_2)
}
\end{equation*}
Similarly convolution makes $\Sh_{\Lambda_1}(X_1)$ into a right $\Sh_0(H)$-module. 

Consider the object $M_{1}\times^{H}M_{2}\in\frR_{X_{1}\times^{H}X_{2}}$. The product Lagrangian $\un\L_{1}\times\un \L_{2}\subset T^{*}(M_{1}/H)\times T^{*}(H\bs M_{2})$ pulls back to a Lagrangian $\L_{1}\times^{H}\L_{2}\subset T^{*}(M_{1}\times^{H}M_{2})$. 

\begin{lemma}\label{lem:monod gl}
External tensor product induces an equivalence
\begin{equation*}
\xymatrix{
\Sh_{\Lambda_1}(M_1) \otimes_{\Sh_0(H)}  \Sh_{\Lambda_2}(M_2) \ar[r]^-\sim & 
\Sh_{\Lambda_1 \times^H \Lambda_2}(M_1\times^H M_2)
}
\end{equation*}
\end{lemma}

\begin{proof}
On the one hand,  the left hand side is calculated as the colimit of the bar complex
\begin{equation*}
\xymatrix{
\Sh_{\Lambda_1}(M_1) \otimes  \Sh_{\Lambda_2}(M_2) &  \ar@<0.5ex>[l] \ar@<-0.5ex>[l]  
\Sh_{\Lambda_1}(M_1) \otimes \Sh_0(H) \otimes  \Sh_{\Lambda_2}(M_2) &  \ar@<0.75ex>[l] \ar[l]   \ar@<-0.75ex>[l]  \cdots 
}
\end{equation*}
By Lemma~\ref{lem:stack ext tens}, this is equivalent to the diagram with maps given by $!$-pushforwards with singular support
\begin{equation}\label{Hmon left}
\xymatrix{
\Sh_{\Lambda_1 \times \Lambda_2}(M_1 \times M_2) &  \ar@<0.5ex>[l] \ar@<-0.5ex>[l]  
\Sh_{\Lambda_1 \times 0 \times \Lambda_2}(M_1 \times H \times M_2) &  \ar@<0.75ex>[l] \ar[l]   \ar@<-0.75ex>[l]  \cdots 
}
\end{equation}
On the other hand, by Lemma \ref{l:descent}, the right hand side can be calculated as the colimit of $!$-pushforwards with singular support
\begin{equation}\label{Hmon right}
\xymatrix{
\Sh_{\Lambda_1 \times \Lambda_2}(M_1 \times M_2) & \ar@<0.5ex>[l] \ar@<-0.5ex>[l]  
\Sh_{\Lambda_1 \times 0 \times \Lambda_2}(M_1 \times H \times M_2) & \ar@<0.75ex>[l] \ar[l]   \ar@<-0.75ex>[l]   \cdots 
}
\end{equation}
Since \eqref{Hmon left} and \eqref{Hmon right} are the same, we get  a natural equivalence between $\Sh_{\Lambda_1}(M_1) \otimes_{\Sh_0(H)}  \Sh_{\Lambda_2}(M_2) $ and $\Sh_{\Lambda_1 \times^H \Lambda_2}(M_1\times^H M_2)$, which is easily checked to be given by external tensor product. 
\end{proof}

The equivalence in Lemma \ref{lem:monod gl} is functorial with respect to pushforward and pullback. The proof follows the same argument as Lemma \ref{lem:monod gl} so we omit it.

\begin{lemma}\label{l:monod gl func} Suppose $(X_{1}, M_{1}, \L_{1}, X_{2}, M_{2}, \L_{2})$ and $(X'_{1}, M'_{1}, \L'_{1}, X'_{2}, M'_{2}, \L'_{2})$ are two situations as in Lemma \ref{lem:monod gl} (with the same group $H$). Let $(f_{i}, f_{M,i}): (X_{i}, M_{i})\to (X'_{i}, M'_{i})$ be morphisms in $\frR^{\Sch}_{\Sta}$ for $i=1,2$ such that $f_{M,i*}(\Sh_{\L_{i}}(M_{i}))\subset\Sh_{\L'_{i}}(M'_{i})$. Define $(f=(f_{1},f_{2}),f_{M}=(f_{1,M}, f_{2,M})): (X_{1}\twtimes{H}X_{2}, M_{1}\twtimes{H} M_{2})\to (X'_{1}\twtimes{H}X'_{2}, M'_{1}\twtimes{H} M'_{2})$, another morphism in $\frR^{\Sch}_{\Sta}$.  Then $f_{M*}$ sends $\Sh_{\L_{1}\twtimes{H}\L_{2}}(M_{1}\twtimes{H}M_{2})$ to $\Sh_{\L'_{1}\twtimes{H}\L'_{2}}(M'_{1}\twtimes{H}M'_{2})$, and the equivalence in Lemma \ref{lem:monod gl} intertwines $f_{1,M*}\ot f_{2,M*}$ on the left side and $f_{M*}$ on the right. 

The same is true for $f^{*}$ if $f_{M,i}^{*}(\Sh_{\L'_{i}}(M'_{i}))\subset \Sh_{\L_{i}}(M_{i})$, $i=1,2$.
\end{lemma}


\section{Nearby cycles}\label{app:B}
This appendix contains some background material called upon in \S\ref{ss: nilp shvs}.
As in the prior appendix, we will work with sheaves on relative real analytic spaces over complex algebraic stacks,  with singular supports in real subanalytic Lagrangians. Our aim is to record some basic properties of nearby cycles -- estimates on singular support, existence of adjoints -- as well as a key commuting relation over higher dimensional bases following~\cite{N}.

\subsection{Relative singular support estimates}
The material here is expanded upon in~\cite[\S3]{NS}.

\subsubsection{Relative singular support for manifolds}

Let $M, B$ be manifolds, and $\pi:M\to B$ a submersion. 
Consider on $M$ the short exact sequence of vector bundles
\begin{equation*}
\xymatrix{
0 \ar[r] & \pi^*(T^*B) \ar[r] & T^*M \ar[r]^-\Pi & T^*_\pi \ar[r] & 0
}
\end{equation*}
where $ T^*_\pi$ is the relative cotangent bundle, and $\Pi$ is the natural projection.

\begin{defn}\label{def:rel ss}
The {\em relative singular support} of $\cF\in \Sh(M)$ is the closed conic subset 
\begin{equation*}
\xymatrix{
ss_\pi(\cF) = \ol{\Pi(ss(\cF))}  \subset T^*_\pi
}
\end{equation*} 
where $ss(\cF)  \subset T^*M$ is the usual singular support.
\end{defn}

Let $j_A:A\incl B$ be the inclusion of a locally closed submanifold, and $J_A:M_A = M \times_B A \to M$  the corresponding inclusion. For  details of the following, see \cite[\S 3.2]{NS}.

\begin{lemma}\label{lem:rel ss}
For $\cF_A\in \Sh(M_A)$, $\cF\in \Sh(M)$, we have
\begin{equation*}
\xymatrix{
ss_\pi(J_{A*}\cF_A) \subset \ol{ss_\pi(\cF_A)}
&
ss_\pi(J_A^*\cF) \subset ss_\pi(\cF)|_{M_A}
}
\end{equation*}
where the closure $\ol{ss_\pi(\cF_A)}$ is taken inside of $T^*_\pi$.
\end{lemma}

\subsubsection{Relative singular support for stacks}\label{sss:rel sing supp stacks}

We  apply a generalization of Lemma~\ref{lem:rel ss} to real stacks in the proof of Proposition~\ref{p:sh functor from C}(1) and (2). Let us spell this out here and explain how to deduce it from  Lemma~\ref{lem:rel ss}.

Consider the situation
\begin{equation*}
\xymatrix{ X_{A}\ar[d]^{\pi_{A}} \ar@{^{(}->}[r]^{J_{A}} & X_{B}\ar[r]\ar[d]^{\pi_{B}} & X\ar[d]^{\pi}\\
A\ar@{^{(}->}[r]^{j_{A}} & B\ar[r] & S}
\end{equation*}
Here $\pi: X\to S$ is a smooth map of smooth stacks, $A\subset B\in\frR_{/S}$ are smooth and $j_{A}:A\incl B$ is locally closed. Let $X_{A}=\pi^{\#}A, X_{B}=\pi^{\#}B\in \frR_{/X}$.

We have the relative cotangent bundle $T^{*}_{\pi}$ over $X$. It is defined as the relative spectrum of the symmetric algebra on the 0th cohomology of the relative tangent complex of $\pi$. Base change to $B$ and $A$ we get the relative cotangent bundles $T^{*}_{\pi_{B}}$ and $T^{*}_{\pi_{A}}$ over $X_{B}$ and $X_{A}$. 

A subanalytic cone $\L\subset T^{*}_{\pi_{B}}$ is an assignment for each $(V,v)\in \Sch^{sm}_{/X}$ a $\RR_{>0}$-conic subanalytic subset $\L_{V,v}\subset (T^{*}_{\pi\c v})\times_{S}B$ (this is the relative cotangent bundle of the submersion $\t_{B}: V_{B}:=V\times_{S}B\to B$), compatible with smooth pullbacks with respect to morphisms in  $(V,v)\in \Sch^{sm}_{/X}$. 

For a subanalytic cone $\L\subset T^{*}X_{B}$, define $\Pi_{B}(\L)$ to be a subanalytic cone of $T^{*}_{\pi_{B}}$ as follows: for each $(V,v)\in\Sch^{sm}_{/X}$, $\Pi_{B}(\L)_{V,v}$ is the image of $\L_{V,v}$ under the natural projection $T^{*}V_{B}\to (T^{*}_{\pi\c v})\times_{S}B=T^{*}_{\t_{B}}$. 

For $\cF\in\Sh(X_{B})$, we define its relative singular support $ss_{\pi_{B}}(\cF)$ as the closure of $\Pi_{B}(ss(\cF))\subset T^{*}_{\pi_{B}}$. It is closed and conic.

In the above situation, we have a generalization of Lemma~\ref{lem:rel ss}.

\begin{lemma}\label{lem:rel ss stacks}
For $\cF_A\in \Sh(X_A)$, $\cF_{B}\in \Sh(X_{B})$, we have
\begin{equation*}
\xymatrix{
ss_{\pi_{B}}(J_{A*}\cF_A) \subset \ol{ss_{\pi_{A}}(\cF_A)}
&
ss_{\pi_{A}}(J_A^*\cF_{B}) \subset ss_{\pi_{B}}(\cF_{B})|_{X_A}
}
\end{equation*}
where the closure $\ol{ss_{\pi_{A}}(\cF_A)}$ is taken inside of $T^*_{\pi_{B}}$.
\end{lemma}
\begin{proof}
First, for a smooth cover $u:U\to S$ by a scheme $U$, write $u_{B}: \wt B=U\times_{S}B\to B$, $\wt\pi_{B}: \wt X_{B}=X_{B}\times_{S}U\to \wt B$, and similarly define $u_{A}: \wt A\to A$ and $\wt\pi_{A}: \wt X_{A}\to \wt A$. Let $J_{\wt A}: \wt X_A \incl \wt X_{B}$ and $j_{\wt A}: \wt A\incl \wt B$.
Then under the identification
$u^{*}_{B}(T^*_{\pi_{B}}) \simeq T^*_{\wt \pi_{B}}$, standard identities imply
$\oll{u_{B}}(ss_{\pi_{B}}(J_{A*}\cF_A)) = ss_{\wt \pi_{B}}(J_{\wt A*}u^*_{A}\cF_A)$,  $\oll{u_{A}}(\ol{ss_{\pi_{A}}(\cF_A)} )= \ol{ss_{\wt \pi_{A}} (u^*_{A}\cF_A)}$, and similarly $\oll{u_{A}}(ss_{\pi_{A}}(J_A^*\cF_{B})) = ss_{\wt \pi_{A}}(J_{\wt A}^*u^*_{A}\cF_{B})$, $\oll{u_{B}}(ss_{\pi_{B}}(\cF_{B})|_{X_A}) = ss_{\wt \pi_{B}}(u^*_{B}\cF_{B})|_{\wt X_A}$. 
Thus to check the assertions, we may perform such a base-change and so assume $S$ itself is a smooth scheme, hence $B$ and $A$ are manifolds.

Next,  given a smooth cover $v:V\to X$ by a scheme $V$, consider the smooth map of smooth schemes  $\t=\pi\c v: V\to B$. Let $v_{A}: V_{A}=V\times_{S}A\to X_{A}$ and $\t_{A}: V_{A}\to A$ be the restrictions of $v$ and $\t$ over $A$, and similarly define $v_{B}: V_{B}\to X_{B}$, $\t_{B}: V_{B}\to B$ and let $K_{A}: V_{A}\to V_{B}$ be the inclusion. By definition, $ss_{\pi_{B}}(J_{A*}\cF_A)$ pulls back to $ss_{\t_{B}}(v^{*}_{B}J_{A*}\cF_A))\subset T^{*}_{\t_{B}}$. Since $v^{*}_{B}J_{A*}\cF_A=K_{A*}v_{A}^{*}\cF_{A}$, hence $ss_{\pi_{B}}(J_{A*}\cF_A)$ pulls back to $ss_{\t_{B}}(K_{A*}v_{A}^{*}\cF_{A})\subset T^{*}_{\t_{B}}$. Similarly, $\oll{v_{A}}(\ol{ss_{\pi_{A}}(\cF_A)})=\ol{ss_{\t_{A}}(v^{*}_{A}\cF_{A})}\subset T^{*}_{\t_{A}}$, $\oll{v_{A}}(ss_{\pi_{A}}(J_A^*\cF_{B})) = ss_{\tau_{A}}(K_{A}^*v^*_{B}\cF_{B})$, $\oll{v_{B}}(ss_{\pi_{B}}(\cF_{B})|_{X_A}) = ss_{\tau_{B}}(v^*_{B}\cF_{B})|_{ V_A}$. Thus we reduce to the same assertions for the map $\t:V\to S$ and the sheaves $v^{*}_{A}\cF_{A}\in\Sh(V_{A})$ and $v^{*}_{B}\cF_{B}\in\Sh(V_{B})$. This is the case provided by Lemma~\ref{lem:rel ss} and so we are done.
\end{proof}

\begin{cor}\label{c:check rel ss} In the situation of \S\ref{sss:rel sing supp stacks}, let $B=\sqcup_{\a}B_{\a}$ be a finite partition of $B$ into locally closed subanalytic subsets. Let $\pi_{\a}: X_{B_{\a}}\to B_{\a}$ be the base change of $\pi$ to $B_{\a}$, and $J_{\a}: X_{B_{\a}}\incl X_{B}$ be the inclusion. Let $\L\subset T^{*}_{\pi_{B}}$ be a closed subanalytic cone. Then for any $\cF\in \Sh(X_{B})$, $ss_{\pi_{B}}(\cF)\subset \L$ if and only if $ss_{\pi_{\a}}(J_{\a}^{*}\cF)\subset \L|_{X_{B_{\a}}}$.
\end{cor}
\begin{proof}
If $ss_{\pi_{B}}(\cF)\subset \L$, then $ss_{\pi_{\a}}(J_{\a}^{*}\cF)\subset \L|_{X_{B_{\a}}}$ by the second inclusion in Lemma \ref{lem:rel ss stacks}.

Conversely, suppose $ss_{\pi_{\a}}(J_{\a}^{*}\cF)\subset \L|_{X_{B_{\a}}}$ for all $\a$. Write $\cF$ as a successive cone of $J_{\a*}J_{\a}^{*}\cF$, it suffices to show that  $ss_{\pi_{B}}(J_{\a*}J_{\a}^{*}\cF)\subset \L$. By the first inclusion of Lemma \ref{lem:rel ss stacks}, $ss_{\pi_{B}}(J_{\a*}J_{\a}^{*}\cF)\subset \ol{ss_{\pi_{\a}}(J_{\a}^{*}\cF)}$ (closure taken in $T^{*}_{\pi_{B}}$). By assumption, $ss_{\pi_{\a}}(J_{\a}^{*}\cF)\subset \L$ and $\L$ is closed, hence $\ol{ss_{\pi_{\a}}(J_{\a}^{*}\cF)}\subset \L$ and $ss_{\pi_{B}}(J_{\a*}J_{\a}^{*}\cF)\subset \L$.
\end{proof}


\subsection{Compatibility of nearby cycles} Much of the material here is quoted from~\cite{N}.

\subsubsection{Unbiased nearby cycles}\label{sss:unbiased mfd}

Let $[n] = \{0, \ldots, n\}$. Set  $A = \RR^{[n]}$ with coordinates $t_0, \ldots, t_n$.

Let $M$ be a manifold, and $\pi:M\to A$ a submersion. 

For $a\subset [n]$, set $A^{+}_a \subset A$
 to be the positive quadrant where $t_i >0$, for $i\in a$, and $t_i = 0$, for $i\not \in a$. 
 Set $M^{+}_a = M \times_{A} A^{+}_a$.
 
For $a \subset b\subset [n]$, consider the inclusions of subspaces of $M$:
\begin{equation*}
\xymatrix{
 M^{+}_a \ar[r]^-{j_a^{a, b}}  & M^{+}_a \cup M^{+}_b & \ar[l]_-{j_b^{a, b} } M^{+}_b 
}
\end{equation*}

%
%
%
%

\begin{defn}
For $a \subset b\subset [n]$, define the {\em unbiased nearby cycles} by
\begin{equation*}
\xymatrix{
\psi^{b}_a: \Sh(M^{+}_b) \ar[r] & \Sh(M^{+}_a) 
&
\psi^b_a =(j_a^{a, b})^* (j_b^{a, b})_* 
}
\end{equation*} 
\end{defn}

\begin{remark}\label{rem:equiv nc} Suppose $U \subset M$ is any subspace with $M^{+}_a, M^{+}_b \subset U$.
Consider the  the commutative diagram
of inclusions  of subspaces of $M$:
\begin{equation*}
\xymatrix{
& U      & \\
 \ar[ur]^-{j^U_a} M^+_a \ar[r]_-{j_a^{a, b}}  & M^+_a \cup M^+_b \ar[u]_-{j^U_{a, b}} & \ar[l]^-{j_b^{a, b} } M^+_b \ar[ul]_-{j^U_b}
 }
\end{equation*} 
Note the unit of adjunction is an equivalence $ (j^U_{a, b})^* (j^U_{a, b})_* \simeq\id $ (as for any subspaces), and so we have a canonical equivalence
\begin{equation*}
\xymatrix{
\psi^b_a =(j_a^{a, b})^* (j_b^{a, b})_*   \simeq    (j_a^{a, b})^* (j^U_{a, b})^* (j^U_{a, b})_* (j_b^{a, b})_*   
\simeq   (j_a^U)^* (j_b^U)_*
}
\end{equation*} 
Among such subspaces $U$, we have the minimal choice $U = M^+_a \cup M^+_b$  as in the definition of $ \psi^b_a $, and the maximal choice $U = M$.

\end{remark}

Consider on $M$ the sequence of vector bundles
\begin{equation*}
\xymatrix{
0 \ar[r] & \pi^*(T^*A) \ar[r] & T^*M \ar[r]^-\Pi & T^*_\pi \ar[r] & 0
}
\end{equation*}
where $ T^*_\pi$ is the relative cotangent bundle, and $\Pi$ is the natural projection.

As an immediate application of the two identities of Lemma~\ref{lem:rel ss}, we have the following.

\begin{cor}
For $\cF_b\in \Sh(M^+_b)$, we have
\begin{equation*}
\xymatrix{
ss_\pi(\psi^b_a \cF_b) \subset \ol{ss_\pi(\cF_b)}|_{M^+_a}
}
\end{equation*}
where the closure $\ol{ss_\pi(\cF_b)}$ is taken inside of $T^*_\pi$.
\end{cor}

\subsubsection{Base-change map}

Fix $a \subset b  \subset c \subset [n]$.
Then there is a natural  map of functors 
\begin{equation*}
\xymatrix{
 \psi_{a}^{c} \ar[r]  & \psi^{b}_{a} \circ  \psi^{c}_{b}   :\Sh(M^+_{c})\ar[r] &  \Sh(M^+_{a}) 
}
\end{equation*}
Namely, consider  the commutative diagram of inclusions of subspaces of $M$ with Cartesian square:
\begin{equation*}
\xymatrix{
 \ar[d]_-{j_b^{b, c}} M^+_b \ar[r]^-{j_b^{a, b}} & M^+_a \cup M^+_b\ar[d]_-{j_{a, b}^{a, b, c}} & \ar[l]_-{j_a^{a, b}} M^+_a  \ar[dl]^-{j_a^{a, b, c}}  \\
 M^+_b \cup M^+_c \ar[r]^-{j_{b, c}^{a, b, c}} & M^+_a \cup M^+_b \cup M^+_c & \\
 M^+_c\ar[u]^-{j_c^{b, c}} \ar[ur]_-{j_c^{a, b, c}}  & & 
}
\end{equation*}
Following Remark~\ref{rem:equiv nc}, we have a canonical equivalence 
$ \psi_{a}^{c}    \simeq (j_a^{a, b, c})^* (j_c^{a, b, c})_*$, and thus   base-change provides a map 
\beq\label{eq:bc map}
\xymatrix{
 \psi_{a}^{c}    \simeq (j_a^{a, b, c})^* (j_c^{a, b, c})_* 
 \simeq (j_a^{a, b})^*(j_{a, b}^{a, b, c})^*  (j_{b, c}^{a, b, c})_*  (j_c^{b, c})_*    \ar[r] & 
 (j_a^{a, b})^*(j_{b}^{a, b})_*  (j_{b}^{b, c})^*  (j_c^{b, c})_* 
 \simeq
  \psi^{b}_{a} \circ  \psi^{c}_{b} 
}
\eeq

\begin{remark}\label{rem:equiv bc}  Suppose $U, V \subset M$ are any subspaces with $M^+_a, M^+_b \subset U$, 
$M^+_b, M^+_c \subset V$, and $U \cap V = M^+_b$. Then we have a commutative diagram of inclusions of subspaces of $M$ with Cartesian square:
\begin{equation*}
\xymatrix{
 \ar[d]_-{j_b^{V}} M^+_b \ar[r]^-{j_b^{U}} & U \ar[d]_-{j_{U}^{U \cup V}} & \ar[l]_-{j_a^{U}} M^+_a  \ar[dl]^-{j_a^{U \cup V}}  \\
V  \ar[r]^-{j_{V}^{U \cup V}} & U \cup V & \\
 M^+_c\ar[u]^-{j_c^{V}} \ar[ur]_-{j_c^{U \cup V}}  & & 
}
\end{equation*}
Using the identifications of Remark~\ref{rem:equiv nc},   base-change provides a map 
\beq\label{eq:bc map again}
\xymatrix{
 \psi_{a}^{c}    \simeq (j_a^{U \cup V})^* (j_c^{U \cup V})_* 
 \simeq (j_a^{U})^*(j_{U}^{U \cup V})^*  (j_{V}^{U \cup V})_*  (j_c^{V})_*    \ar[r] & 
 (j_a^{U})^*(j_{b}^{U})_*  (j_{b}^{V})^*  (j_c^{V})_* 
 \simeq
  \psi^{b}_{a} \circ  \psi^{c}_{b} 
}
\eeq
 Note we recover  \eqref{eq:bc map} in the minimal case when $U = M^+_a \cup M^+_b$, $V= M^+_b \cup M^+_c$.
It is simple to check the map \eqref{eq:bc map again} is independent of the choice of $U, V$, and in particular is an equivalence if and only if the map \eqref{eq:bc map} is an equivalence.
\end{remark}

More generally, let $a_{\bu}$ be a flag of subsets $a_{0}\subsetneq a_{1}\subsetneq \cdots\subsetneq a_{\ell}\subset [n]$. There is a natural transformation (see \cite[Lemma 3.2.9]{N})
\begin{equation*}
\xymatrix{
 r_{a_{\bu}}: \psi_{a_{0}}^{a_{\ell}} \ar[r]  & \psi(a_{\bu}):=\psi^{a_{1}}_{a_{0}} \circ  \cdots\c\psi^{a_{\ell}}_{a_{\ell-1}}   :\Sh(M^+_{a_{\ell}})\ar[r] &  \Sh(M^+_{a_{0}}) 
}
\end{equation*}

For our applications, we will focus on a flag $a_{\bu}$ with $a_{0}=\vn$ and  $a_{\ell} = [n]$. In this case, 
we will write $A_{>0}=\RR^{[n]}_{>0}$, $M_{>0} = M^{+}_c = M|_{A_{>0}}$, $\pi_{>0}: M_{>0}\to \RR^{[n]}_{>0}$  the restriction of $\pi$, $M^{+}_0 = M_0$. We also write $\psi=\psi^{[n]}_{\vn}: \Sh(M_{>0})\to \Sh(M_{0})$

\begin{theorem}[{\cite[Theorem 4.3.1]{N}}]\label{thm:comm nc}
Suppose $\Lambda \subset T^*M_{>0}$ is a closed conic Lagrangian such that
\begin{enumerate}

\item $\Lambda$ is $\pi_{>0}$-non-characteristic: the intersection 
$$ \Lambda\cap \pi_{>0}^*(T^*A_{>0})
$$ lies in the zero-section of $T^*M_{>0}$.

\item $\Lambda$ is  $\pi$-Thom at the origin: the zero-fiber of the closure of the projection 
 $$
 \xymatrix{
 \ol{\Pi(\Lambda)}|_0  \subset T^*_\pi|_0 \simeq T^*M_0
 }$$ is isotropic.
\end{enumerate}

Then  for $\cF \in \Sh_\Lambda(M_{>0})$, the  map induced by base-change  is an equivalence  in $\Sh(M_{0})$
\begin{equation*}
\xymatrix{
 r_{a_{\bu}}(\cF): \psi(\cF) \ar[r]^-\sim  & \psi(a_{\bu})(\cF). 
}
\end{equation*}
\end{theorem}

\subsubsection{Stacks and real base}\label{sss:notation nc stack}

We apply a generalization of Theorem~\ref{thm:comm nc} to  stacks  in the proof of Proposition~\ref{p:sh functor from C}(2). Let us spell this out here and explain how to deduce it from  Theorem~\ref{thm:comm nc}.

Let $B$ be a smooth scheme over $\CC$ with a quasi-finite map $g: B\to \AA^{[n]}$. Let $B^{+}\subset B(\CC)$ be a subanalytic subset (in the sense defined in \S\ref{sss:subset}) such that $g$ restricts to a homeomorphism onto an open neighborhood of $0$ in $A^{+}=\RR_{\ge0}^{[n]}\subset \AA^{[n]}(\CC)$. Let $g^{+}: B^{+}\to A^{+}$ be the restriction of $g$.  Recall from \S\ref{sss:unbiased mfd} the positive quadrants $A^{+}_{a}\subset A$ for $a\subset [n]$. Set $B_{a}^{+}=(g^{+})^{-1}(A^{+}_{a})\subset B^{+}$.

Let $X\in\Sta$ be an algebraic stack, and $\pi: X\to B$ be a smooth map. For $a\subset [n]$, let $X^{+}_a = X \times_{B} B_{a}^{+}$ be the base change of $\pi$ to $B_{a}^{+}$, viewed as a subanalytic subset of $X$.

%


 

For $a \subset b\subset [n]$, we have the inclusions of subanalytic  subsets of $X$:
\begin{equation*}
\xymatrix{
 X^{+}_a \ar[r]^-{j_a^{a, b}}  & X^{+}_a \cup X^+_b & \ar[l]_-{j_b^{a, b} } X^+_b 
}
\end{equation*}
For $a \subset b\subset [n]$, we have the  unbiased nearby cycles 
\begin{equation*}
\xymatrix{
\psi^{b}_a: \Sh(X^+_b) \ar[r] & \Sh(X^+_a) 
&
\psi^b_a =(j_a^{a, b})^* (j_b^{a, b})_* 
}
\end{equation*} 

For a flag  $a_{\bu}$ of subsets  $a_{0}\subsetneq a_{1}\subsetneq \cdots\subsetneq a_{\ell}\subset [n]$. There is a natural transformation (see \cite[Lemma 3.2.9]{N})
\begin{equation*}
\xymatrix{
 r_{a_{\bu}}: \psi_{a_{0}}^{a_{\ell}} \ar[r]  & \psi(a_{\bu}):=\psi^{a_{1}}_{a_{0}} \circ  \cdots\c\psi^{a_{\ell}}_{a_{\ell-1}}   :\Sh(X^+_{a_{\ell}})\ar[r] &  \Sh(X^+_{a_{0}}) 
}
\end{equation*}

%

Now we  focus on the case $ a_{0}=\vn$ and $a_{\ell}= [n]$. In this case,   
we set $X_{>0} = X^+_{[n]} =  X|_{B_{>0}}$,  $\pi_{>0}: X_{>0}\to B_{>0}$ be the restriction of $\pi$. Let $X_0$ be the fiber of $X$ over the point $B_{\vn}^{+}\subset B^{+}$. Also write $\psi = \psi^{[n]}_{\vn}: \Sh(X_{>0})\to \Sh(X_{0})$.


\begin{theorem}\label{th:comm nc stack}
Suppose $\Lambda \subset T^*X_{>0}$ is a closed conic Lagrangian such that

\begin{enumerate}

\item $\Lambda$ is $\pi_{>0}$-non-characteristic: the intersection 
$$ \Lambda\cap \pi^*_{>0}(T^*B_{>0})
$$ lies in the zero-section of $T^*X_{>0}$. 
\item $\Lambda$ is  $\pi$-Thom at the origin: consider the closure $ \ol{\Pi(\Lambda)}$ of the projection $\Pi(\L)\subset T^{*}_{\pi_{>0}}$ in the relative cotangent $T^*_{\pi_{>0}}$, then its zero-fiber
 $$
 \xymatrix{
 \ol{\Pi(\Lambda)}|_0  \subset T^*_\pi|_{B_{\vn}^{+}} \simeq T^*(X_{0})
 }$$ is isotropic.
\end{enumerate}

Then  for $\cF \in \Sh_\Lambda(X_{>0})$, the  map induced by base-change  is an isomorphism  in $\Sh(X_{0})$
\beq\label{eq:bc map for stacks}
\xymatrix{
 r_{a_{\bu}}(\cF): \psi(\cF) \ar[r]^-\sim  & \psi(a_{\bu})(\cF).
}
\eeq
\end{theorem}

\begin{proof}
To check \eqref{eq:bc map for stacks} is an isomorphism it suffices to check smooth locally. Thus consider $(V,v)\in \Sch^{sm}_{/X}$ and let $V_{>0}, V^{+}_{a}, V_{0}$ be the preimages of $X_{>0},  X^{+}_{a}$ and $X_{0}$ under $v$. Let $v_{>0}: V_{>0}\to X_{>0}$ and $v_{0}: V_{0}\to X_{0}$ be the restrictions of $v$. It suffices to check \eqref{eq:bc map for stacks} is an isomorphism after applying $v_{0}^{*}$ to both sides. Since $\psi$ and $\psi(a_{\bu})$ commute with smooth pullback, we reduce to checking
\begin{equation*}
v_{0}^{*}r_{a_{\bu}}(\cF): \psi_{V}(v^{*}_{>0}\cF)\to \psi_{V}(a_{\bu})(v^{*}_{>0}\cF)
\end{equation*}
is an isomorphism in $\Sh(V_{0})$. Here $\psi_{V}, \psi_{V}(a_{\bu})$ are the counterparts of $\psi$ and $\psi(a_{\bu})$ for the family $\pi\c v: V\to B$.  Moreover, the non-characteristic and Thom hypotheses hold for $(\L, \cF)$ exactly means that they hold for $(\oll{v_{>0}}(\L), v^*_{>0}\cF)$ for any $(V,v)\in \Sch^{sm}_{/X}$. Thus we may reduce to the case of $\pi\c v: V\to B$. This is the case provided by Theorem~\ref{thm:comm nc} and so we are done.
\end{proof}

%
%




\begin{thebibliography}{99}


\bibitem{BN}
D. Ben-Zvi, D. Nadler,
Betti geometric Langlands.
Algebraic geometry: Salt Lake City 2015, 3--41, 
Proc. Sympos. Pure Math., 97.2, Amer. Math. Soc., Providence, RI, 2018.

\bibitem{BN-spec}
D. Ben-Zvi, D. Nadler,
Betti spectral gluing. 
Adv. Math. 380 (2021).

\bibitem{BNP-hecke}
D. Ben-Zvi, D. Nadler, A. Preygel,
A spectral incarnation of affine character sheaves, Compositio Math. 153 (2017), no. 9, 1908--1944.




\bibitem{Be}
R. Bezrukavnikov,
On two geometric realizations of an affine Hecke algebra. 
Publ. Math. Inst. Hautes \'Etudes Sci. 123 (2016), 1--67.

\bibitem{BG}
A. Braverman, D. Gaitsgory,
Geometric Eisenstein series. Invent. Math. 150 (2002), no. 2, 287--384.

\bibitem{CY}
T-H. Chen, A. Yom Din,
A formula for the geometric Jacquet functor and its character sheaf analogue. 
Geom. Funct. Anal. 27 (2017), no. 4, 772--797.


\bibitem{BZF}
E. Frenkel, D. Ben-Zvi, 
Vertex algebras and algebraic curves. Second edition. 
Mathematical Surveys and Monographs, 88. 
American Mathematical Society, Providence, RI, 2004. xiv+400 pp. 

\bibitem{Dr}
V. Drinfeld,
On algebraic spaces with an action of $\Gm$.
arXiv:1308.2604.


\bibitem{GKRV} 
D. Gaitsgory, D. Kazhdan, N. Rozenblyum, Y. Varshavsky,
A toy model for the Drinfeld-Lafforgue shtuka construction, arXiv:1908.05420.



\bibitem{GR}
D. Gaitsgory, N.Rozenblyum, 
A study in derived algebraic geometry. Vol. I. Correspondences and duality. 
Mathematical Surveys and Monographs, 221. American Mathematical Society, Providence, RI, 2017. xl+533pp.



\bibitem{Gies}
D. Gieseker,
A degeneration of the moduli space of stable bundles. 
J. Differential Geom. 19 (1984), no. 1, 173--206.

\bibitem{Gin}
V. Ginzburg, 
The global nilpotent variety is Lagrangian. 
Duke Math. J. 109 (2001), no. 3, 511--519. 

\bibitem{H}
H. Hironaka,
Subanalytic sets. 
Number theory, algebraic geometry and commutative algebra, in honor of Yasuo Akizuki, pp. 453--493. Kinokuniya, Tokyo, 1973.

\bibitem{J}
M. O. Jansen,
The stratified homotopy type of the reductive Borel-Serre compactification,
 arXiv:2012.10777.

\bibitem{KS}
M. Kashiwara, P. Schapira, 
Sheaves on manifolds. With a chapter in French by Christian Houzel. Corrected reprint of the 1990 original. 
Grundlehren der Mathematischen Wissenschaften, 292. Springer-Verlag, Berlin, 1994. x+512 pp.

\bibitem{Lau}
G. Laumon, 
Un analogue global du c\^one nilpotent. 
Duke Math. J. 57 (1988), no. 2, 647--671.

\bibitem{NY-2pt}
Y. Li, D. Nadler, Z. Yun,
Betti geometric Langlands over the cylinder. In preparation. 


\bibitem{Lhtt}
J. Lurie, Higher Topos Theory, Annals of Mathematical Studies, Princeton University Press (2009).

\bibitem{Lha}
J. Lurie, Higher Algebra. Available at {\tt math.ias.edu/\~{}lurie/papers/HA.pdf}.




\bibitem{N}
D. Nadler, A microlocal criterion for commuting nearby cycles,
arXiv:2003.11477.

\bibitem{NS}
D. Nadler and V. Shende, Sheaf quantization in Weinstein symplectic manifolds,
arXiv:2007.10154.


\bibitem{NY}
D. Nadler, Z. Yun,
Spectral action in Betti geometric Langlands. 
Israel J. Math. 232 (2019), no. 1, 299--349.


\bibitem{NY-comp}
D. Nadler, Z. Yun,
Compatibilities of automorphic gluing functor. Preprint.



\bibitem{Solis}
P. Solis,
A complete degeneration of the moduli of G-bundles on a curve, arXiv:1311.6847.

\bibitem{Y-tilt}
Z. Yun,
Weights of mixed tilting sheaves and geometric Ringel duality. Selecta Math. (N.S.) 14 (2009), no. 2, 299--320. 

\end{thebibliography}
\end{document}